\documentclass[a4paper,12pt,reqno]{amsart}

\usepackage[margin=2cm]{geometry}
\usepackage{hyperref}
\usepackage{amsmath,mathtools}
\usepackage{enumerate}
\usepackage{tikz}
\usepackage[all]{xy}
\usepackage{verbatim}
\usepackage{setspace}
\usepackage{enumitem}
\usepackage{mathdots}
\usepackage{longtable}
\usepackage{amssymb}
\usepackage{graphicx}
\usepackage{rotating}
\usetikzlibrary{patterns}
\usetikzlibrary{calc}

\xyoption{rotate}

\DeclareMathOperator{\Hom}{Hom}
\DeclareMathOperator{\uHom}{\underline{Hom}}

\DeclareMathOperator{\End}{End}
\DeclareMathOperator{\uEnd}{\underline{End}}

\DeclareMathOperator{\Ext}{Ext}

\DeclareMathOperator{\im}{Im}
\DeclareMathOperator{\Ker}{Ker}

\DeclareMathOperator{\tp}{top}
\DeclareMathOperator{\rad}{rad}
\DeclareMathOperator{\soc}{soc}

\DeclareMathOperator{\mod*}{mod}

\DeclareMathOperator{\id}{id}

\DeclareMathOperator{\proj}{proj}
\DeclareMathOperator{\inj}{inj}

\DeclareMathOperator{\fin}{fin}
\DeclareMathOperator{\ufin}{\underline{fin}}
\DeclareMathOperator{\Aut}{Aut}

\newcommand{\inv}{^{-1}}
\newcommand{\wh}{\widehat}
\newcommand{\wt}{\widetilde}

\newcommand{\ps}[1]{^{(#1)}}
\newcommand{\psp}[1]{^{(#1)'}}

\newcommand{\dynA}{{A}}

\newcommand{\dynC}{{C}}

\newcommand{\dynafA}{\widetilde{{A}}}

\newcommand{\dynafC}{\widetilde{{C}}}

\newcommand{\stp}{\varepsilon}
\newcommand{\crs}{\eta}
\newcommand{\orel}{Z\ps{1}}
\newcommand{\crel}{Z\ps{2}}
\newcommand{\PQ}{\mathcal{P}_Q}
\newcommand{\ov}{Q_{0}\ps{1}}
\newcommand{\cv}{Q_{0}\ps{2}}
\newcommand{\oa}{Q_{1}\ps{1}}
\newcommand{\ca}{Q_{1}\ps{2}}
\newcommand{\unfQ}{\wh{Q}}
\newcommand{\unfZ}{\wh{Z}}
\newcommand{\unfA}{\wh{A}}

\newcommand{\wrd}{\mathcal{W}}
\newcommand{\str}{\mathcal{S}}
\newcommand{\bnd}{\mathcal{B}}
\newcommand{\fw}{\kappa}

\newcommand{\fiso}{\Lambda}
\newcommand{\viso}{\fiso_v}
\newcommand{\finite}{\mathbb{F}}
\newcommand{\rational}{\mathbb{Q}}
\newcommand{\complex}{\mathbb{C}}
\newcommand{\iu}{\mathrm{i}\mkern1mu}
\newcommand{\real}{\mathbb{R}}
\newcommand{\integer}{\mathbb{Z}}
\newcommand{\spint}{\mathbb{Z}_{>0}}

\newcommand{\tikznode}[2]{
	\ifmmode
		\tikz[remember picture,baseline=(#1.base),inner sep=0pt] \node (#1) {$#2$};
	\else
		\tikz[remember picture,baseline=(#1.base),inner sep=0pt] \node (#1) {#2};
	\fi
}

\theoremstyle{plain}
\newtheorem{thm}{Theorem}[section]
\newtheorem*{thm*}{Theorem}
\newtheorem{lem}[thm]{Lemma}
\newtheorem*{lem*}{Lemma}
\newtheorem{cor}[thm]{Corollary}
\newtheorem*{cor*}{Corollary}
\newtheorem{prop}[thm]{Proposition}
\newtheorem*{prop*}{Proposition}

\theoremstyle{definition}
\newtheorem{defn}[thm]{Definition}

\newtheorem{exam}[thm]{Example}

\newtheorem*{exam*}{Example}

\theoremstyle{remark}
\newtheorem{rem}[thm]{Remark}

\numberwithin{equation}{section}
\numberwithin{figure}{section}
\allowdisplaybreaks

\begin{document}
	\title[Folded Gentle Algebras]{Folded Gentle Algebras}
	\author{Drew Damien Duffield}
	\maketitle
	\begin{abstract}
	    We use folding techniques to define a new class of gentle-like algebras that generalise the iterated tilted algebras of type $C$ and $\widetilde{C}$, which we call folded gentle algebras. We show that folded gentle algebras satisfy many of the same properties of gentle algebras, and that the proof of these properties follows directly from folding arguments. As a subclass of clannish algebras (with irreducible quadratic relations), we show that the classification of indecomposable modules (in terms of symmetric and asymmetric strings and bands) can be recovered from folding techniques, and that this proof technique provides further explicit detail on the classification of band modules. In particular, our paper includes a classification of indecomposable modules over the algebra $K\langle x,y\rangle/ \langle p(x),q(y) \rangle$, where $p$ and $q$ are monic, irreducible, quadratic polynomials over $K$. In addition, we classify the Auslander-Reiten sequences of folded gentle algebras, showing that irreducible morphisms between string modules are given by adding/deleting hooks and cohooks to/from strings. Finally, we show that the class of folded gentle algebras is closed under derived equivalence. 
	\end{abstract}
	\tableofcontents
	
	\section{Introduction}
	Since their inception, the class of \emph{gentle algebras} have proven to be incredibly influential to mathematics. Their origins go back to the theory of iterated tilted algebras of Dynkin/Euclidean types $A$ and $\widetilde{A}$ (\cite{AssemA,AssemHappel,ASIteratedTilted}), for which gentle algebras are a natural generalisation. As combinatorially interesting algebras in their own right, gentle algebras received much initial study. They are a subclass of \emph{special biserial algebras} (also incredibly well-studied), and as such, are a class of \emph{tame} algebras (in the sense of the tame-wild dichotomy of Drozd in \cite{Drozd}). In particular, the indecomposable modules over special biserial algebras (and hence also gentle algebras) have been classified. For gentle algebras, the indecomposable modules are given by string and band modules (\cite{WaldWasch}). The \emph{Auslander-Reiten quiver} of the module category of a gentle algebra is also well-understood. For example, the \emph{Auslander-Reiten sequences} of gentle algebras and special biserial algebras have been classified (\cite{ButlerRingel,WaschSkow}). All of this combines together to give a very concrete understanding of the module category of gentle algebras.
	
	What makes gentle algebras a particularly useful subclass of biserial algebras, however, is that we also have a good understanding of their \emph{bounded derived category}. For example, they are known to be \emph{derived-tame} and there is a classification of the indecomposable objects in the bounded derived category of a gentle algebra in \cite{BekkertGentle}. In addition, a basis of the morphisms between indecomposable objects in the bounded derived category is given in \cite{Arnesen} and the mapping cones of these morphisms is considered in \cite{Canakci}. Above all, one of most striking results in the theory of gentle algebras is that they are closed under derived equivalence \cite{SchroerZim}. There are very few classes of algebras that satisfy this property, making them particularly important.
	
	Aside from these very nice properties, there has been a resurgence of interest in gentle algebras due to their natural appearance in numerous other areas of mathematics. For example, they appear in cluster theory as $m$-cluster tilted algebras, $m$-Calabi-Yau tilted algebras (\cite{Elsener}), and as Jacobian algebras associated to triangulations of surfaces with marked points (\cite{ABCP}). They also appear in areas of geometry and mathematical physics such as with Dimer models (\cite{Broomhead}), enveloping algebras of Lie algebras (\cite{Huerfano}), and partially wrapped Fukaya categories (\cite{Opper}).
	
	There is another class of algebras that is closely related to the class of gentle algebras, called \emph{skewed-gentle algebras}. They were originally defined by Geiss and de la Pe\~{n}a in \cite{SkewedGentle}, and are a subclass of clannish algebras (defined in \cite{CBClans}). These algebras have many properties in common with gentle algebras and are just as influential, also appearing in cluster theory, geometry and mathematical physics. For example, they are closely related to iterated tilted algebras of Dynkin/Euclidean types $D$ and $\widetilde{D}$. There is a classification of their indecomposable modules in terms of symmetric and asymmetric string and band modules (\cite{CBClans}), the Auslander-Reiten sequences of the module category have been classified (\cite{GeissClanMaps}), and the indecomposable objects in the bounded derived category of a skewed-gentle algebra have been classified (\cite{BekkertSkewedGentle}). However, unlike gentle algebras, skewed-gentle algebras are not closed under derived equivalence.
	
	The classes of gentle and skewed-gentle algebras thus form generalisations of algebras associated to the two simply-laced Dynkin/Euclidean types that consist of infinite families ($A$, $\widetilde{A}$, $D$ and $\widetilde{D}$). However, this begs the question of whether similar gentle-like classes of algebras exist for the other infinite families of Dynkin/Euclidean diagrams, such as $B$, $\widetilde{B}$, $C$ and $\widetilde{C}$. Algebras of finite/tame representation type indeed exist for these families of diagrams. They are given by the species of the appropriate type, described in \cite{DRIndec}. Moreover, there exist iterated tilted algebras based on these types, described in \cite{AssemC}.
	
	In this paper, we consider the type $C$ and $\widetilde{C}$ case, and indeed show that such a gentle-like class of algebras exists for these types. Not only do we define this class of algebras, which we call \emph{folded gentle algebras}, we also show that they satisfy many of the same properties as gentle algebras, including closure under derived equivalence. The algebras we define are essentially gentle algebras with the addition of certain loops (which we call \emph{crease loops}) that satisfy irreducible quadratic relations over the ground field $K$. Consequently, the ground field $K$ in this paper is never algebraically closed, which is in contrast to the setting of gentle algebras. (That said, these algebras can contain simple modules isomorphic to $\complex$ whenever $K=\real$.) Folded gentle algebras actually turn out to be a special subclass of clannish algebras in the original sense of \cite{CBClans}. Specifically, they arise from \emph{irreducible clans}. On this point, it is worth remarking that this is the key difference between folded gentle algebras and skewed-gentle algebras --- skewed gentle algebras are defined using special loops that satisfy quadratic relations that are reducible over every field, whereas folded gentle algebras are defined using loops with quadratic relations that must be irreducible over the ground field. This leads to the key difference in representation-theoretic behaviour (from Dynkin/Euclidean type $D$ or $\widetilde{D}$ to type $C$ or $\widetilde{C}$). It is also worth remarking that the indecomposable modules of clannish algebras which contain irreducible quadratic relations were recently classified in \cite{SemilinearClan}.
	
	The main inspiration behind the name of the algebras central to this paper, their definition, and the proof of their properties, is a technique called a  \emph{(crystallographic) folding}. Beyond exploring the specific theory and properties of folded gentle algebras, one of the primary aims of this paper is to showcase just how useful and powerful folding arguments can be to representation theory. Crystallographic foldings arise from group symmetry on simply-laced Dynkin or Euclidean diagrams. For example, Dynkin diagrams of type $A_{2n+1}$ have a group action of $\integer_2$ which fixes the central vertex and reflects all other edges and vertices about the central vertex. Dynkin diagrams of type $D_n$ also have a group action of $\integer_2$ which reflect the vertices and edges at the `fork' of the diagram. When one takes the quotient with respect to this group action (which is called folding), one obtains a non-simply-laced Dynkin/Euclidean diagram. Specifically, there are crystallographic foldings $A_{2n-1} \rightarrow C_{n}$, $\widetilde{A}_{2n-1} \rightarrow \widetilde{C}_{n}$, $D_{n+1} \rightarrow B_{n}$ and $\widetilde{D}_{n+1} \rightarrow \widetilde{B}_{n}$ (amongst others). This folding manifests itself in other ways: For each folding, there is an embedding of the corresponding Coxeter groups and there is a projection of the corresponding root systems (see \cite{FockGoncharov06} and \cite[Chapter 11]{SteinbergLectures} for details). Closely related to this paper is work that explores foldings in the context of the representation theory of finite-dimensional algebras, such as in \cite{DengDuI,DengDuII,Tanisaki}.
	
	This paper takes a foundational approach and does not necessarily require that the reader has specialist knowledge, although an understanding of the basics of representation theory is beneficial. In Section~\ref{sec:Defn}, we begin with outlining the required preliminaries to understand the earlier topics covered in this paper. This then leads us to the definition of folded gentle algebras and some examples. All folded gentle algebras in this paper are assumed to be finite-dimensional, but we make no assumptions on the ground field other than that it must admit a field extension of degree 2.
	
	In Section~\ref{sec:BasicProps}, we outline some basic representation-theoretic properties of folded gentle algebras. This includes the first result of the paper, which is that folded gentle algebras are examples of \emph{biserial algebras} (in the sense of \cite{WaldWasch,WaschSkow}) over a field that is not algebraically closed.
	\begin{thm*}[Theorem~\ref{thm:Biserial}]
		Let $A$ be a folded gentle algebra. Then $A$ is a biserial algebra.
	\end{thm*}
	
	In Section~\ref{sec:Strings}, we outline some preliminary notions from various papers (such as \cite{SemilinearClan}\cite{ButlerRingel}, \cite{CBClans}  and \cite{Deng}) on \emph{words, strings, bands} and their associated modules. We also adapt these notions to the setting of folded gentle algebras, as the crease loops require additional considerations to be made. In particular, words, strings and bands of folded gentle algebras contain \emph{crease symbols}, which can result in symmetric structures and additional equivalences to consider. We then precisely define the classes of modules we wish to consider in the paper, which are \emph{symmetric} and \emph{asymmetric} \emph{string modules} and \emph{band modules}. The are essentially the same as the definitions given in \cite{SemilinearClan}, except for some differences in notation: we define the module structures combinatorially in terms of word concatenations. The reason for the slight difference in notation is that it better suits the combinatorial proofs of results later in the paper.
	
	In Section~\ref{sec:UnfoldedGentle}, we introduce both \emph{folding} and \emph{unfolding} in the context of gentle and folded gentle algebras. Specifically, we show that for any folded gentle algebra, there is a corresponding gentle algebra that is related to the folded gentle algebra via an \emph{unfolding procedure}. We call the corresponding gentle algebra an \emph{unfolded gentle algebra}. We then show that this induces the existence of a group action of $\integer_2$ on both the quiver and the algebra of the unfolded gentle algebra. Taking the quotient of the quiver of the unfolded gentle algebra with respect to this group action (a.k.a folding) takes one back to the folded gentle algebra. At the end of this section, we explore the technical implications of this on the combinatorics of words, strings and bands in folded gentle algebras and unfolded gentle algebras, and how they are related to each other. In particular, we show the following, which is essentially Definitions~\ref{def:Z2Words}, \ref{def:FoldedStrings} and \ref{def:FoldedBands} and Lemmata~\ref{lem:FoldedStrings}, \ref{lem:FoldedBands}, \ref{lem:OddParity} and \ref{lem:FoldedSymmetric}.
	
	\begin{lem*}
		Let $A$ be a finite-dimensional folded gentle algebra and let $\wh A$ be the unfolded gentle algebra corresponding to $A$ via the unfolding procedure. Let $\integer_2=\{1,g\}$ be the group that acts on $\wh A$. For any gentle or folded gentle algebra $A'$, denote by $\mathcal{S}_{A'}$ and $\mathcal{B}_{A'}$ the classes of all strings and bands of $A'$ respectively, up to the equivalences $\approx$ defined in Section~\ref{sec:Strings}. Then the following hold.
		\begin{enumerate}[label=(\alph*)]
			\item $\integer_2$ acts on $\str_{\wh A}$ and $\bnd_{\wh A}$. 
			\item There is a bijection $\str_A \cong \str_{\wh A} / \integer_2$. In particular, there exists a surjection $\rho\colon \str_{\wh A} \rightarrow \str_A$ such that
			\begin{enumerate}[label=(\roman*)]
				\item $\rho(\wh w) \approx \rho(\wh w')$ if and only if $\wh w' \approx \wh w$ or $\wh w' \approx g\wh w$; and
				\item $\rho(\wh w) \in \str_A$ is symmetric if and only if $\wh w \in \str_{\wh A}$ is $\integer_2$-invariant (that is, $g \wh w \approx \wh w$).
			\end{enumerate}
			\item There is a bijection $\bnd_A \cong \bnd_{\wh A} / \integer_2$. In particular, there exists a surjection $\rho\colon \bnd_{\wh A} \rightarrow \bnd_A$ such that
			\begin{enumerate}[label=(\roman*)]
				\item $\rho(\wh w) \approx \rho(\wh w')$ if and only if $\wh w' \approx \wh w$ or $\wh w' \approx g\wh w$;
				\item $\rho(\wh w) \in \bnd_A$ has an odd number of crease symbols (and is hence asymmetric) if and only if $\wh w \approx (g\wh w')\wh w'$ for some subword $\wh w'$ of $\wh w$ (and $\wh w$ is hence $\integer_2$-invariant);
				\item $\rho(\wh w) \in \bnd_A$ is asymmetric with an even number of crease symbols if and only if $\wh w$ is not $\integer_2$-invariant; and
				\item $\rho(\wh w) \in \bnd_A$ is symmetric if and only if $\wh w \approx  (g\wh w')\inv \wh w'$ for some subword $\wh w'$ of $\wh w$ (and $\wh w$ is hence $\integer_2$-invariant).
			\end{enumerate}
		\end{enumerate}
	\end{lem*}
	
	Section~\ref{sec:FoldedModule} leads towards some of our most important results. In this section, we use (un)folding techniques to give a complete description of the module category of a folded gentle algebra by using known results on the module category of its corresponding unfolded gentle algebra. We define a functor $U$ (called the \emph{unfolding functor}) which describes how to `unfold' the module category of a folded gentle algebra into the module category of its corresponding (unfolded) gentle algebra. We then give a very explicit correspondence between the string and band modules of the folded gentle algebra and the string and band modules of the unfolded gentle algebra, which is detailed in Propositions~\ref{prop:UStrings}, \ref{prop:UEvenBands}, \ref{prop:UOddBands} and \ref{prop:USymBand}. 
	
	\begin{prop*}[Proposition~\ref{prop:UStrings}]
		Let $M \in \mod*A$ and let $\wh w \in \str_{\wh A}$. Then the following are equivalent.
		\begin{enumerate}[label=(\alph*)]
			\item $U(M)$ contains a direct summand isomorphic to the string module associated to $\wh w$.
			\item $U(M)$ contains a direct summand isomorphic to a direct sum of two string modules: one associated to $\wh w$ and the other to $g\wh w$.
			\item $M$ contains a direct summand isomorphic to the string module associated to $\rho(\wh w)$.
		\end{enumerate}
	\end{prop*}
	
	\begin{prop*}[Proposition~\ref{prop:UEvenBands}]
		Let $M \in \mod*A$ and let $\wh w \in \bnd_{\wh A}$ such that $\wh w$ is not $\integer_2$-invariant. Then the following are equivalent.
		\begin{enumerate}[label=(\alph*)]
			\item $U(M)$ contains a direct summand isomorphic to a band module associated to $\wh w$.
			\item $U(M)$ contains a direct summand isomorphic to a direct sum of two band modules: one associated to $\wh w$ and the other to $g\wh w$.
			\item $M$ contains a direct summand isomorphic to a band module associated to $\rho(\wh w)$.
		\end{enumerate}
	\end{prop*}
	
	\begin{prop*}[Propositions~\ref{prop:UOddBands}, \ref{prop:USymBand}]
		Let $M \in \mod*A$ and let $\wh w \in \bnd_{\wh A}$ such that $\wh w$ is $\integer_2$-invariant. Then the following are equivalent.
		\begin{enumerate}[label=(\alph*)]
			\item $U(M)$ contains a direct summand isomorphic to a band module associated to $\wh w$.
			\item $U(M)$ contains a direct summand isomorphic to a band module associated to $g\wh w$.
			\item $M$ contains a direct summand isomorphic to a band module associated to $\rho(\wh w)$.
		\end{enumerate}
	\end{prop*}
	
	It is worth remarking here that in the Proposition concerning the $\integer_2$-invariant case, the direct summand in (b) is sometimes distinct from the direct summand in (a), and sometimes it is not distinct. Precisely when this happens for a given band module is dependent on the $K$-vector space automorphism that describes the band module. It is also worth remarking that the precise statements of Propositions~\ref{prop:UEvenBands}, \ref{prop:UOddBands} and \ref{prop:USymBand} are even more specific --- we describe this correspondence precisely in terms of each band module associated to $\wh w \in \bnd_{\wh A}$ (of which there are infinitely many). 
	
	As well as showcasing the interesting correspondence of indecomposable modules between folded gentle algebras and gentle algebras, these propositions are useful for several other results in the paper. They can also be used to recover the classification of the iso-classes of indecomposable modules of folded gentle algebras (without repetitions) using folding techniques. Whilst the indecomposable modules have already been classified due to \cite{SemilinearClan}, we have kept the classification of indecomposables via folding as a proof of concept, as it showcases a new classification technique that is distinct from matrix reduction and functorial filtration. We also remark that the use of folding techniques in this paper results in an explicit classification of the symmetric band modules in terms of the $\integer_2$-orbits of certain classes of powers of irreducible polynomials. Specifically, we define three collections of iso-classes of indecomposable modules, denoted by $\mathcal{M}_1$ (Definition~\ref{defn:IndM1}), $\mathcal{M}_2$ (Definition~\ref{def:AsymBandsClass}) and $\mathcal{M}_3$ (Definition~\ref{def:SymBandInds}). The modules in $\mathcal{M}_1$ are the collection of all string modules, up to an equivalence previously defined in Section~\ref{sec:Strings}. The modules in $\mathcal{M}_2$ are asymmetric band modules, again up to an equivalence of bands defined in Section~\ref{sec:Strings}, and further indexed by a set 
	\begin{equation*}
		\Pi = \{p^n \in K[x] : n \in \spint \text{ and } p \text{ is monic, irreducible, and } p(0) \neq 0\}.
	\end{equation*}
	The modules in $\mathcal{M}_3$ are symmetric band modules, again up to an equivalence of bands, and further indexed by a set $\Pi_w = \Pi / \integer_2$, where the action of $\integer_2$ on $\Pi$ is determined by each band $w$.
	\begin{thm*}[Theorem~\ref{thm:IndClassification}]
		Let $A$ be a folded gentle algebra. The combined collection of all objects in $\mathcal{M}_1 \cup \mathcal{M}_2 \cup \mathcal{M}_3$ forms a complete collection of all iso-classes of indecomposable objects of $\mod*A$ without repetitions.
	\end{thm*}
	As a consequence of the classification, we automatically recover the following further classification.
	\begin{cor*}[Corollary~\ref{cor:Kxy}]
		Let $K$ be a field and let $A = K\langle x,y \rangle / \langle p(x), q(y) \rangle$ be such that
		\begin{align*}
			p(x) &= x^2 - \lambda_1 x - \lambda_2, \\
			q(x) &= y^2 - \lambda'_1 y - \lambda'_2
		\end{align*}
		are both irreducible over $K$. Let
		\begin{equation*}
			\Pi = \{p^n \in K[x] : n \in \spint \text{ and } p \text{ is monic, irreducible, and } p(0) \neq 0\}
		\end{equation*}
		and let $\integer_2 = \{1,g\}$ be the group of order 2 with $g^2 = 1$. Define an action of $\integer_2$ on $\Pi$ by defining, for each $p \in \Pi$, the polynomial $gp$ as the characteristic polynomial of the matrix $\lambda_2\lambda'_2\phi_p\inv$, where $\phi_p$ is the companion matrix of $p$. Then the isomorphism classes of indecomposable objects in the category $\fin A$ are indexed by the set $\Pi/\integer_2$.
	\end{cor*}
	
	The final main result of Section~\ref{sec:FoldedModule} is the classification of the Auslander-Reiten sequences of a folded gentle algebra. In particular, the band modules belong to homogeneous tubes. Moreover, the irreducible morphisms between string modules are given combinatorially by adding/deleting hooks and/or cohooks to strings, as is the case for gentle algebras (c.f. \cite{ButlerRingel}).
	\begin{thm*}
		Let $A$ be a folded gentle algebra. Then the Auslander-Reiten sequences of $A$ are precisely those given in Theorem~\ref{thm:ARSeq}.
	\end{thm*}
	
	We also show that there is a $\integer_2$-action on the Auslander-Reiten quiver of an unfolded gentle algebra $\wh A$, and that the Auslander-Reiten quiver of the folded gentle algebra associated to $\wh A$ is given by the quotient with respect to this group action.
	
	\begin{cor*}[Corollary~\ref{cor:ARZ2}]
		Let $A$ be a folded gentle algebra and $\wh A$ be the unfolded gentle algebra of $A$ corresponding to $A$ via the unfolding procedure. Let $\Gamma_A$ be the valued Auslander-Reiten quiver of $A$, let $\Gamma'_A$ be the underlying unvalued translation quiver of $\Gamma_A$, and let $\Gamma_{\wh A}$ be the Auslander-Reiten quiver of $\wh A$. 
		\begin{enumerate}[label = (\alph*)]
			\item There exists a group action of $\integer_2$ on $\Gamma_{\wh A}$ such that $\Gamma'_{A} \cong \Gamma_{\wh A} / \integer_2$.
			\item Let $\wh M$ be a vertex of $\Gamma_{\wh A}$ that is fixed under the action of $\integer_2$ and which corresponds to a string module. Then:
			\begin{enumerate}[label=(\roman*)]
				\item $[\wh M] \in \Gamma_A$ corresponds to the string module of a symmetric string;
				\item the arrow of source $[\wh M] \in \Gamma_A$ has a valuation of $(1,2)$; and
				\item the arrow of target $[\wh M] \in \Gamma_A$ has a valuation of $(2,1)$;
			\end{enumerate}
		\end{enumerate}
	\end{cor*}
	
	In Section~\ref{sec:DerEquiv}, we provide some details on the repetitive algebra for folded gentle algebras. This includes an explicit description of the quiver and relations of the repetitive algebra of a folded gentle algebra (Lemma~\ref{lem:RepAlgQuiver}). In particular, we show that the repetitive algebra of a folded gentle algebra is related to the quiver and relations of the repetitive algebra of its corresponding unfolded gentle algebra (described in \cite{Schroer}), and this relationship is once again given by the (un)folding procedure given in Section~\ref{sec:UnfoldedGentle}. This allows us to prove our final result.
	\begin{thm*}[Theorem~\ref{thm:DerClosed}]
	Let $A$ be a finite-dimensional folded gentle algebra and let $B$ be an algebra that is derived equivalent to $A$. Then $B$ is a folded gentle algebra. Moreover, there are only finitely many such algebras (up to Morita equivalence) that are derived equivalent to $A$.
	\end{thm*}
	The proof of the theorem is essentially the same as that in \cite{SchroerZim}. In fact the result for folded gentle algebras directly follows from applying a folding argument on the result in \cite{SchroerZim}.
	
	There are numerous examples throughout the paper. However, we end with a comprehensive worked example in Appendix~\ref{sec:Example}, which showcases every aspect of the theory developed in this paper.
	
	We thus show that folded gentle algebras, which generalise the iterated tilted algebras of types $C$ and $\widetilde{C}$, satisfy many of the same remarkable properties as gentle algebras. The author believes that folded gentle algebras therefore have the potential to be just as influential to mathematics as gentle and skewed-gentle algebras, and would not be surprised if they also appear naturally in other fields. Furthermore, the author would not be at all surprised if the folding techniques employed in this paper could be used to prove further properties of folded gentle algebras. In particular, the author strongly suspects that, like gentle algebras, folded algebras are derived-tame, and if so, that the classification of indecomposable objects in the bounded derived category should follow from a folding argument. It may also be possible to form a geometric model of the bounded derived category of folded gentle algebras that is similar to the geometric model for gentle algebras in \cite{Opper}. In addition, the author believes that the folding techniques applied in this paper can be readily applied to other settings with the aim of defining and exploring further new classes of `folded' algebras that inherit important properties from their `unfolded' counterparts. Examples of such candidate algebras could be folded analogues of the more general class of special biserial algebras, the class of skewed-gentle algebras, and Brauer graph/configuration algebras, amongst others.

	% =================================================
	% DEFINITIONS SEC
	% =================================================
	\section{The Definition of Folded Gentle Algebras} \label{sec:Defn}
	\subsection{Preliminaries and notation} \label{sec:Preliminaries}
	Throughout this paper, $Q= (Q_0,Q_1,s,t)$ is a quiver, where $Q_0$ is a set of vertices, $Q_1$ is a set of arrows, and the functions
	\begin{align*}
		s\colon &Q_1 \rightarrow Q_0 \\
		t\colon &Q_1 \rightarrow Q_0
	\end{align*}
	are source and target functions respectively.
	
	By a \emph{path} of $Q$, we mean either a \emph{stationary path} (a path of length 0) $\stp_v$ for some $v \in Q_0$, or a concatenation of arrows $\alpha_1\ldots\alpha_n$ such that each $\alpha_i \in Q_1$ and $t(\alpha_i)=s(\alpha_{i+1})$. By a \emph{subpath} of a path $p=\alpha_1\ldots\alpha_n$, we mean either a stationary path $p' \in \{s(\alpha_i),t(\alpha_i) : 1 \leq i \leq n\}$ or a path $p'=\alpha_i\alpha_{i+1}\ldots\alpha_{j}$ for some $1 \leq i \leq j \leq n$. We denote by $\PQ$ the set of all paths of $Q$.
	
	The source and target functions of $Q$ may be extended to the set $\PQ$ in the natural way. That is, for any $p=\alpha_1\ldots\alpha_n \in \PQ$, we define $s(p)=s(\alpha_1)$ and $t(p)=t(\alpha_n)$, or $s(p)=t(p)=v$ if $p$ is the stationary path at $v$. We define the \emph{length} of $p$ to be $|p|=n$ (or $|p|=0$ if $p$ is stationary). A subset $Z \subseteq \PQ$ is called an \emph{admissible set of zero relations of $Q$} if either $Z=\emptyset$ or $|p|\geq 2$ for any $p \in Z$.
	
	Given a field $K$, we denote by $KQ$ the \emph{path algebra} of $Q$ over $K$. This is the vector space given by $K$-linear combinations of elements of $\PQ$ with ring multiplication defined by concatenation of paths. For any paths $p,p' \in \PQ$ whose concatenation $pp' \not\in\PQ$, we define $pp' = 0_K \in KQ$. 
	
	\subsection{Gentle and folded gentle pairs/triples}
	We will now provide the definition of the family of quivers (with relations) that will be the main focus of this paper. 
	\begin{defn} \label{def:FoldedGentleTriple}
		Let $K$ be a field, $Q$ be a finite quiver, and $Z$ be a subset of the path algebra $KQ$. Suppose in addition that
		\begin{align*}
			Q_0 &= \ov \cup \cv, \\
			Q_1 &= \oa \cup \ca, \\
			Z &= \orel \cup \crel
		\end{align*}
		for some disjoint subsets $\ov$, $\cv$, $\oa$, $\ca$, $\orel$ and $\crel$ such that
		\begin{equation*}
			\ca = \{\crs_v : v \in \cv\}
		\end{equation*}
		is a set of loops (that is, $s(\crs_v)=t(\crs_v)=v \in \cv$), $\orel$ is an admissible set of zero relations on the subquiver $Q \setminus \ca$, and 
		\begin{equation*}
			\crel = \{\crs_v^2 - \lambda_{v,1} \crs_v - \lambda_{v,2} \stp_v : v \in \cv\}
		\end{equation*}
		for some $\lambda_{v,1},\lambda_{v,2} \in K$ (for each $v \in \cv$). We call the triple $(K,Q,Z)$ a \emph{folded gentle triple} if the following conditions are satisfied.
		\begin{enumerate}[label=(FG\arabic*)]
			\item For any $v \in Q_0$, there exist at most two arrows of source $v$ and at most two arrows of target $v$. \label{en:FG1}
			\item For any arrow $\beta \in \oa$, there exists at most one arrow $\alpha \in Q_1$ such that $\alpha\beta \not\in \orel$, and there exists at most one arrow $\gamma$ such that $\beta\gamma \not\in \orel$. \label{en:FG2}
			\item For any arrow $\beta \in \oa$, there exists at most one arrow $\alpha \in Q_1$ such that $\alpha\beta \in \orel$, and there exists at most one arrow $\gamma$ such that $\beta\gamma \in \orel$. \label{en:FG3}
			%\item For any $\alpha,\beta \in Q_1$ with $t(\alpha)=s(\beta) \in \cv$, we have $\alpha\beta \in Z$. \label{en:FG4}
			\item $|p|=2$ for any $p \in \orel$. \label{en:FG4}
			\item For each relation $\eta_v^2 - \lambda_{v,1} \eta_v - \lambda_{v,2}\stp_v \in \crel$, the polynomial $x^2-\lambda_{v,1} x - \lambda_{v,2} \in K[x]$ is irreducible. \label{en:FG5}
		\end{enumerate}
		We call the arrows in the set $\oa$ the \emph{ordinary arrows} of $Q$ and the arrows in the set $\ca$ the \emph{crease loops} of $Q$ (or shortly, \emph{creases}). We call the vertices in $\ov$ the \emph{ordinary vertices} of $Q$, and we call the vertices in $\cv$ the \emph{crease vertices} of $Q$. Finally, we call the zero relations in $\orel$ the \emph{ordinary relations} of $Q$ and the relations in $\crel$ the \emph{crease relations}.
	\end{defn}
	
	\begin{rem}
		Note that as a result of how folded gentle triples are defined, no relation in $\orel$ contains a crease loop. As a consequence, we have the following axiom that is redundant but nevertheless useful to point out.
		\begin{enumerate}[label=(FG\arabic*)]
			\setcounter{enumi}{5}
			\item For any $\alpha,\beta \in \oa$ with $t(\alpha)=s(\beta) \in \cv$, we have $\alpha\beta \in \orel$. \label{en:FG6}
		\end{enumerate}
	\end{rem}
	
	The axioms \ref{en:FG1}-\ref{en:FG4} above will be familiar to readers with experience in working with gentle algebras: these are precisely the axioms used to define gentle algebras (c.f. \cite{AssemHappel,ASIteratedTilted}). This motivates the following definition.
	
	\begin{defn}
		Let $(K,Q,Z)$ be a folded gentle triple. We call $(K,Q,Z)$ a \emph{gentle triple} (or equivalently, we call $(Q,Z)$ a \emph{gentle pair}) if $\cv=\emptyset$.
	\end{defn}
	
	Readers familiar with skewed-gentle algebras will also be familiar with the above definitions for the aforementioned reason, as well as the addition of \emph{special loops} (c.f. \cite{SkewedGentle}). However, every special loop $\crs$ of a skewed-gentle algebra satisfies an idempotent relation $\crs^2 - \crs$. This is never the case for a folded gentle triple, as the polynomial $x^2 - x$ is reducible over every field (which would violate \ref{en:FG5}). This subtle difference in the quadratic relations between special loops and the crease loops defined above leads to markedly different behaviour. In particular, the algebras that we will define from folded gentle triples are neither Morita nor derived equivalent to skewed-gentle algebras. In addition, when this paper later discusses words, strings and bands, we will consider inverses of crease loops to be distinct symbols from crease loops (in contrast to special loops in the theory of both clannish algebras and skewed-gentle algebras). We have thus given our `special loops' a different name (crease loops) to distinguish them with regards to these properties. The name `crease loops' also better fits the idea behind folding that will be presented later in the paper. Those who instead work more frequently with irreducible clans/clannish algebras or with semilinear clannish algebras can consider the words `special' and `crease' to be interchangeable, though again, some caution is advised with our consideration of words, strings and bands.
	
	\begin{defn}
		Let $A$ be a finite-dimensional $K$-algebra.
		\begin{enumerate}[label=(\alph*)]
			\item We call $A$ a \emph{folded gentle algebra} if $A$ is Morita equivalent to a quotient algebra $KQ / \langle Z \rangle$ for some folded gentle triple $(K,Q,Z)$ that is not a gentle triple.
			\item We call $A$ a \emph{gentle algebra} if $A$ is Morita equivalent to a quotient algebra $KQ / \langle Z \rangle$ for some gentle pair $(Q,Z)$.
		\end{enumerate}
	\end{defn}
	
		Henceforth, whenever we say that $(K,Q,Z)$ is a folded gentle triple, we will assume that $\cv\neq\emptyset$ and reserve the term gentle triple/pair exclusively for the case where $\cv=\emptyset$. In addition, all gentle and folded gentle algebras are assumed to be finite-dimensional in this paper.

	Given a folded gentle algebra $A$ associated to a triple $(K,Q,Z)$, it is straightforward to verify that each $\stp_v \in A$ is an idempotent element of $A$, where $\stp_v \in \PQ$ is stationary. In particular, it follows that $\sum_{v\in Q_0}\stp_v = 1_{A}$.
	
	\begin{rem} \label{rem:InverseCrease}
		Let $A$ be a folded gentle algebra. It follows from axiom \ref{en:FG5} that for any $v \in \cv$, the $K$-subalgebra $\langle \stp_v, \crs_v \rangle$ of $A$ is isomorphic to the field $K[x] / ( x^2 - \lambda_{v,1} x - \lambda_{v,2} )$. Consequently, there exists an element $\crs_v\inv = \frac{1}{\lambda_{v,2}}(\crs_v - \lambda_{v,1}\stp_v) \in A$.
	\end{rem}
	
	\begin{rem}
		The above construction of folded gentle algebras could equivalently be expressed in the language of a quotient algebra of a tensor algebra over a $K$-species (see \cite{DRIndec,GabrielIndec,Simson} for an account of species and \cite{AssemC} for the subclass of iterated tilted algebras of type $\dynC_n$). Species are defined over a much more general framework. However for the specific class of algebras we are investigating in this paper, the author is of the opinion that the general framework of species makes the construction harder to understand than is necessary. Thus we have given the alternative construction above, which makes the comparison with gentle algebras much more straightforward.
	\end{rem}
	
	\begin{exam} \label{ex:C3}
		Consider the folded gentle triple $(\real,Q,Z)$ given by
		\begin{equation*}
			\begin{tikzpicture}
				\draw [anchor = east] (-0.4,0) node {$Q\colon$};
				\draw (-0.2,0) node [draw] {1};
				\draw[<-] (-0.1,0.4) .. controls (0.1,0.7) and (-0.1,0.8) .. (-0.2,0.8) .. controls (-0.3,0.8) and (-0.5,0.7) .. (-0.3,0.4);
				\draw (0.2,0.6) node {\footnotesize$\eta_1$};
				\draw [->] (0.2,0) -- (0.8,0);
				\draw (0.5,0.2) node {\footnotesize$\alpha$};
				\draw (1,0) node {2};
				\draw [->] (1.2,0) -- (1.8,0);
				\draw (1.5,0.2) node {\footnotesize$\beta$};
				\draw (2,0) node {3};
			\end{tikzpicture}
		\end{equation*}
		with ordinary vertices $2,3 \in \ov$, crease vertex $1 \in \cv$ (displayed as a box vertex for highlighting purposes), and $Z = \crel = \{\crs_1^2 + \stp_1\}$. We also have $\alpha,\beta \in \oa$ and $\crs_1 \in \ca$. The corresponding folded gentle algebra is
		\begin{equation*}
			\real Q / \langle Z\rangle = \langle \stp_1, \stp_2, \stp_3, \alpha, \beta, \alpha\beta, \crs_1, \crs_1\alpha, \crs_1\alpha\beta\rangle_\real
		\end{equation*}
		whose $\real$-basis is given by the listed set of generators. It follows that $\real Q / \langle Z \rangle$ is isomorphic to the matrix algebra
		\begin{equation*}
			\begin{pmatrix}
				\complex & \complex & \complex \\
				0 & \real & \real \\
				0 & 0 & \real
			\end{pmatrix}.
		\end{equation*}
		It is also isomorphic to the tensor algebra of a linearly oriented $\real$-species of Dynkin type $\dynC_3$.
	\end{exam}
	
	\begin{exam}\label{ex:C2tilde}
		Consider the folded gentle triple $(\rational,Q,Z)$ given by
		\begin{equation*}
			\begin{tikzpicture}
				\draw [anchor = east] (-0.4,0) node {$Q\colon$};
				\draw (-0.2,0) node [draw] {1};
				\draw[<-] (-0.1,0.4) .. controls (0.1,0.7) and (-0.1,0.8) .. (-0.2,0.8) .. controls (-0.3,0.8) and (-0.5,0.7) .. (-0.3,0.4);
				\draw (0.2,0.6) node {\footnotesize$\eta_1$};
				\draw [->] (0.2,0) -- (0.8,0);
				\draw (0.5,0.2) node {\footnotesize$\alpha$};
				\draw (1,0) node {2};
				\draw [->] (1.2,0) -- (1.8,0);
				\draw (1.5,0.2) node {\footnotesize$\beta$};
				\draw (2.2,0) node [draw] {3};
				\draw[<-] (2.3,0.4) .. controls (2.5,0.7) and (2.3,0.8) .. (2.2,0.8) .. controls (2.1,0.8) and (1.9,0.7) .. (2.1,0.4);
				\draw (2.6,0.6) node {\footnotesize$\eta_3$};
			\end{tikzpicture}
		\end{equation*}
		with $2 \in \ov$ and $1,3 \in \cv$ and $Z = \crel=\{\crs^2_1 +\stp_1, \crs^2_3 - \crs_3 - \stp_3\}$. The corresponding folded gentle algebra then has $\rational$-basis
		\begin{equation*}
			\{\stp_1, \stp_2, \stp_3, \alpha, \beta, \crs_1, \crs_3, \crs_1 \alpha, \beta\crs_3, \alpha\beta, \crs_1\alpha\beta, \alpha\beta\crs_3, \crs_1\alpha\beta\crs_3\}.
		\end{equation*}
		One can show that this algebra is isomorphic to the tensor algebra of a linearly oriented $\rational$-species of type $\dynafC_2$. It is also isomorphic to the matrix algebra
		\begin{equation*}
			\begin{pmatrix}
				\rational(\iu) & \rational(\iu) & \rational(\iu) \otimes_{\rational} \rational(\varphi) \\
				0 & \rational & \rational(\varphi) \\
				0 & 0 & \rational(\varphi)
			\end{pmatrix},
		\end{equation*}
		where $\varphi = \frac{\sqrt{5} + 1}{2}$ is the golden ratio, which satisfies the identity $\varphi^2 = \varphi +1$.
	\end{exam}
	
	\begin{exam} \label{ex:C3TildeD4}
		Consider the folded gentle triple $(\finite_2,Q,Z)$ given by
		\begin{equation*}
			\begin{tikzpicture}
				\draw [anchor = east] (-1.4,0) node {$Q\colon$};
				\draw (-1.2,0) node [draw] {1};
				\draw[<-] (-1.1,0.4) .. controls (-0.9,0.7) and (-1.1,0.8) .. (-1.2,0.8) .. controls (-1.3,0.8) and (-1.5,0.7) .. (-1.3,0.4);
				\draw (-0.8,0.6) node {\footnotesize$\eta_1$};
				\draw [<-] (-0.8,0) -- (-0.2,0);
				\draw (-0.5,0.2) node {\footnotesize$\alpha$};
				\draw (0,0) node {2};
				\draw [<-] (0.2,0) -- (0.8,0);
				\draw (0.5,0.2) node {\footnotesize$\beta$};
				\draw (1,0) node {3};
				\draw [->] (1.2,0) -- (1.8,0);
				\draw (1.5,0.2) node {\footnotesize$\gamma$};
				\draw (2.2,0) node [draw] {4};
				\draw[<-] (2.3,0.4) .. controls (2.5,0.7) and (2.3,0.8) .. (2.2,0.8) .. controls (2.1,0.8) and (1.9,0.7) .. (2.1,0.4);
				\draw (2.6,0.6) node {\footnotesize$\eta_4$};
				\draw [->] (2.6,0) -- (3.2,0);
				\draw (2.9,0.2) node {\footnotesize$\delta$};
				\draw (3.4,0) node {5};
				\draw [->] (3.6,0.1) -- (4.2,0.3);
				\draw (3.9,0.5) node {\footnotesize$\zeta_1$};
				\draw (4.4,0.4) node {6};
				\draw [->] (3.6,-0.1) -- (4.2,-0.3);
				\draw (3.9,-0.5) node {\footnotesize$\zeta_2$};
				\draw (4.4,-0.4) node {7};
				\draw [dashed] (1.717,-0.1294) arc (-165.0022:-15:0.5);
				\draw [dashed](3.0116,-0.0957) arc (-166.1583:-41.5385:0.4);
			\end{tikzpicture}
		\end{equation*}
		with $\orel=\{\gamma\delta, \delta\zeta_2\}$ (zero relations are indicated on the quiver with dashed lines) and $\crel = \{\crs_1^2+\crs_1+\stp_1,\crs_4^2+\crs_4+\stp_4\}$. The corresponding folded gentle algebra $\finite_2 Q/ \langle Z \rangle$ has $\finite_2$-basis
		\begin{equation*}
			\{\stp_1, \stp_2, \stp_3, \stp_4, \stp_5, \stp_6, \stp_7, \alpha, \beta, \gamma, \delta, \zeta_1, \zeta_2, \beta\alpha, \delta\zeta_1, \crs_1, \crs_4, \alpha\crs_1, \beta\alpha\crs_1, \gamma\crs_4,\gamma\crs_4\delta,\gamma\crs_4\delta\zeta_1, \crs_4\delta\zeta_1\}.
		\end{equation*}
		In particular, $\finite_2 Q / \langle Z \rangle$ is a finite set given precisely by sums of pairwise distinct $\finite_2$-basis elements.
	\end{exam}

	\begin{exam} \label{ex:TripleCrease}
		Consider the folded gentle triple $(\rational,Q,Z)$ given by the following.
		\begin{equation*}
			\begin{tikzpicture}
				\draw[<-] (0.9,0.4) .. controls (1.1,0.7) and (0.9,0.8) .. (0.8,0.8) .. controls (0.7,0.8) and (0.5,0.7) .. (0.7,0.4);
				\draw (1.2,0.6) node {\footnotesize$\eta_2$};
				\draw[<-] (2.3,0.4) .. controls (2.5,0.7) and (2.3,0.8) .. (2.2,0.8) .. controls (2.1,0.8) and (1.9,0.7) .. (2.1,0.4);
				\draw (2.6,0.6) node {\footnotesize$\eta_3$};
				
				\draw (-0.4,0) node {1};
				\draw [->] (-0.2,0) -- (0.4,0);
				\draw (0.1,0.2) node {\footnotesize$\alpha$};
				\draw (0.8,0) node [draw] {2};
				\draw [->] (1.2,0) -- (1.8,0);
				\draw (1.5,0.2) node {\footnotesize$\beta$};
				\draw (2.2,0) node [draw] {3};
				\draw [->] (2.6,0) -- (3.2,0);
				\draw (2.9,0.2) node {\footnotesize$\gamma$};
				\draw (3.4,0) node {4};
				\draw[->] (3.6,-0.2) .. controls (4,-0.4) and (4.2,-0.2) .. (4.2,0) .. controls (4.2,0.2) and (4,0.4) .. (3.6,0.2);
				\draw (4.4,0) node {\footnotesize$\delta$};

				\draw [dashed](3.8,-0.2) .. controls (3.9,-0.15) and (3.9,-0.1) .. (3.9,0) .. controls (3.9,0.1) and (3.9,0.15) .. (3.8,0.2);
				\draw [dashed](0.2174,-0.1436) arc (-166.1537:-13.8462:0.6);
				\draw [dashed](1.6174,-0.1436) arc (-166.1537:-13.8462:0.6);	
			\end{tikzpicture}
		\end{equation*}
		Then $\oa= \{\alpha,\beta,\gamma,\delta\}$ and $\ca = \{\crs_2,\crs_3\}$. Moreover, $\ov=\{1,4\}$, $\cv=\{2,3\}$. Defining $\orel= \{\alpha\beta,\beta\gamma,\delta^2\}$ and $\crel= \{\crs_2^2+\stp_2, \crs_3^2-2\crs_3 + 5 \stp_3\}$, it follows that $\rational Q / \langle Z \rangle$ has basis
		\begin{multline*}
			\{\stp_1,\stp_2,\stp_3,\stp_4,\alpha,\beta,\gamma,\delta, \gamma\delta, \crs_2, \crs_3, \alpha\crs_2, \crs_2\beta, \beta\crs_3, \crs_3\gamma, \alpha\crs_2\beta, \crs_2\beta\crs_3, \beta\crs_3\gamma,  \\
			\crs_3\gamma\delta, \alpha\crs_2\beta\crs_3,
			\crs_2\beta\crs_3\gamma,\beta\crs_3\gamma\delta,\alpha\crs_2\beta\crs_3\gamma, \crs_2\beta\crs_3\gamma\delta, \alpha\crs_2\beta\crs_3\gamma\delta\}.
		\end{multline*}
	\end{exam}
	
	% =================================================
	% BASIC PROPERTIES SEC
	% =================================================
	\section{Basic Properties of Folded Gentle Algebras} \label{sec:BasicProps}
	In this section, we will outline of some elementary representation theory of folded gentle algebras. The primary intention is to provide readers who may be unfamiliar with the representation theory of species some essential background in the context of folded gentle algebras. However, it is also intended to serve as a useful comparison to gentle algebras. This will then be followed by some algebraic structural properties. Throughout $A$ is a folded gentle algebra associated to a folded gentle triple $(K,Q,Z)$.
	
	\subsection{Elementary representation theoretic properties} The following largely results from the representation theory of species.
	
	\subsubsection{Simple $A$-modules} Firstly, we will discuss simple $A$-modules. Since $K$ is not an algebraically closed field, we can no longer assume that the simple $A$-modules are $1$-dimensional over $K$. The following definition/notation is useful for describing the structure of simple $A$-modules and will be used later in the paper.
	
	\begin{defn}
		For each $v \in Q_0$, define a field $\viso$ by
		\begin{equation*}
			\viso =
			\begin{cases}
				K & \text{if } v \in \ov \\
				K[x]/(x^2 - \lambda_1 x - \lambda_2)	& \text{if } v \in \cv \text{ and } \crs_v^2 - \lambda_1 \crs_v - \lambda_2 \stp_v \in \crel.
			\end{cases}
		\end{equation*}
		We call $\viso$ the field associated to $v \in Q_0$.
	\end{defn}
	
	By \cite[Theorem 2.3]{Simson}, a complete collection of pairwise non-isomorphic simple $A$-modules is in bijective correspondence with the vertices of $Q$. In particular, for each $v \in Q_0$, we have a simple $A$-module $S(v)$ whose underlying vector space is $\viso$ and whose $A$-action is defined such that $\stp_v$ acts by an identity endomorphism, $\crs_v$ (if $v \in \cv$) acts by multiplication by $x$ on the field $K[x]/(x^2 - \lambda_1 x - \lambda_2)$, and any arrow or stationary path $\alpha\neq \stp_v,\crs_v$ acts by the zero morphism. Consequently, 
	\begin{equation*}
		\End_A S(v) \cong \viso
	\end{equation*}
	and $\dim_{K}S(v)=2$ for each $v \in \cv$.
	
	\subsubsection{Morphisms of right $A$-modules} A consequence of Schur's Lemma is that a morphism of right $A$-modules $f\colon M \rightarrow M'$ can be written as a tuple $(f_v)_{v \in Q_0}$ of linear maps, where $f_v\colon M\stp_v \rightarrow M'\stp_v$ and each $M\stp_v$ (resp. $M'\stp_v$) is the image subspace given by the action of $\stp_v$ on $M$ (resp. $M'$). Thus, for any arrow $\alpha \in A$ with $s(\alpha)=u$ and $t(\alpha) = v$, we have $f_v(m\alpha)=f_u(m) \alpha$ for any $m \in M\stp_u \subseteq M$. Note that $f_v$ is actually a $\viso$-linear map as well as a $K$-linear map.
	
	\subsubsection{Projective and injective $A$-modules} It is easy to verify that we also have a unique decomposition of $A$ into indecomposable projective right $A$-modules given by $\stp_v A$ for each $v \in Q_0$. Dually, we have a unique decomposition of $A$ into indecomposable projective left $A$-modules given by $A \stp_v$ for each $v \in Q_0$. Thus, each vertex $v \in Q_0$ corresponds to both an indecomposable projective right $A$-module $P(v) = \stp_v A$ and indecomposable injective right $A$-module $I(v) = D(A\stp_v)$, where $D= \Hom_{K}({-},K)$ is the standard $K$-duality.
	
	Taking axioms \ref{en:FG1}-\ref{en:FG5} into consideration, the above characterisation leads to the following remarks about morphisms between indecomposable projective $A$-modules.
	
	\begin{rem} \label{rem:ProjMorphs}
		Let $v,v' \in Q_0$.
		\begin{enumerate}[label=(\alph*)]
			\item If $v \in \cv$, then $\Hom_A(P(v),P(v'))$ is a $\viso$-vector space. Specifically, suppose that $\viso = K[x] / (q)$ for some quadratic relation $q$ and suppose $f \in \Hom_A(P(v),P(v'))$. Then there exists a morphism $(\lambda+\mu x)f\in \Hom_A(P(v),P(v'))$, where $\lambda,\mu \in K$ and $xf(a) = f(\crs_v a)$.
			\item If $v' \in \cv$, then $\Hom_A(P(v),P(v'))$ is a $\fiso_{v'}$-vector space. If we similarly suppose that $\fiso_{v'} = K[x] / (q)$ and $f \in \Hom_A(P(v),P(v'))$, then there exists a morphism $(\lambda+\mu x)f\in \Hom_A(P(v),P(v'))$, where $\lambda,\mu \in K$ and $xf(a) = \crs_{v'} f(a)$.
			\item If $v,v' \in \cv$, then $\Hom_A(P(v),P(v'))$ has the structure of a $\fiso_{v'}$-$\viso$-bimodule in the obvious way. In particular, $\Hom_A(P(v),P(v'))$ is a $\fiso_{v'} \otimes_K \viso$-module.
			%\item If $v\neq v'$ and there exist arrows $\alpha,\alpha' \in Q_1$ with $t(\alpha)=v=t(\alpha')$, then any morphism $f \in \Hom_A(P(v),P(v'))$ factors through a linear combination of morphisms $g,g' \in \Hom_A(P(v),P(s(\alpha)) \oplus P(s(\alpha')))$ given by $g(\stp_v) = \alpha$ and $g'(\stp_v)=\alpha'$.
		\end{enumerate}
	\end{rem}
	
	\subsubsection{Decomposition of $A$-modules} By virtue of being a module category of a finite-dimensional $K$-algebra, the category $\mod*A$ is Krull-Schmidt: every object in $\mod*A$ has a unique (up to isomorphism and permutation of direct summands) decomposition into indecomposable objects, and an object is indecomposable if and only if its endomorphism algebra is local.
	
	\subsection{Biseriality}
	The aim of this subsection is to show that folded string algebras satisfy an analogous condition to (special) biserial algebras, the latter of which are typically defined over an algebraically closed field. First, some definitions.
	
	\begin{defn} \label{def:Biserial}
		Let $A$ be a finite-dimensional algebra. A (left or right) $A$-module $M$ is said to be \emph{uniserial} if it has a unique composition series, or equivalently, if $\rad^i M / \rad^{i+1} M$ is simple or zero for each $i$, where $\rad M$ is the radical of the (left or right) module $M$.
		
		A (left or right) $A$-module $M$ is said to be \emph{biserial} if there exist (possibly trivial) uniserial submodules $U_1,U_2 \subseteq M$ such that $\rad M \cong U_1 + U_2$ and $U_1 \cap U_2$ is simple or zero. We say $A$ is a biserial algebra if every left and right indecomposable projective $A$-module is biserial.
	\end{defn}
	
	Ultimately, we aim to show the following in this subsection.
	
	\begin{thm}\label{thm:Biserial}
		Let $A=KQ / \langle Z \rangle$ be a folded gentle algebra.
		\begin{enumerate}[label=(\alph*)]
			\item Let $M$ be a left $A$-module and let $0\neq m \in M$ such that $\stp_v m = m$ for some $v \in Q_0$. Then the submodule $A m = \langle a m : a \in A\rangle \subseteq M$ is biserial.
			\item Let $M$ be a right $A$-module and let $0\neq m \in M$ such that $m\stp_v = m$ for some $v \in Q_0$. Then the submodule $mA = \langle ma : a \in A\rangle \subseteq M$ is biserial.
		\end{enumerate}
		Consequently, $A$ is a biserial algebra.
	\end{thm}
	
	To achieve this, we must prove some technical lemmata, the first of which is adapted from \cite{VHW} and \cite[Lemma 2.7]{GSMultiserial}.
	
	\begin{lem}\label{lem:maAUniserial}
		Let $m \in M$ be as in Theorem~\ref{thm:Biserial}(b). For any arrow $\alpha \in \oa$, the module $m\alpha A = \langle m\alpha a : a \in A \rangle$ is a (possibly trivial) uniserial module.
	\end{lem}
	\begin{proof}
		If $m\alpha A$ is either simple or zero, then the lemma is trivially satisfied. Thus, we will assume that this is not the case. This implies that there exists an arrow $\beta_1 \in Q_1$ such that $m\alpha\beta_1 \neq 0$. By axiom \ref{en:FG2}, there exists no other arrow $\gamma\neq \beta_1 \in Q_1$ such that $m\alpha\gamma \neq 0$. Thus, the arrow $\beta_1$ is unique with regard to this property.
		
		Two possibilities now exist: either $t(\alpha) \in \ov$ or $t(\alpha) \in \cv$. In the former case, we note that $\beta_1 \in \oa$ and if there exists an arrow $\beta_2 \in Q_1$ such that $m\alpha\beta_1\beta_2 \neq 0$, then it is again unique by axiom \ref{en:FG2}. On the other hand if $t(\alpha) \in \cv$, then $\beta_1 \in \ca$ by axiom \ref{en:FG6}. In this case $m\alpha\beta^2_1 \neq 0$, but there may also exist an arrow $\beta_2 \in \oa$ such that $m\alpha\beta_1\beta_2 \neq 0$. At first glance, this appears to imply that we have two arrows $\beta_1$ and $\beta_2$ that can be multiplied on the right of $m\alpha\beta_1$ to yield a non-zero element of the submodule $m\alpha A$. However, we also have a relation $\beta_1^2 - \lambda_1 \beta_1 - \lambda_2 \stp_{t(\beta_1)} \in \crel$ for some $\lambda_1,\lambda_2 \in K$, which implies
		\begin{equation*}
			m\alpha\beta^2_1\beta_2 = \lambda_1 m\alpha\beta_1\beta_2 + \lambda_2 m\alpha\beta_2 =\lambda_1 m\alpha\beta_1\beta_2
		\end{equation*}
		by \ref{en:FG6}.
		
		Since $Q$ is finite by definition, an iterative use of the above argument results in a unique path $\alpha\beta_1\ldots\beta_n$ that is of maximal length with respect to the condition that $m\alpha\beta_1\ldots\beta_n \neq 0$, and is such that
		\begin{equation*}
			m\alpha A = \langle m\alpha, m\alpha\beta_1, \ldots, m\alpha\beta_1\ldots\beta_n \rangle_K
		\end{equation*}
		Now rewrite the path $\alpha\beta_1\ldots\beta_n$ as
		\begin{equation} \tag{$\ast$} \label{eq:UniserialPath}
			\alpha \delta_1\delta_2\ldots\delta_{r_1} \crs_{v_1}\delta_{r_1+1}\delta_{r_1+2}\ldots\delta_{r_2} \crs_{v_2} \delta_{r_2+1}\delta_{r_2+2}\ldots \delta_{r_s} \crs_{v_s}
		\end{equation}
		where each $\delta_i \in \oa$, each $\crs_i \in \ca$, and we omit the last symbol ($\crs_{v_s}$) if $\beta_n \not\in\ca$. It is then straightforward to verify that
		\begin{equation*}
			\rad^j m\alpha A = m p_j A
		\end{equation*}
		where $p_j$ is the subpath of (\ref{eq:UniserialPath}) given by removing everything to the right of $\delta_j$ (if $j \neq r_l$ for any $l$) or $\crs_{v_l}$ (if $j = r_l$ for some $l$). Finally, we note that there is a vector space isomorphism
		\begin{equation*}
			\rad^j m\alpha A / \rad^{j+1} m\alpha A \cong
			\begin{cases}
				\langle m p_{j-1} \delta_j \rangle_K	& \text{if } j \neq r_l \text{ for any } l\\
				\langle m p_{j-1} \delta_{r_l},   m p_{j-1}\delta_{r_l} \crs_{v_l}  \rangle_K	& \text{if } j = r_l \text{ for some } l,
			\end{cases}
		\end{equation*}
		which in particular extends to an $A$-module isomorphism
		\begin{equation*}
			\rad^j m\alpha A / \rad^{j+1} m\alpha A \cong S(t(\delta_j)),
		\end{equation*}
		and we are done.
	\end{proof}
	
	\begin{lem} \label{lem:radmASum}
		Let $m \in M$ be as in Theorem~\ref{thm:Biserial}(b). Then either $\rad mA$ is uniserial or $\rad mA = U_1 + U_2$ for some uniserial submodules $U_1$ and $U_2$ such that $U_1 \not\subseteq U_2$ and $U_2 \not\subseteq U_1$.
	\end{lem}
	
	\begin{proof}
		We begin by making the trivial observation that if $\rad mA = U_1 + U_2$ for uniserial submodules $U_1$ and $U_2$ such that either $U_1 \subseteq U_2$ or $U_2 \subseteq U_1$, then $\rad mA$ is uniserial. Thus, it is sufficient to show that there exist uniserials $U_1$ (and possibly $U_2$) such that $\rad mA = U_1$ (resp. $\rad mA = U_1+U_2$) in all possible cases
	
		There are two possibilities to consider: either $v \in \ov$ or $v \in \cv$. First suppose that $v \in \ov$. Then $\tp mA = mA / \rad mA$ is a 1-dimensional $K$-vector space. Now since $m \stp_v \neq 0$, we must have $m\alpha \neq 0$ for some $\alpha \in Q_1$ only if $s(\alpha) =v$. By axiom \ref{en:FG1}, there exist at most two arrows of source $v$, say $\alpha$ and $\alpha'$. By construction, both $\alpha,\alpha' \in \oa$. It is therefore easy to see that we have 
		\begin{equation*} \label{eq:radmACase1} \tag{\ref*{lem:radmASum}.1}
			\rad mA = m\alpha A + m\alpha' A.
		\end{equation*}
		By Lemma~\ref{lem:maAUniserial}, both $m\alpha A$ and $m\alpha' A$ are uniserial, and so we are done.
	
		Now suppose that $v \in \cv$. Then $v$ is incident to a crease $\crs_v \in \ca$ and (by \ref{en:FG1}) at most one ordinary arrow, say $\alpha \in \oa$. In this case $\tp mA$ is a 2-dimensional $K$-vector space, indexed by the elements $\stp_v,\crs_v \in A$. One can verify that we have
		\begin{equation*} \label{eq:radmACase2} \tag{\ref*{lem:radmASum}.2}
			\rad mA = m\alpha A + m\crs_v\alpha A.
		\end{equation*}
		Now $m\crs_v \in M$ is necessarily non-zero and satisfies $m\crs_v\stp_v = m\crs_v$, so by Lemma~\ref{lem:maAUniserial}, both $m\alpha A$ and $m\crs_v\alpha A$ are uniserial, as required. This completes the proof.
	\end{proof}
	
	We can now prove the theorem above.
	
	\begin{proof}[Proof of Theorem~\ref{thm:Biserial}]
		We will prove (b), as the proof for (a) is dual. If $mA$ is uniserial, then the theorem is automatically satisfied, so we will assume throughout the proof that $mA$ is not uniserial. This implies that $\rad mA$ is not uniserial and thus, by Lemma~\ref{lem:radmASum}, that $\rad mA = U_1 + U_2$ for some uniserials $U_1$ and $U_2$ such that $U_1 \not\subseteq U_2$ and $U_2 \not\subseteq U_1$. In particular, it follows from the proof of Lemma~\ref{lem:radmASum} that $U_1$ and $U_2$ are of the form
		\begin{equation*}
			U_1 = ma_1 A  \qquad \text{and} \qquad U_2 = ma_2A
		\end{equation*}
		for some distinct basis elements $a_1,a_2 \in A$.
		
		If $U_1 \cap U_2$ is simple or zero, then we are done. So we will also assume that this is not the case, and show that this leads to a contradiction. Now as a submodule of $U_1$ and $U_2$, it follows that $U_1 \cap U_2$ is uniserial. In particular, this implies that there exist $p_1,p_2 \in \PQ$ such that
		\begin{equation*}
			U_1 \cap U_2 = m a_1 p_1 A = m a_2 p_2 A
		\end{equation*}
		and we may assume that $p_1$ and $p_2$ are maximal with respect to this property --- this maximality assumption implies that if $t(p_1)=t(p_2) \in \cv$, then the last arrow of $p_1$ and $p_2$ is a crease.
	
		Write $p_1 = \beta_{n_1}\ldots\beta_{1}$ and $p_2 = \beta'_{n_2}\ldots\beta'_{1}$. Since $U_1 \cap U_2$ is neither simple nor zero by assumption, there exists $\gamma \in \oa$  such that $m a_1 p_1 \gamma = m a_2 p_2 \gamma \neq 0$. But then $\beta_{1}\gamma, \beta'_{1}\gamma \not\in Z$. By axiom \ref{en:FG2}, this is possible only if $\beta_{1} =\beta'_{1}$. Proceeding iteratively with this argument, we obtain $\beta_{i} = \beta'_{i}$ and $n_1 = n_2$.
	
		From here, the proof splits into three cases. For the first case, suppose that $v$ is the source of precisely two distinct ordinary arrows $\alpha,\alpha' \in \ov$. Then it follows from the proof of Lemma~\ref{lem:radmASum} that $U_1 = m\alpha A$ and $U_2 = m \alpha' A$ (this is from Equation~\ref{eq:radmACase1}). Since we have assumed that $\rad mA$ is not uniserial, $U_1,U_2 \neq 0$. But then the same iterative argument we used before to show that $\beta_i = \beta'_i$ can be used to show that we must have $\alpha = \alpha'$, which contradicts the assumption that $\alpha$ and $\alpha'$ are distinct. Thus if $U_1 \cap U_2$ is neither simple nor zero, then $v$ cannot be the source of precisely two distinct ordinary arrows.
	
		This leads us to the second case, where $v$ is the source of precisely one ordinary arrow $\alpha \in \oa$. Since our initial assumption is that $\rad mA$ is not uniserial, it follows from the proof of Lemma~\ref{lem:radmASum} that $v \in \cv$ and thus that there exists a crease $\crs_v \in \ca$ such that $U_1 = m\alpha A$ and $U_2 = m \crs_v\alpha A$ (this is from Equation~\ref{eq:radmACase2}). We thus have
		\begin{align*}
			U_1 \cap U_2 = m\alpha \beta_n \ldots \beta_1 A = m\crs_v\alpha \beta_n \ldots \beta_1 A,
		\end{align*}
		
		There are two subcases to consider: namely that either $t(\beta_1) \in \ov$ (in which case $\beta_1 \in \oa$) or $t(\beta_1) \in \cv$ (in which case $\beta_1 \in \ca$). Suppose the former is true. Then $\tp(U_1 \cap U_2)$ is a 1-dimensional $K$-vector-space, as it must be isomorphic to $S(t(p))$. But we also have
		\begin{align*}
			\tp(U_1 \cap U_2) &= \rad^{n+1} mA / \rad^{n+2} mA \\
			&\cong \langle m\alpha \beta_n \ldots \beta_1, m\crs_v\alpha \beta_n \ldots \beta_1 \rangle_{K}.
		\end{align*}
		However these statements are only compatible if $m\crs_v\alpha \beta_n \ldots \beta_1 = \mu m\alpha \beta_n \ldots \beta_1$ for some non-zero $\mu \in K$. Now consider the automorphism $f \in \End_A(mA)$ defined such that
		\begin{equation*}
			f(m) = m - \mu (\lambda_1 \mu+\lambda_2)\inv m \crs_v,
		\end{equation*}
		where $\crs_v^2 -\lambda_1 \crs_v + \lambda_2 \in \crel$. Then $f(m \crs_v \alpha \beta_n \ldots \beta_1) = 0$. In particular, this implies that
		\begin{equation*}
			U_1 \cap U_2 = m\crs_v\alpha \beta_n \ldots \beta_1 A \cong 0,
		\end{equation*}
		which obviously contradicts our assumption that $U_1 \cap U_2$ is non-zero. So we cannot have $t(\beta_1) \in \ov$.
		
		Suppose instead that $t(\beta_1) \in \cv$. Then $\tp(U_1 \cap U_2)$ is a 2-dimensional $K$-vector-space. But then we also have
		\begin{equation*}
			\tp(U_1 \cap U_2) \cong \langle m\alpha \beta_n \ldots \beta_1 ,m\alpha \beta_n \ldots \beta_2, m\crs_v\alpha \beta_n \ldots \beta_1, m\crs_v\alpha \beta_n \ldots \beta_2\rangle_{K}.
		\end{equation*}
		This is again only possible if $m\crs_v\alpha \beta_n \ldots \beta_1 = \mu m\alpha \beta_n \ldots \beta_1$ for some non-zero $\mu \in K$. The same isomorphism above then shows that we again must have $U_1 \cap U_2 \cong 0$ --- a contradiction. Thus, if we assume that $U_1 \cap U_2$ is neither simple nor zero, then $v$ cannot be the source of precisely one ordinary arrow. 
		
		By \ref{en:FG1}, we have a final case to consider --- namely, where $v$ is the source of no ordinary arrows. But then $mA$ must be simple, and hence uniserial, and we have already ruled this case out. We thus conclude that if $mA$ is not uniserial, then $U_1 \cap U_2$ must be simple or zero. Hence, $mA$ is biserial.
		
		To show that this implies that $A$ is biserial, we simply use the fact that each indecomposable projective right $A$-module is isomorphic to a module $\stp_v A$ for some $v \in Q_0$, and every indecomposable projective left $A$-module is isomorphic to a module $A\stp_v$ for some $v \in Q_0$. Each such module is clearly biserial from the proof above (and its dual), and so we are done.
	\end{proof}
		
	% ======================================================
	\section{Words, Strings and Bands} \label{sec:Strings}
	% ======================================================
	There is a classification of the indecomposable modules of gentle algebras, skewed-gentle algebras, and clannish algebras (see \cite{ButlerRingel}, \cite{CBClans} and \cite{Deng} for an account). The indecomposable modules over these algebras are given by \emph{string modules} and \emph{band modules}. Inspired by this classification, we will define string and band modules for folded gentle algebras, which are also indecomposable due to \cite{SemilinearClan}.
	
	\subsection{Words}
	We will begin with some terminology that is commonplace in the literature for string algebras and gentle algebras, which we adapt to folded gentle algebras. Most of the terminology we will define will be applicable to any finite-dimensional algebra (except for where we state otherwise or use the terms ordinary/crease).
	
	For each arrow $\alpha \in Q_1$, we define a \emph{formal inverse} $\alpha\inv$. We denote the set of all formal inverses by $Q_1\inv$ and we call the elements of the set $Q_1 \cup Q_1\inv$ \emph{symbols} of $Q$. For folded gentle algebras, the set of symbols can therefore be partitioned into subsets $\oa \cup (\oa)\inv$ and $\ca \cup (\ca)\inv$. We call the subset $\oa \cup (\oa)\inv$ the set of \emph{ordinary symbols} and the subset $\ca \cup (\ca)\inv$ the set of \emph{crease symbols}. The source and target functions of $Q$ can then be extended to the set of symbols:
	\begin{align*}
		s\colon Q_1 \cup Q_1\inv &\rightarrow Q_0, \\
		t\colon Q_1 \cup Q_1\inv &\rightarrow Q_0,
	\end{align*}
	where $s(\alpha\inv)=t(\alpha)$ and $t(\alpha\inv) =s(\alpha)$ for each $\alpha\inv \in Q_1\inv$.
	
	By a \emph{word} of $Q$, we mean a concatenation $w=\sigma_1\ldots\sigma_n$ of symbols of $Q$ such that $t(\sigma_i)=s(\sigma_{i+1})$ for each $1 \leq i < n$. The \emph{source} of $w$ is defined to be $s(w) = s(\sigma_1)$ and the \emph{target} of $w$ is defined to be $t(w) = t(\sigma_n)$. The \emph{inverse} of $w$ is defined to be the word $w\inv = \sigma_n\inv\ldots\sigma_1\inv$, where we define $(\alpha\inv)\inv = \alpha \in Q_1$ for any formal inverse $\alpha\inv \in Q_1\inv$. We say a word is \emph{reduced} if $\sigma_i \neq \sigma_{i+1}\inv$ for each $1 \leq i <n$. The \emph{length} of $w$ is defined to be $|w|=n$. A \emph{subword} of $w$ is a word $w' = \sigma_i \sigma_{i+1} \ldots \sigma_j$ for some $1 \leq i \leq j \leq n$. A subword $w'$ of $w$ is said to be \emph{proper} if $|w'|<|w|$.
	
	For technical reasons, we allow for the existence of words of length 0, which we call \emph{simple words}, and we say that there are precisely $|Q_0|$ distinct simple words associated to the quiver $Q$. By a slight abuse of notation, we will denote the simple word corresponding to $v \in Q_0$ by $\stp_v$, and we define $s(\stp_v)=v=t(\stp_v)$ and $\stp_v\inv = \stp_v$. Note that a consequence of the definitions is that a simple word never has a proper subword.
	
	For any words $w$ and $w'$ such that $t(w)=s(w')$, we define a word $ww'$ given by their concatenation subject to the following identifications. Firstly, simple words give rise to trivial concatenations: we define $\stp_u w = w$ and $w \stp_v = w$ for any word $w$ such that $s(w)=u$ and $t(w)=v$. Secondly, if $w= w_1 w_2$ and $w' = w_3 w_4$ with $w_2 = w_3\inv$, then we identify $ww' = w_1 w_4$. Consequently, every word has a reduced expression.
	
	For any given words $w$ and $w'$, we define two equivalence relations. The first equivalence relation is given by $w' \sim_1 w$ if and only if $w' = w$ or $w' = w\inv$. The second equivalence relation is given by $w \sim_2 w'$ if and only if $w = w'$ up to the identification $\crs_v = \crs_v\inv$ for each $\crs_v \in \ca$. It is not difficult to check that these equivalence relations commute. We then say that two words $w$ and $w'$ are equivalent (written $w \approx w'$) if there exists a word $w''$ such that $w \sim_1 w'' \sim_2 w'$.
	
	\begin{defn}
		By $\wrd_Q$ we denote the collection of all words, up to equivalence, of the quiver $Q$.
	\end{defn}
	
	\subsection{Strings and bands} \label{sec:StringsBands}
	We now have all we need to define strings and bands.
	
	\begin{defn}
		Let $A= KQ / \langle Z \rangle$ be a gentle or folded gentle algebra. We call a word $w \in \wrd_Q$ a \emph{string} of $A$ if
		\begin{enumerate}[label = (S\arabic*)]
			\item $w$ is reduced; \label{en:FS1}
			\item $w$ avoids the relations in $Z$: there exists no subword $w'$ of $w$ such that $w' \in \orel$ or $(w')\inv \in \orel$, and there exists no proper power of a crease symbol in $w$; and \label{en:FS2}
			\item $s(w) \in \cv$ if and only if $w$ starts with a crease symbol, and $t(w) \in \cv$ if and only if $w$ ends with a crease symbol. \label{en:FS3}
		\end{enumerate}
		By \ref{en:FS3}, a simple word $\stp_v$ is a string if and only if $v \in \ov$. We call such a string a \emph{simple string}.
		
		We call a word $w$ of $Q$ a \emph{band} of $A$ if
		\begin{enumerate}[label = (B\arabic*)]
			\item $w$ is neither a simple word nor a crease symbol; \label{en:FB1}
			\item $w$ has a cyclic structure and $w^m$ satisfies \ref{en:FS1} and \ref{en:FS2} for any $m \geq 0$; and \label{en:FB2}
			\item there exists no proper subword $w'$ such that $w \approx (w')^m$ for $m>0$. \label{en:FB3}
		\end{enumerate}
	\end{defn}
	
	We then say that two strings $w$ and $w'$ are equivalent (written $w \approx w'$) if they are equivalent as words. For bands, we define a third equivalence relation by $w \sim_3 w'$ if and only if $w = w'$ up to cyclic permutation/rotation. Again, it is not difficult to check that $\sim_3$ is pairwise commutative with $\sim_1$ and $\sim_2$. We thus say that two bands $w$ and $w'$ are equivalent (again written $w \approx w'$) if there exist bands $w''$ and $w'''$ such that $w \sim_1 w'' \sim_2 w''' \sim_3 w'$.
	
	\begin{defn}
		We denote by $\mathcal{S}_A$ the collection of all classes of strings of $A$ up to equivalence. Likewise, we denote by $\mathcal{B}_A$ the collection of all bands of $A$ up to equivalence.
	\end{defn}
	
	For brevity, we will often abuse notation throughout the paper by writing $w \in \str_A$ (or $w \in \bnd_A$) to refer to a \emph{representative} of the class $[w] \in \str_A$ (or $[w] \in \bnd_A$).
	
	We say a vertex $v$ is at a \emph{peak} of a string $w=\sigma_1\ldots\sigma_n$ if either $v=s(\sigma_1)$ with $\sigma_1 \in Q_1$, or $v=t(\sigma_i)=s(\sigma_{i+1})$ with $\sigma_i\inv,\sigma_{i+1} \in Q_1$, or $v=s(\sigma_n)$ with $\sigma_n \in Q_1\inv$. Dually, we say $v$ is at a \emph{deep} of $w$ if either $v=s(\sigma_1)$ with $\sigma_1 \in Q_1\inv$, or $v=t(\sigma_i)=s(\sigma_{i+1})$ with $\sigma_i\inv,\sigma_{i+1} \in Q_1\inv$, or $v=s(\sigma_n)$ with $\sigma_n \in Q_1$.
		
	We say a string is \emph{direct} if it is simple, or it is a crease, or if every ordinary symbol is an arrow. We say a string is \emph{inverse} if it is simple, or it is a crease, or if every ordinary symbol is a formal inverse. We say that a non-simple string $w$ is \emph{symmetric} if $w \sim_2 w\inv$ and \emph{asymmetric} otherwise. For technical purposes, we state that simple strings are never symmetric. We say that a band $w$ is \emph{symmetric} if there exists a band $w'$ such that $w \sim_2 w' \sim_3 w\inv$, and we say that $w$ is \emph{asymmetric} otherwise.
		
		\begin{exam} \label{ex:C3Strings}
		Let $(K,Q,Z)$ be as in Example~\ref{ex:C3}. There are no bands, and thus only finitely many strings. Namely, we have precisely nine strings up to equivalence:
		\begin{align*}
			&\crs_1,	&	&\crs_1\alpha,		&	&\alpha\inv\crs_1\alpha, \\
			&\stp_2, 	&	&\beta,				&	&\alpha\inv\crs_1\alpha\beta, \\
			&\stp_3,	&	&\crs_1\alpha\beta,	&	&\beta\inv\alpha\inv\crs_1\alpha\beta.
		\end{align*}
		There are precisely three symmetric strings in this list, namely $\crs_1$, $\alpha\inv\crs_1\alpha$ and $\beta\inv\alpha\inv\crs_1\alpha\beta$.
	\end{exam}
	
	\begin{exam} \label{ex:C3TildeD4Bands}
		Let $(K,Q,Z)$ be as in Example~\ref{ex:C3TildeD4}. There are infinitely many strings in this triple. There is also precisely one band up to equivalence given by $\crs_1\alpha\inv\beta\inv\gamma\crs_4\gamma\inv\beta\alpha$, which is symmetric.
	\end{exam}
	
	\begin{exam} \label{ex:TripleCreaseBands}
		Let $(K,Q,Z)$ be as in Example~\ref{ex:TripleCrease}. There are infinitely many bands in this triple. For example, there is a family of bands given by $(\gamma\delta\gamma\inv\crs_3)^n\gamma\delta\inv\gamma\inv\crs_3$ with $n \geq 0$. This is symmetric for $n=1$ but is otherwise asymmetric. There is also a symmetric band given by $\crs_2\beta\crs_3\beta\inv$.
	\end{exam}
	
	\subsection{Modules associated to strings and bands}
	Let $A = KQ / \langle Z \rangle$ be a folded gentle algebra. We will show how one obtains $A$-modules that correspond to the strings and bands of the folded gentle triple $(K,Q,Z)$.
	
	\subsubsection{String modules} \label{sec:StringModules}
	Let $w = \sigma_1\ldots\sigma_n \in \str_A$. For each $0 \leq i \leq n$, denote by $b_{i}$ the subword $\sigma_1\ldots\sigma_i$ of $w$, where $b_{0} = \stp_{s(w)}$. From this, we define a $K$-vector space
	\begin{equation*}
		M(w)= \langle b_{ i} : 0 \leq i \leq n\rangle_K
	\end{equation*}
	whose elements are given by $K$-linear combinations of the subwords $b_{ i}$. Thus, the subwords $b_{ i}$ form a basis of $M(w)$, which we call the \emph{standard basis} of $M(w)$. We define the action of an arrow or stationary path $a \in A$ on a subword $b_{ i}$ by the concatenation $b_{ i} a$, subject to the following identifications:
	\begin{enumerate}[label=(SM\arabic*)]
		\item \label{en:SM1} $b_{ i} a$ is identified with its reduction. That is, if $\sigma_i = \alpha\inv\in Q_1\inv$, then $b_{ i} \alpha = b_{i-1}\alpha\inv\alpha = b_{i-1}$.
		\item \label{en:SM2} If $\sigma_i = \alpha\in Q_1$, then $b_{ i-1} \alpha = b_i$.
		\item \label{en:SM3} If $\sigma_i=\crs_v \in \ca$ and $\crs_v^2 - \lambda_1\crs_v - \lambda_2 \stp_v \in \crel$, then $b_{ i} \crs_v = b_{ i-1}\crs_v^2 = \lambda_1 b_{ i} + \lambda_2 b_{ i-1}$. On the other hand if $\sigma_i=\crs_v\inv \in (\ca)\inv$, then $b_{ i-1} \crs_v = b_{ i-1}(\lambda_1 + \lambda_2 \crs_v\inv) = \lambda_1 b_{ i-1} + \lambda_2 b_{ i}$.
		\item \label{en:SM4} If, after the identifications \ref{en:SM1}-\ref{en:SM3} are considered, $b_{ i} a$ is not a $K$-linear combination of the elements in $\{b_{ i} : 0 \leq i \leq n\}$, then we define $b_{ i} a=0$.
	\end{enumerate}
	When extended linearly and distributively, this action enriches $M(w)$ with the structure of an $A$-module, as required.
	
	\begin{defn}
		Let $w$ be a string and let $M(w)$ be the $A$-module constructed above. We call $M(w)$ the \emph{string module} corresponding to $w$.
	\end{defn}
	
	\begin{rem} \label{rem:StringInverseAction}
		By a slight abuse of notation, we can define the `action' of a formal inverse $\alpha\inv \in Q_1 \inv$ on $M(w)$ by a similar set of rules to \ref{en:SM1}-\ref{en:SM4} (given in the obvious way by inverting the symbols). This `action' is useful later in the paper, and essentially describes the preimage of the action of $\alpha$ on $M(w)$.
	\end{rem}
	
	\begin{lem} \label{lem:StringIsos}
		Suppose $w=\sigma_1\ldots\sigma_n$ and $w'=\sigma'_1\ldots\sigma'_{n}$ are strings.
		\begin{enumerate}[label=(\alph*)]
			\item If $w \sim_1 w'$ then $M(w) \cong M(w')$.
			\item If $w \sim_2 w'$ then $M(w) \cong M(w')$.
		\end{enumerate}
	\end{lem}
	\begin{proof}
		(a) Obviously if $w' = w$, then $M(w) \cong M(w')$, so suppose instead that $w' = w\inv$. It is easy to verify that $M(w)$ has the same module structure as $M(w')$. Let $b_0,\ldots,b_n$ be the standard basis for $M(w)$ and let $b'_0,\ldots,b'_n$ be the standard basis for $M(w')$. The $A$-module isomorphism $\omega\colon M(w) \rightarrow M(w')$ is then given by $\omega(b_i) = b_{n-i}$.
		
		(b) The statement can be proved by induction. Suppose $w''$ is a string obtained from $w$ by inverting a single crease symbol, say $\sigma_j$. We will show that $M(w) \cong M(w'')$. Without loss of generality, we may assume that $\sigma_j = \crs_v \in \ca$, as otherwise the symbol $\sigma_j\inv$ of $w''$ is in $\ca$ instead, and we can simply relabel $w$ as $w''$ and vice versa.
		
		Let $b_0,\ldots,b_n$ be the standard basis for $M(w)$ and let $b''_0,\ldots,b''_n$ be the standard basis for $M(w'')$. We define a $K$-linear map $\omega\colon M(w) \rightarrow M(w'')$ as follows. Firstly, we define $\omega(b_i)=b''_i$ for each $0 \leq i < j$. Then define 
		\begin{equation*}
			\omega(b_j)=\lambda_1 b''_{j-1} + \lambda_2 b''_j = \lambda_1 \sigma_1\ldots\sigma_{j-1} + \lambda_2 \sigma_1\ldots\sigma_{j-1}\crs_v\inv,
		\end{equation*}
		where $\crs_v^2 - \lambda_1 \crs_v - \lambda_2\stp_v \in \crel$. For each $j < i \leq n$, we define $\omega(b_{i})$ inductively by the following $K$-linear concatenation rule: First we write $\omega(b_{i-1}) = \sum_{k=0}^{i-1} \mu_{k} b''_k$ for some $\mu_k \in K$, and define
		\begin{equation*}
			\omega(b_{i}) =  \omega(b_{i-1})\sigma_{i} = \sum_{k=0}^{i-1} \mu_{k} b''_k\sigma_{i},
		\end{equation*}
		where each individual summand $\mu_{k} b''_k\sigma_{i}$ is identified according to \ref{en:SM1}-\ref{en:SM4} and Remark~\ref{rem:StringInverseAction}.
%		\begin{equation*}
%			b''_k\sigma_{i} =
%			\begin{cases}
%				b''_{k-1} & \text{if } \sigma_{k} = \sigma_{i}\inv, \\
%				b''_{k+1} & \text{if } \sigma_{k+1} = \sigma_{i}, \\
%				\lambda'_{1,u} b''_{k} + \lambda'_{2,u} b''_{k-1} & \text{if } \sigma_{k} = \sigma_{i} = \crs_u \in \ca, \\
%				\frac{1}{\lambda'_{2,u}} b''_{k-1} -\frac{\lambda'_{1,u}}{\lambda'_{2,u}} b''_{k}  & \text{if } \sigma_{k} = \sigma_{i}= \crs_u\inv \in (\ca)\inv, \\
%				0 & \text{otherwise},
%			\end{cases}
%		\end{equation*}
%		and where $\crs_u^2 - \lambda'_{1,u}\crs_u - \lambda'_{2,u}\stp_u \in \crel$. 
		It is then easy to verify that $\omega$ is an isomorphism such that $\omega(b_k)\alpha = \omega(b_k \alpha)$, and is thus an $A$-module isomorphism, as required. Hence, $M(w) \cong M(w'')$. 
		
		To see that $M(w) \cong M(w')$, we note that there must exist a sequence of strings $w=w_0,w_1,\ldots, w_{m-1},w_m=w'$ such that $w_i$ is obtained from $w_{i-1}$ by inverting a single crease symbol. By the above proof, this induces a chain of isomorphisms
		\begin{equation*}
			M(w) = M(w_0) \cong M(w_1) \cong \ldots \cong M(w_{m-1}) \cong M(w_m) = M(w'),
		\end{equation*}
		and so we are done.
	\end{proof}
	
	\begin{exam} \label{ex:C3String}
		Let $w=\alpha\inv\crs_1\alpha\beta$ be the string from Example~\ref{ex:C3Strings}. The corresponding string module $M(w)$ has underlying vector space $\real^5$ with basis elements
		\begin{align*}
			b_0 &= \stp_2, & b_1&=\alpha\inv, & b_2 &= \alpha\inv\crs_1,	\\
			b_3&=\alpha\inv\crs_1\alpha, & b_4 &= \alpha\inv\crs_1\alpha\beta.
		\end{align*}
		One can also interpret the vector space structure of $M(w)$ as $\real \oplus \complex \oplus \real \oplus \real$, as the vector subspace $\langle b_1,b_2 \rangle_{\real} \cong \complex$. We have the following actions on $M(w)$ by the basis elements of the $\real$-algebra given in Example~\ref{ex:C3}:
		\begin{align*}
			b_1\stp_1 &= \alpha\inv \stp_1 = \alpha\inv = b_1, 	&	
			b_1\alpha &= \alpha\inv \alpha = \stp_2 = b_0, \\
			b_2\stp_1 &= \alpha\inv\crs_1 \stp_1 = \alpha\inv\crs_1= b_2, 	&	
			b_2\alpha &= \alpha\inv\crs_1\alpha = b_3, \\
			b_1\crs_1 &= \alpha\inv\crs_1 =b_2,	&	
			b_1\crs_1\alpha &= \alpha\inv\crs_1\alpha = b_3, \\
			b_2\crs_1 &= \alpha\inv\crs_1^2 = -\alpha\inv = -b_1,	&	
			b_2\crs_1\alpha &= \alpha\inv\crs_1^2\alpha = -\alpha\inv\alpha = -b_0, \\
			b_0\stp_2 &= \stp_2^2 = \stp_2 = b_0,	&	
			b_3\beta &= \alpha\inv\crs_1\alpha\beta= b_4, \\
			b_3\stp_2 &= \alpha\inv\crs_1\alpha\stp_2 = \alpha\inv\crs_1\alpha=b_3,	&	
			b_2\alpha\beta &= \alpha\inv\crs_1\alpha\beta = b_4, \\
			b_4\stp_3 &= \alpha\inv\crs_1\alpha\beta \stp_3 = \alpha\inv\crs_1\alpha\beta =b_4,	&	
			b_1\crs_1\alpha\beta &= \alpha\inv\crs_1\alpha\beta = b_4,
		\end{align*}
		with all other actions by basis elements being the zero action. So for example, $b_2\beta=0$ because $\alpha\inv\crs_1\beta$ is not a $\real$-linear combination of the words $b_0,\ldots, b_4$ (it is not even a word itself). On the other hand, $b_0 \beta=0$, because whilst $b_0 \beta=\stp_0\beta=\beta$ is a word, it is also not a $\real$-linear combination of the words $b_0,\ldots, b_4$.
	\end{exam}
	
	\subsubsection{Asymmetric band modules} \label{sec:ABandModules}
	First some notation. Let $w=\sigma_1\ldots\sigma_n$ be an asymmetric band. Since $w$ has a cyclic structure and $A$ is assumed to be finite-dimensional, $w$ must contain at least one arrow symbol in $\oa$. Let $r\leq n$ be the greatest index such that $\sigma_r \in \oa$. Let $m \in \spint$ and consider $m$ distinct copies of the word $w$, which we denote by $w\ps{1}=\sigma_1\ps{1}\ldots\sigma_n\ps{1},\ldots,w\ps{m}=\sigma_1\ps{m}\ldots\sigma_n\ps{m}$ with each $\sigma_i\ps{j} = \sigma_i$.  Similar to the previous subsection, we denote by $b_i\ps{j}$ the subword $\sigma_1\ps{j}\ldots\sigma_i\ps{j}$ of $w\ps{j}$. Finally, let $\phi \in \Aut(K^m)$.
	
	From this, we define a $K$-vector space 
	\begin{equation*}
		M(w,m,\phi)= \langle b_i\ps{j} : 1 \leq i \leq n \text{ and } 1 \leq j \leq m\rangle_K,
	\end{equation*}
	and we call the above listed set of generators the \emph{standard basis} of $M(w,m,\phi)$. The action of any arrow or stationary path $a \in A$ on a subword $b_i\ps{j} \in M(w,m,\phi)$ is then defined by the concatenation $b_i\ps{j}a$, subject to the following identifications:
	\begin{enumerate}[label=(ABM\arabic*)]
		\item For each $j$, we identify $b_n\ps{j}=\sigma_1\ps{j}\ldots\sigma_n\ps{j}=\stp_{s(w)}\ps{j}=b_0\ps{j}$. \label{en:ABM1}
		\item $b_{ r-1}\ps{j} \alpha = \phi(b_r\ps{j})$ whenever $\alpha=\sigma_r\in\oa$, and where we view $\langle b_r\ps{1},\ldots,b_r\ps{m}\rangle_K \cong K^m$. \label{en:ABM2} 
		\item Let $\alpha \in Q_1$ and suppose that $\alpha \neq \sigma_r$ or $i \neq r-1$. Further suppose that $b_i \alpha = \sum_{k=1}^n \lambda_k b_k$ with respect to the identifications \ref{en:ABM1} and \ref{en:SM1}-\ref{en:SM4}. Then $b_i\ps{j} \alpha = \sum_{k=1}^n \lambda_k b_k\ps{j}$ for each $j$. \label{en:ABM3} 
	\end{enumerate}
	Extending this action linearly and distributively then enriches $M(w,m,\phi)$ with the structure of an $A$-module.
	
	\begin{defn} \label{def:ASBandModules}
		Let $w$ be an asymmetric band, $m \in \spint$, and $\phi \in \Aut(K^m)$. We call the module $M(w,m,\phi)$, as constructed above, an \emph{(asymmetric) quasi-band module} corresponding to $w$. If in addition $(K^m,\phi)$ is an $m$-dimensional indecomposable representation of the algebra $K[x,x\inv]$, then we call $M(w,m,\phi)$ an \emph{(asymmetric) band module}.
	\end{defn}
	
	\begin{rem}
		Since $K[x,x\inv]$ is a principal ideal domain, the condition that $(K^m,\phi)$ is an indecomposable representation of $K[x,x\inv]$ implies that $\phi$ is similar to the companion matrix of some $p^n \in K_j[x]$ for some $n \in \spint$, where $p$ is monic and irreducible, $p \neq x$ and $\deg(p^n)=m$. That is, if
		\begin{equation*}
		p^n=x^{m} + a_{m - 1} x^{m-1} + \ldots + a_1 x + a_0,
		\end{equation*}
		then $\phi$ is similar to the matrix
		\begin{equation*}
			C_{p,m} =
			\begin{pmatrix}
				0			&	0		&	\cdots	&	0		&	-a_0			\\
				1			&	0		&	\cdots	&	0		&	-a_1			\\
				0			&	\ddots	&	\ddots	&	\vdots	&	\vdots		\\
				\vdots		&	\ddots	&	1		&	0		&	-a_{m-2}	\\
				0			&	\cdots	&	0		&	1		&	-a_{m-1}.
			\end{pmatrix}.
		\end{equation*}
	\end{rem}
	
	\begin{rem} \label{rem:BandInverseAction}
		We have an analogous remark for band modules to that given for string modules in Remark~\ref{rem:StringInverseAction}. That is, one can also define the `action' of a formal inverse $\alpha\inv$ on $M(w,m,\phi)$ to be given by the preimage of $\alpha$.
	\end{rem}
	
	\begin{lem} \label{lem:BandIsos}
		Let $w=\sigma_1\ldots\sigma_n$ be an asymmetric band let $m \in \spint$ and let $\phi \in \Aut(K^m)$. Then the following hold.
		\begin{enumerate}[label=(\alph*)]
			\item $M(w,m,\phi) \cong M(w\inv,m,\phi\inv)$.
			\item Suppose $\sigma_k =\crs_v \in \ca$ for some $1\leq k \leq n$ and let $w' = \sigma_1\ldots\sigma_{k-1}\sigma_k\inv\sigma_{k+1}\ldots\sigma_n$. Suppose that $\crs_v^2 - \lambda_1 \crs_v - \lambda_2 \stp_v \in \crel$. Then $M(w,m,\phi) \cong M(w',m,\lambda_2\inv\phi)$.
			\item Suppose $w' \sim_3 w$. Then  $M(w,m,\phi) \cong M(w',m,\phi)$.
		\end{enumerate}
	\end{lem}
	\begin{proof}
		In all cases, we will explicitly state the isomorphism. We will prove (c) first, as we use it in the proof for both (a) and (b).
	
		(c) The proof works by induction on rotation. Let $r \leq n$ be the greatest index such that $\sigma_r \in \oa$ and let $\{b_i\ps{j}:1 \leq i \leq n, 1 \leq j \leq m\}$ be the standard basis of $M(w,m,\phi)$. Let $w_1 = \sigma'_{1}\ldots\sigma'_{n}$ with $\sigma'_i = \sigma_{i-1}$ (with respect to the cyclic ordering of the indices that identifies $n+1 = 1$). That is, $w_1$ is given by a single rightward rotation of symbols. Let $\{c_i\ps{j}:1 \leq i \leq n, 1 \leq j \leq m\}$ be the standard basis of $M(w_1,m,\phi)$ and let and $r' \leq n$ be the greatest index such that $\sigma'_{r'} \in \oa$. Here, we need to note that it follows from the way that asymmetric band modules are defined that $\phi(b_i\ps{j})a =\phi(b_i\ps{j}a)$ for any $a \in A$. It is then easy to see that we have an isomorphism $\omega\colon M(w,m,\phi) \rightarrow M(w_1,m,\phi)$ given by $\omega(b_i\ps{j}) = \phi(c_{i+1}\ps{j})$ for $r'-1 \leq i \leq r-1$ (with respect to the cyclic ordering of the indices) and $\omega(b_i\ps{j}) = c_{i+1}\ps{j}$ otherwise.
		
		We therefore have $M(w,m,\phi) \cong M(w_1,m,\phi)$. Now define
		\begin{equation*}
			w_i = \sigma_{n-i+1}\ldots\sigma_n\sigma_1\ldots\sigma_{n-i},
		\end{equation*}
		the rightward rotation of $w$ by $i$ symbols. Then an iterative use of the above argument shows that we have
		\begin{equation*}
			M(w,m,\phi) = M(w_0,m,\phi) \cong M(w_1,m,\phi) \cong \ldots \cong  M(w_{n},m,\phi) = M(w,m,\phi).
		\end{equation*}
		Thus, for any $w' \sim_3 w$, we have $M(w,m,\phi) \cong M(w',m,\phi)$ as required.
	
		(a) Again, let $r \leq n$ be the greatest index such that $\sigma_r \in \oa$ and let $\{b_i\ps{j}:1 \leq i \leq n, 1 \leq j \leq m\}$ be the standard basis of $M(w,m,\phi)$. Since $A$ is finite-dimensional, $w$ must have at least one formal inverse. By (c), we may assume without loss of generality that $r=n$ and $s(\sigma_r)$ is at a peak of the band. In particular, this implies that $r'=n-1$ is the greatest index such that $\sigma_{r'} \in (\oa)\inv$. Now consider the rotated band
		\begin{equation*}
			w\inv_{2} = \sigma_{n-2}\inv\sigma_{n-3}\inv\ldots\sigma_1\inv\sigma_n\inv\ldots\sigma_{n-1}\inv
		\end{equation*}
		of $w\inv$. Let $\{c_i\ps{j}:1 \leq i \leq n, 1 \leq j \leq m\}$ be the standard basis of $M(w\inv_{2},m,\phi\inv)$. As a caution to the reader regarding notation, recall that this means that each $c_n\ps{j}$ is a representative of the complete word $w\inv_{2}$. It is then straightforward to verify that we have an isomorphism $\omega\colon M(w,m,\phi)\rightarrow M(w\inv_{2},m,\phi\inv)$ given by $\omega(b_{n-1}\ps{j})=\phi(c_{n-1}\ps{j})$ and $\omega(b_{n-i}\ps{j})=c_{n-2+i}\ps{j}$ for all $i \neq 1$, with indices modulo $n$. Since $M(w\inv,m,\phi\inv) \cong M(w_{2}\inv,m,\phi\inv)$ by (c), we are done.
		
		(b) We can assume without loss of generality that $k=1$, as we can otherwise rotate the band by (c). Let $\{b_i\ps{j}:1 \leq i \leq n, 1 \leq j \leq m\}$ be the standard basis of $M(w,m,\phi)$, and let $r \leq n$ be the greatest index such that $\sigma_r \in \oa$. Let $\phi' \in \Aut(K^m)$ and let $\{c_i\ps{j}:1 \leq i \leq n, 1 \leq j \leq m\}$ be the standard basis of $M(w',m,\phi')$. We will construct an isomorphism $\omega\colon M(w,m,\phi) \rightarrow M(w',m,\phi')$, and then show that $\omega$ is an $A$-module isomorphism only if $\phi' = \lambda_2\inv \phi$.
		
		The construction is almost identical to the proof of Lemma~\ref{lem:StringIsos}(b). First define
		\begin{equation*}
			\omega(b_1\ps{j}) = \omega(\crs_v\ps{j}) = \lambda_1 \stp_v\ps{j} + \lambda_2 (\crs_v\ps{j})\inv = \lambda_1 c_0\ps{j} + \lambda_2 c_1\ps{j} = \lambda_1 c_n\ps{j} + \lambda_2 c_1\ps{j}.
		\end{equation*}
		We then define $\omega(b_i\ps{j})$ inductively for $i< r$ by 
		\begin{equation*}
			\omega(b_i\ps{j}) = \omega(b_{i-1}\ps{j} \sigma_{i}\ps{j}) = \omega(b_{i-1}\ps{j})\sigma_{i},
		\end{equation*}
		 where $\omega(b_{i-1}\ps{j})\sigma_{i}$ is interpreted as the usual $A$-action if $\sigma_i \in Q_1$ or by Remark~\ref{rem:BandInverseAction} if $\sigma_i \in Q_1\inv$. At $i=r-1$, we have
		\begin{equation*}
			\omega(b_{r-1}\ps{j}) = \lambda_2 c_{r-1}\ps{j} + \lambda_1 d_{r-1},
		\end{equation*}
		where $d_{r-1}$ is some $K$-linear combination of basis vectors. So far, we have defined $\omega$ such that it is compatible with the right action of $A$. If we wish to continue the inductive process with the aim of ensuring that $\omega$ is an $A$-module isomorphism, then we must have the equality
		\begin{align*}
			\omega(b_{r-1}\ps{j})\cdot \sigma_r &=\omega(b_{r-1}\ps{j} \cdot \sigma_r) = \omega(\phi(b_{r}\ps{j})) = \phi\omega(b_{r}\ps{j}) \\
			&= \lambda_2\phi'( c_{r}\ps{j}) + \lambda_1 d_{r}.
		\end{align*}
		Continuing this inductive process into the range $r \leq i \leq n$, we obtain further equalities
		\begin{equation*}
			\phi\omega(b_i\ps{j}) = \lambda_2\phi'( c_{i}\ps{j}) + \lambda_1 d_{i}.
		\end{equation*}
		Since $w$ is asymmetric, the term $d_{i}$ must eventually vanish for some $1 < i \leq n$. So at $i=n$, we are left with $\phi\omega(b_n\ps{j}) = \lambda_2\phi'( c_{n}\ps{j})$. Thus, we must have
		\begin{align*}
			&\phi\omega(b_n\ps{j}) \crs_v = \phi\omega(b_n\ps{j}\crs_v) = \phi\omega(b_1\ps{j}) \\
			&= \lambda_2\phi'( c_{n}\ps{j})\crs_v = \lambda_2\phi'( c_{n}\ps{j})(\lambda_1\stp_v+\lambda_2\crs_v\inv) = \lambda_2\phi'( \lambda_1 c_{n}\ps{j} + \lambda_2 c_{1}\ps{j}),
		\end{align*}
		which from the very first equality, implies that we must have $\phi' = \lambda_2\inv\phi$.
		
		In conclusion, we have constructed an $A$-module isomorphism $M(w,m,\phi)\cong M(w',m,\lambda_2\inv\phi)$, as required.
	\end{proof}
	
	\begin{exam}
		Consider the $\rational$-algebra from Example~\ref{ex:TripleCrease}. As stated in Example~\ref{ex:TripleCreaseBands}, this has an asymmetric band $w=\gamma\delta\inv\gamma\inv\crs_3$. Suppose $m=3$. The module $M(w,m,\phi)$ then has underlying vector space $\rational^{12} \cong \rational^3 \oplus \rational^3 \oplus \rational(\iu)^3$. In particular, the standard basis of $M(w,m,\phi)$ is given by the words
		\begin{align*}
			b_1\ps{j} &= \gamma\ps{j}, & 
			b_2\ps{j} &= \gamma\ps{j}(\delta\inv)\ps{j}, \\
			b_3\ps{j} &= \gamma\ps{j}(\delta\inv)\ps{j}(\gamma\inv)\ps{j}, &
			b_4\ps{j} &= \gamma\ps{j}(\delta\inv)\ps{j}(\gamma\inv)\ps{j}\crs_3\ps{j},
		\end{align*}
		with $1 \leq j \leq 3$. Suppose that $\phi$ is given by the companion matrix of the irreducible polynomial $x^3 -x^2-2x+1 \in \rational[x,x\inv]$. Then the non-zero actions of basis elements of the algebra (listed in Example~\ref{ex:TripleCrease}) on $M(w,m,\phi)$ are given by the following:
		\begin{align*}
			b_3\ps{j} \stp_3 &= b_3\ps{j}, &
			b_4\ps{1} \gamma &= b_1\ps{2}, &
			b_3\ps{j} \gamma\delta &= b_1\ps{j}, \\
			b_4\ps{j} \stp_3 &= b_4\ps{j}, &
			b_4\ps{2} \gamma &=  b_1\ps{3}, &
			b_3\ps{j} \crs_3\gamma &= b_4\ps{j}\gamma, \\
			b_1\ps{j} \stp_4 &= b_1\ps{j}, &
			b_4\ps{3} \gamma &= -b_1\ps{1}+2b_1\ps{2}+b_1\ps{3}, &
			b_4\ps{j} \crs_3\gamma &= 2b_4\ps{j}\gamma  - 5b_2\ps{j}, \\
			b_2\ps{j} \stp_4 &= b_2\ps{j}, &
			b_3\ps{j} \gamma &=  b_2\ps{j}, &
			b_4\ps{j} \crs_3\gamma\delta &= -5b_1\ps{j}, \\
			b_3\ps{j}\crs_3 &=  b_4\ps{j}, &
			b_2\ps{j}\delta &=  b_1\ps{j}, \\
			b_4\ps{j}\crs_3 &= 2b_4\ps{j} - 5b_3\ps{j}.
		\end{align*}
	\end{exam}
	
	\subsubsection{Symmetric band modules}
	Let $w$ be a symmetric band. Note that this implies that $w$ can be uniquely written (up to equivalence) in the form
	\begin{equation*}
		w= \crs\sigma_1\ldots\sigma_n\crs'\sigma_n\inv\ldots\sigma_1\inv
	\end{equation*}
	for some distinguished crease symbols $\crs,\crs' \in \ca \cup (\ca)\inv$ about which the band is symmetric. We will assume this form throughout the definition. Let $\crs^2 - \mu_{1} \crs - \mu_{2} \stp_{s(\crs)}$ and $(\crs')^2 - \mu'_{1} \crs' - \mu'_{2} \stp_{s(\crs')}$ be the corresponding crease relations of $\crs$ and $\crs'$ (see Remark~\ref{rem:InverseCrease} for inverses of creases). Again, let $m \in \spint$. In this case, we will consider $2m$ distinct copies of the subword $\sigma_1\ldots\sigma_n$, where we label the $j$-th copy by $\sigma_1\ps{j}\ldots\sigma_n\ps{j}$. Denote by $b_i\ps{j}$ the subword $\sigma_1\ps{j}\ldots\sigma_i\ps{j}$ with $b_0\ps{j} = \stp_{s(\sigma_1)}\ps{j}$. Finally, we let $\phi \in \Aut(K^{2m})$ be such that $\phi^2 - \mu'_{1} \phi - \mu'_{2} =0$.
	
	We then define a $K$-vector space
	\begin{equation*}
		M(w,m,\phi) = \langle b_i\ps{j} : 0 \leq i \leq n \text{ and } 1 \leq j \leq 2m \rangle_K,
	\end{equation*}	
	and we call the above listed set of generators the \emph{standard basis} of $M(w,m,\phi)$. The action of any ordinary arrow, crease symbol, or stationary path $a \in A$ on a subword $b_i\ps{j} \in M(w,m,\phi)$ is then defined by the concatenation $b_i\ps{j}a$, subject to the following identifications:
	\begin{enumerate}[label=(SBM\arabic*)]
		\item \label{en:SBM1} $b_0\ps{j} \crs = b_0\ps{j+m}$ and $b_0\ps{j+m} \crs = \mu_{1}b_0\ps{j+m} + \mu_{2}b_0\ps{j}$ (assuming $1 \leq j\leq m$).
		\item \label{en:SBM2} $b_n\ps{j} \crs' = \phi(b_n\ps{j})$.
		\item Let $\alpha \in Q_1$ and suppose that $b_i \alpha = \sum_{k=1}^n \lambda_k b_k$ with respect to the identifications \ref{en:SM1}-\ref{en:SM4}. Then $b_i\ps{j} \alpha = \sum_{k=1}^n \lambda_k b_k\ps{j}$ for each $j$. \label{en:SBM4} 
	\end{enumerate}

	Extending this action linearly and distributively then enriches $M(w,m,\phi)$ with the structure of an $A$-module.
	
	\begin{defn}
		Let $w$ be a symmetric band, $m \in \spint$, and $\phi \in \Aut(K^{2m})$. We call the module $M(w,m,\phi)$, as constructed above, a \emph{(symmetric) quasi-band module} corresponding to $w$. If, in addition, $\left(K^{2m},\left(
		\begin{smallmatrix}
			0	&	\mu_2 \\
			1	&	\mu_1
		\end{smallmatrix}\right)
		\otimes_K \id_m, \phi\right)$ is an indecomposable representation of the algebra
		\begin{equation*}
			\fiso_w = K\langle x,y \rangle / \langle x^2 - \mu_1 x - \mu_2, y^2 - \mu'_1 y - \mu'_2 \rangle,
		\end{equation*}
		then we call $M(w,m,\phi)$ a \emph{(symmetric) band module}.
	\end{defn}
	
	The representation theory of the algebra $\fiso_w$ is highly dependent on the choice of field $K$ and field extensions $\fiso_{s(\crs)}$ and $\fiso_{s(\crs')}$. It is not hard to show that any representation of $\fiso_w$ is isomorphic to a representation
	\begin{equation*}
		\left(K^{2m},\left(
		\begin{smallmatrix}
			0	&	\mu_2 \\
			1	&	\mu_1
		\end{smallmatrix}\right)
		\otimes_K \id_m, \psi\inv\left(\left(
		\begin{smallmatrix}
			0	&	\mu'_2 \\
			1	&	\mu'_1
		\end{smallmatrix}\right)\otimes_K \id_m\right)\psi\right)
	\end{equation*}
	for some $\psi \in \Aut(K^{2m})$. Thus, an equivalent condition for a symmetric quasi-band module to be a band module is for $(\fiso_{v_0}^m,\fiso_{v_n}^m,\psi)$ to be an indecomposable homogeneous representation of the $K$-species	
	\begin{equation*}
		\xymatrix@1{\fiso_{s(\crs)} \ar[rr]^-{\fiso_{s(\crs)} \otimes_K \fiso_{s(\crs')}} && \fiso_{s(\crs')}},
	\end{equation*}
	See \cite[pp. 2-4, Theorem (b); Sec. 5]{DRIndec} and \cite{RingelSpecies} for details. This paper will later give an alternative characterisation of these indecomposable representations in terms of (quasi-)band modules of a Kronecker quiver (see the proof of Corollary~\ref{cor:Kxy} for details).

	\section{Unfolded Gentle Algebras} \label{sec:UnfoldedGentle}
	In this section, we will justify our choice of the name `folded gentle algebras' by showing that this class of algebras is closely related to the quotient of a (covariant) group action on both a quiver $\unfQ$ and set of zero relations $\unfZ$ for some gentle pair $(\unfQ,\unfZ)$. Such a procedure is commonly referred to in the literature as \emph{folding}.
	
	\subsection{The unfolding procedure} We will begin by constructing a gentle pair $(\unfQ,\unfZ)$ that corresponds to a folded gentle triple $(K,Q,Z)$, which we will later show satisfies the properties that we require. We call this the \emph{unfolding procedure}.
	
	So let $(K,Q,Z)$ be a folded gentle triple (that is not a gentle triple). We define $\unfQ$ in the following way.
	\begin{enumerate}[label=(U.\roman*)]
		\item For each ordinary vertex $v \in \ov$, there exist precisely two corresponding vertices $\wh{v}, \wh{v}' \in \unfQ_0$. \label{en:U1}
		\item For each crease vertex $v \in \cv$, there exists precisely one corresponding vertex $\wh{v}=\wh{v}' \in \unfQ_0$. \label{en:U2}
			\item For each ordinary arrow $\alpha\colon u \rightarrow v \in \oa$, there exist precisely two corresponding arrows in $\unfQ_1$, namely $\wh{\alpha}\colon \wh{u} \rightarrow \wh{v}$ and $\wh{\alpha}'\colon \wh{u}' \rightarrow \wh{v}'$. \label{en:U3}
	\end{enumerate}
	This completes the definition of $\unfQ$. The set $\unfZ$ is then defined as follows.
	\begin{enumerate}[label=(U.\roman*)]
		\setcounter{enumi}{3}
		\item For each relation $\alpha\beta \in \orel$, we have precisely two corresponding relations $\wh{\alpha}\wh{\beta}, \wh{\alpha}'\wh{\beta}' \in \unfZ$. \label{en:U4}
	\end{enumerate}
	This completes the definition of $\unfZ$.
	
	This defines a pair $(\unfQ,\unfZ)$, and we consider $\unfQ$ to have no crease vertices. Thus to show $(\unfQ,\unfZ)$ is a gentle pair, we need only verify \ref{en:FG1}-\ref{en:FG4}. This is straightforward to check. For each ordinary vertex $v \in \ov$ with up to two incoming and up to two outgoing arrows the unfolding procedure results in the following.
	\begin{center}
		\begin{tikzpicture}
			\coordinate (v2) at (1,0);
			\coordinate (v1) at ($(v2)-(1,-0.5)$);
			\coordinate (v3) at ($(v2)+(1,0.5)$);
			\coordinate (v4) at ($(v2)+(-1,-0.5)$);
			\coordinate (v5) at ($(v2)+(1,-0.5)$);
			
			\coordinate (sq) at (3,0);
			
			\coordinate (v2p) at (5,0) {} {};
			\coordinate (v1p) at ($(v2p)+(-1,0.5)$);
			\coordinate (v4p) at ($(v2p)+(-1,-0.5)$);
			\coordinate (v3p) at ($(v2p)+(1,0.5)$);
			\coordinate (v5p) at ($(v2p)+(1,-0.5)$);
			\coordinate (v2pp) at (8,0) {} {} {};
			\coordinate (v1pp) at ($(v2pp)+(-1,0.5)$);
			\coordinate (v4pp) at ($(v2pp)+(-1,-0.5)$);
			\coordinate (v3pp) at ($(v2pp)+(1,0.5)$);
			\coordinate (v5pp) at ($(v2pp)+(1,-0.5)$);
			
			\draw (v2) node {$v$};
			
			\draw ($(v1)+0.5*(v2)-0.5*(v1)+(0.1,0.3)$) node {\footnotesize$\alpha$};
			\draw[->, shorten <=1.5ex, shorten >= 1.5ex] (v1) -- (v2);
			\draw ($(v2)+0.5*(v3)-0.5*(v2)+(-0.1,0.3)$) node {\footnotesize$\beta$};
			\draw[->, shorten <=1.5ex, shorten >= 1.5ex] (v2) -- (v3);
			\draw ($(v4)+0.5*(v2)-0.5*(v4)+(0.1,-0.3)$) node {\footnotesize$\gamma$};
			\draw[->, shorten <=1.5ex, shorten >= 1.5ex] (v4) -- (v2);
			\draw ($(v2)+0.5*(v5)-0.5*(v2)+(-0.1,-0.3)$) node {\footnotesize$\delta$};
			\draw[->, shorten <=1.5ex, shorten >= 1.5ex] (v2) -- (v5);
			
			\draw (sq) node {$\rightsquigarrow$};
			
			\draw (v2p) node {$\wh{v}$};
			\draw (v2pp) node {$\wh{v}'$};
			
			\draw ($(v1p)+0.5*(v2p)-0.5*(v1p)+(0.1,0.3)$) node {\footnotesize$\wh{\alpha}$};
			\draw[->, shorten <=1.5ex, shorten >= 1.5ex] (v1p) -- (v2p);
			\draw ($(v4p)+0.5*(v2p)-0.5*(v4p)+(0.1,-0.3)$) node {\footnotesize$\wh{\gamma}$};
			\draw[->, shorten <=1.5ex, shorten >= 1.5ex] (v4p) -- (v2p);
			\draw ($(v3p)+0.5*(v2p)-0.5*(v3p)+(-0.1,0.3)$) node {\footnotesize$\wh{\beta}$};
			\draw[->, shorten <=1.5ex, shorten >= 1.5ex] (v2p) -- (v3p);
			\draw ($(v5p)+0.5*(v2p)-0.5*(v5p)+(-0.1,-0.3)$) node {\footnotesize$\wh{\delta}$};
			\draw[->, shorten <=1.5ex, shorten >= 1.5ex] (v2p) -- (v5p);
			\draw ($(v1pp)+0.5*(v2pp)-0.5*(v1pp)+(0.1,0.3)$) node {\footnotesize$\wh{\alpha}'$};
			\draw[->, shorten <=1.5ex, shorten >= 1.5ex] (v1pp) -- (v2pp);
			\draw ($(v4pp)+0.5*(v2pp)-0.5*(v4pp)+(0.1,-0.3)$) node {\footnotesize$\wh{\gamma}'$};
			\draw[->, shorten <=1.5ex, shorten >= 1.5ex] (v4pp) -- (v2pp);
			\draw ($(v3pp)+0.5*(v2pp)-0.5*(v3pp)+(-0.1,0.3)$) node {\footnotesize$\wh{\beta}'$};
			\draw[->, shorten <=1.5ex, shorten >= 1.5ex] (v2pp) -- (v3pp);
			\draw ($(v5pp)+0.5*(v2pp)-0.5*(v5pp)+(-0.1,-0.3)$) node {\footnotesize$\wh{\delta}'$};
			\draw[->, shorten <=1.5ex, shorten >= 1.5ex] (v2pp) -- (v5pp);

			\draw (4,-1.2) node {$\alpha\beta, \gamma\delta \in Z \qquad \qquad \qquad \qquad \wh{\alpha}\wh{\beta}, \wh{\alpha}'\wh{\beta}', \wh{\gamma}\wh{\delta}, \wh{\gamma}'\wh{\delta}' \in \unfZ$};
			\draw [dashed](0.7168,0.2057) arc (144.0075:36:0.35);
			\draw [dashed](0.7168,-0.2057) arc (-144.0075:-36:0.35);
			\draw [dashed](4.7056,0.1892) arc (147.2726:32.7273:0.35);
			\draw [dashed](4.7056,-0.1892) arc (-147.2726:-32.7273:0.35);
			\draw [dashed](7.7056,0.1892) arc (147.2726:32.7273:0.35);
			\draw [dashed](7.7056,-0.1892) arc (-147.2726:-32.7273:0.35);
		\end{tikzpicture}
	\end{center}
	For each crease vertex $v \in \cv$ incident to up to two ordinary arrows, the unfolding procedure results in the following.
	\begin{center}
		\begin{tikzpicture}
\coordinate (lab) at (0,0.2);
			\coordinate (v2) at (1,0);
			\coordinate (v1) at ($(v2)-(1,0)$);
			\coordinate (v3) at ($(v2)+(1,0)$);
			\coordinate (Z) at ($(v2)-(0,1.2)$);
			
			\coordinate (sq) at (3,0);
			
			\coordinate (v2p) at (5,0);
			\coordinate (v1p) at ($(v2p)+(-1,0.5)$);
			\coordinate (v1pp) at ($(v2p)+(-1,-0.5)$);
			\coordinate (v3p) at ($(v2p)+(1,0.5)$);
			\coordinate (v3pp) at ($(v2p)+(1,-0.5)$);
			\coordinate (Zp) at ($(v2p)-(0,1.2)$);
			
			%\draw (v1) node {$u_1$};
			\draw (v2) node {$v$};
			%\draw (v3) node {$u_3$};
			
			\draw ($(v1)+0.5*(v2)-0.5*(v1)+(lab)$) node {\footnotesize$\alpha$};
			\draw[->, shorten <=1.5ex, shorten >= 1.5ex] (v1) -- (v2);
			\draw ($(v2)+0.5*(v3)-0.5*(v2)+(lab)$) node {\footnotesize$\beta$};
			\draw[->, shorten <=1.5ex, shorten >= 1.5ex] (v2) -- (v3);
			\draw [->](1.1,0.2) .. controls (1.3,0.5) and (1.1,0.6) .. (1,0.6) .. controls (0.9,0.6) and (0.7,0.5) .. (0.9,0.2);
			\draw ($(v2)+(0,0.8)$) node {\footnotesize$\crs_v$};
			\draw [dashed](0.61,-0.089) arc (-167.1449:-12.8571:0.4);
			
			\draw (Z) node {$\alpha\beta \in Z$};
			
			\draw (sq) node {$\rightsquigarrow$};
			
			\draw (v2p) node {$\wh{v}$};
			
			\draw ($(v1p)+0.5*(v2p)-0.5*(v1p)+(0.1,0.3)$) node {\footnotesize$\wh{\alpha}$};
			\draw[->, shorten <=1.5ex, shorten >= 1.5ex] (v1p) -- (v2p);
			\draw ($(v1pp)+0.5*(v2p)-0.5*(v1pp)+(0.1,-0.3)$) node {\footnotesize$\wh{\alpha}'$};
			\draw[->, shorten <=1.5ex, shorten >= 1.5ex] (v1pp) -- (v2p);
			\draw ($(v3p)+0.5*(v2p)-0.5*(v3p)+(-0.1,0.3)$) node {\footnotesize$\wh{\beta}'$};
			\draw[->, shorten <=1.5ex, shorten >= 1.5ex] (v2p) -- (v3p);
			\draw ($(v3pp)+0.5*(v2p)-0.5*(v3pp)+(-0.1,-0.3)$) node {\footnotesize$\wh{\beta}'$};
			\draw[->, shorten <=1.5ex, shorten >= 1.5ex] (v2p) -- (v3pp);
			\draw [dashed](4.6873,-0.2494) arc (-141.4252:-38.5714:0.4);
			\draw [dashed](4.6873,0.2494) arc (141.4252:38.5714:0.4);
			
			\draw (Zp) node {$\wh{\alpha}\wh{\beta}, \wh{\alpha}'\wh{\beta}' \in \unfZ$};
		\end{tikzpicture}
	\end{center}
	
	\begin{defn}
		Let $(K,Q,Z)$ be a folded gentle triple. We call the gentle pair $(\unfQ,\unfZ)$ constructed from $(K,Q,Z)$, as above, the corresponding \emph{unfolded gentle pair}. We call the gentle algebra $K \unfQ / \langle \unfZ \rangle$ the \emph{unfolded gentle algebra} corresponding to the folded gentle algebra $K Q / \langle Z \rangle$.
	\end{defn}
	
	Readers familiar with skewed-gentle algebras will recognise the above construction. The unfolded gentle algebra $\wh A$ associated to a folded gentle algebra $A$ is precisely the skew-group algebra $\wt{B}$ associated to the skewed-gentle algebra $B$, where $B$ is given by replacing each crease relation of $A$ with a special (idempotent) relation. See \cite[Proposition 4.6]{SkewedGentle} for details.
	
	\begin{exam} \label{ex:C3Unf}
		The folded gentle triple $(\real,Q,Z)$ from Example~\ref{ex:C3} unfolds to a gentle pair $(\unfQ,\unfZ)$ of Dynkin type $\dynA_5$, as shown below.
		\begin{equation*}
			\begin{tikzpicture}
				\draw [anchor = east] (-5.6,0) node {$Q\colon$};
				\draw (-5.4,0) node [draw] {1};
				\draw[<-] (-5.3,0.4) .. controls (-5.1,0.7) and (-5.3,0.8) .. (-5.4,0.8) .. controls (-5.5,0.8) and (-5.7,0.7) .. (-5.5,0.4);
				\draw (-5,0.6) node {\footnotesize$\eta_1$};
				\draw [->] (-5,0) -- (-4.4,0);
				\draw (-4.7,0.2) node {\footnotesize$\alpha$};
				\draw (-4.2,0) node {2};
				\draw [->] (-4,0) -- (-3.4,0);
				\draw (-3.7,0.2) node {\footnotesize$\beta$};
				\draw (-3.2,0) node {3};

				\draw (-2.1,0) node {$\rightsquigarrow$};
				
				\draw [anchor = east] (-0.4,0) node {$\unfQ\colon$};
				\draw (-0.2,0) node {$\wh{1}$};
				\draw [->] (0.1,0.1) -- (0.7,0.5);
				\draw (0.3,0.5) node {\footnotesize$\wh{\alpha}$};
				\draw (1,0.6) node {$\wh{2}$};
				\draw [->] (1.2,0.6) -- (1.8,0.6);
				\draw (1.5,0.8) node {\footnotesize$\wh{\beta}$};
				\draw (2,0.6) node {$\wh{3}$};
				\draw [->] (0.1,-0.1) -- (0.7,-0.5);
				\draw (0.2,-0.5) node {\footnotesize$\wh{\alpha}'$};
				\draw (1,-0.6) node {$\wh{2}'$};
				\draw [->] (1.2,-0.6) -- (1.8,-0.6);
				\draw (1.5,-0.4) node {\footnotesize$\wh{\beta}'$};
				\draw (2,-0.6) node {$\wh{3}'$};
			\end{tikzpicture}
		\end{equation*}
		This represents the classical folding of Dynkin diagrams $\dynA_5 \rightarrow \dynC_3$.
	\end{exam}

	\begin{exam} \label{ex:C2tildeUnf}
		The folded gentle triple $(\rational,Q,Z)$ from Example~\ref{ex:C2tilde} unfolds to a gentle pair $(\unfQ,\unfZ)$ of affine Dynkin/Euclidean type $\dynafA_4$, as shown below.
		\begin{equation*}
			\begin{tikzpicture}
				\draw [anchor = east] (-5.6,0) node {$Q\colon$};
				\draw (-5.4,0) node [draw] {1};
				\draw[<-] (-5.3,0.4) .. controls (-5.1,0.7) and (-5.3,0.8) .. (-5.4,0.8) .. controls (-5.5,0.8) and (-5.7,0.7) .. (-5.5,0.4);
				\draw (-5,0.6) node {\footnotesize$\eta_1$};
				\draw [->] (-5,0) -- (-4.4,0);
				\draw (-4.7,0.2) node {\footnotesize$\alpha$};
				\draw (-4.2,0) node {2};
				\draw [->] (-4,0) -- (-3.4,0);
				\draw (-3.7,0.2) node {\footnotesize$\beta$};
				\draw (-3,0) node [draw] {3};
				\draw[<-] (-2.9,0.4) .. controls (-2.7,0.7) and (-2.9,0.8) .. (-3,0.8) .. controls (-3.1,0.8) and (-3.3,0.7) .. (-3.1,0.4);
				\draw (-2.6,0.6) node {\footnotesize$\eta_3$};
				
				\draw (-2.1,0) node {$\rightsquigarrow$};
				
				\draw [anchor = east] (-0.4,0) node {$\unfQ\colon$};
				\draw (-0.2,0) node {$\wh{1}$};
				\draw [->] (0.1,0.1) -- (0.7,0.5);
				\draw (0.3,0.5) node {\footnotesize$\wh{\alpha}$};
				\draw (1,0.6) node {$\wh{2}$};
				\draw [->] (1.2,0.5) -- (1.8,0.1);
				\draw (1.7,0.5) node {\footnotesize$\wh{\beta}$};
				\draw (2.1,0) node {$\wh{3}$};
				\draw [->] (0.1,-0.1) -- (0.7,-0.5);
				\draw (0.2,-0.5) node {\footnotesize$\wh{\alpha}'$};
				\draw (1,-0.6) node {$\wh{2}'$};
				\draw [->] (1.2,-0.5) -- (1.8,-0.1);
				\draw (1.7,-0.5) node {\footnotesize$\wh{\beta}'$};
			\end{tikzpicture}
		\end{equation*}
		This represents the classical folding of affine Dynkin/Euclidean diagrams $\dynafA_4 \rightarrow \dynafC_2$.
	\end{exam}
	
	\begin{exam}
		The folded gentle triple $(\finite_2,Q,Z)$ from Example~\ref{ex:C3TildeD4} unfolds to a gentle pair $(\unfQ,\unfZ)$ as follows.
		\begin{equation*}
			\begin{tikzpicture}
				\draw [anchor = east] (-9.9,0) node {$Q\colon$};
				\draw (-9.7,0) node [draw] {1};
				\draw[<-] (-9.6,0.4) .. controls (-9.4,0.7) and (-9.6,0.8) .. (-9.7,0.8) .. controls (-9.8,0.8) and (-10,0.7) .. (-9.8,0.4);
				\draw (-9.3,0.6) node {\footnotesize$\eta_1$};
				\draw [<-] (-9.3,0) -- (-8.7,0);
				\draw (-9,0.2) node {\footnotesize$\alpha$};
				\draw (-8.5,0) node {2};
				\draw [<-] (-8.3,0) -- (-7.7,0);
				\draw (-8,0.2) node {\footnotesize$\beta$};
				\draw (-7.5,0) node {3};
				\draw [->] (-7.3,0) -- (-6.7,0);
				\draw (-7,0.2) node {\footnotesize$\gamma$};
				\draw (-6.3,0) node [draw] {4};
				\draw[<-] (-6.2,0.4) .. controls (-6,0.7) and (-6.2,0.8) .. (-6.3,0.8) .. controls (-6.4,0.8) and (-6.6,0.7) .. (-6.4,0.4);
				\draw (-5.9,0.6) node {\footnotesize$\eta_4$};
				\draw [->] (-5.9,0) -- (-5.3,0);
				\draw (-5.6,0.2) node {\footnotesize$\delta$};
				\draw (-5.1,0) node {5};
				\draw [->] (-4.9,0.1) -- (-4.3,0.3);
				\draw (-4.6,0.5) node {\footnotesize$\zeta_1$};
				\draw (-4.1,0.4) node {6};
				\draw [->] (-4.9,-0.1) -- (-4.3,-0.3);
				\draw (-4.6,-0.5) node {\footnotesize$\zeta_2$};
				\draw (-4.1,-0.4) node {7};
				
				\draw (-3,0) node {$\rightsquigarrow$};
				
				\draw [anchor = east] (-1.4,0) node {$\unfQ\colon$};
				\draw (-1.2,0) node {$\wh{1}$};
				
				\draw [<-] (-0.9,0.1) -- (-0.3,0.5);
				\draw (-0.7,0.5) node {\footnotesize$\wh{\alpha}$};
				\draw (0,0.6) node {$\wh{2}$};
				\draw [<-] (0.2,0.6) -- (0.8,0.6);
				\draw (0.5,0.8) node {\footnotesize$\wh{\beta}$};
				\draw (1,0.6) node {$\wh{3}$};
				\draw [->] (1.2,0.5) -- (1.9,0.1);
				\draw (1.6,0.5) node {\footnotesize$\wh{\gamma}$};
				
				\draw [<-] (-0.9,-0.1) -- (-0.3,-0.5);
				\draw (-0.7,-0.6) node {\footnotesize$\wh{\alpha}'$};
				\draw (0,-0.6) node {$\wh{2}'$};
				\draw [<-] (0.2,-0.6) -- (0.8,-0.6);
				\draw (0.5,-0.4) node {\footnotesize$\wh{\beta}'$};
				\draw (1,-0.6) node {$\wh{3}'$};
				\draw [->] (1.2,-0.5) -- (1.9,-0.1);
				\draw (1.6,-0.6) node {\footnotesize$\wh{\gamma}'$};
				
				\draw (2.2,0) node {$\wh{4}$};
				
				\draw [->] (2.5,0.1) -- (3.2,0.7);
				\draw (2.8,0.6) node {\footnotesize$\wh{\delta}$};
				\draw (3.4,0.8) node {$\wh{5}$};
				\draw [->] (3.6,0.9) -- (4.2,1.1);
				\draw (3.9,1.3) node {\footnotesize$\wh{\zeta}_1$};
				\draw (4.4,1.2) node {$\wh{6}$};
				\draw [->] (3.6,0.7) -- (4.2,0.5);
				\draw (3.9,0.3) node {\footnotesize$\wh{\zeta}_2$};
				\draw (4.4,0.4) node {$\wh{7}$};
				
				\draw [->] (2.5,-0.1) -- (3.2,-0.7);
				\draw (2.8,-0.6) node {\footnotesize$\wh{\delta}'$};
				\draw (3.4,-0.8) node {$\wh{5}'$};
				\draw [->] (3.6,-0.7) -- (4.2,-0.5);
				\draw (3.9,-0.3) node {\footnotesize$\wh{\zeta}'_1$};
				\draw (4.4,-0.4) node {$\wh{6}'$};
				\draw [->] (3.6,-0.9) -- (4.2,-1.1);
				\draw (3.9,-1.3) node {\footnotesize$\wh{\zeta}'_2$};
				\draw (4.4,-1.2) node {$\wh{7}'$};
				
				\draw (-7.2,-2) node {$Z=\{\gamma\delta,\delta\zeta_2\}$};
				\draw (1.4,-2) node {$\unfZ=\{\wh{\gamma}\wh{\delta},\wh{\gamma}'\wh{\delta}',\wh{\delta}\wh{\zeta}_2,\wh{\delta}'\wh{\zeta}'_2\}$};
				
				\draw [dashed](-6.885,-0.1335) arc (-167.1449:-12.8571:0.6);
				\draw [dashed](-5.49,-0.089) arc (-167.1449:-38.5714:0.4);
				\draw [dashed](1.8091,-0.3117) arc (-141.4315:-38.5714:0.5);
				\draw [dashed](1.7495,0.2169) arc (154.2908:38.5714:0.5);
				\draw [dashed](3.01,-0.711) arc (167.1449:321.4286:0.4);
				\draw [dashed](3.0873,0.5506) arc (-141.4252:-38.5714:0.4);
			\end{tikzpicture}
		\end{equation*}
	\end{exam}
	
	\begin{exam} \label{ex:DoubleCreaseUnf}
		Let $(\rational,Q,Z)$ be as in Example~\ref{ex:TripleCrease}. Then the corresponding unfolded gentle pair is as follows.
		\begin{equation*}
			\begin{tikzpicture}
				\draw[<-] (-7.1,0.4) .. controls (-6.9,0.7) and (-7.1,0.8) .. (-7.2,0.8) .. controls (-7.3,0.8) and (-7.5,0.7) .. (-7.3,0.4);
				\draw (-6.8,0.6) node {\footnotesize$\eta_2$};
				\draw[<-] (-5.7,0.4) .. controls (-5.5,0.7) and (-5.7,0.8) .. (-5.8,0.8) .. controls (-5.9,0.8) and (-6.1,0.7) .. (-5.9,0.4);
				\draw (-5.4,0.6) node {\footnotesize$\eta_3$};
				
				\draw (-8.4,0) node {1};
				\draw [->] (-8.2,0) -- (-7.6,0);
				\draw (-7.9,0.2) node {\footnotesize$\alpha$};
				\draw (-7.2,0) node [draw] {2};
				\draw [->] (-6.8,0) -- (-6.2,0);
				\draw (-6.5,0.2) node {\footnotesize$\beta$};
				\draw (-5.8,0) node [draw] {3};
				\draw [->] (-5.4,0) -- (-4.8,0);
				\draw (-5.1,0.2) node {\footnotesize$\gamma$};
				\draw (-4.6,0) node {4};
				\draw[->] (-4.4,-0.2) .. controls (-4,-0.4) and (-3.8,-0.2) .. (-3.8,0) .. controls (-3.8,0.2) and (-4,0.4) .. (-4.4,0.2);
				\draw (-3.6,0) node {\footnotesize$\delta$};

				\draw [dashed](-4.2,-0.2) .. controls (-4.1,-0.15) and (-4.1,-0.1) .. (-4.1,0) .. controls (-4.1,0.1) and (-4.1,0.15) .. (-4.2,0.2);
				\draw [dashed](-7.7826,-0.1436) arc (-166.1537:-13.8462:0.6);
				\draw [dashed](-6.3826,-0.1436) arc (-166.1537:-13.8462:0.6);

				\draw (-2.6,0) node {$\rightsquigarrow$};
				
				\draw (-1.8,0.6) node {$\wh{1}$};
				\draw [->] (-1.5,0.5) -- (-0.9,0.1);
				\draw (-1.2,0.6) node {\footnotesize$\wh{\alpha}$};
				\draw (-1.8,-0.6) node {$\wh{1}'$};
				\draw [->] (-1.5,-0.5) -- (-0.9,-0.1);
				\draw (-1.2,-0.6) node {\footnotesize$\wh{\alpha}'$};
				\draw (-0.6,0) node [] {$\wh{2}$};
				\draw [->] (-0.2,0.1) -- (0.4,0.1);
				\draw (0.1,0.3) node {\footnotesize$\wh{\beta}$};
				\draw [->] (-0.2,-0.1) -- (0.4,-0.1);
				\draw (0.1,-0.4) node {\footnotesize$\wh{\beta}'$};
				\draw (0.8,0) node [] {$\wh{3}$};
				\draw [->] (1,0.1) -- (1.8,0.5);
				\draw (1.4,0.6) node {\footnotesize$\wh{\gamma}$};
				\draw (2,0.6) node {$\wh{4}$};
				\draw[->] (2.2,0.4) .. controls (2.6,0.2) and (2.8,0.4) .. (2.8,0.6) .. controls (2.8,0.8) and (2.6,1) .. (2.2,0.8);
				\draw (3,0.6) node {\footnotesize$\wh{\delta}$};
				\draw [->] (1,-0.1) -- (1.8,-0.5);
				\draw (1.4,-0.6) node {\footnotesize$\wh{\gamma}'$};
				\draw (2,-0.6) node {$\wh{4}'$};
				\draw[->] (2.2,-0.8) .. controls (2.6,-1) and (2.8,-0.8) .. (2.8,-0.6) .. controls (2.8,-0.4) and (2.6,-0.2) .. (2.2,-0.4);
				\draw (3.1,-0.6) node {\footnotesize$\wh{\delta}'$};

				\draw (-6.1,-2.2) node {$Z=\{\alpha\beta,\beta\gamma,\delta^2\}$};
				\draw (0.6,-2.2) node {$\unfZ=\{ \wh{\alpha}\wh{\beta}, \wh{\alpha}'\wh{\beta}', \wh{\beta}\wh{\gamma}, \wh{\beta}'\wh{\gamma}', \wh{\delta}^2, (\wh{\delta}')^2 \}$};

				\draw [dashed](-0.9909,-0.3117) arc (-141.4315:-25.7143:0.5);
				\draw [dashed](-0.9909,0.3117) arc (141.4315:25.7143:0.5);
				\draw [dashed](0.3495,0.2169) arc (154.2908:38.5714:0.5);
				\draw [dashed](0.3495,-0.2169) arc (-154.2908:-38.5714:0.5);
				\draw [dashed](2.4,0.4) .. controls (2.5,0.45) and (2.5,0.5) .. (2.5,0.6) .. controls (2.5,0.7) and (2.5,0.75) .. (2.4,0.8);
				\draw [dashed](2.4,-0.8) .. controls (2.5,-0.75) and (2.5,-0.7) .. (2.5,-0.6) .. controls (2.5,-0.5) and (2.5,-0.45) .. (2.4,-0.4);
			\end{tikzpicture}
		\end{equation*}
	\end{exam}
	
	\subsection{Folding in the context of gentle algebras} \label{sec:Folding}
	We will now review the reverse process (\emph{folding}) and its relationship to the unfolding procedure defined above. Folding is typically performed by considering the quotient of a group action. In the case where we have started off with a folded gentle triple and obtained a gentle pair by unfolding, it is clear what this group action should be, as outlined in the following lemma.

	\begin{lem} \label{lem:Folding}
		Let $(K,Q,Z)$ be a folded gentle triple and let $(\unfQ,\unfZ)$ be the corresponding unfolded gentle pair. Then there exists an automorphic (covariant) $\integer_2$ group action on $\unfQ$ and $\unfZ$ such that $\unfQ / \integer_2 \cong Q \setminus \ca$ and $\unfZ / \integer_2 \cong Z \setminus \crel$. 
	\end{lem}
	\begin{proof}
		Let $\integer_2 = \{1, g\}$. The unfolding procedure has been deliberately defined such that we have $g\wh{v} = \wh{v}'$ and $g \wh{\alpha} = \wh{\alpha}'$. This gives rise to an automorphic action on the quiver $\unfQ$. For any field $\wh{K}$, one can extend the $\integer_2$-action to the path algebra $\wh{K}\unfQ$ by defining
		\begin{align*}
			g \stp_{\wh{v}} &= \stp_{\wh{v}'} \\
			g(\wh{\alpha}_1\ldots\wh{\alpha}_m) &= \wh{\alpha}'_1\ldots\wh{\alpha}'_m \\
			g(\lambda_1\wh{p}_1 + \ldots \lambda_n \wh{p}_n) &= \lambda_1(g\wh{p}_1) + \ldots + \lambda_n(g\wh{p}_n).
		\end{align*}
		It follows from the definition of the unfolding procedure that the set $\unfZ$ is closed under this action. In particular, it is a straightforward consequence of the definition that the orbits of both $\unfQ$ and $\unfZ$ under the $\integer_2$-action correspond to the vertices and ordinary arrows of $Q$ and ordinary relations in $Z$.
	\end{proof}
	
	So far we have defined unfolded gentle pairs that correspond to folded gentle triples. Now we state a broader definition and show how one obtains folded gentle triples from this broader setting.
	
	\begin{defn}
		Let $\wh{K}$ be a field and $(\unfQ,\unfZ)$ be a gentle pair such that $\unfQ$ is connected. We say that $(\unfQ,\unfZ)$ is an \emph{unfolded gentle pair} if there exists a non-trivial $\integer_2$ group action on $\unfQ$, which when extended to a $\integer_2$-action on the path algebra $\wh{K}\unfQ$, the set $\unfZ$ is closed under the action of $\integer_2$.
		
		We say that $\wh{K}\unfQ/ \langle \unfZ \rangle$ is an \emph{unfolded gentle algebra} if $(\unfQ,\unfZ)$ is an unfolded gentle pair and $\wh{K}$ admits a field extension of degree 2. We call the action of the non-trivial element of $\integer_2$ on $\wh Q$ and $\wh{K}\unfQ/ \langle \unfZ \rangle$ the \emph{folding action}.
	\end{defn}
	
	\begin{rem} \label{rem:NoFixedArrows}
		It is a consequence of axioms \ref{en:FG1}-\ref{en:FG4} that for any unfolded gentle pair $(\unfQ,\unfZ)$, there exists no arrow of $\unfQ$ that is fixed under the folding action.
	\end{rem}
	
	So suppose that $K$ is a field that admits a degree 2 extension. To obtain a folded gentle triple $(K,Q,Z)$ from an arbitrary unfolded gentle pair $(\unfQ, \unfZ)$ with folding action given by $g \in \integer_2$, one can simply consider the quiver $Q$ and set of relations $Z$ obtained by a quotient of the group action of $\integer_2$, and add creases and crease relations as appropriate. More formally, we have the following \emph{folding procedure}.
	\begin{enumerate}[label=(F.\roman*)]
		\item The vertices of $Q$ correspond to the orbits of vertices of $\unfQ$ under $\integer_2$. In particular, we define
		\begin{align*}
			\ov &= \{[\wh{v}] : \wh{v} \in \unfQ_0, g\wh{v} \neq \wh{v}\}, \\
			\cv &= \{[\wh{v}] : \wh{v} \in \unfQ_0, g\wh{v} = \wh{v}\}.
		\end{align*} \label{en:F1}
		\item For each crease vertex $[\wh{v}] \in \cv$, we have a corresponding crease $\crs_{[\wh{v}]} \in \ca$. \label{en:F3}
		\item The ordinary arrows of $Q$ correspond to the orbits of arrows of $\unfQ$ under $\integer_2$. That is, we define $\oa = \{[\wh{\alpha}] : \wh{\alpha} \in \unfQ_1\}$. \label{en:F2}
		\item The ordinary relations are defined to be the orbits of relations in $Z$ under $\integer_2$. That is, $\orel = \{[\wh{\alpha}\wh{\beta}] : \wh{\alpha}\wh{\beta} \in \unfZ\}$. \label{en:F4}
		\item Finally, we (arbitrarily) add crease relations that satisfy \ref{en:FG5}. \label{en:F5}
	\end{enumerate}
	It is easy to check that, up to an appropriate choice of crease relations, this procedure is a left inverse to the unfolding procedure, and thus, that $(K,Q,Z)$ is a folded gentle triple.
	
	\begin{rem}
		The above shows that folding and unfolding is not a unique process: On the one hand, depending on the ground field, there can be infinitely many non-isomorphic folded gentle algebras obtained by performing \ref{en:F1}-\ref{en:F5} on an unfolded gentle algebra (given by varying the crease relations from \ref{en:F5}). On the other hand, given a folded gentle algebra $A$, there can be more than one unfolded gentle algebra $\wh A$ that `folds' onto $A$ (such as, for example, when $\wh Q$ contains loops). Nevertheless, the representation theory of a folded gentle algebra is intimately related to the representation theory of any gentle algebra that folds onto it.
	\end{rem}
	
	Foldings and unfoldings can also be described via certain quiver morphisms. Such morphisms are useful devices that we will need later in the paper. 
	
	\begin{defn} \label{def:QuiverMorphisms}
		Let $(K,Q,Z)$ be a folded gentle triple that corresponds to $(\unfQ,\unfZ)$ via the (un)folding procedure. Define a surjective covariant quiver morphism
		\begin{equation*}
			\pi\colon \unfQ \rightarrow Q \setminus \ca
		\end{equation*}
		that maps each vertex in $\unfQ_0$ to its corresponding $\integer_2$-orbit in $Q_0$ and each arrow in $\unfQ_1$ to its corresponding $\integer_2$-orbit in $Q \setminus \ca$. We call $\pi$ the \emph{quiver folding morphism}.
		
		Consider a map
		\begin{equation*}
			\iota\colon Q \setminus \ca \rightarrow \unfQ.
		\end{equation*}
		We call $\iota$ a \emph{quiver unfolding morphism} if the following hold.
		\begin{enumerate}[label=(QU\arabic*)]
			\item $\iota$ is a covariant quiver morphism that is a right inverse to $\pi$; and \label{en:QU1}
			%\item $\iota$ respects concatenation of paths: if $\alpha_1\ldots\alpha_n \in \PQ$ then $\iota(\alpha_1)\ldots\iota(\alpha_n) \in \mathcal{P}_{\unfQ}$;
			\item $\iota$ respects the relations of $\orel$: if $\alpha\beta \in \orel$ then $\iota(\alpha\beta)=\iota(\alpha)\iota(\beta) \in \unfZ$. \label{en:QU2}
		\end{enumerate}
	\end{defn}
	
	\begin{exam}
		Consider Example~\ref{ex:DoubleCreaseUnf}. The quiver folding morphism $\pi$ is given by $\pi(\wh{v})=\pi(\wh{v}')=v$ and $\pi(\wh{\alpha})=\pi(\wh{\alpha}')=\alpha$. There are precisely two possible quiver unfolding morphisms $\iota_1$ and $\iota_2$. Specifically, we have
		\begin{align*}
%			\iota_1(v)			&=\wh v			&	\iota_2(v)			&=\wh{v}' \\
%			\iota_1(\crs_v)	&=\stp_{\wh v}	&	\iota_2(\crs_v)	&=\stp_{\wh v} \\
%			\iota_1(\alpha)	&=\wh\alpha		&	\iota_2(\alpha)	&=\wh\alpha' \\
%			\iota_1(\beta)	&=\wh{\beta}' 	&	\iota_2(\beta)	&=\wh{\beta} \\
%			\iota_1(\gamma)	&=\wh\gamma	&	\iota_2(\gamma)	&=\wh\gamma' \\
%			\iota_1(\delta)	&=\wh\delta		&	\iota_2(\delta)	&=\wh\delta' \\
			\iota_1(v) &= \wh{v}, &  \iota_1(\zeta) &=\wh{\zeta}, \\
			\iota_2(v) &= \wh{v}', &  \iota_2(\zeta) &=\wh{\zeta}',
%			\iota_1(v) &=\wh{v}, & \iota_1(\crs_v) &=\stp_{\wh v}, & \iota_1(\alpha) &=\wh\alpha, & \iota_1(\beta) &=\wh{\beta}', & \iota_1(\gamma) &=\wh\gamma, & \iota_1(\delta) &=\wh\delta, \\
%			\iota_2(v) &=\wh{v}', & \iota_2(\crs_v) &=\stp_{\wh v}, & \iota_2(\alpha) &=\wh\alpha', & \iota_2(\beta) &=\wh{\beta}, & \iota_2(\gamma) &=\wh\gamma', & \iota_2(\delta) &=\wh\delta',
		\end{align*}
		for each $v \in Q_0$ and $\zeta \in Q_1$. In particular $\iota_2 = g\iota_1$, where $g \in \integer_2$ is the folding action.
	\end{exam}
	
	It is clear that for any unfolded gentle algebra $\wh A$ and any folded gentle algebra $A$ related to $\wh A$ via either the unfolding or folding procedure, the quiver folding morphism $\pi$ always exists. It is also not difficult to verify that if $\wh A$ is related to a folded gentle algebra $A$ via the unfolding procedure, then a quiver unfolding morphism always exists. However, there may not be a quiver unfolding morphism for all other unfolded gentle algebras. For example, let $A$ be the folded gentle algebra in Examples~\ref{ex:TripleCrease} and~\ref{ex:DoubleCreaseUnf}, and consider the unfolded gentle algebra $\wh A'$ given by replacing the loops $\wh\delta\colon\wh 4 \rightarrow \wh 4$ and $\wh \delta'\colon\wh 4' \rightarrow \wh 4'$ in Example~\ref{ex:DoubleCreaseUnf} by arrows $\wh\delta\colon \wh 4 \rightarrow \wh 4'$ and $\wh\delta'\colon \wh 4' \rightarrow \wh 4$, and by replacing the relations $\wh \delta^2$ and $(\wh\delta')^2$ with $\wh \delta\wh \delta'$ and $\wh \delta'\wh \delta$. In this case, $\wh A'$ still folds onto $A$, but no quiver unfolding morphism exists (the condition that fails is \ref{en:QU2}). For this reason, we will primarily focus on unfolded gentle algebras obtained via the unfolding procedure for the rest of the paper.
	
	\subsection{(Un)folding strings and bands}
	 The quiver folding ($\pi$) and unfolding ($\iota$) morphisms can naturally be extended to the set of ordinary symbols. In addition, the folding action of $g \in \integer_2$ can be extended to strings and bands in the natural way. This motivates a theory of (un)folding for strings and bands, which establishes a deep relationship between the combinatorics of gentle algebras and their folded counterparts. Such a theory also illuminates intricate differences in behaviours between certain classes of strings and certain classes of bands. For example, we will later see that certain classes of strings and bands that are invariant (up to equivalence) under the folding action are closely related to symmetric strings and bands in a folded gentle algebra, which exhibit slightly different behaviour to their asymmetric counterparts. Henceforth, $A = K Q / \langle Z \rangle$ will refer to a folded gentle algebra and $\wh A = K \wh Q / \langle \wh Z \rangle$ will refer to an unfolded gentle algebra related to $A$ via the unfolding procedure.
	
	\begin{defn} \label{def:Z2Words}
		For any arrow $\wh \alpha \in \wh Q_1$ and $\alpha \in Q_1$, define $\pi(\wh\alpha\inv)=(\pi(\wh\alpha))\inv$ and $\iota(\alpha\inv)=(\iota(\alpha))\inv$. For any word $\wh w = \wh \sigma_1\ldots\wh\sigma_n \in \wrd_{\wh Q}$, define $g \wh w = (g\wh\sigma_1)\ldots(g\wh\sigma_n)$, where $g\wh\alpha\inv = (g\wh\alpha)\inv$ for any $\wh\alpha\inv \in \unfQ_1\inv$.
	\end{defn}
	
	Clearly if $\wh w \in \str_{\wh A}$ then $g\wh w \in \str_{\wh A}$. Likewise for any band $\wh w \in \bnd_{\wh A}$, it follows that $g\wh w \in \bnd_{\wh A}$. Thus, $\integer_2$ acts on $\str_{\wh A}$ and $\bnd_{\wh A}$, and hence we may define the following.
	
	\begin{defn}
		We say that a string or band $\wh{w}$ of $\wh{A}$ is \emph{$\integer_2$-invariant} if $g \wh w \approx \wh w$.
	\end{defn}
	
	One can also define an analogue of the (un)folding morphisms $\pi$ and $\iota$ for words and strings. This will be an incredibly useful tool when we investigate the module category of folded gentle algebras. Specifically, we have the following.
	
	\begin{defn} \label{def:FoldedWords}
		Define a function
		\begin{equation*}
			\fw\colon \wrd_{\wh Q} \rightarrow \wrd_{Q}
		\end{equation*}
		as follows. If $\wh w \in \wrd_{\wh Q}$ is a simple word, then define $\fw(\wh w) = \stp_{\pi(s(\wh w))}$. Otherwise, for any symbol $\wh \sigma \in \wh Q_1 \cup \wh Q_1\inv$, denote $\sigma= \pi(\wh\sigma)$ and define
		\begin{equation*}
			\fw(\wh\sigma) =
			\begin{cases}
					\crs_{s(\sigma)} \sigma
						& \text{if } \wh\sigma\not\in\im\iota, s(\sigma)\in\cv \text{ and } t(\sigma)\not\in\cv, \\
					\sigma\crs_{t(\sigma)}\inv				
						& \text{if } \wh\sigma\not\in\im\iota, s(\sigma)\not\in\cv \text{ and } t(\sigma)\in\cv, \\
					\crs_{s(\sigma)} \sigma\crs_{t(\sigma)}\inv	
					& \text{if } \wh\sigma\not\in\im\iota \text{ and } s(\sigma),t(\sigma)\in\cv, \\
					\sigma				& \text{otherwise.}
				\end{cases}
		\end{equation*}
		For any non-simple word $\wh w=\wh\sigma_1\ldots\wh\sigma_n\in \wrd_{\wh Q}$, we then define
		\begin{equation*}
			\fw(\wh w) = \fw(\wh\sigma_1)\fw(\wh\sigma_2)\ldots\fw(\wh\sigma_n).
		\end{equation*}
		We call $\fw(\wh w)$ the \emph{folded word} of $\wh w$.
		
		Now define a function
		\begin{equation*}
			\theta\colon \wrd_{ Q} \rightarrow \wrd_{\wh Q}
		\end{equation*}
		in the following way. Firstly, if $w$ is either simple or a crease symbol, then define $\theta(w) = \stp_{\iota(s(w))}$. Otherwise, write $w=\sigma_1\ldots\sigma_n$ and let $\sigma_{i_1},\ldots,\sigma_{i_r}$ be the complete list of crease symbols of $w$ with $1 \leq i_1 < \ldots < i_r \leq n$. In addition, define $i_0=0$ and $i_{r+1} = n+1$. Next, for each $0 \leq j \leq r+1$, we then define 
		\begin{equation*}
			\wh w_j = \iota(\sigma_{i_j+1})\ldots\iota(\sigma_{i_{j+1}-1}).
		\end{equation*}
		Finally, define
		\begin{equation*}
			\theta(w) = (g^0 \wh w_0) (g^1 \wh w_1) \ldots (g^r \wh w_r) (g^{r+1} \wh w_{r+1})
		\end{equation*}
		We call $\theta(w)$ the \emph{unfolded word} of $w$.
	\end{defn}
		
	\begin{rem}
		Note that if $i_1 = 1$, the word $\theta(w)$ should instead be interpreted as the word $(g^1 \wh w_1) \ldots (g^r \wh w_r) (g^{r+1} \wh w_{r+1})$. Similarly, if $i_r=n$, then we omit the $g^{r+1} \wh w_{r+1}$ term.
	\end{rem}
		
	It is not difficult to verify that if $w \in \str_A$, then $\theta(w) \in \str_{\wh A}$. In this case we call $\theta(w)$ the \emph{unfolded string} of $w$. The similar statement for $\fw$ is not necessarily true (the condition that can fail is \ref{en:FS3}). Thus, we need another function for this purpose.

	\begin{defn}\label{def:FoldedStrings}
		Define a function
		\begin{equation*}
			\rho\colon \str_{\wh A} \rightarrow \str_{ A} : \wh w \mapsto \sigma \fw(\wh w) \sigma',
		\end{equation*}
		where $\sigma \in \{ \stp_{\pi(s(\wh w))}, \crs_{\pi(s(\wh w))}\inv\}$ and $\sigma' \in \{  \stp_{\pi(t(\wh w))}, \crs_{\pi(t(\wh w))}\}$ are uniquely determined such that $\sigma \fw(\wh w) \sigma' \in \str_A$. In particular, $\sigma = \crs_{\pi(s(\wh w))}\inv$ if and only if $\pi(s(\wh w)) \in \cv$ and the first symbol of $\wh w$ is in $\im \iota$, and $\sigma' = \crs_{\pi(t(\wh w))}$ if and only if $\pi(t(\wh w)) \in \cv$ and the last symbol of $\wh w$ is in $\im \iota$. We call $\rho(\wh w)$ the \emph{folded string} of $\wh w$.
	\end{defn}
	
	The process of folding strings is well-defined up to equivalence and the folding action, as the next lemma shows.
	\begin{lem} \label{lem:FoldedStrings}
		Let $\wh w, \wh w' \in  \str_{\wh A}$. Then $\rho(\wh w) \approx \rho(\wh w')$ if and only if $\wh w \approx \wh w'$ or $\wh w \approx g\wh w'$.
	\end{lem}
	\begin{proof}
		The result follows from the fact that the ordinary symbols of $Q$ correspond to $\integer_2$-orbits of ordinary symbols of $\wh Q$, and that $\pi$ maps symbols to their corresponding $\integer_2$-orbit. As for crease symbols, it follows from \ref{en:FG1}, \ref{en:FG2}, \ref{en:FG6}, \ref{en:FS2}, \ref{en:QU2} and the definition of $\fw$ that for any ordinary symbol $\sigma_i$ of a string $\sigma_1\ldots\sigma_n \in \str_A$ such that $t(\sigma_i)\in \cv$, there exists a unique crease $\crs_{t(\sigma_i)} \in \ca$ such that $\sigma_{i+1} \in \{\crs_{t(\sigma_i)}, \crs_{t(\sigma_i)}\inv\}$. Similarly, for any ordinary symbol $\sigma_i$ such that $s(\sigma_i)\in \cv$, there exists a unique crease $\crs_{s(\sigma_i)} \in \ca$ such that $\sigma_{i-1} \in \{\crs_{s(\sigma_i)}, \crs_{s(\sigma_i)}\inv\}$. Thus, it follows entirely by construction that $\rho(\wh w) \approx \rho(\wh w')$ if and only if $\wh w \approx \wh w'$ or $\wh w \approx g\wh w'$.
	\end{proof}
	
	The process of (un)folding bands is somewhat less straightforward. This is because there exist words $w$ of a folded gentle algebra such that $s(w)=t(w)$ but $s(\theta(w)) \neq t(\theta(w))$. This occurs only if $t(\theta(w))=g s(\theta(w))$, which occurs if and only if the number of crease symbols of $w$ is odd. In this case, we have $s(\theta(w^2)) = t(\theta(w^2))$ instead. This motivates the following definitions.
	
	\begin{defn} \label{def:Parity}
		Let $w$ be a band of a folded gentle algebra. We say that $w$ is of \emph{odd parity} if $w$ contains an odd number of crease symbols. We say that $w$ is of \emph{even parity} otherwise.
		
		Let $\wh w$ be a band of an unfolded gentle algebra. We say that $\wh w$ is of \emph{odd parity} if there exists a subword $\wh w'$ of $\wh w$ such that $\wh w=\wh w'(g\wh w')$. We say that $\wh w$ is of \emph{even parity} otherwise.
	\end{defn}
	
	Armed with this definition, we can now define how one (un)folds bands.
	
	\begin{defn} \label{def:FoldedBands}
		Let $\wh b \in \bnd_{\unfA}$ and let $\wh w = \wh b \in \wrd_{\wh Q}$ be the equivalent word. Define a function
		\begin{equation*}
			\rho\colon \bnd_{\wh A} \rightarrow \bnd_A
		\end{equation*}
		as follows. If $\wh b$ is of even parity, then we define $\rho(\wh b)$ to be the band of $A$ whose underlying word is $\kappa(\wh w)$. On the other hand if $\wh b$ is of odd parity, then there exist subwords $w'$ and $w''$ of $\kappa(\wh w)$ such that $\kappa(\wh w) = w' w'' \sim_2 (w')^2$. In this case we define $\rho(\wh b)$ to be the band of $A$ whose underlying word is $w'$. In either case, we call $\rho(\wh b)$ the \emph{folded band} of $\wh b$.
		
		Let $b \in \bnd_A$ and let $w=b \in \wrd_Q$ be the equivalent word. Define a function
		\begin{equation*}
			\theta\colon \bnd_{ A} \rightarrow \bnd_{\wh A}
		\end{equation*}
		as follows. If $b$ is of even parity, then we define $\theta(b)$ to be the band of $\unfA$ whose underlying word is equivalent to $\theta(w)$. Otherwise, we define $\theta(b)$ to be the band of $\unfA$ whose underlying word is equivalent to $\theta(w^2)$. In either case, we call $\theta( b)$ the \emph{unfolded band} of $ b$.
	\end{defn}
	
	\begin{lem} \label{lem:FoldedBands}
		Let $\wh w, \wh w' \in  \bnd_{\wh A}$. Then $\rho(\wh w) \approx \rho(\wh w')$ if and only if $\wh w \approx \wh w'$ or $\wh w \approx g\wh w'$.
	\end{lem}
	\begin{proof}
		This result follows as a result of the same argument given in the proof of Lemma~\ref{lem:FoldedStrings}.
	\end{proof}
	
	In fact, we can say more about the relationship between $\rho(\wh w)$ and $\rho(g \wh w)$ for bands. The next lemma shows that the pre-action of $\integer_2$ on a band $\wh w$ results in every crease symbol of $\rho(\wh w)$ being inverted.
	
	\begin{lem} \label{lem:Z2Invert}
		Let $\wh w \in \bnd_{\wh A}$, and write $\rho(\wh w) = \sigma_1\ldots\sigma_n$. Let $\sigma_{i_1},\ldots,\sigma_{i_r}$, be a complete list of crease symbols of $\rho(\wh w)$, with $i_1 < \ldots < i_r$. Then $\rho(g\wh w) \sim_3 \sigma'_1\ldots\sigma'_n$, with $\sigma'_i = \sigma_i$ whenever $i \neq i_j$ (for any $j$), and $\sigma'_{i_j} = \sigma\inv_{i_j}$ for each $1 \leq j \leq r$.
	\end{lem}
	\begin{proof}
		Write $\wh w = \wh\sigma_1\ldots\wh\sigma_{\wh n}$. Throughout the proof, indices of $\wh w$ are taken modulo $\wh n$, and indices of $\rho(\wh w)$ are taken modulo $n$. It follows from Definition~\ref{def:FoldedWords} that for each symbol $\wh\sigma_i$ such that $\pi(s(\wh\sigma_i)),\pi(t(\wh\sigma_i)) \in \ov$, we have $\fw(\wh\sigma_i) = \fw(g\wh\sigma_i)$. So we need only consider symbols $\wh\sigma_i$ for which $\fw(\wh\sigma_i)$ contains a crease symbol.
	
		It follows from Definition~\ref{def:FoldedWords}, \ref{en:QU2} and the unfolding procedure that, for each $1 < j \leq r$, $\sigma_{i_j} \in \ca$ if and only if $\sigma_{i_{j-1}} \in (\ca)\inv$. That is, crease symbols alternate between arrows and formal inverses. Thus, it is sufficient to show that $\sigma'_{i_1} = \sigma_{i_1}\inv$. Let $k$ be the index such that $\fw(\wh\sigma_k)$ contains the symbol $\sigma_{i_1}$. Then either $\fw(\wh\sigma_k) = \sigma_{i_1 - 1}\sigma_{i_1}$ or $\fw(\wh\sigma_k)$ contains $\sigma_{i_1} \sigma_{i_1+1}$ as a subword (which is an equality if $t(\sigma_{i_1+1}) \in \ov$). In the former case, we have $\wh\sigma_{k} \not\in \im \iota$ and $\wh\sigma_{k+1} \in \im \iota$. Thus, $g\wh\sigma_{k} \in \im \iota$ and $g\wh\sigma_{k+1} \not\in \im \iota$. Since we necessarily have $\pi(t(g\wh\sigma_{k}))\in\cv$ in this case, it follows from Definition~\ref{def:FoldedWords} that $\fw(g\wh\sigma_{k+1})$ contains $\sigma_{i_1}\inv \sigma_{i_1+1}$ as a subword (which is an equality if $t(\sigma_{i_1+1}) \in \ov$). Similarly, in the latter case, we have $g\wh\sigma_{k-1} \not\in \im \iota$ and $g\wh\sigma_{k} \in \im \iota$, and thus $\fw(g\wh\sigma_{k-1})$ contains $\sigma_{i_1 - 1}\sigma_{i_1}\inv$ as a subword (which is an equality if $s(\sigma_{i_1 - 1}) \in \ov$). In conclusion, the first crease symbol of $\rho(g\wh w)$  (after a possible rightwards rotation if $i_1 = 1$) is the inverse of the first crease symbol of $\rho(\wh w)$. The result then follows.
	\end{proof}
	
	It is not difficult to verify that $\theta$ is a right inverse to $\rho$ on both the classes of strings and bands. Consequently, $\rho$ induces a surjective map from the class of strings of $A$ to the class of strings of $\unfA$, and a surjective map from the class of bands of $A$ to the class of bands of $\unfA$. Similarly, $\theta$ induces an injective map from the class of strings of $A$ to the class of strings of $\unfA$, and an injective map from the class of bands of $A$ to the class of bands of $\unfA$. Additionally, bands are mapped to bands of the same parity under the maps $\rho$ and $\theta$, as the following lemma shows.
	
	 \begin{lem} \label{lem:OddParity}
	 	Let $A$ be a folded gentle algebra. Then the following are equivalent.
	 	\begin{enumerate}[label = (\alph*)]
	 		\item $w \in \bnd_A$ is of odd parity.
	 		\item $w \in \bnd_A$ is asymmetric and of odd parity.
			\item $\theta(w) \in \bnd_{\wh A}$ is $\integer_2$-invariant and of odd parity.
			\item $\theta(w) \in \bnd_{\wh A}$ is of odd parity.
		\end{enumerate}
	 \end{lem}
	 \begin{proof}
	 	(a) $\Rightarrow$ (b): The result follows from the contrapositive statement: If $w$ is symmetric, then $w \approx \crs w' \crs' (w')\inv$ for some substring $w$ and crease symbols $\crs$ and $\crs'$. Thus, $w$ must contain an even number of crease symbols, and hence is of even parity as required.
	 	
	 	(b) $\Rightarrow$ (c): Since $w$ is of odd parity, it has at least one crease symbol. We can therefore assume without loss of generality (by Lemma~\ref{lem:BandIsos}(c)) that the final symbol of $w$ is a crease symbol. By \ref{en:FB2}, this implies that the first symbol of $w$ is ordinary. Now consider the word $\wt w \in \wrd_Q$ that is equivalent to $w \in \bnd_A$ and write $\theta(\wt w) = \wh\sigma_1 \ldots \wh \sigma_{n}$. Since the first symbol of $w$ is ordinary, this implies that $\wh \sigma_1 \in \im \iota$. Since $w$ is of odd parity, we have $\theta(w) = \theta(\wt w^2)$ by definition. In particular, there is an odd number of crease symbols in $w$, so by Definition~\ref{def:FoldedWords}, we must have $\wh\sigma_n \in \im \iota$. Now by Definition~\ref{def:FoldedWords}, this implies that we must have $\theta(\wt w^2) = \theta(\wt w) (g \theta(\wt w))$. Hence $\theta(w)$ is of odd parity. Clearly, $\theta(w)$ is also $\integer_2$-invariant, as required.
	 	
	 	(c) $\Rightarrow$ (d): This is obvious from the statement itself.
	 	
	 	(d) $\Rightarrow$ (a): We can assume (by the equivalence of rotation) that $\theta(w) = \wh w_0 \ldots \wh w_{r-1}$ for some non-simple substrings $\wh w_0,\ldots,\wh w_{r-1}$, where each symbol of $\wh w_i$ is in $\im \iota$ if and only if $i$ is even. In particular, the word $\fw \theta(w) \in \wrd_Q$ has precisely $r$ crease symbols. Now since $\theta(w)$ is of odd parity, we have $\theta(w) = \wh w' (g \wh w')$ for some substring $\wh w'$. But then there must exist some $r'$ such that $\theta(w) = \wh w_0 \ldots \wh w_{r'-1}\wh w_{r'} \ldots \wh w_{2r'-1}$ with $\wh w_i = g \wh w_{i + r'}$. This is possible only if $r'$ is odd. Moreover $w = \rho \theta(w) = \fw(\wh w_0 \ldots \wh w_{r'-1})$. Thus, $w$ has an odd number of crease symbols, as required.
	 \end{proof}
	 
	 A somewhat useful technical result is that every band $w$ of odd parity has a symbol ending on a vertex $v$ such that $v \in \ov$ (if $w \in \bnd_A$) or such that $g v \neq v$ (if $w \in \bnd_{\wh A}$). This is outlined in the following lemma.
	 
	 \begin{lem} \label{lem:OddOrdVert}
	 	Let $A$ be a finite-dimensional folded gentle algebra. Let $w=\sigma_1\ldots\sigma_n \in \bnd_A$ be of odd parity. Then there exists some $1 \leq i \leq n$ such that $t(\sigma_i) \in \ov$.
	 \end{lem}
	 \begin{proof}
	 	Suppose for a contradiction that $t(\sigma_i) \in \cv$ for all $i$. Let $\alpha_1\ldots\alpha_n \in Q_1$ be the (not necessarily distinct) arrows corresponding to the symbols $\sigma_1,\ldots,\sigma_n$. By \ref{en:FB1}, $w$ is neither a simple word nor a crease symbol, and hence \ref{en:FB2} implies that there exists some $i$ such that $\alpha_i \in \oa$. By Lemma~\ref{lem:OddParity}, $w$ is asymmetric, so we can use Lemma~\ref{lem:BandIsos}(c) to assume without loss of generality that $\alpha_1 \in \oa$. Now since $s(\alpha_1),t(\alpha_1) \in \cv$, there exist crease loops $\crs_{s(\alpha_1)},\crs_{t(\alpha_1)}\in \ca$. By \ref{en:FG1} this implies that if there exists any other arrow $\alpha' \in Q_1 \setminus \{\alpha_1,\crs_{s(\alpha_1)}\}$ incident to the vertex $s(\alpha_1)$, then it is unique and is such that $s(\alpha')\neq t(\alpha') = s(\alpha)$. Likewise, \ref{en:FG1} implies that if there exists any other arrow $\alpha'' \in Q_1 \setminus \{\alpha_1,\crs_{t(\alpha_1)}\}$ incident to the vertex $t(\alpha_1)$, then it is also unique and is such that $t(\alpha')\neq s(\alpha') = t(\alpha)$. Moreover, \ref{en:FG6} then implies that $\alpha'\alpha_1, \alpha_1\alpha'' \in \orel$. Proceeding inductively, \ref{en:FB2} implies that each $\sigma_{2i+1}$ must be an ordinary symbol (and hence each $\alpha_{2i+1}\in \oa$) and each $\sigma_{2i}$ is a crease symbol (and hence each $\alpha_{2i}\in \ca$). In particular, $n$ must be even, as \ref{en:FG6},  \ref{en:FB2} and $\sigma_1$ being ordinary together imply that $\sigma_n$ is a crease symbol. In addition, the arrows $\alpha_{1},\alpha_3,\ldots,\alpha_{n-1}$ trace a linearly oriented subquiver $Q'\subseteq Q$ that is of Dynkin type $A_m$ (for some $m$) --- one can verify that $Q'$ is not a cycle (that is, it is not of extended Dynkin type $\wt A_m$), as the algebra would otherwise be infinite-dimensional. Now since $s(w)=t(w)$ by \ref{en:FB2}, and since $Q'$ is a linearly oriented subquiver of Dynkin type $A_m$, it follows that $\frac{n}{2}$ must also be even (for any arrow $\beta$ in a cyclic walk in $Q'$, the formal inverse $\beta\inv$ must be walked on to return to the start). But then 
	 	\begin{equation*}
	 		|\{\text{crease symbols of } w\}| = |\{\sigma_2,\sigma_4,\ldots,\sigma_n\}| = \frac{n}{2}
	 	\end{equation*}
	 	is even, which contradicts the supposition that $w$ is of odd parity. Thus, the assumption that $t(\sigma_i) \in \cv$ for all $i$ must be false, as required.
	 \end{proof}
	
	Finally, the relationship between $\integer_2$-invariant strings/bands and symmetric strings/bands is outlined by the following.
	
	\begin{lem} \label{lem:FoldedSymmetric}
		Let $A$ be a folded gentle algebra.
		\begin{enumerate}[label = (\alph*)]
			\item A string $w \in \str_A$ is symmetric if and only if $\theta(w)$ is $\integer_2$-invariant.
			\item A band $w \in \bnd_A$ is symmetric if and only if $\theta(w)$ is $\integer_2$-invariant and of even parity.
		\end{enumerate}
	\end{lem}
	\begin{proof}
		(a) Suppose that $\theta(w) = \wh \sigma_1\ldots \wh\sigma_n$ is a $\integer_2$-invariant string. Since $g\wh\sigma_i \neq \wh\sigma_i$ for any $1 \leq i \leq n$ (by Remark~\ref{rem:NoFixedArrows}), we cannot have $g\theta(w) = \theta(w)$. Thus, we must have $g\theta(w) = (\theta(w))\inv$, and hence, $\wh \sigma_i = g\wh \sigma_{n+1-i}\inv$ for each $1 \leq i \leq n$. In particular, $\pi(\wh \sigma_i) = \pi(\wh \sigma_{n+1-i}\inv)$ for each $1 \leq i \leq n$. So $w$ must be symmetric. Conversely, if $w$ is a symmetric string, then we must have $w= \sigma_1\ldots\sigma_r \crs_v \sigma_r\inv \ldots\sigma_1\inv$ for some crease symbol $\crs_v$. Thus, $\theta(w) = \sigma_1\ldots\sigma_r (g\sigma_r)\inv \ldots(g\sigma_1)\inv$, which is clearly $\integer_2$-invariant, as required.
		
		(b) Firstly, $\theta(w)$ is a $\integer_2$-invariant band of even parity if and only if $g\theta(w) \sim_3 (\theta(w))\inv$. This follows because the only other possibility is that $g\theta(w) \sim_3 \theta(w)$, but then $\theta(w)$ is of odd parity. Secondly, $g\theta(w) \sim_3 (\theta(w))\inv$ if and only if $\theta(w) \sim_3 \wh\sigma_1 \ldots \wh \sigma_n (g\wh\sigma_n)\inv\ldots(g\wh\sigma_1)\inv$. But this is true if and only if $w \approx \crs_{\pi(s(\wh\sigma_1))}w'\crs_{\pi(t(\wh\sigma_n))}(w')\inv$ for some substring $w'$, which is clearly symmetric.
	\end{proof}
	
	% =====================================================
	\section{The Module Category of Folded Gentle Algebras} \label{sec:FoldedModule}
	% =====================================================
	In this section, we will provide a description of the module category of folded gentle algebras, which includes a classification of the indecomposable modules and the Auslander-Reiten sequences. Throughout this section, $(K,Q,Z)$ is a folded gentle triple and $A=K Q / \langle Z \rangle$ is the corresponding folded gentle algebra. Moreover, $(\unfQ,\unfZ)$ will be the unfolded gentle pair corresponding to $(K,Q,Z)$ by the unfolding procedure, and $g \in \integer_2$ will be the folding action on $(\unfQ,\unfZ)$. The map $\pi\colon \unfQ \rightarrow Q \setminus \ca$ will be the quiver folding morphism and $\iota$ will be some fixed choice of quiver unfolding morphism. By $\unfA$, we mean the unfolded gentle algebra $K \unfQ / \langle \unfZ \rangle$.
	
	\subsection{Unfolding via functors} \label{sec:UnfFunctor}
	One can define a categorical notion of the unfolding procedure. This involves defining an exact functor, faithful, $K$-linear functor
	\begin{equation*}
		U\colon \mod*A \rightarrow \mod* \unfA.
	\end{equation*}
	
	For any object $M \in \mod* A$, we define $\wh M = U(M)$ as follows. The underlying $K$-vector space of $\wh M$ is given by
	\begin{equation*}
		\wh M = \bigoplus_{\wh{v} \in \unfQ_0} \wh M_{\wh{v}},
	\end{equation*}
	where each $\wh M_{\wh{v}}$ is a distinct copy of $M_v = M \stp_{\pi(\wh{v})}$, the vector subspace of $M$ induced by the action of $\stp_{\pi(\wh{v})}$. Consequently, there is a single copy of each vector subspace $M_v \subseteq \wh M$ when $v \in \cv$, and precisely two distinct copies of each vector subspace $M_v \subseteq \wh M$ when $v \in \ov$.
	
	We can now enrich $\wh M$ with the structure of an $\unfA$-module. For each  $\wh{m} \in \wh{M}$ and $\wh{v} \in \unfQ_0$, define
	\begin{equation*}
		\wh{m}\stp_{\wh v} =
		\begin{cases}
			\wh{m}	& \text{if } \wh m \in \wh M_{\wh v}, \\
			0			& \text{otherwise,}
		\end{cases}
	\end{equation*}
	For the action of non-idempotent elements of $\wh A$ on $\wh M$, we first note that for each vertex $\wh v \in \wh Q_0$ such that $\pi(\wh v) \in \cv$, there is a natural automorphic $A$-action $M_{\wh v} \rightarrow  M_{\wh v}$ given by right multiplication by $\crs_{s(\pi(\wh v))}$. In addition, for each arrow $\wh \alpha \in \wh Q_1$, there is a natural right $A$-action $\pi(\wh\alpha)\colon \wh M_{s(\wh\alpha)} \rightarrow \wh M_{t(\wh\alpha)}$. We can thus define $\wh m \wh \alpha = \wh m \fw(\wh\alpha)$ for each $\wh{m} \in \wh{M}$ and $\wh{\alpha} \in \unfQ_1$, where $\fw$ is as in Definition~\ref{def:FoldedWords}. 
	
	For any morphism $f\in \Hom_{A}(M,M')$, we define $\wh f = U(f)$ by first writing $f = (f_{v})_{v \in Q_0}$ and $\wh f = (\wh f_{\wh v})_{\wh{v} \in \unfQ_0}$, where each $f_{v}\colon {M}\stp_v \rightarrow {M}'\stp_v$ and $\wh f_{\wh v}\colon \wh{M}_{\wh v} \rightarrow \wh{M}_{\wh v}'$ is $K$-linear, and then by defining $\wh f_{\wh v} = f_{\pi(\wh{v})}$.
	
	\begin{defn} \label{def:UnfoldingFunctor}
		We call the functor $U$, as constructed above, the \emph{unfolding functor} of $\mod*A$.
	\end{defn}
	
	A number of useful remarks result from this definition.
	
	\begin{rem} \label{rem:UImage}
		The functor $U$ is clearly exact, $K$-linear and faithful. There are a number of other useful properties that we will use throughout this section.
		\begin{enumerate}[label=(\alph*)]
			\item The image of $U$ is closed under compositions of morphisms and is thus an abelian subcategory of $\mod*\wh A$. In particular, $\mod*A$ is equivalent to the subcategory $\im U \subset \mod*\wh A$. However, this subcategory is not full. This can result in decomposable objects in $\mod*\wh A$ having a local endomorphism algebra in $\im U$. Thus, indecomposable objects in the category $\im U$ are not necessarily indecomposable objects in $\mod*\wh A$. In addition, there exist some objects that are not isomorphic in $\im U$ but which are isomorphic in $\mod*\wh A$.
			\item Every object in $\mod*A$ and $\im U$ can be viewed simultaneously as a right $A$-module and right $\wh A$-module in the natural way. This is described in the definition of $U$. In particular, every $\wh A$-action on an object in $\im U$ has an equivalent $A$-action given $K$-linearly by the function $\fw$. Conversely, every object in $\mod*A$ has the structure of an $\wh A$-module in the natural way via the map $\fw$. This will frequently be exploited in this section.
			\item As already mentioned in the definition of $U$, for any object $\wh M \in \im U$ and any crease loop $\crs_v \in \ca$, it follows that the right action of $\crs_v$ on $\wh M$ is a non-trivial automorphic action on the subspace $\wh M \stp_{\iota(v)}$ and a zero action on all subspaces $\wh M \stp_{\iota(u)}$ with $u \neq v$. Since $\crs_v$ satisfies an irreducible quadratic relation by \ref{en:FG5}, the elements $\wh m, \wh m \crs_v \in \wh M\stp_{\iota(v)}$ are linearly independent.
		\end{enumerate}
	\end{rem}
	
	\subsection{Unfolding String Modules}
	We will now turn our attention to the relationship between string modules of both $A$ and $\wh A$ induced by the unfolding functor $U$.
	
	\begin{prop} \label{prop:UStrings}
		Let $M \in \mod*A$ and let $\wh w = \wh\sigma_1\ldots \wh\sigma_n \in \str_{\unfA}$. Then the following are equivalent.
		\begin{enumerate}[label=(\alph*)]
			\item $M(\wh w)$ is isomorphic to a direct summand of $U(M)$.
			\item $M(\wh w) \oplus M(g\wh w)$ is isomorphic to a direct summand of $U(M)$.
			\item $M(\rho(\wh w))$ is isomorphic to a direct summand of $M$.
		\end{enumerate}
		In particular, $U(M(\rho(\wh w))) \cong M(\wh w) \oplus M(g\wh w)$.
	\end{prop}
	\begin{proof}
		For readability purposes, we will adopt the following notation throughout the proof. We will write $w = \rho(\wh w)$, $\sigma_i = \pi (\wh \sigma_i)$ for each $i$, $v = \pi(\wh v)$ for each $\wh v \in \unfQ_0$, and $\alpha = \pi(\wh \alpha)$ for each $\wh \alpha \in \unfQ_1$.
		
		(a) $\Rightarrow$ (c): Since $U$ is additive, there exists an indecomposable direct summand $N \subseteq M$ such that $M(\wh w) \subseteq U(N)$. Recall that $M(\wh w)$ has basis $\wh B=\{\wh b_0, \ldots, \wh b_n\}$, where each $\wh b_i$ is the subword $\wh \sigma_1\ldots\wh\sigma_i$, and that the right $\wh A$-action on $M(\wh w)$ is given by concatenation (subject to the identifications \ref{en:SM1}-\ref{en:SM4}). We claim that $N \cong M(w)$. To prove this claim, we will use the definition of $U$ and $M(\wh w)$ to inductively construct a basis for $N$ via the actions of $\wh A$ on $M(\wh w)$. We will then show that this is precisely a basis for $M(w)$. 
		
		So first consider the basis element $\wh b_0 = \stp_{s(\wh w)}$. Note that if $g s(\wh w) = s(\wh w)$, then $s(w) \in \cv$, in which case, there exists an element $\wh b_0 \crs_{s(w)} \in N\stp_{s(w)} = U(N) \stp_{s(\wh w)}$ which is $K$-linearly independent to $\wh b_0$ (since $\crs_{s(w)}$ satisfies an irreducible quadratic relation). With this in mind, define a word of $A$
		\begin{equation*}
			w_0 =
			\begin{cases}
				\crs_{s(w)}\inv		&	\text{if } s(w) \in \cv \text{ and } \wh\sigma_1 \in \im \iota, \\
				\stp_{s(w)}		&	\text{otherwise}.
			\end{cases}
		\end{equation*}
		For each $1 \leq i \leq n$, define a word $w_i = w_0 \kappa(\wh b_i)$ of $A$, where $\kappa$ is as defined in Definition~\ref{def:FoldedWords}. Now note that by the definition of $\fw$ (Definition~\ref{def:FoldedWords}), $\rho$ (Definition~\ref{def:FoldedStrings}) and $U$ (Definition~\ref{def:UnfoldingFunctor}), each $w_i$ is a subword of $w=\rho(\wh w)$. Moreover, for each $w_i$ such that $t(w_i)=\pi(t(\wh b_i)) \in \cv$, either $w_i \crs_{t(w_i)}$ or $w_i \crs_{t(w_i)}\inv$ is also a subword of $w$ that is distinct from the set of subwords $\{w_i : 1 \leq i \leq n\}$. For each $i$ such that $t(w_i) \in \cv$, we will therefore denote by $\wt{w}_i$ whichever word in $\{w_i \crs_{t(w_i)},w_i \crs_{t(w_i)}\inv\}$ is a subword of $w$ that is not in $\{w_i : 1 \leq i \leq n\}$.
		
		Now consider the $K$-vector subspace
		\begin{equation*}
			W = \langle w_i : 0 \leq i \leq n \rangle_K + \langle \wt w_i : t(w_i) \in \cv \rangle_K \subseteq U(N).
		\end{equation*}
		By Remark~\ref{rem:UImage}(b), $U(N)$ has the structure of an $A$-module. By the same remark, $W$ is closed under the action of $A$, and is thus an $A$-submodule. Moreover, it is clear by definition that the set of all $w_i$ and $\wt w_i$ is the standard basis for $M(w)$. Thus, $M(w) \subseteq N$.
		
		The last thing to check is that $N \cong M(w)$, which is a straightforward. Suppose for a contradiction that $M(w)$ is a proper submodule of $N$. Since $N$ is assumed to be indecomposable, this implies that there exists an element $x \in N / M(w)$ such that $x \alpha \neq 0$ and $x \alpha \in M(w)$ for some $\alpha \in Q_1$. By the definition of $U$, this implies that $x \in N\stp_{s(\alpha)} = U(N)\stp_{\iota(s(\alpha))}=U(N)\stp_{g\iota(s(\alpha))}$ and either $x \iota(\alpha) \in M(\wh w)$ or $x g\iota(\alpha) \in M(\wh w)$. But then $M(\wh w)$ is not a direct summand of $U(N)$, which is a contradiction. Thus, $N \cong M(w)$ as required.
		
		(c) $\Rightarrow$ (b), and the final statement: Consider the respective standard bases $\wh B = \{\wh b_0, \ldots, \wh b_n\}$ and $\wh B' = \{\wh b'_0, \ldots, \wh b'_n\}$ of the strings $\wh w, g\wh w \in \str_{\wh A}$. Let $w_0,\ldots,w_n \in \wrd_Q$ be the words corresponding to $\wh b_0,\ldots,\wh b_n \in \wh B$ given in the proof of (a) $\Rightarrow$ (c). Similarly define 
		\begin{equation*}
			w'_0 =
			\begin{cases}
				\crs_{s(w)}\inv	&	\text{if } s(w) \in \cv \text{ and } w_0 = \stp_{s(w)}, \\
				\stp_{s(w)}		&	\text{otherwise,}
			\end{cases}
		\end{equation*}
		and for each $i$, define words $w'_i=w'_0\kappa(\wh b'_i)  \in \wrd_Q$. 
		
		Now note the following implications of this construction: Firstly, as previously remarked, $w_n$ is a subword of $\rho(\wh w)$. Similarly, we have by construction that $w'_n$ is a subword of $\rho(g \wh w)$. Secondly, for each $i$, the $j$-th ordinary symbol of $w_i$ is the same as the $j$-th ordinary symbol of $w'_i$, and the $j$-th crease symbol of $w_i$ is the inverse of the $j$-th crease symbol of $w'_i$. Consequently, $\rho(\wh w) \sim_2 \rho(g \wh w)$, and thus $M(\rho(\wh w)) \cong M(\rho(g\wh w))$ by Lemma~\ref{lem:StringIsos}(b). Thirdly, the last symbol of $w_i$ is the same as the last symbol of $w'_i$ precisely if $t(w_i),t(w'_i) \in \ov$. Otherwise, we must have (by construction) $t(w_i),t(w'_i) \in \cv$, in which case $w'_i \in \{w_i \crs_{t(w_i)}, w_i \crs_{t(w_i)}\inv\}$. Thus,
		\begin{equation*}
			\{w_i : 1 \leq i \leq n\} \cup \{w'_i: t(w_i) \in \cv\}
		\end{equation*}
		is a basis for $M(w)$, and each $w'_i$ such that $t(w'_i) \in \ov$ is linearly dependent to an element in $M(w)$
		
		Now from the definition of $U$, we have two copies of the vector space $M(w)\stp_{v}$ in $U(M(w))$ whenever $v \in \ov$: namely with $U(M(w))\stp_{\iota(v)}$ and $U(M(w))\stp_{g\iota(v)}$. In particular, we have $w_i, w'_i \in U(M(w))\stp_{\iota(v)} + U(M(w))\stp_{g\iota(v)}$. By defining $w_i \in U(M(w))\stp_{t(\wh \sigma_i)}$ and $w'_i \in U(M(w))\stp_{gt(\wh \sigma_i)}$, we can see from the definition of $U$ that we actually have
		\begin{align*}
			M(\wh w) &\cong \langle  w_0, \ldots,  w_n \rangle_K, \\
			M(g\wh w) &\cong \langle  w'_0, \ldots,  w'_n \rangle_K.
		\end{align*}
		In particular, $U(M(w)) \cong M(\wh w) +  M(g\wh w)$. That this sum is actually a direct sum follows from the fact that $w_i$ and $w'_i$ are linearly independent whenever $t(w_i)=t(w'_i) \in \cv$. Thus, the action of $\wh A$ on $U(M(w))$ is closed on both $\langle w_0, \ldots, w_n \rangle_K$ and $\langle w'_0, \ldots, w'_n \rangle_K$. So $U(M(w)) \cong M(\wh w) \oplus M(g\wh w)$, as required.
		
		(b) $\Rightarrow$ (a): This is obvious from the statement itself.
	\end{proof}
	
	\begin{cor} \label{cor:UProj}
		For any $v \in Q_0$, we have
		\begin{align*}
			U(P(v)) &\cong P(\iota(v)) \oplus P(g \iota(v)), \\
			U(I(v)) &\cong I(\iota(v)) \oplus I(g \iota(v)).
		\end{align*}
		Consequently, $U$ induces exact, faithful, $K$-linear functors
		\begin{align*}
			U_{\proj}&\colon\proj A \rightarrow \proj\unfA,  \\
			U_{\inj}&\colon\inj A \rightarrow \inj\unfA,
		\end{align*}
		that are injective on iso-classes of objects, where $\proj A$ and $\proj \unfA$ are the full subcategories of projective objects of $\mod*A$ and $\mod*\unfA$ respectively, and $\inj A$ and $\inj \unfA$ are the full subcategories of injective objects of $\mod*A$ and $\mod*\unfA$ respectively.
	\end{cor}
	\begin{proof}
		Since $A$ is a special biserial algebra whose relations consist only of zero relations and crease relations, $P(v)$ and $I(v)$ are both string modules. In particular, a basis of $P(v) = \stp_v A$ is given by every direct string of source $v$. There are at most two such direct strings that are of maximal length, say $w$ and $w'$. It is then easy to verify that $P(v) \cong M(w'w\inv)$. The argument for $I(v)$ is dual. The statement is then an immediate consequence of the previous proposition.
	\end{proof}
	
	The above corollary leads to other useful properties concerning projective and injective resolutions.
	
	\begin{cor} \label{cor:UProps}
		Let $U$ be the unfolding functor of $\mod*A$. Then the following hold.
		\begin{enumerate}[label=(\alph*)]
			\item $U$ preserves projective and injective resolutions.
			\item $U$ maps minimal projective resolutions to minimal projective resolutions, and minimal injective resolutions to minimal injective resolutions.
		\end{enumerate}
	\end{cor}
	\begin{proof}
		(a) This is a direct consequence of Corollary~\ref{cor:UProj} and the fact that $U$ is exact and additive.
		
		(b) Let $\Omega^i(M)$ be the $i$-th syzygy of $M \in \mod*A$ and $f_i\colon P_i \rightarrow \Omega^i(M)$ be its projective cover. Then by definition, $f_i$ is an epimorphism, $P_i$ is a projective object, and $\Ker f_i \cong \Omega^{i+1}(M)$ is superfluous in $P_i$. Since $U$ is exact, $U(f_i)\colon U(P_i) \rightarrow U(\Omega^i(M))$ is surjective and $\Omega^{i+1}(U(M)) \cong U(\Omega^{i+1}(M))$. It also follows that $U(P_i)$ is a projective object by Corollary~\ref{cor:UProj} and the additivity of $U$. Now suppose for a contradiction that $U(\Omega^{i+1}(M))$ is not superfluous in $U(P_i)$. Then there exists a projective direct summand $\wh P$ of $U(\Omega^{i+1}(M))$. Since indecomposable projective objects of $\mod*\unfA$ are string modules, Proposition~\ref{prop:UStrings} implies that there is a corresponding direct summand $P$ of $\Omega^{i+1}(M)$ that is isomorphic to the module corresponding to the folded string. In particular $P$ must be projective, which contradicts the statement that $\Omega^{i+1}(M)$ is superfluous in $P_i$. Thus, $\Omega^{i+1}(U(M))$ is superfluous in $U(P_i)$, and hence $U(f_i)$ is a projective cover. Consequently, $U$ preserves minimal projective resolutions, as required. The proof for minimal injective resolutions is dual.
	\end{proof}
	
	\subsection{Unfolding Band Modules} We now consider the relationship between the band modules of $A$ and $\wh A$. For bands, that are not $\integer_2$-invariant, we have a similar result to string modules.
	
	\begin{prop} \label{prop:UEvenBands}
		Let $M \in \mod*A$ and let $\wh w \in \bnd_{\unfA}$. Suppose that $\wh w$ is not $\integer_2$-invariant. Then the following are equivalent.
		\begin{enumerate}[label=(\alph*)]
			\item $M(\wh w, m, \phi)$ is isomorphic to a direct summand of $U(M)$.
			\item $M(\wh w, m, \phi) \oplus M(g\wh w, m, \lambda_{\wh w} \phi)$ is isomorphic to a direct summand of $U(M)$ for some non-zero $\lambda_{\wh w} \in K$ that depends on $\wh w$.
			\item $M(\rho(\wh w), m, \phi)$ is isomorphic to a direct summand of $M$.
		\end{enumerate}
		In particular, $U(M(\rho(\wh w), m, \phi)) \cong M(\wh w, m, \phi) \oplus M(g\wh w, m, \lambda_{\wh w} \phi)$ for some non-zero $\lambda_{\wh w} \in K$ that depends on $\wh w$.
	\end{prop}
	\begin{proof}
		The proof similar to the proof of Proposition~\ref{prop:UStrings}, with some minor alterations to accommodate the band structure of the module.
	
		(a) $\Rightarrow$ (c): Let $N$ be an indecomposable direct summand of $M$ such that $M(\wh w, m, \phi) \subseteq U(N)$. Let $\wh w = \wh \sigma_1\ldots \wh\sigma_n$ and define $w_0 =\stp_{\pi(s(\wh w))}$. For each $1 \leq i \leq n$, define a word $w_i = \fw(\wh\sigma_1\ldots\wh\sigma_i)$. Since $\wh w$ is not $\integer_2$-invariant, we then have $w_n=\rho(\wh w)$ by Lemma~\ref{lem:OddParity} and Definition~\ref{def:FoldedBands}. Next define $m$ copies of each word $w_i$, denoted by $w_i\ps{j} = w_i$ for $1 \leq j \leq m$. Under the identification $w_0\ps{j} = w_n\ps{j}$ for each $j$, there is then an obvious $\wh A$-module isomorphism
		\begin{equation*}
			\omega\colon \langle w_i\ps{j} : 1 \leq i \leq n \text{ and } 1 \leq j \leq m\rangle_K \rightarrow M(\wh w, m,\phi),
		\end{equation*}
		that maps each $w_i\ps{j}$ to the $j$-th copy of the subword $\wh\sigma_1\ldots\wh\sigma_i$ in the standard basis.
		
		Similar to the proof of Proposition~\ref{prop:UStrings}, for each $i$ such that $t(w_i) \in \cv$, there exists a subword $\wt w_i\ps{j} \in \{w_i\ps{j}\crs_{t(w_i)},w_i \ps{j}\crs_{t(w_i)}\inv\}$ of $\rho(\wh w)$ such that $\wt w_i\ps{j} \not\in \{w_i\ps{j}: 1 \leq i \leq n \text{ and } 1 \leq j \leq m\}$. Again, by Remark~\ref{rem:UImage}, it then follows that the vector subspace
		\begin{equation*}
			W = \langle w_i\ps{j}: 1 \leq i \leq n \text{ and } 1 \leq j \leq m \rangle_K + \langle \wt w_i\ps{j}: t(w_i\ps{j}) \in \cv \rangle_K \subseteq U(N)
		\end{equation*}
		is an $A$-submodule of $U(N)$. Moreover, the set of all $w_i\ps{j}$ and $\wt w_i\ps{j}$ is clearly the standard basis of $M(\rho(\wh w), m, \phi)$. So $W \cong M(\rho(\wh w), m, \phi) \subseteq N$. The proof of the claim that $N \cong M(\rho(\wh w),m,\phi)$ is identical to the corresponding proof in (a) $\Rightarrow$ (c) of Proposition~\ref{prop:UStrings}, and so we are done.
		
		(c) $\Rightarrow$ (b), and the final statement: Let $w_i\ps{j}$ be the same words as above for each $i$ and $j$. For each $i$, define $w'_i = \fw((g\wh\sigma_1)\ldots(g\wh\sigma_i))$, and define $m$ copies of each $w'_i$, with the $j$-th copy denoted by $(w'_i)\ps{j}$. Now by similar arguments to those used in the proof of (c) $\Rightarrow$ (b) in Proposition~\ref{prop:UStrings}, we have have $w_n = \rho(\wh w) \sim_2 \rho(g\wh w) = w'_n$. Again, $t( w_i) = t( w'_i)$ for each $i$. We (again) also have $t( w_i) \in \cv$ if and only if $t( w'_i) \in \cv$. Finally, we (again) have that $ w_i\ps{j}$ and $( w'_i)\ps{j}$ are linearly dependent if and only if $t( w_i) = t( w'_i) \in \ov$.
		
		Now note that since $\wh w$ is not $\integer_2$-invariant, $\rho(\wh w)$ is asymmetric by Lemma~\ref{lem:FoldedSymmetric}. Hence, $M(\rho(\wh w), m, \phi) \cong M(\rho(g \wh w), m, \lambda_{\wh w} \phi)$ for some $\lambda_{\wh w} \in K$ by Lemma~\ref{lem:BandIsos}(b), which is non-zero by \ref{en:FG5}. To see that $\lambda_{\wh w}$ depends on $\wh w$, we simply note from the proof of Lemma~\ref{lem:BandIsos}(b) that $\lambda_{\wh w}$ arises from a product of particular non-zero coefficients of the quadratic relations of the crease symbols of $\wh w$. It is then easy to verify via similar arguments to that used in the proof of (c) $\Rightarrow$ (b) in Proposition~\ref{prop:UStrings} that we have 
		\begin{align*}
			M(\wh w, m, \phi) &\cong \langle  w_i\ps{j} : 1 \leq i \leq n \text{ and } 1 \leq j \leq m \rangle_K, \\
			M(g\wh w, m, \lambda_{\wh w}\phi) &\cong \langle ( w'_i)\ps{j} : 1 \leq i \leq n \text{ and } 1 \leq j \leq m \rangle_K, \\
			U(M(w, m, \phi)) &\cong M(\wh w, m, \phi) \oplus M(g\wh w, m, \lambda_{\wh w}\phi),
		\end{align*}
		and so we are done.
		
		(b) $\Rightarrow$ (a): This is evident from the statement itself.
	\end{proof}
	
	\begin{rem} \label{rem:LamConst}
		We can say precisely what the value of $\lambda_{\wh w}$ is for each band $\wh w \in \bnd_{\wh w}$. Let $\sigma_{i_1},\ldots, \sigma_{i_r}$ be the complete list of crease symbols of $\rho(\wh w)$, and suppose we have relations $\sigma_{i_j}^2-\lambda_{1,j}\sigma_{i_j}-\lambda_{2,j}\stp_s(\sigma_{i_j})=0$. It then follows from Lemma~\ref{lem:BandIsos}(b) and Lemma~\ref{lem:Z2Invert} that
		\begin{equation*}
			\lambda_{\wh w} = \prod_{j=1}^r \lambda_{2,j}\inv.
		\end{equation*}
	\end{rem}
	
	For bands that are $\integer_2$-invariant, the theory is more complicated. Firstly, $\integer_2$-invariant bands of odd and even parity behave differently to each other. Thus, we must consider both of these subcases carefully. Secondly, in the string and non-$\integer_2$-invariant band case, the string/band modules of $\mod* \wh A$ always appear in the image of $U$ as pairs of indecomposables (that are related to each other via the folding action on the string/band). This can also happen for $\integer_2$-invariant bands, but in certain cases, it is possible for the image of $U$ to be indecomposable. To account for this, it is helpful to instead consider the maximal quasiband direct summands of an object in the image of $U$.
	
	\begin{defn}
		Let $A$ be a gentle or folded gentle algebra and let $M \in \mod*A$. Suppose $N \subseteq M$ is a direct summand that is isomorphic to a quasi-band module $M( w,m,\phi)$. We say $ N$ is a \emph{maximal quasi-band direct summand} if there exists no direct summand $ L \subseteq  M /  N$ that is isomorphic to a quasi-band module of $ w$.
	\end{defn}
	
	\begin{prop} \label{prop:UOddBands}
		Let $A$ be a folded gentle algebra and $M \in \mod*A$. Let $\wh w$ be a band of $\wh A$ of odd parity and let $\lambda_{\wh w} \in K$ be as in Remark~\ref{rem:LamConst}. Then the following are equivalent.
		\begin{enumerate}[label = (\alph*)]
			\item $M(\wh w, m, \wh \phi)$ is isomorphic to a maximal quasi-band direct summand of $U(M)$.
			\item There exists $\phi \in \Aut(K^m)$ such that $\wh \phi$ is similar to $\lambda_{\wh w} \phi^2$ and $M(\rho(\wh w),m,\phi)$ is isomorphic to a maximal quasi-band direct summand of $M$.
		\end{enumerate}
		In particular, for any $M \in \mod*A$ with $U(M) \cong M(\wh w, m, \wh \phi)$, $M$ must be isomorphic to an asymmetric quasi-band module $M(\rho(\wh w),m,\phi)$ and $U(M) \cong U(M(\rho(\wh w),m,\phi)) \cong M(\wh w,m,\lambda_{\wh w} \phi^2)$.
	\end{prop}
	\begin{proof}
		We will begin by proving that $U(M) \cong M(\wh w, m, \lambda_{\wh w}\phi^2)$ when $M \cong M(\rho(\wh w),m,\phi)$. First, some notation. Write $w= \rho(\wh w) =\sigma_1\ldots\sigma_n$ and $\wh w = \wh\sigma_1\ldots\wh\sigma_{2\wh n}$. In particular, we have $\wh \sigma_{\wh n+i} = g\wh \sigma_{i}$ for each $1 \leq i \leq \wh n$ by Definition~\ref{def:Parity}. For the purposes of readability, denote for each $i$ the vertex $u_i = t(\sigma_i)$. By Lemma~\ref{lem:OddOrdVert}, there exists some $i$ such that $u_i \in \ov$. Thus by Lemma~\ref{lem:BandIsos}(c) and the fact that $\wh w$ is of odd parity, we may assume without loss of generality that $\wh w$ is rotated such that $s(w)=u_{n} \in \ov$ and that the first crease symbol of $w$ is not a formal inverse. Let $\sigma_{i_1},\ldots,\sigma_{i_p}$ be the crease symbols of $w$ with $1= i_0 < i_1 < i_2 < \ldots < i_p < i_{p+1} = n$. Note, that by the definition of $\fw$ (Definition~\ref{def:FoldedWords}) and combinatorics of $\iota$ (specifically arising from \ref{en:QU2} and \ref{en:U2}-\ref{en:U4} and \ref{en:FG6}), it follows that each $\sigma_{i_{2k}} \in (\ca)\inv$ and $\sigma_{i_{2k+1}} \in \ca$. Now consider the equivalent band $w' = \rho(g\wh w) = \sigma'_1\ldots\sigma'_n$. By Lemma~\ref{lem:Z2Invert}, we have $\sigma'_i = \sigma_i$ whenever $\sigma_i$ is ordinary and $\sigma'_i = \sigma_i\inv$ whenever $\sigma_i$ is a crease symbol. In particular, this implies that each $\sigma_{i_{2k}} \in \ca$ and each $\sigma_{i_{2k+1}} \in (\ca)\inv$.  Also note that we then have $\fw(\wh w) = w w'$.
		
		Define functions
		\begin{align*}
			\omega\colon &\{1,2,\ldots,\wh n\}  \rightarrow \{1,\ldots,n\} \setminus (\{i_{2k+1} : 0 \leq k \leq \lfloor\tfrac{p}{2}\rfloor\} \cup \{i_{2k}-1 : 1 \leq k \leq \lfloor\tfrac{p}{2}\rfloor\}) \\
			\omega'\colon &\{\wh n,\wh n + 1,\ldots,2\wh n\}  \rightarrow \{1,\ldots,n\} \setminus (\{i_{2k+1}-1 : 0 \leq k \leq \lfloor\tfrac{p}{2}\rfloor\} \cup \{i_{2k} : 1 \leq k \leq \lfloor\tfrac{p}{2}\rfloor\})
		\end{align*}
		determined such that $\omega(k)<\omega(l)$ and $\omega'(k)<\omega'(l)$ if and only if $k <l$. It is not hard to verify (from the definitions of $\rho(\wh w)$ and $\fw(\wh w)$ in Definitions~\ref{def:FoldedWords} and \ref{def:FoldedBands}) that both $\omega$ and $\omega'$ are bijective, and thus, uniquely determined. In particular, note that the excluded sets in the codomains of these functions are determined by the indices immediately preceding a symbol in $(\ca)\inv$ or immediately following a symbol $\ca$ --- this is by design to ensure that
		\begin{equation*}
			\fw(\wh\sigma_1,\ldots\wh\sigma_{i}) =
			\begin{cases}
				\sigma_1\ldots\sigma_{\omega(i)}					&\text{if } 1 \leq i \leq \wh n \\
				w\sigma'_1\ldots\sigma'_{\omega'(i)}	&\text{if } \wh n < i \leq 2\wh n.
			\end{cases}
		\end{equation*}
		
		Let $\{b_i\ps{j} : 1 \leq i \leq n, 1\leq j \leq m\}$ be the standard basis of $M(w,m,\phi)$. We will use this basis to construct a basis $\{\wh b_i\ps{j} : 1 \leq i \leq 2\wh n, 1 \leq j \leq m\}$ for $U(M)$ which we will then show is a basis of a quasi-band module isomorphic to $M(\wh w, m, \lambda_{\wh w}\phi^2)$. To achieve this, it is helpful to recall from the definition of $U$ that the vector space structure of $U(M)$ is such that, for each $i$, $U(M)\stp_{\iota(u_i)}$ and $U(M)\stp_{g \iota(u_i)}$ are copies of $M\stp_{u_i}$. So first, for each $j$, define $\wh b_{2\wh n}\ps{j} = \wh b_{0}\ps{j} \in U(M)$ to be the element of $U(M)\stp_{\iota(u_{n})}$ that is a copy of $b_{n}\ps{j} \in M\stp_{u_n}$. Next, for each $1 \leq i \leq \wh{n}$ and each $1 \leq j \leq m$, define $\wh b_{i}\ps{j} \in U(M)$ to be the copy of the element $b_{n}\ps{j}\cdot\fw(\wh\sigma_1\ldots\wh\sigma_{i}) \in M\stp_{u_{\omega(i)}}$, with
		\begin{equation*}
			\wh b_{i}\ps{j} \in
			\begin{cases}
				U(M)\stp_{\iota(u_{\omega(i)})}	&\text{if } i_{2k} -1 < \omega(i) < i_{2k+1} \text{ for some } k, \\
				U(M)\stp_{g\iota(u_{\omega(i)})}	&\text{if } i_{2k+1} < \omega(i) < i_{2k} -1 \text{ for some } k.
			\end{cases}
		\end{equation*}
		Now let $r$ be the greatest index such that $\wh\sigma_{r} \in \wh Q_1$. For each $\wh n < i < r$, define each $\wh b_{i}\ps{j} \in U(M)$ to be a copy of the element $b_{n}\ps{j}\cdot w\fw(\wh\sigma'_1\ldots\wh\sigma'_{i}) \in M\stp_{u_{\omega'(i)}}$, and for each $r \leq i < 2 \wh n$, define each $\wh b_{i}\ps{j} \in U(M)$ to be a copy of the element $b_{n}\ps{j}\cdot\fw((\wh\sigma'_{i+1}\ldots\wh\sigma'_{2\wh n})\inv) \in M\stp_{u_{\omega'(i)}}$. In both cases, each $\wh b_{i}\ps{j}$ is defined to be the specific copy (of the appropriate element of $M\stp_{u_{\omega'(i)}}$) given such that
		\begin{equation*}
			\wh b_{i}\ps{j} \in
			\begin{cases}
				U(M)\stp_{g\iota(u_{\omega'(i)})}	&\text{if } i_{2k} < \omega'(i) < i_{2k+1} -1 \text{ for some } k, \\
				U(M)\stp_{\iota(u_{\omega'(i)})}	&\text{if } i_{2k+1}-1 < \omega'(i) < i_{2k}  \text{ for some } k,
			\end{cases}
		\end{equation*}
		This completes the construction of the set $\{\wh b_i\ps{j} : 1 \leq i \leq 2\wh n, 1 \leq j \leq m\}$. 

		We will now verify that this set is a basis for $U(M)$. Note by the construction that for each $b_i\ps{j}$ with $u_i \in \ov$, we have defined two corresponding elements $\wh b_{\omega\inv(i)}\ps{j} \in U(M)\stp_{g^a\iota(u_i)}$ and $\wh b_{\wh n + \omega\inv(i)}\ps{j} \in U(M)\stp_{g^{a+1}\iota(u_i)}$ (for some $a \in \{0,1\}$). Since $U(M)\stp_{\iota(u_i)}$ and $U(M)\stp_{g\iota(u_i)}$ are distinct subspaces of $U(M)$ when $u_i \in \ov$, it is clear that $\wh b_{\omega\inv(i)}\ps{j}$ and $\wh b_{\wh n + \omega\inv(i)}\ps{j}$ are linearly independent in $U(M)$ despite both corresponding to copies of the $b_i\ps{j}$. We now need to consider the vertices $u_i \in \cv$, for which it is useful to note that $\{i : u_i \in \cv\} = \{i_k -1, i_k : 1 \leq k \leq p\}$. It is then clear from our construction and the definition of $\omega$ and $\omega'$ that to each pair of linearly independent vectors $b_{i_k-1}\ps{j},b_{i_k}\ps{j}$, we have defined corresponding elements $\wh b_{\omega\inv(l)}\ps{j},\wh b_{(\omega')\inv(l')}\ps{j} \in U(M)\stp_{\iota(u_i)}=U(M)\stp_{g\iota(u_i)}$ for some $l,l' \in \{i_k -1, i_k\}$ with $l \neq l'$. Consequently, the set $\{\wh b_i\ps{j} : 1 \leq i \leq \wh n, 1 \leq j \leq m \}$ is linearly independent. Moreover, by dimension counting, we can see that this is indeed a basis for $U(M)$.
		
		To see that $\{\wh b_i\ps{j} : 1 \leq i \leq \wh n, 1 \leq j \leq m \}$ is the basis for some quasi-band module $M(\wh w, m,\wh \phi)$, we simply note by Remark~\ref{rem:UImage} that the action of each symbol $\wh \sigma_i$ of $\wh w$ induces an isomorphism of vector subspaces $\langle \wh b_{i-1}\ps{j} : 1 \leq j \leq m \rangle \rightarrow \langle \wh b_{i}\ps{j} : 1 \leq j \leq m \rangle$. Thus, it only remains to show that $\wh \phi$ is similar to $\lambda_{\wh w}\phi^2$, where $\lambda_{\wh w} \in K$ is as given in Remark~\ref{rem:LamConst}. This follows largely by construction and an application of Lemma~\ref{lem:BandIsos}(b). Recall that $w = \fw(\wh\sigma_1\ldots\wh\sigma_{\wh n})$ and $\fw(\wh w) = w w'$. Also recall by Lemma~\ref{lem:Z2Invert} that $\fw(g(\wh\sigma_1\ldots\wh\sigma_{\wh n}))$ is given by inverting every crease symbol of $\fw(\wh\sigma_1\ldots\wh\sigma_{\wh n})$. Thus by Lemma~\ref{lem:BandIsos}(b), $M(w,m,\phi) \cong M(w',m,\lambda_{\wh w}\phi)$ for any $\phi \in \Aut(K^m)$. Now since $U(M(w,m,\phi)) \cong M(\wh w, m, \wh \phi)$ and $\fw(\wh w) = w w'$, Remark~\ref{rem:UImage} implies that the action of performing a full rotation of $\wh w$ is  equivalent to performing a full rotation of $w$ followed by a full rotation of $w'$. That is,
		\begin{align*}
			\wh b_{\wh n}\ps{j} &= \wh b_{2\wh n}\ps{j} \cdot w & &\text{ which is a copy of } & b_{n}\ps{j}\cdot w &= \phi(b_{n}\ps{j}), \\
			\wh b_{2\wh n}\ps{j} \cdot \fw(\wh w) &= \wh b_{\wh n}\ps{j} \cdot w' & &\text{ which is a copy of } &  \phi(b_{n}\ps{j})\cdot w' &=\lambda_{\wh w}\phi^2(\wh b_{n}\ps{j}).
		\end{align*}
		This is only possible if $\wh \phi$ is similar to $\lambda_{\wh w} \phi^2$, and thus $U(M(\rho(\wh w),m,\phi))\cong M(\wh w, m, \lambda_{\wh w} \phi^2)$, as required.
	
		(a) $\Rightarrow$ (b): We use much of the same notation as before. For this section of the proof, denote for each $i$ the vertex $\wh v_i = t(\wh \sigma_i)$ and the vertex $v_i = \pi(t(\wh \sigma_i))$. Note that $v_i$ here is not necessarily the same as $u_i$ from the previous section of the proof. The orientation/rotation of $\wh w$ is the same as before (so $s(w)=v_{2\wh n} \in \ov$ and the first crease symbol of $w$ is not a formal inverse). Here, we let $\{\wh b_i\ps{j} : 1 \leq i \leq 2n, 1 \leq j \leq m\}$ be the standard basis of $M(\wh w, m, \wh \phi)$. For each $i$, denote by $\wh B_i$ the vector subspace $\langle \wh b_{i}\ps{j} : 1 \leq j \leq m \rangle \subseteq M(\wh w, m,\wh\phi)$. By the definition of $U$, there then exists for each $1 \leq i \leq \wh n$ vector subspaces $B_i, B_{i+\wh n} \subseteq M\stp_{v_i}$ such that $\wh B_i$ is a copy of $B_i$ and $\wh B_{i+\wh n}$ is a copy of $B_{i+\wh n}$.
		
		We claim that $B_{\wh n} = B_{2 \wh n}$. Suppose for a contradiction that this claim is false. Since each symbol $\wh \sigma_i$ induces an isomorphism $B_{i-1} \rightarrow B_i$, this supposition implies that there exists $b \in B_{\wh n}$ such that $b \not\in \wh B_{2\wh n}$. In addition, for any such element, we have
		\begin{equation*}
			b \cdot ((g\wh\sigma_1)\ldots(g\wh\sigma_{\wh n}) \wh\sigma_1\ldots\wh\sigma_{\wh n})^k = \wh\phi^k(b) \not \in B_{2\wh n}
		\end{equation*}
		for any $k$. Thus, $C_{\wh n}=B_{\wh n} / (B_{\wh n} \cap B_{2\wh n})$ is non-zero and closed under $\wh\phi$. By a similar argument, the vector space $C_{2\wh n}=B_{2\wh n} / (B_{\wh n} \cap B_{2\wh n})$ is also non-zero and closed under $\wh\phi$. Thus, $C_{\wh n }\oplus (B_{\wh n} \cap B_{2\wh n}) \oplus C_{2\wh n} \subseteq M\stp_{v_{\wh n}}$ with $C_{\wh n} \oplus (B_{\wh n} \cap B_{2\wh n}) \cong B_{\wh n}$ and $(B_{\wh n} \cap B_{2\wh n}) \oplus C_{2\wh n} \cong B_{2\wh n}$. Now note that both $\wh w$ and $g\wh w$ act by automorphisms on $C_{\wh n} \oplus (B_{\wh n} \cap B_{2\wh n}) \oplus C_{2\wh n}$. But from the definition of $U$, we know that both $U(M) \stp_{\wh v_{\wh n}}$ and $U(M) \stp_{\wh v_{2\wh n}}$ are copies of $C_{\wh n} \oplus (B_{\wh n} \cap B_{2\wh n}) \oplus C_{2\wh n}$. Since both $C_{\wh n}$ and $C_{2\wh n}$ are non-zero by assumption, this is only possible if there exists a direct summand of $U(M) / M(\wh w, m, \wh\phi)$ that is a quasi-band module of $\wh w$. But then $M(\wh w, m, \wh\phi)$ is not a maximal quasi-band direct summand, which contradicts the statement of (a). Thus, we must have $B_{\wh n} = B_{2\wh n}$.
		
		Now, for each $1\leq i \leq \wh n$, let $w_i = \fw(\wh\sigma_1\ldots\wh\sigma_i)$ and $w'_i = \fw(\wh\sigma_{\wh n+1}\ldots\wh\sigma_{\wh n+i})$. Then each $w_i$ is a subword of $w$ and each $w'_i$ is a subword of $w'$. In particular, $t(w_i)=t(w'_i)=v_i$ for each $i$, and $w_i \sim_2 w_i$ if and only if $v_i \in \ov$. In addition, for each $i$ such that $v_i\in\cv$, it follows by construction that $w_i$ ends with the formal inverse of a crease if and only if $\wh \sigma_i \not\in \im \iota$ if and only if $\wh \sigma_{\wh n+i} \in \im \iota$ if and only if $w'_i$ ends with an ordinary symbol. Likewise, $w'_i$ ends with the formal inverse of a crease if and only if $w_i$ ends with an ordinary symbol. %Now, for each $1\leq i \leq 2n$ and $1\leq j \leq m$, let $b_i\ps{j} \in B_i$ be the vector corresponding to $\wh b_i\ps{j} \in \wh B_i$. In addition, let $r$ be the greatest index such that $\wh \sigma_r \in \wh Q_1$. Then 
		Now since $B_{\wh n} = B_{2\wh n}$, the combinatorics that we have described implies that 
		\begin{equation*}
			B_{\wh n+i} =
			\begin{cases}
				B_i					&\text{if } v_i \in \ov, \\
				B_i \crs_{v_i}		&\text{if } v_i \in \cv \text{ and } \wh\sigma_i\not\in\im\iota, \\
				B_i \crs_{v_i}\inv	&\text{if } v_i \in \cv \text{ and } \wh\sigma_i\in\im\iota.
			\end{cases}
		\end{equation*}
		
		Now consider the vector subspace 
		\begin{equation*}
			N = \sum_{i=1}^{2\wh n} B_i = \sum_{1 \leq i \leq \wh n : v_i \in \ov} B_i + \sum_{1 \leq i \leq 2\wh n : v_i \in \cv} B_i \subseteq M
		\end{equation*}
		Then it is clear that $N$ is closed under the action of $A$, and is thus an $A$-submodule of $M$. In particular, $U(N) = \langle \wh b_i\ps{j} : 1 \leq i \leq 2\wh n, 1 \leq j \leq m \rangle$, and thus $N$ must be a direct summand of $M$ by similar arguments to that used in the proof of (a) $\Rightarrow$ (c) in Proposition~\ref{prop:UStrings}. Moreover, every symbol $\sigma$ of $w$ acts on $N$ via an isomorphism between subspaces $B_i, B_j \subseteq N$, where $i$ and $j$ are determined by $s(\sigma)$ and $t(\sigma)$. Since $w$ is cyclic, this implies that $N \cong M(\rho(\wh w),m,\phi)$ for some $\phi \in \Aut(K^m)$. But since $U(N) \cong U(M(\rho(\wh w),m,\phi)) \cong M(\wh w, m, \wh \phi)$, this implies that $\wh \phi$ must be similar to $\lambda_{\wh w} \phi^2$.
		
		We have thus shown that there exists $\phi \in \Aut(K^m)$ such that $\wh \phi$ is similar to $\lambda_{\wh w} \phi^2$ and $M(\rho(\wh w),m,\phi)$ is isomorphic to a direct summand of $N \subseteq M$. Finally, we must show that $M(\rho(\wh w),m,\phi)$ is a maximal quasi-band direct summand of $M$. Suppose for a contradiction that this is not the case. Then there exists a quasi-band module direct-summand $M(\rho(\wh w),m',\phi') \subseteq M / N$. But since $U$ is additive, it follows from the final statement of the proposition (which we have already proved to be true) that we then have $U(M) \cong M(\wh w,m,\wh\phi) \oplus M(\wh w, m', \lambda_{\wh w} (\phi')^2) \oplus \wh X$ for some direct summand $\wh X \subseteq U(M)$. But then $M(\wh w,m,\wh\phi)$ is not a maximal quasi-band direct summand of $U(M)$, which gives the required contradiction. Thus, $M(\rho(\wh w),m,\phi)$ is a maximal quasi-band direct summand of $M$.
		
		(b) $\Rightarrow$ (a): Suppose $M \cong M(\rho(\wh w),m,\phi) \oplus X$, where $X$ contains no direct summand that is isomorphic to a quasi-band module of $\rho(\wh w)$. By the additivity of $U$ and the final statement of the proposition, we then have $U(M) \cong M(\wh w, m, \lambda_{\wh w}\phi^2) \oplus U(X)$. To complete the proof, we must show that $M(\wh w, m, \lambda_{\wh w}\phi^2)$ is a maximal quasi-band direct summand of $U(M)$. Suppose for a contradiction that this is not the case. Then there exists a direct summand of $U(X)$ that is isomorphic to $M(\wh w, m', \wh \phi')$. But then from the statement of (a) $\Rightarrow$ (b) (which we have already proven), this implies that $\wh \phi'$ is similar to $\lambda_{\wh w}(\phi')^2$ for some $\phi' \in \Aut(K^m)$ and that $M(\rho(\wh w),m',\phi')$ is isomorphic to a direct summand of $X$. This contradicts the initial assumption that $X$ contains no such direct summand. Thus, $M(\wh w, m, \lambda_{\wh w}\phi^2)$ must be a maximal quasi-band direct summand of $U(M)$, as required.
	\end{proof}
	
	We are now left with what is perhaps the most complicated case: $\integer_2$-invariant bands of even parity in $\bnd_{\wh A}$, which correspond to symmetric bands in $\bnd_A$ by Lemma~\ref{lem:FoldedSymmetric}(b). To make our next results as explicit as possible, we will say that a band $\wh w \in \bnd_{\wh A}$ that is $\integer_2$-invariant and of even parity is in \emph{standard form} if $\rho(\wh w) \sim_2 \crs'\sigma_n\inv\ldots\sigma_1\inv\crs\sigma_1\ldots\sigma_n$ for some crease symbols $\crs$ and $\crs'$, and is such that $\sigma_r \in Q_1$ for some $1 \leq r \leq n$. Every $\integer_2$-invariant band of even parity can be written in such a form by considering the equivalent band given by rotation and/or inverse.
	
	\begin{rem} \label{rem:SymBandConst}
		It is worthwhile to make the following remarks regarding $\rho(\wh w)$.
		\begin{enumerate}[label=(\alph*)]
			\item Whilst $\rho(\wh w) \sim_2 \crs'\sigma_n\inv\ldots\sigma_1\inv\crs\sigma_1\ldots\sigma_n$, this is not an equality. Instead, we have $\rho(\wh w) = \crs' \sigma_{2n}\ldots\sigma_{n+1} \crs \sigma_1 \ldots \sigma_n$, where $\sigma_{i+n} = \sigma_i\inv$ if $\sigma_i$ is an ordinary symbol, and $\sigma_{i+n} = \sigma_i$ if $\sigma_i$ is a crease symbol.
			\item The crease symbols of $\rho(\wh w)$, reading from left to right, alternate between being direct to inverse and vice versa. This follows from the combinatorial setup of $\wh Q$ and how the function $\fw$ behaves (Definition~\ref{def:FoldedWords}).
		\end{enumerate}
	\end{rem}

	 \begin{lem} \label{lem:SymBandConst}
	 	Let $M \in \mod*A$ and $\wh w = (g\wh\sigma_{\wh n}\inv)\ldots(g\wh\sigma_{1}\inv)\wh\sigma_{1}\ldots\wh\sigma_{\wh n}$ be a band of $\wh A$ of even parity that is $\integer_2$-invariant and in standard form. Let $w=\rho(\wh w) = \crs' \sigma_{2n}\ldots\sigma_{n+1} \crs \sigma_1 \ldots \sigma_n$. Suppose that a quasiband module $M(\wh w, \wh m, \wh \phi)$ with standard basis $\{\wh b_i\ps{j}: 1 \leq i \leq 2\wh n, 1 \leq j \leq \wh m\}$ is a direct summand of $U(M)$. Then there exists $\mu_{\wh w} \in K$ such that
	 	\begin{align*}
	 		\wh b_{2\wh n}\ps{j} \crs' \sigma_{2n}\ldots\sigma_{n+1} \crs \sigma_1 \ldots \sigma_n 
	 		&= \mu_{\wh w} \wh b_{2\wh n}\ps{j} \crs'\sigma_{n}\inv\ldots\sigma_{1}\inv \crs \sigma_1 \ldots \sigma_n \\
	 		&=\mu_{\wh w}^2 \wh b_{2\wh n}\ps{j} \crs'\sigma_{n}\inv\ldots\sigma_{1}\inv \crs \sigma_{n+1}\inv \ldots \sigma_{2n}\inv +p_{\wh w},
	 	\end{align*}
	 	where $p_{\wh w}$ is a $K$-linear combination of terms in the set
	 	\begin{equation*}
	 		\{\wh b_{2\wh n}\ps{j} \crs' \sigma_{n}\inv\ldots\sigma_{1}\inv \crs \sigma_{n+1}\inv \ldots \sigma_{n+i}\inv: t(\sigma_{n+i}\inv)=t(w), 1 \leq j \leq \wh m\}.
	 	\end{equation*}
	 	Moreover, $\mu_{\wh w} =  \mu_{g\wh w}\inv$.
	 \end{lem}
	 \begin{proof}
	 	 Let $\sigma_{i_1},\ldots,\sigma_{i_r}$ be the complete list of crease symbols in the subword $\sigma_1 \ldots \sigma_n$ of $w$, where $i_1<\ldots<i_r$. Then by Remark~\ref{rem:SymBandConst}(a), we have
		\begin{equation*}
			\wh b_{2\wh n}\ps{j} \crs' \sigma_{2n}\ldots\sigma_{n+1} \crs \sigma_1 \ldots \sigma_n 
			= \wh b_{2\wh n}\ps{j} (\mu_1\sigma_{n}\inv\ldots\sigma_{i_r+1}\inv + \mu_2 \sigma_{n}\inv\ldots\sigma_{i_r+1}\inv\sigma_{i_r}\inv)\sigma_{n+i_r-1}\ldots\sigma_{n+1} \crs \sigma_1 \ldots \sigma_n
		\end{equation*}
		for some $\mu_1,\mu_2\in K$ such that $\sigma_{i_r} = \mu_1\stp_{s(\sigma_{i_r})} + \mu_2\sigma_{i_r}\inv$. Now we note that the substrings $\sigma_{1}\ldots \sigma_{n} \approx \sigma_{2n}\ldots\sigma_{n+1}$ are asymmetric (otherwise both $w$ and $\wh w$ would violate \ref{en:FB3}). Thus, the term
		\begin{equation*} \tag{$\ast$}\label{eq:BndConstCancel}
			\mu_1\wh b_{2\wh n}\ps{j} \sigma_{2n}\ldots\sigma_{i_r+1} \cdot \sigma_{n+i_r-1}\ldots\sigma_{n+1} \crs \sigma_1 \ldots \sigma_n = 0, 
		\end{equation*}
		either by \ref{en:FG2} or \ref{en:FG6}. Iteratively repeating this procedure with the symbols $\sigma_{n+i_{r-1}},\ldots, \sigma_{n+i_1}$ yields 
		\begin{align*}
			\wh b_{2\wh n}\ps{j} \crs' \sigma_{2n}\ldots\sigma_{n+1} \crs \sigma_1 \ldots \sigma_n =\mu_{\wh w} \wh b_n\ps{j} \crs'\sigma_{n}\inv\ldots\sigma_{1}\inv \crs \sigma_1 \ldots \sigma_n
		\end{align*}
		for some $\mu_{\wh w} \in K$, as required.
		
		For the second equality, one can perform the above iterative procedure further with the crease symbols  $\sigma_{i_1},\ldots,\sigma_{i_r}$ within the substring $\sigma_1\ldots\sigma_n$. Since each $\sigma_{i_k} = \sigma_{n+i_k}$, this produces a leading term
		\begin{equation*}
			\mu_{\wh w}^2 \wh b_{2\wh n}\ps{j} \crs' \sigma_{n}\inv\ldots\sigma_{1}\inv \crs \sigma_{n+1}\inv \ldots \sigma_{2n}\inv.
		\end{equation*}
		However, we note that the word $\rho(\wh w)$ may not be sufficiently long enough for other terms to vanish through the action (as in (\ref{eq:BndConstCancel})). This occurs if there exists $k$ such that $\sigma_{i_k}\in\{\crs',(\crs')\inv\}$. Thus, we may additionally be left with a $K$-linear combination $p_{\wh w}$ of terms in the set $\{\wh b_{2\wh n}\ps{j} \crs' \sigma_{n}\inv\ldots\sigma_{1}\inv \crs \sigma_{n+1}\inv \ldots \sigma_{n+i}\inv: t(\sigma_{n+i}\inv)=t(w), 1 \leq j \leq \wh m\}$. Hence,
		\begin{equation*}
			\wh b_{2\wh n}\ps{j} \crs' \sigma_{2n}\ldots\sigma_{n+1} \crs \sigma_1 \ldots \sigma_n 
	 		= \mu_{\wh w}^2 \wh b_{2\wh n}\ps{j} \crs'\sigma_{n}\inv\ldots\sigma_{1}\inv \crs \sigma_{n+1}\inv \ldots \sigma_{2n}\inv +p_{\wh w}.
		\end{equation*}
		
		That $\mu_{\wh w}=\mu_{g\wh w}\inv$ follows from the fact that $g\wh w = \wh w\inv$. In particular, every crease symbol of $\rho(g \wh w)$ is the inverse of the corresponding crease symbol in $\rho(\wh w)$ by Lemma~\ref{lem:Z2Invert}. The result then follows from Remark~\ref{rem:InverseCrease}.
	 \end{proof}
	 
	\begin{prop} \label{prop:USymBand}
		Let $M \in \mod*A$ and let $\wh w = (g\wh\sigma_{\wh n}\inv)\ldots(g\wh\sigma_{1}\inv)\wh\sigma_{1}\ldots\wh\sigma_{\wh n}$ be a band of $\wh A$ of even parity that is $\integer_2$-invariant and in standard form. Let $w=\rho(\wh w) = \crs'\sigma_{2n}\ldots\sigma_{n+1}\crs\sigma_1\ldots\sigma_n$ be the corresponding symmetric band, where $\crs,\crs' \in A$ are crease symbols that satisfy relations $\crs^2 - \lambda_1 \crs - \lambda_2\stp_{s(\crs)}$ and $(\crs')^2 - \lambda'_1 \crs' - \lambda'_2\stp_{s(\crs')}$, respectively. Let $M(\wh w, \wh m, \wh\phi) \in \mod*\wh A$ be a band module. Then the following are equivalent.
		\begin{enumerate}[label=(\alph*)]
			\item $M(\wh w, \wh m, \wh\phi)$ is isomorphic to a direct summand of $U(M)$.
			\item $M(g\wh w, \wh m, \tfrac{\mu_{g\wh w}^2}{\lambda_2\lambda'_2}\wh\phi)$ is isomorphic to a direct summand of $U(M)$, where $\mu_{g\wh w} = \mu_{\wh w}\inv \in K$ is as given in Lemma~\ref{lem:SymBandConst}.
			\item $M(\wh w, \wh m, \mu_{\wh w}^2\lambda_2\lambda'_2\wh\phi\inv)$ is a direct summand of $U(M)$, where $\mu_{\wh w} \in K$ is as given in Lemma~\ref{lem:SymBandConst}.
			\item There exists $m \in \spint$ and a direct summand $\wh N \subseteq U(M)$ such that $\wh N$ is closed under the action of $A$ and is isomorphic to a quasiband module
			\begin{equation*}
				M(\wh w, 2 m, \mu_{\wh w} \phi \phi') \cong
				\begin{cases}
					M(\wh w, \wh m, \wh\phi)	& \text{if } M(\wh w, \wh m, \wh\phi) \crs' \subseteq M(\wh w, \wh m, \wh\phi), \\
					M(\wh w, 2\wh m, \wh\phi \oplus \mu_{\wh w}^2\lambda_2\lambda'_2\wh\phi\inv) & \text{otherwise,}
				\end{cases}
			 \end{equation*}
			where $\phi,\phi' \in \Aut(K^{2m})$ are induced by the actions of $\crs$ and $\crs'$ on the subspaces $M(\wh w, \wh m, \wh\phi)\stp_{s(\wh \sigma_1)}$ and $M(\wh w, \wh m, \wh\phi)\stp_{t(\wh \sigma_{\wh n})}$ respectively, and $\mu_{\wh w} \in K$ is as given in Lemma~\ref{lem:SymBandConst}.
			\item There exists $m \in \spint$, $\psi \in \Aut(K^{2m})$, and a direct summand $\wh N \subseteq U(M)$ such that $\wh N$ is closed under the action of $A$ and is isomorphic to a quasiband module
			\begin{equation*}
				M(\wh w, 2 m, \mu_{\wh w} H \psi) \cong
				\begin{cases}
					M(\wh w, \wh m, \wh\phi)	& \text{if } M(\wh w, \wh m, \wh\phi) \crs' \subseteq M(\wh w, \wh m, \wh\phi), \\
					M(\wh w, 2\wh m, \wh\phi \oplus \mu_{\wh w}^2\lambda_2\lambda'_2\wh\phi\inv) & \text{otherwise,}
				\end{cases}
			 \end{equation*}
			where $H=\left(\begin{smallmatrix} 0 & \lambda_2 \\ 1 & \lambda_1 \end{smallmatrix}\right) \otimes_K \id_{m}$, and $\mu_{\wh w} \in K$ is as given in Lemma~\ref{lem:SymBandConst}.
			\item There exists a symmetric band module $M(w, m, \psi)$ that is isomorphic to a direct summand of $M$ and
			\begin{equation*}
				U(M(w, m, \psi)) \cong
				\begin{cases}
					M(\wh w, \wh m, \wh\phi)	& \text{if } M(\wh w, \wh m, \wh\phi) \crs' \subseteq M(\wh w, \wh m, \wh\phi), \\
					M(\wh w, 2\wh m, \wh\phi \oplus \mu_{\wh w}^2\lambda_2\lambda'_2\wh\phi\inv) & \text{otherwise.}
				\end{cases}
			\end{equation*}
		\end{enumerate}
		
	\end{prop}
	
	\begin{proof}
		(a) $\Rightarrow$ (b): Let $\wh B = \{\wh b_i\ps{j}: 1 \leq i \leq 2\wh n, 1 \leq j \leq \wh m\}$ be the standard basis of $M(\wh w, \wh m, \wh\phi)$. The first step of the proof is to observe the effects of the actions of the words $w=\fw(\wh w)$ and $w \inv = \fw(\wh w\inv) = \fw(g \wh w)$ on the elements $\wh b_{2\wh n}\ps{j}$ (see Remark~\ref{rem:UImage}(b) and Remark~\ref{rem:BandInverseAction}). In particular, it follows from Remark~\ref{rem:UImage}(b) that the action of $w$ on $\wh b_{2\wh n}\ps{j}$ is equivalent to the action of $\wh w$ on $\wh b_{2\wh n}\ps{j}$, which is a full rightward rotation of the band $\wh w$. That is, $\wh b_{2\wh n}\ps{j} \cdot w = \wh\phi(\wh b_{2\wh n}\ps{j})$. Conversely, the action of $g \wh w$ on $\wh b_{2\wh n}\ps{j}$ is a full leftward rotation of the band $\wh w$, and hence $\wh b_{2\wh n}\ps{j} \cdot w\inv = \wh\phi\inv(\wh b_{2\wh n}\ps{j})$.
		
		Since $s(w) \in \cv$, there must exist (by Remark~\ref{rem:UImage}, for each $j$) an element $\wh b_{2\wh n}\ps{j}\crs' \in U(M)$ that is linearly independent to $\wh b_{2\wh n}\ps{j} \in U(M)$. The next step of the proof is to compute the action of $w\inv$ on each $\wh b_{2\wh n}\ps{j}\crs'$. By Lemma~\ref{lem:SymBandConst}, we have
		\begin{align*}
			\wh b_{2\wh n}\ps{j}\crs' w\inv 
			&= \mu_{g\wh w}\wh b_{2\wh n}\ps{j}\crs' \sigma_{2n}\ldots\sigma_{n+1}\crs\inv\sigma_{n+1}\inv\ldots\sigma_{2n}\inv(\crs')\inv \\
			&= \mu_{g\wh w}^2\wh b_{2\wh n}\ps{j}\crs' \sigma_{2n}\ldots\sigma_{n+1}\crs\inv\sigma_1\ldots\sigma_n(\crs')\inv + p_{g\wh w}
		\end{align*}
		for some $\mu_{g\wh w}$ and some $K$-linear combination $p_{g\wh w}$ of terms in $\{\wh b_i\ps{j} \crs' \sigma_{2n}\ldots\sigma_{n+1}\crs\inv\sigma_1\ldots\sigma_i : t(\sigma_i)=t(w), 1 \leq j \leq \wh m\}$. Now using the appropriate identities for $\crs$ and $\crs'$, we obtain
		\begin{align*}
			\wh b_{2\wh n}\ps{j}\crs' w\inv 
			&= \wh b_{2\wh n}\ps{j}\left(\frac{\mu_{g\wh w}^2}{\lambda_2} w -\mu_{g\wh w}\frac{\lambda_1}{\lambda_2} \crs'\right)(\crs')\inv +p'_{g\wh w} \\
			&=\wh b_{2\wh n}\ps{j}\left(\frac{\mu_{g\wh w}^2}{\lambda_2} w \left(\frac{1}{\lambda'_2}\crs' - \frac{\lambda'_1}{\lambda'_2}\stp_{s(w)}\right) -\mu_{g\wh w}\frac{\lambda_1}{\lambda_2} \stp_{s(w)}\right) + p'_{g\wh w}\\
			&=\frac{\mu_{g\wh w}^2}{\lambda_2\lambda'_2} \wh \phi(\wh b_{2\wh n}\ps{j}\crs')
			-\frac{\mu_{g\wh w}^2\lambda'_1}{\lambda_2\lambda'_2}\wh\phi(\wh b_{2\wh n}\ps{j})
			-\frac{\mu_{g\wh w}\lambda_1}{\lambda_2}\wh b_{2\wh n}\ps{j}
			+p'_{g\wh w},
		\end{align*}
		where $p'_{g\wh w}$ is again some $K$-linear combination of terms in $\{\wh b_i\ps{j} \crs' \sigma_{2n}\ldots\sigma_{n+1}\crs\inv\sigma_1\ldots\sigma_i : t(\sigma_i)=t(w), 1 \leq j \leq \wh m\}$.
		
		Next, we consider the action of $w\inv$ on the vector $p'_{g \wh w} \in U(M)$. Since the substring $\sigma_1\ldots \sigma_n$ of $w$ is asymmetric, we must have $p'_{g \wh w}w\inv=0$  by \ref{en:FG2} or \ref{en:FG6}. Through this calculation, we have essentially shown that the word $w\inv$ acts on the subspace $\langle \wh b_{2\wh n}\ps{j},\wh b_{2\wh n}\ps{j}\crs', p'_{g\wh w} : 1 \leq j \leq \wh m\rangle \subseteq U(M)$ via the block matrix
		\begin{equation*}
			\begin{pmatrix}
				\wh\phi\inv & 
				-\frac{\mu_{g\wh w}}{\lambda_2}\left(
				\lambda_1 +\frac{\mu_{g\wh w}\lambda'_1}{\lambda'_2}\wh\phi\right) & 
				0 
				\\
				0 & 
				\frac{\mu_{g\wh w}^2}{\lambda_2\lambda'_2} \wh \phi & 
				0 
				\\
				0 & 
				1 & 
				0
			\end{pmatrix}.
		\end{equation*}
		It is not difficult to verify that a subset of the eigenvalues of this matrix are given by $\frac{\mu_{g\wh w}^2}{\lambda_2\lambda'_2}$-multiples of the eigenvalues of $\wh \phi$. By considering the eigenvectors of this matrix, one can verify that the elements
		\begin{equation*}
			\wh c_{2\wh n}\ps{j} = \frac{\mu_{g\wh w}^2}{\lambda_2\lambda'_2} \wh \phi(\wh b_{2\wh n}\ps{j}\crs') 
			- \wh\phi\inv(\wh b_{2\wh n}\ps{j}\crs') 
			-\frac{\mu_{g\wh w}\lambda_1}{\lambda_2}\wh b_{2\wh n}\ps{j} 
			-\frac{\mu_{g\wh w}^2\lambda'_1}{\lambda_2\lambda'_2}\wh\phi(\wh b_{2\wh n}\ps{j})
			+p'_{g \wh w} 
			- \mu_{\wh w}^2\lambda_2\lambda'_2\wh\phi^{-2}(p'_{g \wh w})
		\end{equation*}
		are such that $\wh c_{2\wh n}\ps{j} w\inv = \frac{\mu_{g\wh w}^2}{\lambda_2\lambda'_2} \wh\phi(\wh c_{2\wh n}\ps{j})$. Now let $\wh b'_i$ be the $i$-th subword of $g\wh w = \wh w\inv$ consisting of the first $i$ symbols. It is then easy to see that the set $\wh C = \{\wh c_i\ps{j}: 1 \leq i \leq 2\wh n, 1 \leq j \leq \wh m\}$ given by $\wh c_i\ps{j} = \wh c_{2\wh n}\ps{j}\fw(b'_i)$ is a basis for $M(g\wh w, \wh m, \frac{\mu_{g\wh w}^2}{\lambda_2\lambda'_2} \wh \phi)$. By construction, the subspace of $\langle \wh C \rangle \subseteq U(M)$ is clearly closed under the action of $\wh A$, and thus, $M(g\wh w, \wh m, \frac{\mu_{g\wh w}^2}{\lambda_2\lambda'_2} \wh \phi)$ is an $\wh A$-submodule of $U(M)$.
		
		It remains to show that  $\langle \wh C \rangle \cong M(g\wh w, \wh m, \frac{\mu_{g\wh w}^2}{\lambda_2\lambda'_2} \wh \phi)$ is actually a direct summand of $U(M)$. Suppose for a contradiction that this is not the case. Since $M(g\wh w, \wh m, \frac{\mu_{g\wh w}^2}{\lambda_2\lambda'_2} \wh \phi)$ is a band module, it is indecomposable. Thus, $M(g\wh w, \wh m, \frac{\mu_{g\wh w}^2}{\lambda_2\lambda'_2} \wh \phi)$ is a proper submodule of some indecomposable direct summand $\wh X \subseteq U(M)$. In particular, it follows from the classification of indecomposable modules of gentle algebras (for example, in \cite{ButlerRingel}) that $\wh X$ must be isomorphic to a band module $M(g\wh w, \wh m, \frac{\mu_{g\wh w}^2}{\lambda_2\lambda'_2}\wh \phi')$, where $\wh \phi'$ contains all of the eigenvalues of $\wh \phi$. By the construction of $U$ and Remark~\ref{rem:UImage}, this is only possible if $\langle \wh B \rangle \cong M(\wh w, \wh m, \wh \phi)$ is a proper $\wh A$-submodule of an indecomposable direct summand $M(\wh w, \wh m, \wh \phi') \subseteq U(M)$. But then this contradicts the assumption that $\langle \wh B \rangle$ is a direct summand of $U(M)$. Thus, $\langle\wh C\rangle \cong M(g\wh w, \wh m, \frac{\mu_{g\wh w}^2}{\lambda_2\lambda'_2} \wh \phi)$ is  a direct summand of $U(M)$, as required.
		
		(b) $\Rightarrow$ (a): 
		This follows by rotating $g\wh w$ such that it is in standard form and relabelling it as $\wh w$. In addition, by relabelling $\frac{\mu_{g\wh w}^2}{\lambda_2\lambda'_2} \wh \phi$ as $\wh \phi$ and considering the identities in Remark~\ref{rem:InverseCrease} and Lemma~\ref{lem:SymBandConst}, the result follows from the same procedure used in the proof for (a) $\Rightarrow$ (b).
		
		(b) $\Leftrightarrow$ (c): 
		This follows from Lemma~\ref{lem:BandIsos}(a) (which is applicable since $\wh w$ is asymmetric, despite $w$ being symmetric) and the fact that $g\wh w = \wh w \inv$. Also note that $\mu_{\wh w} = \mu_{g\wh w}\inv$ by Lemma~\ref{lem:SymBandConst}.

		((a) $\Leftrightarrow$ (b) $\Leftrightarrow$ (c)) $\Rightarrow$ (d):
		Again, let $\wh B$ and $\wh C$ be bases of $M(\wh w, \wh m, \wh\phi)$ and $M(g\wh w, \wh m, \tfrac{\mu_{g\wh w}^2}{\lambda_2\lambda'_2}\wh\phi)$ constructed in the proof of (a) $\Rightarrow$ (b). Consider the $\wh A$-module
		\begin{equation*}
			\wh N = \langle\wh B \rangle
			+
			\langle \wh C \rangle
			\subseteq U(M),
		\end{equation*}
		which is a direct summand of $U(M)$, since we have already proven that both $\langle\wh B \rangle$ and $\langle \wh C \rangle$ are (not necessarily distinct) direct summands of $U(M)$.  We will show that this module is closed under the action of $A$, and is hence an $A$-submodule of $U(M)$.
		
		Let $\wh b_i$ and $\wh b'_i$ be the $i$-th subwords of $\wh w$ and $g\wh w = \wh w\inv$ (respectively), each consisting of the first $i$ symbols. For each $i$, let $b_i = \fw(\wh b_i)$ and $b'_i = \fw(\wh b'_i)$. From Remark~\ref{rem:SymBandConst}(b) and Definition~\ref{def:FoldedWords}, we can see that for each $i$ and each $j>1$, the $j$-th crease symbol of $\crs' b'_i$ is the inverse of the $j$-th crease symbol of $b_i$. Moreover, for each $i$, we have $t(b_i) \in \cv$ if and only if $t(\crs' b'_i) \in \cv$. In addition, for each $i$ such that $t(b_i) \in \cv$, $b_i$ ends on a crease symbol if and only if $\crs' b'_i$ does not end on a crease symbol. Now note that $\wh N$ is closed under the action of the words $b_i$ and $\crs' b'_i$ by construction. Consequently, $\wh N$ is closed under the action of all symbols (and their inverses) in $w$. In addition, for any symbol $\sigma$ such that neither $\sigma$ nor $\sigma\inv$ is a symbol of $w$, we have $\wh N \sigma = \wh N \sigma\inv = 0$. Thus, $\wh N$ is closed under the action of $A$, as required, and hence, $\wh N$ is an $A$-submodule of $U(M)$.
		
		We will now show that we have
		\begin{equation*}
			\wh N \cong
			\begin{cases}
					M(\wh w, \wh m, \wh\phi)	& \text{if } M(\wh w, \wh m, \wh\phi) \crs' \subseteq M(\wh w, \wh m, \wh\phi), \\
					M(\wh w, 2\wh m, \wh\phi \oplus \mu_{\wh w}^2\lambda_2\lambda'_2\wh\phi\inv) & \text{otherwise.}
				\end{cases}
		\end{equation*}
		First suppose that $\langle\wh B \rangle$ is closed under the action of $\crs'$. For this, we note that $\langle\wh B \rangle$ is closed under the action of each $b'_i$, since each $b'_i$ is the inverse of some subword at the end of $b_{2\wh n}$. It therefore follows that $\langle\wh B \rangle$ is closed under the action of each $\crs' b'_i$. But from the construction of $\langle \wh C \rangle$, we must therefore have $\langle\wh B \rangle \cap \langle \wh C \rangle \neq 0$. Suppose for a contradiction that $\langle\wh B \rangle \neq \langle \wh C \rangle$. Note that by construction, for each $i$, every arrow of $\wh Q_1$ acts on each subspace $\langle \wh b_i\ps{j}, \wh c_i\ps{j} : 1 \leq j \leq 2\wh n\rangle$ by either an isomorphism onto the subspace $\langle \wh b_k\ps{j}, \wh c_k\ps{j} : 1 \leq j \leq 2\wh n\rangle$ (with $k = i \pm 1$) or by a zero action. But then if $\langle B \rangle \neq \langle\wh B \rangle \cap \langle \wh C \rangle \neq \langle C \rangle$, then this is only possible if $\wh \phi$ has a direct sum decomposition. But then $M(\wh w, \wh m, \wh\phi)$ is not a band module. Thus, in this case we must have $\langle\wh B \rangle = \langle \wh C \rangle$, as required. So $\wh N \cong M(\wh w, \wh m, \wh\phi)$.
		
		Now suppose that $\langle\wh B \rangle$ is not closed under the action of $\crs'$. Then by construction, at least some $c_i\ps{j}$ are linearly independent to the vector space $\langle\wh B \rangle$. But since $\langle\wh B \rangle$ and $\langle \wh C \rangle$ are isomorphic to band modules, they are each closed under the action of $\wh A$. This is possible only if $\langle\wh B \rangle \cap \langle \wh C \rangle = 0$. Thus, we must have $\wh N \cong M(\wh w, \wh m, \wh\phi) \oplus M(g\wh w, \wh m, \tfrac{\mu_{g\wh w}^2}{\lambda_2\lambda'_2}\wh\phi) \cong M(\wh w, 2\wh m, \wh\phi \oplus \mu_{\wh w}^2\lambda_2\lambda'_2\wh\phi\inv)$, as required.
		
		Finally, we will show that $\wh N \cong M(\wh w, 2 m, \mu_{\wh w} \phi \phi')$ for some $m \in \spint$ and some $\phi,\phi' \in \Aut(K^{2m})$ determined by the actions of $\crs$ and $\crs'$, respectively. We begin by noting that we have already shown that $\wh N$ is a quasiband module of $\wh w$. The first step is to therefore to determine the quantity $m$. In the case where $\wh N \cong M(\wh w, 2\wh m, \wh\phi \oplus \mu_{\wh w}^2\lambda_2\lambda'_2\wh\phi\inv)$, we clearly have $m = \wh m$. For the case where $\wh N \cong M(\wh w, \wh m, \wh\phi)$, we note that for each $b \in M(\wh w, \wh m, \wh\phi)$ such that $b \crs' \neq 0$, we have linearly independent elements $b, b\crs' \in M(\wh w, \wh m, \wh\phi)$. In particular, such elements of $M(\wh w, \wh m, \wh\phi)$ always come in pairs since $\crs'$ satisfies a quadratic relation. Since all symbols of $w$ act by isomorphisms between subspaces of dimension $\wh m$, this implies that $\wh m$ must be even. So $m = \frac{\wh m}{2}$ in this case.
		
		So we have now shown that $\wh N \cong M(\wh w, 2 m, \wh\phi')$, where $\wh \phi' = \wh \phi$ if $M(\wh w, \wh m, \wh\phi) \crs' \subseteq M(\wh w, \wh m, \wh\phi)$ and $\wh \phi' = \wh \phi \oplus \mu_{\wh w}^2\lambda_2\lambda'_2 \wh \phi\inv$ otherwise. It remains to show that $\wh \phi'$ is similar to $\mu_{\wh w} \phi \phi'$. Since $\wh N$ is a direct summand of $U(M)$ (which follows from (a), (c), and the isomorphism of $\wh N$ that we previously computed), there must exist a direct summand $N \subseteq M$ such that $U(N) \cong \wh N$ (since $\wh N$ is closed under the action of $A$ and $U$ is exact and additive). 
		 Now recall from the definition of $U$ that for each $1 \leq i \leq \wh n$ with $\pi(t(\wh\sigma_i)) \in \ov$, the subspaces $\wh N_i = U(N) \stp_{t(\wh\sigma_i)} \subseteq U(N)$ and $\wh N'_i = U(N) \stp_{t(g\wh\sigma_i)}\subseteq U(N)$ are disjoint copies of the subspace $N\stp_{t(\pi(\wh\sigma_i))}$. For each $1 \leq i \leq \wh n$ with $\pi(t(\wh\sigma_i)) \in \cv$, we have $t(\wh\sigma_i) = t(g\wh\sigma_i)$ and so $\wh N_i = U(N) \stp_{t(\wh\sigma_i)} = N\stp_{t(\pi(\wh\sigma_i))} = U(N) \stp_{t(g\wh\sigma_i)} = \wh N'_i$. In particular, this implies that for each $0 \leq i \leq \wh n$, the subspaces $\wh N_i, \wh N'_i \subseteq U(N)$ have a common basis, say $B_i$.
		 
		 Now let $a_n\ps{1},\ldots a_n\ps{2m}$ be the vectors $\wh b_{2\wh n}\ps{1},\ldots \wh b_{2\wh n}\ps{2m}$ written with respect to the basis $B_{\wh n}$. For each $n \geq i \geq 1$ and $1 \leq j \leq 2m$, inductively define vectors $a_{i-1}\ps{j} = a_{i}\ps{j} \sigma_i\inv$. We then have vectors $a_0\ps{1},\ldots a_0\ps{2m} \in B_{0}$ that are related to the vectors $a_n\ps{1},\ldots a_n\ps{2m} \in B_n$ by the identity action of the word $\sigma_1\ldots\sigma_n$. Next, let $b \in \langle a_n\ps{1},\ldots, a_{n}\ps{2m}\rangle \subseteq \langle B_{\wh n} \rangle = \wh N \stp_{s(\wh w)}$ and note (by Lemma~\ref{lem:SymBandConst}) that
		\begin{equation*}
			b \wh w = b \mu_{\wh w}\crs' \sigma_n\inv\ldots\sigma_1\inv \crs \sigma_1\ldots\sigma_n.
		\end{equation*}
		But $\wh b_{2 \wh n}\ps{j} \wh w = \wh\phi'(\wh b_{2 \wh n}\ps{j})$ for each $j$. Thus,  $\wh \phi'$ must be similar to $\mu_{\wh w} \phi \phi'$ with respect to the common bases $B_i$. Thus, $\wh N \cong M(\wh w, 2m, \wh \phi') \cong M(\wh w, 2m, \mu_{\wh w} \phi \phi')$, as required.
		
		(d) $\Rightarrow$ (a): This is obvious from the statement itself.
		
		((a) $\Leftrightarrow$ (b) $\Leftrightarrow$ (c) $\Leftrightarrow$ (d)) $\Rightarrow$ (e):
		Every representation $\phi$ of $\crs$ satisfies the relation $\phi^2 - \lambda_1\phi-\lambda_2$. Any such $\phi$ is thus similar to the matrix $H$. That is, there exists $h \in \Aut(K^{2m})$ such that $H=h\phi h\inv$. But then
		\begin{equation*}
			h\phi \phi' h\inv = h \phi h\inv h \phi' h\inv = H h \phi' h\inv.
		\end{equation*}
		Thus, $M(\wh w, 2m, \mu_{\wh w}\phi\phi') \cong M(\wh w, 2m, \mu_{\wh w}H h \phi' h\inv)$. The result then follows by defining $\psi=h \phi' h\inv$.
		
		(e) $\Rightarrow$ (a):
		This is obvious from the statement itself
		
		(e) $\Rightarrow$ (f):
		We have already shown that there exists a direct summand $N = \langle a_i\ps{j} : 0 \leq i \leq n, 1 \leq j \leq 2m \rangle \subseteq M$ such that $U(N) \cong M(\wh w, 2m, \mu_{\wh w} H \psi)$. It is easy to verify that the basis $\{a_i\ps{j} : 0 \leq i \leq n, 1 \leq j \leq 2m\}$ constructed above is also the standard basis for the module $M(w, m, \psi)$. So $N \cong M(w, m, \psi)$. It remains to show that $M(w, m, \psi)$ is indecomposable, and is thus a band module (and not just a quasi-band module).
		
		There are two cases to consider: either the subspace $\langle \wh B \rangle  = \langle \wh b_i\ps{j} : 1 \leq i \leq 2\wh n, 1 \leq j \leq \wh m\rangle \cong  M(\wh w, \wh m, \wh \phi)$ is closed under the action of $\crs'$ or it is not. So first suppose that this subspace is closed under the action of $\crs'$. Then $U(M(w, m, \psi)) \cong M(\wh w, \wh m, \wh \phi)$ and is thus indecomposable (since $M(\wh w, \wh m, \wh \phi)$ is a band module of a gentle algebra). Now suppose for a contradiction that $M(w, m, \psi)$ is decomposable. Then there exists a non-trivial idempotent $f \in \End_A(M(w, m, \psi))$. But then from the definition of $U$, there must exist a non-trivial idempotent $U(f) \in \End_A(U(M(w, m, \psi)))$, which would imply that $U(M(w, m, \psi))$ is decomposable --- a contradiction to the fact that $M(\wh w, \wh m, \wh \phi)$ is a band module. Thus, $M(w, m, \psi)$ is indecomposable in this case.
		
		For the other case, suppose instead that $\langle \wh B \rangle \cong  M(\wh w, \wh m, \wh \phi)$ is not closed under the action of $\crs'$. Then
		\begin{align*}
			U(M(w, m, \psi)) &\cong M(\wh w, 2\wh m, \wh \phi \oplus \mu_{\wh w}^2 \lambda_2\lambda'_2 \wh \phi\inv) \\
			&\cong M(\wh w, m, \wh \phi) \oplus M(\wh w, m, \mu_{\wh w}^2 \lambda_2\lambda'_2 \wh \phi\inv) \\
			&\cong \langle \wh B\rangle \oplus \langle \wh C\rangle,
		\end{align*}
		where $\wh C$ is as in the proof of (a) $\Rightarrow$ (b). Now suppose for a contradiction that $M(w,m,\psi)$ is decomposable, say $M(w, m, \psi) \cong N' \oplus N''$. But then we must have $U(N') \cong M(\wh w, m, \wh \phi)$ and $U(N'') \cong M(\wh w, m, \mu_{\wh w}^2 \lambda_2\lambda'_2 \wh \phi\inv)$ (up to relabelling of direct summands), since $M(\wh w, m, \wh \phi)$ and $M(\wh w, m, \mu_{\wh w}^2 \lambda_2\lambda'_2 \wh \phi\inv)$ are both band modules (and thus indecomposable), and $U$ is additive. But then $\langle \wh B\rangle \cong M(\wh w, m, \wh \phi)$ is an indecomposable object in the image of $U$. By (d), this is true only if $\langle \wh B\rangle$ is closed under the action of $\crs'$ --- a contradiction to our original supposition. Thus, $M(w, m, \psi)$ is indecomposable. Hence, $M(w, m, \psi)$ is a symmetric band module.
		
		(e) $\Rightarrow$ (a):
		This is obvious from the statement itself.
	\end{proof}
	
	The above Proposition implies that if $M(\wh w, \wh m, \wh \phi)$ is a maximal quasi-band direct summand of $U(M)$ for some $\integer_2$-invariant band $\wh w$ that is of even parity, then $\wh m$ is even and $\mu_{g\wh w}\wh \phi$ is similar to $H\psi$, where $H$ and $\psi$ satisfy some quadratic relations. It may not be immediately obvious to the reader that this decomposition of $\mu_{g\wh w}\wh \phi$ into $H\psi$ is well-defined, but this is indeed the case, as the next lemma shows.
	
	\begin{lem}\label{lem:SymInj}
		Let $M(w, m, \psi)$ and $M(w, m, \psi')$ be symmetric band modules of $A$. Then
		\begin{equation*}
			U(M(w, m, \psi)) \cong U(M(w, m, \psi')) \Rightarrow M(w, m, \psi) \cong M(w, m, \psi').
		\end{equation*}
	\end{lem}
	\begin{proof}
		For readability, write $M= M(w, m, \psi)$ and $M' = M(w, m, \psi')$. Let $B=\{\wh b_i\ps{j} : 1 \leq i \leq 2\wh n, 1 \leq j \leq 2m\}$ and $C=\{\wh c_i\ps{j} : 1 \leq i \leq 2\wh n, 1 \leq j \leq 2m\}$ be the standard bases of $\wh M = U(M) \cong M(\wh w, 2 m, \mu_{\wh w} H \psi)$ and $\wh M' = U(M') \cong M(\wh w, 2 m, \mu_{\wh w} H \psi')$, respectively. Let $\wh M_i = \langle\wh b_i\ps{j} : 1 \leq j \leq 2m\rangle$ and $\wh M'_i=\langle \wh c_i\ps{j} : 1 \leq j \leq 2m\rangle$. We will construct an isomorphism $f\colon \wh M \rightarrow \wh M'$ such that $f|_{\wh M_1}=\ldots=f|_{\wh M_{2\wh n}}$ commutes with $H=\left(\begin{smallmatrix} 0 & \lambda_2 \\ 1 & \lambda_1 \end{smallmatrix}\right) \otimes_K \id_{m}$.
		
		Let $u=t(\wh b_{\wh n}\ps{j}) = t(\crs) \in \cv$ and $v=t(\wh b_{2\wh n}\ps{j}) = t(\crs')\in \cv$. One can then assign a semisimple $A$-module structure to the vector subspaces $\wh M_{\wh n}$, $\wh M'_{\wh n}$, $\wh M_{2\wh n}$ and $\wh M'_{2\wh n}$ via $A$-module isomorphisms
		\begin{align*}
			\xymatrix{S(u)^m = \langle \stp_u, \crs_u \rangle^m \cong \langle \stp_u\ps{j}, \crs_u\ps{j} : 1 \leq j \leq m\rangle \ar[r]^-{\omega_u}_-{\cong} & \langle\wh b_{\wh n}\ps{j} : 1 \leq j \leq 2m\rangle = \wh M_{\wh n}}, \\
			\xymatrix{S(v)^m = \langle \stp_v, \crs_v \rangle^m \cong \langle \stp_v\ps{j}, \crs_v\ps{j} : 1 \leq j \leq m\rangle \ar[r]^-{\omega_v}_-{\cong} & \langle\wh b_{2\wh n}\ps{j} : 1 \leq j \leq 2m\rangle = \wh M_{2\wh n}}, \\
			\xymatrix{S(u)^m \cong \langle \stp_u\ps{j}, \crs_u\ps{j} : 1 \leq j \leq m\rangle \ar[r]^-{\sigma_u}_-{\cong} & \langle\wh c_{\wh n}\ps{j} : 1 \leq j \leq 2m\rangle = \wh M'_{\wh n}}, \\
			\xymatrix{S(v)^m \cong \langle \stp_v\ps{j}, \crs_v\ps{j} : 1 \leq j \leq m\rangle \ar[r]^-{\sigma_v}_-{\cong} & \langle\wh c_{2\wh n}\ps{j} : 1 \leq j \leq 2m\rangle = \wh M'_{2\wh n}},
		\end{align*}
		where 
		\begin{align*}
			\stp_u\ps{j} \crs &= \crs_u\ps{j},
			& \crs_u\ps{j} \crs &=\lambda_1 \crs_u\ps{j} + \lambda_2 \stp_u\ps{j},
			& \wh b_{\wh n}\ps{j} \crs &= H(\wh b_{\wh n}\ps{j}),
			\\
			\stp_v\ps{j} \crs' &= \crs_v\ps{j},
			& \crs_v\ps{j} \crs' &=\lambda'_1 \crs_v\ps{j} + \lambda'_2 \stp_v\ps{j},
			& \wh b_{2\wh n}\ps{j} \crs' &= \psi(\wh b_{2\wh n}\ps{j}),
		\end{align*}
		and $\wh c_{\wh n}\ps{j} \crs = H(\wh c_{\wh n}\ps{j})$ and $\wh c_{2\wh n}\ps{j} \crs' = \psi'(\wh c_{2\wh n}\ps{j})$. We thus have
		\begin{align*}
			\omega_u\inv(\wh b_{\wh n}\ps{j}) &= \stp_u\ps{j}, &
			\omega_u\inv(\wh b_{\wh n}\ps{j+m}) &= \crs_u\ps{j},
			\\
			\sigma_u\inv(\wh c_{\wh n}\ps{j}) &= \stp_u\ps{j}, &
			\sigma_u\inv(\wh c_{\wh n}\ps{j+m}) &= \crs_u\ps{j}.
		\end{align*}
		Now let $H' =\left(\begin{smallmatrix} 0 & \lambda'_2 \\ 1 & \lambda'_1 \end{smallmatrix}\right) \otimes_K \id_{m}$. Since $\psi^2 = \lambda'_1 \psi + \lambda'_2$ and $(\psi')^2 = \lambda'_1 \psi' + \lambda'_2$, there must exist matrices $X=(x_{k,l}),Y=(y_{k,l}) \in \Aut(K^{2m})$ such that
		\begin{align*}
			\psi &= X H' X\inv, & \psi' &= Y H' Y\inv, \\
			\omega_v\inv(\wh b_{2\wh n}\ps{j}) &=\sum_{k=1}^{m} (x_{k,j} \stp_v\ps{k} + x_{k+m,j} \crs_v\ps{k}), &
			\sigma_v\inv(\wh c_{2\wh n}\ps{j}) &=\sum_{k=1}^{m} (y_{k,j} \stp_v\ps{k} + y_{k+m,j} \crs_v\ps{k}),\\
			\omega_v\inv(\wh b_{2\wh n}\ps{j}) \crs' &= \omega_v\inv\psi(\wh b_{2\wh n}\ps{j}), & 
			\sigma_v\inv(\wh c_{2\wh n}\ps{j}) \crs' &= \sigma_v\inv\psi'(\wh c_{2\wh n}\ps{j}).
		\end{align*}
		That is, the bases of $\wh M_{2 \wh n}$ and $\wh M'_{2 \wh n}$ are each given by a change of basis from the standard basis of $S(v)^m$, and this is such that the action of $\crs'$ on $\wh M_{2 \wh n}$ and $\wh M'_{2 \wh n}$ is given by the automorphisms $\psi$ and $\psi'$ respectively.
		
		Now consider the $\wh A$-module structures of $\wh M$ and $\wh M'$. Firstly, for any $k$ and $l$, each basis element $\wh b_k\ps{j}$ of $\wh M_{k}$ is identified with the basis element $\wh b_l\ps{j}$ of $\wh M_{l}$ via the action of some subword of a rotation of $\wh w$. A similar statement holds for the basis elements of $\wh M'_{k}$ and $\wh M'_{l}$. In particular, we remark that each $\wh b_{2\wh n}\ps{j}$ is identified with $\wh b_{\wh n}\ps{j}$ and each $\wh c_{2\wh n}\ps{j}$ is identified with $\wh c_{\wh n}\ps{j}$.
		
		Now since it is already given that $\wh M$ and $\wh M'$ are isomorphic $\wh A$-modules, it follows that $\wh M'_{2\wh n}$ is given by a change of basis $f|_{\wh M_{2\wh n}}$ of $\wh M_{2\wh n}$. In addition, $\wh M'_{\wh n}$ is given by a change of basis $f|_{\wh M_{\wh n}}=f|_{\wh M_{2\wh n}}$ of $\wh M_{\wh n}$. But for any $\wh b_{\wh n}\ps{j} \in \wh M_{\wh n}$ and $\wh c_{\wh n}\ps{j} \in \wh M'_{\wh n}$, we have $\wh b_{\wh n}\ps{j} \crs = H(\wh b_{\wh n}\ps{j})$ and $\wh c_{\wh n}\ps{j}\crs = H(\wh b_{\wh n}\ps{j})$. Since $\wh M,\wh M' \in \im U$, there must therefore exist a choice of $f$ such that $f|_{\wh M_{\wh n}}$ commutes with $H$.
		
		For brevity, write $h = f|_{\wh M_{\wh n}} = f|_{\wh M_{2\wh n}}$. Since $f$ is an isomorphism of $\wh A$-modules, we then have
		\begin{equation*}
			H \psi' = h H \psi h\inv = h H h\inv h \psi h\inv = H h \psi h\inv.
		\end{equation*}
		Thus,  $\psi' = h \psi h\inv$. In particular, $(K^{2m}, H, \psi)$ and $(K^{2m}, H, \psi')$ are isomorphic representations of the algebra
		\begin{equation*}
			K\langle x, y \rangle / \langle x^2 - \lambda_1 x - \lambda_2, y^2 - \lambda'_1 y - \lambda'_2 \rangle.
		\end{equation*}
		Hence, $M(w, m, \psi) \cong M(w, m, \psi')$, as required.
	\end{proof}
	
	\begin{rem} \label{rem:UIndImage}
		Determining precisely when the image (under $U$) of a symmetric band module is indecomposable appears to be quite a hard problem in general. By the above Proposition, there is a necessary condition for a band module $M(\wh w, \wh m, \wh \phi) \in \mod*\wh A$ to appear as the indecomposable image of some symmetric band module in $\mod*A$. Namely, $\wh \phi$ must be similar to $\mu_{\wh w}^2 \lambda_2\lambda'_2 \wh \phi\inv$. Observationally, this condition also appears to be sufficient if $K = \real$. However, it is not at all obvious if this condition is sufficient in general.
		
		On the other hand, it is not hard to prove that for any $\wh \phi \in \Aut(K^{\wh m})$, the module $M(\wh w, \wh m, \wh \phi)$ appears as a direct summand in the image of some symmetric quasi-band module $M(w, \wh m, \phi')$ under $U$. Namely, for
		\begin{equation*}
			\phi' = 
			\begin{pmatrix}
				\lambda'_1 & \mu_{g \wh w}\wh\phi \\
				\mu_{\wh w}\lambda'_2 \wh\phi\inv & 0
			\end{pmatrix}.
		\end{equation*}
		Consequently, for any band module $M(\wh w, \wh m, \wh \phi)$ of $\mod* \wh A$, there exists a symmetric band module $M(\wh w, m, \phi')$ such that either $U(M(\wh w, m, \phi')) \cong M(\wh w, \wh m, \wh \phi)$ or $U(M(\wh w, m, \phi')) \cong M(\wh w, \wh m, \wh \phi) \oplus M(\wh w, \wh m, \mu_{\wh w}^2 \lambda_2\lambda'_2 \wh \phi\inv)$.
	\end{rem}
	
	\subsection{Classification of indecomposable modules}
	We will now use the unfolding functor to provide a classification of the iso-classes of indecomposable modules of a folded gentle algebra $A$ (without repetitions). We begin with a technical result that follows from the propositions proved earlier in this section.

	\begin{lem} \label{lem:UEssInj}
		Let $M,N \in \mod*A$. Suppose $U(M)$ does not contain a direct summand that is isomorphic to a quasiband module $M(\wh w,m,\wh \phi)$ with $\wh w \in \bnd_{\wh A}$ of odd parity. Then $U(M) \cong U(M)$ only if $M \cong N$.
	\end{lem}
	\begin{proof}
		Since $U$ is additive, it is sufficient to prove the lemma for $M$ indecomposable, so we will suppose this throughout the proof. Also note that from the classification of indecomposable modules over gentle algebras, we know that $U(M)$ has a unique (up to isomorphism) decomposition into string and band modules. Thus given the assumption of the lemma statement, there are three cases we need to consider: either $U(M)$ contains a direct summand that is isomorphic to a string module $M(\wh w)$, a band module $M(\wh w,m,\phi)$ with $\wh w$ not $\integer_2$-invariant, or a band module $M(\wh w,\wh m,\wh \phi)$ with $\wh w$ being $\integer_2$-invariant and of even parity.
		
		Case 1: Suppose $U(M)$ contains a direct summand isomorphic to a string module $M(\wh w)$. Then by Proposition~\ref{prop:UStrings}, $M \cong M(\rho(\wh w))$ and $U(M) \cong M(\wh w) \oplus M(g\wh w)$. Now if $U(N) \cong U(M)$, then $U(N) \cong M(\wh w) \oplus M(g\wh w)$. Thus, Proposition~\ref{prop:UStrings} implies that we also have $N \cong  M(\rho(\wh w))$, and hence $M \cong N$, as required.
		
		Case 2: Suppose $U(M)$ contains a direct summand isomorphic to a band module $M(\wh w,m,\phi)$ with $\wh w$ not $\integer_2$-invariant. Then by Proposition~\ref{prop:UEvenBands}, $M \cong M(\rho(w),m,\phi)$ and $U(M) \cong M(\wh w, m, \phi) \oplus M(\wh w, m, \lambda_{\wh w}\phi)$, where $\lambda_{\wh w} \in K$ is as in Remark~\ref{rem:LamConst}. If $U(N) \cong U(M)$, then $U(N) \cong M(\wh w, m, \phi) \oplus M(\wh w, m, \lambda_{\wh w}\phi)$, and thus it again follows from Proposition~\ref{prop:UEvenBands} that $N \cong  M(\rho(\wh w), m, \phi) \cong M$, as required.
		
		Case 3: Suppose $U(M)$ contains a direct summand isomorphic to a band module $M(\wh w,\wh m,\wh\phi)$ with $\wh w$ being $\integer_2$-invariant and of even parity. Then the result follows by Lemma~\ref{lem:SymInj}
	\end{proof}
	
	We will now define the collections of objects that form the classification of indecomposable modules of a folded gentle algebra. First, the collection of string modules of $\mod*A$.
	
	\begin{defn} \label{defn:IndM1}
		Define $\mathcal{M}_1$ to be the complete collection of all string modules of $A$ up to equivalence of strings. That is,
		\begin{equation*}
			\mathcal{M}_1 = \{M(w) : w \in \str_A\}.
		\end{equation*}
	\end{defn}		

	Next, the collection of asymmetric band modules. For this we need some additional notation. Let $\Pi$ be the set
	\begin{equation*}
		\Pi = \{p^n \in K[x] : n \in \spint \text{ and } p \text{ is monic, irreducible, and } p(0) \neq 0\}.
	\end{equation*}
	For any polynomial $p \in K[x]$, denote by $\phi_p$ its companion matrix. The condition that $p(0) \neq 0$ for any $p \in \Pi$ implies that $p \neq x^n$ for any $n$, which further implies that $\phi_p$ is an automorphism of $K^{\deg(p)}$.
	
	\begin{defn} \label{def:AsymBandsClass}
		 Define $\mathcal{M}_2$ to be the complete collection of all asymmetric band modules $M(w,\deg(p),\phi_p)$ up to equivalence of $w$, with $p \in \Pi$. That is,
		\begin{equation*}
			\mathcal{M}_2 = \{M(w,\deg(p),\phi_p) : w \in \bnd_A, w \text{ asymmetric}, p \in \Pi\}.
		\end{equation*}
	\end{defn}
	
	Finally, the collection of symmetric band modules. As the most complicated class of objects, this requires even more notation. So define
	\begin{equation*}
		\bnd_A^s = \{\rho(\wh w) \in \bnd_A : \rho(\wh w) \text{ is symmetric}, \wh w \in \bnd_{\wh A} \text{ is in standard form}\}.
	\end{equation*}
	For each $w \in \bnd_A^s$, let $\crs_w$ and $\crs_{w}'$ be the crease symbols such that $\rho(\theta(w)) = \crs'_w\sigma_{2n}\ldots\sigma_{n+1}\crs_w\sigma_1\ldots\sigma_n$, and suppose that $\crs_w^2-\lambda_1 \crs_w - \lambda_2\stp_{s(\crs_w)}$ and $(\crs'_w)^2-\lambda'_1 \crs'_w - \lambda'_2\stp_{s(\crs'_w)}$ are the corresponding relations of the crease symbols. For each $m \in \spint$, let $H_{w,m} = \left(\left(
	\begin{smallmatrix}
		0	&	\lambda_2 \\
		1	& \lambda_1
	\end{smallmatrix}
	\right)\otimes_K \id_m\right)$. In addition, let $\Phi\ps{2m}_w$ be the collection of all classes of automorphisms $\phi' \in \Aut(K^{2m})$ satisfying $(\phi')^2 - \lambda'_1 \phi' - \lambda'_2 =0$, up to similarity of the matrix $\mu_{\theta(w)}H_{w,m}\phi'$. For technical purposes, we will define $\Phi\ps{2m+1}_w=\emptyset$.
	
	Next, note that there is a group action of $\integer_2 = \{1,g\}$ on $\Pi$ that is dependent on the choice of $w$, which is given in the following way. Let $\phi_p$ be a companion matrix with $p \in \Pi$. Then $\phi_p$ is an indecomposable matrix whose characteristic polynomial is $p$. Consequently, the matrix $\mu_{\theta(w)} \lambda_2\lambda'_2 \phi_p\inv$ is also indecomposable and so must be similar to some companion matrix $\phi_{gp}$, where $gp \in \Pi$ is the characteristic polynomial of $\mu_{\theta(w)} \lambda_2\lambda'_2 \phi_p\inv$. It is easy to see that $g^2 p = p$, and hence that this is indeed a group action. Define $\Pi_w$ to be the set of orbits of $\Pi$ under $\integer_2$.
	
	For each $w \in \bnd_A^s$ and $p \in \Pi_w$, define
	\begin{align*}
		\Psi_{w,p}\ps{1} &= \{[\psi] \in \Phi\ps{\deg(p)}_w : \mu_{\theta(w)}H_{w,\deg(p)} \psi \text{ is similar to } \phi_p\}, \\
		\Psi_{w,p}\ps{2} &=
		\begin{cases}
			\emptyset & \text{if } \Psi_{w,p}\ps{1} \neq \emptyset,\\
			\{[\psi] \in \Phi\ps{2 \deg(p)}_w : \mu_{\theta(w)}H_{w,2\deg(p)} \psi \text{ is similar to } \phi_p \oplus \phi_{gp}\}	& \text{if } \Psi_{w,p}\ps{1} = \emptyset,
		\end{cases}\\
		\Psi_{w,p} &= \Phi_{w,p}\ps{1} \cup \Phi_{w,p}\ps{2}.
	\end{align*}
	For any $[\psi] \in \Psi_{w,p}$, we define $m_\psi = \frac{1}{2}\deg(p)$ if $[\psi] \in \Psi_{w,p}\ps{1}$ and $m_\psi = \deg(p)$ if $[\psi] \in \Psi_{w,p}\ps{2}$.
	
	\begin{rem} \label{rem:PsiDi}
		By Proposition~\ref{prop:USymBand}, Lemma~\ref{lem:SymInj} and Remark~\ref{rem:UIndImage}, it follows that for each $p \in \Pi_w$, we have $|\Psi_{w,p}\ps{1}|,|\Psi_{w,p}\ps{2}| \in \{0,1\}$, with $|\Psi_{w,p}\ps{1}| =1$ if and only if $|\Psi_{w,p}\ps{2}| =0$. Thus, we always have $|\Psi_{w,p}|=1$. 
	\end{rem}
	
	\begin{defn} \label{def:SymBandInds}
		Define $\mathcal{M}_3$ to be the complete collection of all symmetric band modules $M(w,m_{\psi_p},\psi_p)$ up to equivalence of $w$, with $[\psi_p] \in \Psi_{w,p}$ and $p \in \Pi_w$. That is,
		\begin{equation*}
			\mathcal{M}_3 = \{M(w,m_{\psi_p},\psi_p) : w \in \bnd^s_A, \psi_p \in \Psi_{w,p}, p \in \Pi_w\}.
		\end{equation*}
	\end{defn}
	
	\begin{lem}\label{lem:NoRepetitions}
		Let $M \in \mathcal{M}_i$ and $N \in \mathcal{M}_j$.
		\begin{enumerate}[label=(\alph*)]
			\item If $i \neq j$, then $M \not\cong N$.
			\item If $i = j$ but $M$ and $N$ are distinct objects, then $M \not\cong N$.
		\end{enumerate}
		Consequently $M \not\cong N$ for any distinct objects $M,N \in \mathcal{M}_1 \cup \mathcal{M}_2 \cup \mathcal{M}_3$
	\end{lem}
	\begin{proof}
		(a) By Proposition~\ref{prop:UStrings}, the objects of $\mathcal{M}_1$ map to a direct sum of string modules of $\wh A$ under $U$. By Propositions~\ref{prop:UEvenBands}, \ref{prop:UOddBands} and \ref{prop:USymBand}, the modules of $\mathcal{M}_{2} \cup \mathcal{M}_{3}$ map to band modules of $\wh A$ under $U$. Thus, $U(M) \not\cong U(N)$ for any $M \in \mathcal{M}_1$ and any $N \in \mathcal{M}_{2} \cup \mathcal{M}_{3}$. Hence, $M \not\cong N$ in this case. So suppose instead that $M \in \mathcal{M}_{2}$ and $N \in \mathcal{M}_{3}$. Then by Lemma~\ref{lem:FoldedSymmetric} and Propositions~\ref{prop:UEvenBands} and \ref{prop:UOddBands}, $U(M)$ is a quasiband module of a band $\wh w$ which is either of odd parity or is of even parity but not $\integer_2$-invariant. On the other hand, by Lemma~\ref{lem:FoldedSymmetric} and Proposition~\ref{prop:USymBand}, $U(N)$ is a quasiband module of a band $\wh w'$ that is $\integer_2$-invariant and of even parity. But then $\wh w \not\approx \wh w'$, and thus $U(M) \not\cong U(N)$ by the classification of indecomposable modules over a gentle algebra. Thus, $M \not\cong N$ in this final case, as required.
		
		(b) Consider the case where $i=1$. Then $M \cong M(w)$ and $N \cong M(w')$, where $w \not\approx w'$. By Proposition~\ref{prop:UStrings}, $U(M) \cong M(\theta(w)) \oplus M(g\theta(w))$ and $U(N) \cong M(\theta(w')) \oplus M(g\theta(w'))$. By Lemma~\ref{lem:FoldedStrings}, this implies that $\theta(w) \not\approx \theta(w')$, and thus that $U(M) \not\cong U(N)$. Hence, $M \not\cong N$, as required.
		
		Now consider the case where $i=2$. Then $M \cong M(w,\deg(p),\phi_p)$ and $N \cong M(w', \deg(p'), \phi_{p'})$. In particular, $w \not\approx w'$ or $p \neq p'$. Suppose $w \not\approx w'$, then Lemma~\ref{lem:FoldedBands} as well as Propositions~\ref{prop:UEvenBands} and \ref{prop:UOddBands} imply that $U(M) \not\cong U(N)$. So $M \not\cong N$ in this subcase. 
		
		So consider the remaining subcase (where $w \approx w'$ but $p \neq p'$). Let $\{b\ps{j}_i : 1 \leq i \leq n, 1 \leq j \leq m\}$ be the standard basis of $M(w,\deg(p),\phi_p)$ and $\{c\ps{j}_i : 1 \leq i \leq n, 1 \leq j \leq m\}$ be the standard basis of $M(w, \deg(p'), \phi_{p'})$. Let $B_i = \langle b\ps{j}_i : 1 \leq j \leq m\rangle$ and $C_i = \langle c\ps{j}_i : 1 \leq j \leq m\rangle$. Suppose for a contradiction that there exists an isomorphism $h \in \Hom_A(M(w,\deg(p),\phi_p),M(w, \deg(p'), \phi_{p'}))$. Then it follows from the band structure that each $h|_{B_i} \colon B_i \rightarrow C_i$ is such that $h|_{B_i} = h|_{B_j}$ for any $i$ and $j$.
		
		Now note that there exists a $K$-linear, exact, faithful functor $F_w \colon \fin K[x] \rightarrow \mod*\wh A$ defined by $F_w((K^a, \psi)) = M(w,a,\psi)$. Thus, both $M(w,\deg(p),\phi_p)$ and $M(w, \deg(p'), \phi_{p'})$ are in the image of $F_w$. Moreover, $F_w$ is defined such that $F_w(f)|_{B_i} = f$ for each $i$, where $f\colon K^{m} \rightarrow K^{m'}$ satisfies $f\phi_{p} = \phi_{p'}f$. From this, it is easy to see that $F_w$ reflects isomorphisms. But then $M(w,\deg(p),\phi_p) \cong M(w, \deg(p'), \phi_{p'})$ only if $p=p'$, which is a contradiction to our assumption that $p\neq p'$. So $M \not\cong N$ in this subcase either. This completes the proof for $i=2$.
		
		Finally, consider the case where $i=3$. We have $M \cong M(w,m_{\psi_{p}},\psi_{p})$ and $N \cong M(w', m_{\psi_{p'}},\psi_{p'})$. In this case, either $w \not\approx w'$ or $p \neq p'$ (since $w$ and $w'$ are assumed to be in standard form). The proof for the subcase where $w \not\approx w'$ is similar to the $i \in \{1,2\}$ cases, with the use of Proposition~\ref{prop:USymBand}. So suppose $p \neq p'$, then by Proposition~\ref{prop:USymBand}, we also have $U(M) \not\cong U(N)$. So $M \not\cong N$, as required.
		
		The final statement of the lemma is then a trivial consequence of both (a) and (b).
	\end{proof}
	
	\begin{thm} \label{thm:IndClassification}
		Let $A$ be a folded gentle algebra. The combined collection of all objects in $\mathcal{M}_1 \cup \mathcal{M}_2 \cup \mathcal{M}_3$ forms a complete collection of all iso-classes of indecomposable objects of $\mod*A$ without repetitions.
	\end{thm}
	\begin{proof}
		First, let us show that the objects of $\mathcal{M}_1$ are indecomposable. Let $M(w) \in \mathcal{M}_1$. Then $U(M(w)) \cong M(\theta(w)) \oplus M(g\theta(w))$ by Proposition~\ref{prop:UStrings}. Suppose for a contradiction that $M(w)$ is decomposable. Since $U$ is additive and exact, and string modules of $\mod*\wh A$ are indecomposable, this is only possible if we have $M(w) \cong N \oplus L$ such that $U(N) \cong M(\theta(w))$ and $U(L) \cong M(g\theta(w))$. But then by Proposition~\ref{prop:UStrings}, $M(\theta(w)) \oplus M(g\theta(w))$ must be a direct summand of $U(N)$, which is a contradiction. So the objects of $\mathcal{M}_1$ are indecomposable, as required.
		
		Now let us show that the objects of $\mathcal{M}_2$ are indecomposable. Denote by $\fin K[x, x\inv]$ the category of finite-dimensional representations of $K[x, x\inv]$. Then for each asymmetric $w \in \bnd_A$, there exists an obvious $K$-linear, exact, faithful functor $F_w\colon \fin K[x, x\inv] \rightarrow \mod*A$ that preserves indecomposable objects (given by $F_w(K^m,\phi) = M(w,m,\phi)$). Since $(K^{\deg(p)},\phi_p) \in \fin K[x, x\inv]$ is indecomposable (by the classification of indecomposable modules over a principal ideal domain), it follows that each $M( w, \deg(p), \phi_p) \in \mathcal{M}_2$ is indecomposable as required.
		
		For $\mathcal{M}_3$, we have already shown that the objects are indecomposable: this is shown in the proof of (e) $\Rightarrow$ (f) in Proposition~\ref{prop:USymBand}.
		
		By Lemma~\ref{lem:NoRepetitions} and the arguments above, we can see that $\mathcal{M}_1\cup\mathcal{M}_2\cup\mathcal{M}_3$ form a set of pairwise non-isomorphic indecomposable $A$-modules. It remains to show that this set forms a complete classification of the indecomposable $A$-modules.
		
		So let $M \in \mod*A$ and consider the module $U(M) \in \mod*\wh A$. By the classification of indecomposable modules over gentle algebras, we may uniquely (up to isomorphism) write
		\begin{equation*}
			U(M) \cong \bigoplus_{\wh w \in \str_{\wh A}} M(\wh w)^{a_{\wh w}} \oplus \bigoplus_{\wh w \in \mathcal{B}_{\wh A}}\bigoplus_{p \in \Pi} M(\wh w, \deg(p), \phi_p)^{b_{\wh w, p}},
		\end{equation*}
		where each $a_{\wh w}, b_{\wh w, p} \in \integer_{\geq 0}$, and both the set $\Pi$ and each $\phi_p$ are as in Definition~\ref{def:AsymBandsClass}. Furthermore, every band of $\wh A$ must belong to precisely one of the following three subtypes: bands of even parity that are not $\integer_2$-invariant, bands of even parity that are $\integer_2$-invariant, and bands of odd parity. Let $\mathcal{B}\ps{0}_{\wh A}\subseteq \mathcal{B}_{\wh A}$ be the subcollection of bands of even parity of $\wh A$ that are not $\integer_2$-invariant, let $\mathcal{B}^s_{\wh A}\subseteq \mathcal{B}_{\wh A}$ be the subcollection of bands of even parity of $\wh A$ that are $\integer_2$-invariant, and let $\mathcal{B}\ps{1}_{\wh A}\subseteq \mathcal{B}_{\wh A}$ be the subcollection of bands of odd parity of $\wh A$. We thus have
		\begin{align*}
			\bigoplus_{\wh w \in \mathcal{B}_{\wh A}} \bigoplus_{p \in \Pi} M(\wh w, \deg(p), \phi_p)^{b_{\wh w, p}} \cong \bigoplus_{\wh w \in \mathcal{B}_{\wh A}\ps{0}} \bigoplus_{p \in \Pi} M(\wh w, \deg(p), \phi_p)^{b_{\wh w, p}} 
			&\oplus \bigoplus_{\wh w \in \mathcal{B}_{\wh A}\ps{1}} M(\wh w, m_{\wh w}, \wh\phi_{\wh w}) \\
			&\oplus \bigoplus_{\wh w \in \mathcal{B}_{\wh A}^s} \bigoplus_{p \in \Pi} M(\wh w, \deg(p), \phi_p)^{b_{\wh w, p}},
		\end{align*}
		where each direct summand labelled by $M(\wh w, m_{\wh w}, \wh\phi_{\wh w})$ is a maximal quasiband module and each $\wh w \in \mathcal{B}_{\wh A}^s$ is in standard form.
		
		By Proposition~\ref{prop:UStrings}, we must have
		\begin{equation*}
			\bigoplus_{\wh w \in \str_{\wh A}}M(\wh w)^{a_{\wh w}} \cong \bigoplus_{\wh w \in \str_{\wh A} / \integer_{2}}(M(\wh w) \oplus M(g\wh w))^{a_{\wh w}} \cong \bigoplus_{\wh w \in \str_{\wh A} / \integer_{2}}U(M(\rho(\wh w)))^{a_{\wh w}}.
		\end{equation*}
		Since $\rho$ induces a surjection $\rho\colon \str_{\wh A}\rightarrow \str_A$, we clearly have each $M(\rho(\wh w)) \in \mathcal{M}_1$. Moreover, by Lemma~\ref{lem:UEssInj}, $M(\rho(\wh w))$ is the only $A$-module (up to isomorphism) such that $U(M(\rho(\wh w))) \cong M(\wh w)$.
		
		By Proposition~\ref{prop:UEvenBands}, we must have
		\begin{align*}
			\bigoplus_{\wh w \in \mathcal{B}_{\wh A}\ps{0}} \bigoplus_{p \in \Pi} M(\wh w, \deg(p), \phi_p)^{b_{\wh w, p}} 
			&\cong \bigoplus_{\wh w \in \mathcal{B}_{\wh A}\ps{0} / \integer_2} \bigoplus_{p \in \Pi} (M(\wh w, \deg(p), \phi_p) \oplus M(g\wh w, \deg(p), \lambda_{\wh w}\phi_p))^{b_{\wh w, p}} \\
			&\cong \bigoplus_{\wh w \in \mathcal{B}_{\wh A}\ps{0} / \integer_2} \bigoplus_{p \in \Pi} U(M(\rho(\wh w), \deg(p), \phi_p))^{b_{\wh w, p}},
		\end{align*}
		where each $\lambda_{\wh w} \in K$ is as in Remark~\ref{rem:LamConst}. Since $\rho\colon\bnd_{\wh A} \rightarrow \bnd_A$ is surjective and $\rho(\wh w)$ is an asymmetric band by Lemma~\ref{lem:FoldedSymmetric}, we clearly have each $M(\rho(\wh w), \deg(p), \phi_p) \in \mathcal{M}_2$. Again, Lemma~\ref{lem:UEssInj} additionally implies that each $M(\rho(\wh w), \deg(p), \phi_p)$ is the only $A$-module (up to isomorphism) that maps onto $M(\wh w, \deg(p), \phi_p)$ via $U$.
		
		By Proposition~\ref{prop:UOddBands}, we must have
		\begin{align*}
			\bigoplus_{\wh w \in \mathcal{B}_{\wh A}\ps{1}} M(\wh w, m_{\wh w}, \wh\phi_{\wh w}) 
			&\cong \bigoplus_{\wh w \in \mathcal{B}_{\wh A}\ps{1}} M(\wh w, m_{\wh w}, \phi_{\wh w}^2) \\
			&\cong \bigoplus_{\wh w \in \mathcal{B}_{\wh A}\ps{1}} U(M(\rho(\wh w), m_{\wh w}, \phi_{\wh w}))
		\end{align*}
		where each $\phi_{\wh w} \in \Aut(K^{m_{\wh w}})$. Now each $\phi_{\wh w}$ can be uniquely expressed in its rational canonical form as a direct sum $\phi_{p_1}\oplus\ldots\oplus\phi_{p_r}$ via an automorphism $f \in \Aut(K^{m_{\wh w}})$, where each $p_i \in \Pi$. This induces an isomorphism of (quasi)band modules
		\begin{equation*}
			M(\rho(\wh w), m_{\wh w}, \phi_{\wh w}) \cong M\left(\rho(\wh w), m_{\wh w}, \bigoplus_{i=1}^r\phi_{p_i}\right) \cong \bigoplus_{i=1}^r M(\rho(\wh w), m_{\wh w}, \phi_{p_i}).
		\end{equation*}
		By Lemma~\ref{lem:FoldedSymmetric}, it is then clear that each $M(\rho(\wh w), m_{\wh w}, \phi_{p_i}) \in \mathcal{M}_2$. In addition, Proposition~\ref{prop:UOddBands} implies that any other $A$-module that maps onto $M(\rho(\wh w), m_{\wh w}, \phi_{\wh w})$ via $U$ must be some other quasiband $A$-module $M(\rho(\wh w), m_{\wh w}, \phi'_{\wh w})$. A similar argument shows that any such module must also decompose into $A$-modules that belong to the collection $\mathcal{M}_2$.
		
		Finally, by Proposition~\ref{prop:USymBand} and Remark~\ref{rem:PsiDi}, we must have
		\begin{align*}
			\bigoplus_{\wh w \in \mathcal{B}_{\wh A}^s} \bigoplus_{p \in \Pi} M(\wh w, \deg(p), \phi_p)^{b_{\wh w, p}} 
			&\cong \bigoplus_{\wh w \in \mathcal{B}_{\wh A}^s} \quad \bigoplus_{p \in \Pi_{\wh w} : \Psi_{\rho(\wh w),p}\ps{1} = \emptyset} (M(\wh w, \deg(p), \phi_p) \oplus M(\wh w, \deg(gp), \phi_{gp}))^{b_{\wh w, p}} \\
			&\quad\oplus \bigoplus_{\wh w \in \mathcal{B}_{\wh A}^s} \quad \bigoplus_{p \in \Pi_{\wh w} : \Psi_{\rho(\wh w),p}\ps{2} = \emptyset} M(\wh w, \deg(p), \phi_p) \\
			&\cong \bigoplus_{\wh w \in \mathcal{B}_{\wh A}^s} \bigoplus_{p \in \Pi_{\wh w}} (M(\wh w, 2m_{\psi_p}, H_{\rho(w),m} \psi_p))^{b_{\wh w, p}} \\
			&\cong \bigoplus_{\wh w \in \mathcal{B}_{\wh A}^s} \bigoplus_{p \in \Pi_{\wh w}} (U(M(\rho(\wh w), m_{\psi_p}, \psi_p)))^{b_{\wh w, p}},
		\end{align*}
		where all notation is consistent with Definition~\ref{def:SymBandInds} (and the exposition immediately preceding it). By the surjectivity of $\rho$ and by Lemma~\ref{lem:FoldedSymmetric}, we clearly have each $M(\rho(\wh w), m_{\psi_p}, \psi_p) \in \mathcal{M}_3$. Again, Lemma~\ref{lem:UEssInj} additionally implies that each $M(\rho(\wh w), m_{\psi_p}, \psi_p)$ is the only $A$-module (up to isomorphism) that maps onto $M(\wh w, \deg(p), \phi_p)$ via $U$.
		
		In conclusion, $\mathcal{M}_1 \cup \mathcal{M}_2 \cup \mathcal{M}_3$ is a collection of pairwise non-isomorphic indecomposable $A$-modules. Suppose for a contradiction that there exists an indecomposable $A$-module $M$ that is not isomorphic to a module from this collection. Since $U(M)$ is a direct sum of string and band modules, and since we have just shown that the only $A$-modules that map onto such a direct sum (via $U$) are modules from the collection $\mathcal{M}_1 \cup \mathcal{M}_2 \cup \mathcal{M}_3$, it follows that $M$ must belong to $\mathcal{M}_1 \cup \mathcal{M}_2 \cup \mathcal{M}_3$, which yields the required contradiction. Thus, the collection $\mathcal{M}_1 \cup \mathcal{M}_2 \cup \mathcal{M}_3$ is a complete classification of iso-classes of indecomposable $A$-modules without repetition, as required.
	\end{proof}
	
	As a consequence of the classification, we obtain the following.
	
	\begin{cor} \label{cor:Kxy}
		Let $K$ be a field and let $A = K\langle x,y \rangle / \langle p(x), q(y) \rangle$ be such that
		\begin{align*}
			p(x) &= x^2 - \lambda_1 x - \lambda_2, \\
			q(x) &= y^2 - \lambda'_1 y - \lambda'_2
		\end{align*}
		are both irreducible over $K$. Define an action of $\integer_2$ on $\Pi$ by defining, for each $p \in \Pi$, the polynomial $gp$ as the characteristic polynomial of the matrix $\lambda_2\lambda'_2\phi_p\inv$, where $\phi_p$ is the companion matrix of $p$. Then the isomorphism classes of indecomposable objects in the category $\fin A$ are indexed by the set $\Pi/\integer_2$.
	\end{cor}
	\begin{proof}
		Consider the folded gentle algebra $B=KQ/\langle Z \rangle$ given by the quiver and relations
		\begin{equation*}
			\begin{tikzpicture}
				\draw (-1.7,0) node {$Q\colon$};
				\draw (-0.5,0) node {1};
				\draw (0.5,0) node {2};
				\draw [->](-0.3,0) -- (0.3,0);
				\draw [->](-0.7,0.1) .. controls (-0.9,0.2) and (-1,0.1) .. (-1,0) .. controls (-1,-0.1) and (-0.9,-0.2) .. (-0.7,-0.1);
				\draw [->](0.7,-0.1) .. controls (0.9,-0.2) and (1,-0.1) .. (1,0) .. controls (1,0.1) and (0.9,0.2) .. (0.7,0.1);
				\draw (-1.2,0) node {\footnotesize $\eta_1$};
				\draw (0,0.2) node {\footnotesize $\alpha$};
				\draw (1.3,0) node {\footnotesize $\eta_2$};
			\end{tikzpicture}
		\end{equation*}
		\begin{equation*}
			Z = \{\crs_1^2 + \lambda_1\lambda_2\inv \crs_1 - \lambda_2\inv\stp_1, \crs_2^2 - \lambda'_1 \crs_2 - \lambda'_2 \stp_2 \}.
		\end{equation*}
		Here, we note that $\crs\inv_1$ satisfies $\crs_1^{-2} - \lambda_1 \crs_1\inv - \lambda_2\stp_1$. 
		
		The algebra $B$ has only one band (up to equivalence), which is the symmetric band $w=\crs_2\alpha\inv\crs_1\inv\alpha$. The full subcategory $\mathcal{B} \subset \mod*B$ whose objects are direct sums of band modules is thus equivalent to the category $\fin A$. The unfolding of $B$ is the algebra $\wh B = K\wh Q$, where
		\begin{equation*}
			\begin{tikzpicture}
				\draw (-0.9,0) node {$\widehat Q\colon$};
				\draw (-0.5,0) node {1};
				\draw (0.5,0) node {2};
				\draw [->](-0.3,0.1) -- (0.3,0.1);
				\draw (0,0.3) node {\footnotesize $\widehat\alpha$};
				\draw [->](-0.3,-0.1) -- (0.3,-0.1);
				\draw (0,-0.3) node {\footnotesize $\widehat\alpha'$};
			\end{tikzpicture}
		\end{equation*}
		is the Kronecker quiver. This contains only one band (up to equivalence), which is the $\integer_2$-invariant band of even parity given by $\wh w = (\alpha')\inv\alpha$. Note that $\wh w$ is also in standard form, and $\rho(\wh w) = w$. By Propositions~\ref{prop:UStrings} and \ref{prop:USymBand} and Theorem~\ref{thm:IndClassification}, this implies that the isomorphism classes of the objects in $\mathcal{B}$ are indexed by the set $\Pi / \integer_2$. Hence, the isomorphism classes of objects in $\fin A$ are indexed by the set $\Pi / \integer_2$, as required.
	\end{proof}
	
	\subsection{Auslander-Reiten sequences}
	As well as providing the means of classifying indecomposable objects of $\mod*A$ in terms of indecomposable objects of $\mod*\wh A$, the unfolding functor is capable of producing a classification of the Auslander-Reiten sequences of $\mod*A$. We will begin with the following.
	
	\begin{lem} \label{lem:UIrred}
		Suppose $f \in \Hom_A(M,N)$. If $U(f)$ is irreducible, then $f$ is irreducible.
	\end{lem}
	
	\begin{proof}
		This is a simple consequence of $\mod*A \simeq \im U$ being a subcategory of $\mod*\wh A$. Since $\wh f = U(f)$ is irreducible, $\wh f$ has neither a left inverse nor a right inverse. That is, there exists no map $\wh h \in \Hom_{\wh A}(U(N),U(M))$ such that $\wh h \wh f = \id_{U(M)}$ or $\wh f \wh h = \id_{U(N)}$. Thus, there exists no such map in $\im U$, and hence no such map in $\mod*A$ by the fact that $U$ is exact and faithful. So $f$ has neither a left inverse nor a right inverse.
		
		Suppose $f=f'f''$ for some morphisms $f' \in \Hom_A(L,N)$ and $f'' \in \Hom_A(M,L)$. Then $U(f) = U(f'f'') = U(f')U(f'')$. Since $U(f)$ is irreducible, this implies that either $U(f')$ has a right inverse (say $\wh h' \in \Hom_{\wh A}(U(N),U(L))$ or $U(f'')$ has a left inverse (say $\wh h'' \in \Hom_{\wh A}(U(L),U(M))$). In particular, this implies that $U(L)$ has a direct summand $\wh X$ that is isomorphic to either $U(M)$ or $U(N)$. But $U$ is additive and exact. Thus, this implies that there exists a direct summand $X \subseteq L$ that is isomorphic to either $M$ or $N$. Moreover, since $U$ is exact and faithful, we either have $X = L / \Ker f'$ or $X= \im f''$. Thus, either $f'$ is right invertible or $f''$ is left invertible. Hence, $f$ is irreducible, as required.
	\end{proof}
	
	The above lemma allows us to describe the Auslander-Reiten sequences of $\mod*A$. First we must introduce some further definitions and notation for strings, which we adapt from \cite{ButlerRingel}.
	
	\begin{defn}
		We say a string $w$ \emph{ends in a deep} if there exists no ordinary arrow $\alpha \in \oa$ such that $w \alpha$ is a string. Dually, we say $w$ \emph{starts in a deep} if there exists no ordinary arrow $\alpha \in \oa$ such that $\alpha\inv w$ is a string.
	\end{defn}
	
	\begin{defn}
		Suppose a string $w$ does not end in a deep. Then there exists a string $w_c = w \alpha w'$, where $\alpha \in \oa$ and $w'$ is a (possibly simple) inverse string of maximal length. We say that $w_c$ is obtained from $w$ by \emph{adding a cohook} to the end of $w$. Suppose instead that $w$ does not start in a deep. Then there exists a string $_c w = w' \alpha\inv w$, where $\alpha \in \oa$ and $w'$ is a (possibly simple) direct string of maximal length. We say that $_c w$ is obtained from $w$ by \emph{adding a cohook} to the start of $w$.
	\end{defn}
	
	For any string $w$ that is not simple, axioms \ref{en:FG1} and \ref{en:FG2} ensure that the strings $_c w$ and $w_c$ are uniquely determined. If $w$ is a simple string $\stp_v$ such that $v \in \ov$ is the source of $n \in \{1,2\}$ arrows, then there exist precisely $n$ possible choices for $w_c$. In this case there are also $n$ possible choices for $_c w$, but these strings will be equivalent to the corresponding strings $w_c$. Whenever $n=2$, the choices for $w_c$ (or equivalently $_c w$) can be a potential source of ambiguity when computing rays of irreducible morphisms in the Auslander-Reiten quiver. This ambiguity is practically resolved by fixing one possible choice as $_c w$ and the other possible choice as $w_c$, and then computing the irreducible morphisms around the simple string in the Auslander-Reiten quiver.
	
	\begin{defn}
		Suppose $w \not\sim_2 w'$ for some direct string $w'$. Then $w \sim_2 w_{-h} w_1 w_2$, where $w_1$ is an ordinary formal inverse and $w_2$ is a (possibly simple) direct string. We say the substring $w_{-h}$ is obtained from $w$ by \emph{deleting a hook} from the end of $w$. Suppose instead that $w \not\sim_2 w'$ for some inverse string $w'$. Then $w \sim_2 w_2 w_1 {_{-h}w}$, where $w_1$ is an ordinary arrow and $w_2$ is a (possibly simple) inverse string. We say the substring $_{-h}w$ is obtained from $w$ by \emph{deleting a hook} from the start of $w$.
	\end{defn}
	
	\begin{thm} \label{thm:ARSeq}
		Let $A$ be a folded gentle algebra.
		\begin{enumerate}[label=(\alph*)]
			\item Let $w \in \str_A$.
			\begin{enumerate}[label = (\roman*)]
				\item If $w$ starts and ends in a deep, then the Auslander-Reiten sequence ending in $M(w)$ is
				\begin{equation*}
					0 \rightarrow M(_{-h}w_{-h}) \rightarrow M(_{-h}w) \oplus M(w_{-h}) \rightarrow M(w) \rightarrow 0.
				\end{equation*}
				\item If $w$ starts, but does not end, in a deep, then the Auslander-Reiten sequence ending in $M(w)$ is
				\begin{equation*}
					0 \rightarrow M(_{-h}w_{c}) \rightarrow M(_{-h}w) \oplus M(w_{c}) \rightarrow M(w) \rightarrow 0.
				\end{equation*}
				\item If $w$ ends, but does not start, in a deep, then the Auslander-Reiten sequence ending in $M(w)$ is
				\begin{equation*}
					0 \rightarrow M(_{c}w_{-h}) \rightarrow M(_{c}w) \oplus M(w_{-h}) \rightarrow M(w) \rightarrow 0.
				\end{equation*}
				\item If $w$ neither starts nor ends in a deep, then the Auslander-Reiten sequence ending in $M(w)$ is
				\begin{equation*}
					0 \rightarrow M(_{c}w_{c}) \rightarrow M(_{c}w) \oplus M(w_{c}) \rightarrow M(w) \rightarrow 0.
				\end{equation*}
			\end{enumerate}
			\item Let $w \in \bnd_A$ be asymmetric and let $q$ be a monic and irreducible polynomial. For each $n \geq 1$, denote by $\phi_{q^n}$ the companion matrix of $q^n$. Then for each $n$, there exists an Auslander-Reiten sequence
			\begin{equation*}
				0 \rightarrow M(w, m_n, \phi_{q^n}) \rightarrow M(w, m_{n-1}, \phi_{q^{n-1}}) \oplus M(w, m_{n+1}, \phi_{q^{n+1}}) \rightarrow M(w, m_{n}, \phi_{q^n}) \rightarrow 0,
			\end{equation*}
			where each $m_i = i\deg(q)$ and $M(w, m_{0}, \phi_{q^0})$ is interpreted as the trivial module.
			\item Let $w \in \bnd_A$ be symmetric and let $q$ be a monic and irreducible polynomial. Then for each $M(w,m_{\psi_{q^n}}, \psi_{q^n}) \in \mathcal{M}_3$ from Definition~\ref{def:SymBandInds}, there exists an Auslander-Reiten sequence
			\begin{equation*}
				0 \rightarrow M(w,m_{\psi_{q^n}}, \psi_{q^n}) \rightarrow M(w,m_{\psi_{q^{n-1}}}, \psi_{q^{n-1}}) \oplus M(w,m_{\psi_{q^{n+1}}}, \psi_{q^{n+1}}) \rightarrow M(w,m_{\psi_{q^n}}, \psi_{q^n}) \rightarrow 0,
			\end{equation*}
			where $M(w,m_{\psi_{q^0}}, \psi_{q^0})$ is interpreted as the trivial module.
		\end{enumerate}
	\end{thm}
	
	\begin{proof}
		(a) It is not difficult to verify that the above exact sequences exist via the obvious inclusions/surjections of standard bases. Thus, it remains to show that these are indeed Auslander-Reiten sequences. By Proposition~\ref{prop:UStrings}, $U(M(w)) \cong M(\theta(w)) \oplus M(g\theta(w))$. Now note that it follows from the (un)folding procedure that $w$ starts (resp. ends) in a deep if and only if $\theta(w)$ starts (resp. ends) in a deep if and only if $g\theta(w)$ starts (resp. ends) in a deep. Now since $\wh A$ is a gentle algebra, it is also a string algebra. By the results of \cite[p. 172, Proposition]{ButlerRingel} and \cite[Sec. 5.4]{GeissClanMaps}, this implies that the Auslander-Reiten sequence ending in $\theta(w)$ is given by adding cohooks/deleting hooks in precisely the same way as given above. This is true for $g\theta(w)$ also. In particular, the image under $U$ of the appropriate choice of exact sequence above is a direct sum of Auslander-Reiten sequences (one ending in $\theta(w)$, the other ending in $g \theta(w)$). Consequently, the inclusion/surjection morphisms of the exact sequence in the image of $U$ are both irreducible. By Lemma~\ref{lem:UIrred}, this implies that both the inclusion/surjection morphisms of the choice of exact sequence above are irreducible. Since the starting and ending terms are string modules, they are indecomposable. Thus, the sequence is an Auslander-Reiten sequence, as required.
		
		(b) For any quasiband module $\bigoplus_{i=1}^r M(\wh w, m_{n_i}, \phi_{q_i^{n_i}}) \in \mod*\wh A$, it is known (see for example \cite[p. 165]{ButlerRingel} and \cite[Sec. 5.4]{GeissClanMaps}) that, for each $i$, we have Auslander-Reiten sequences
		\begin{equation*} \tag{$\ast$} \label{eq:whABandAR}
			0 \rightarrow M(\wh w, m_{n_i}, \phi_{q_i^{n_i}})
			\rightarrow M(\wh w, m_{n_i-1}, \phi_{q_i^{n_i-1}}) \oplus M(\wh w, m_{n_i+1}, \phi_{q_i^{n_i+1}})
			\rightarrow M(\wh w, m_{n_i}, \phi_{q_i^{n_i}})
			\rightarrow 0
		\end{equation*}
		in $\mod*\wh A$. Now for any $M(w, m_n, \phi_{q^n}) \in \mathcal{M}_2$, we clearly have an exact sequence
		\begin{equation*} \tag{$\ast\ast$} \label{eq:ABandAR}
			0 \rightarrow M(w, m_n, \phi_{q^n}) \rightarrow M(w, m_{n-1}, \phi_{q^{n-1}}) \oplus M(w, m_{n+1}, \phi_{q^{n+1}}) \rightarrow M(w, m_{n}, \phi_{q^n}) \rightarrow 0
		\end{equation*}
		(to see this, one can consider the $K$-linear, exact, faithful functor $F_w$ from the category $\fin K[x,x\inv]$ used in the proof of Theorem~\ref{thm:IndClassification}). Moreover, Propositions~\ref{prop:UEvenBands} and \ref{prop:UOddBands} imply that we have
		\begin{equation*}
			U(M(w, m_n, \phi_{q^n})) \cong \bigoplus_{i=1}^r M(\theta(w), m_{n_i}, \phi_{q_i^{n_i}})
		\end{equation*}	
		for some $r$ and some monic and irreducible polynomials $q_i$ if $U(M(w, m_n, \phi_{q^n}))$ is indecomposable, and
		\begin{equation*}
			U(M(w, m_n, \phi_{q^n})) \cong \bigoplus_{i=1}^r (M(\theta(w), m_{n_i}, \phi_{q_i^{n_i}}) \oplus M(g\theta(w), m_{n_i}, \phi_{p_i^{n_i}}))
		\end{equation*}	
		for some $r$ and some monic and irreducible polynomials $p_i$ and $q_i$ if $U(M(w, m_n, \phi_{q^n}))$ is not indecomposable. Consequently, the sequence (\ref{eq:ABandAR}) maps to the direct sum of the Auslander-Reiten sequences (\ref{eq:whABandAR}) under $U$. Lemma~\ref{lem:UIrred} then implies that (\ref{eq:ABandAR}) is an Auslander-Reiten sequence, as required.
		
		(c) The proof is similar to (b).
	\end{proof}
	
	\begin{rem} \label{rem:SymMiddleTerm}
		Note that if the string $w$ is symmetric, then it starts in a deep if and only if it ends in a deep. Furthermore, it is easy to see that, if they exist, $_{-h}w_{-h}$ and $_c w_c$ are also symmetric. In addition, $_{-h}w \approx w_{-h}$ and $_{c}w \approx w_{c}$. Thus, the Auslander-Reiten sequence starting/ending in symmetric string can always be viewed as a \emph{non-trivially valued Auslander-Reiten sequence}
		\begin{equation*}
			\xymatrix@1{
				0 \ar[r] &
				\tau M(w) \ar[r]^-{(1,2)} &
				E \ar[r]^-{(2,1)} &
				M(w) \ar[r] &
				0,
			}
		\end{equation*}
		where $E$ is indecomposable. See \cite[Sec. VII.1]{AuslanderReitenSmalo} for an account on valued Auslander-Reiten sequences/quivers.
	\end{rem}
	
	The Auslander-Reiten quiver of a folded gentle algebra can often be obtained by folding the Auslander-Reiten quiver of its corresponding unfolded gentle algebra (with respect to the folding action), as the next example shows.
	
	\begin{exam}
		Consider the folded gentle algebra $A=K Q / Z$ and corresponding unfolded gentle algebra $\wh A = K \wh Q / \wh Z$ given by the following quivers
		\begin{equation*}
			\begin{tikzpicture}
				\draw (-4.5,0) node {$Q\colon$};
				\draw (-4,0) node {$1$};
				\draw (-2.6,0) node {$2$};
				\draw (-1.2,0) node {$3$};
				\draw [->](-3.8,0) -- (-2.8,0);
				\draw [->](-2.4,0) -- (-1.4,0);
				\draw [->](-2.7,0.2) .. controls (-2.8,0.4) and (-2.7,0.5) .. (-2.6,0.5) .. controls (-2.5,0.5) and (-2.4,0.4) .. (-2.5,0.2);
				\draw (-3.3,0.2) node {\footnotesize$\alpha$};
				\draw (-2.6,0.7) node {\footnotesize$\eta_2$};
				\draw (-1.9,0.2) node {\footnotesize$\beta$};
				\draw [dashed](-3.0875,-0.1113) arc (-167.1394:-12.8571:0.5);
				
				\draw (0.2,0) node {$\widehat{Q}\colon$};
				\draw (0.7,0.7) node {$\widehat 1$};
				\draw (0.7,-0.7) node {$\widehat 1'$};
				\draw (1.9,0) node {$\widehat 2$};
				\draw (3.1,0.7) node {$\widehat 3$};
				\draw (3.1,-0.7) node {$\widehat 3'$};
				\draw [->](0.9,0.5) -- (1.7,0.1);
				\draw [->](0.9,-0.5) -- (1.7,-0.1);
				\draw [->](2.1,0.1) -- (2.9,0.5);
				\draw [->](2.1,-0.1) -- (2.9,-0.5);
				\draw (1.3,0.6) node {\footnotesize$\widehat\alpha$};
				\draw (1.3,-0.6) node {\footnotesize$\widehat\alpha'$};
				\draw (2.5,0.6) node {\footnotesize$\widehat\beta$};
				\draw (2.5,-0.6) node {\footnotesize$\widehat\beta'$};
				\draw [dashed](1.5091,0.3117) arc (141.4315:38.5714:0.5);
				\draw [dashed](1.5091,-0.3117) arc (-141.4315:-38.5714:0.5);
			\end{tikzpicture}
		\end{equation*}
		where dashed lines indicate ordinary relations. The Auslander-Reiten quiver of $\mod*A$ is as follows.
		\begin{equation*}
			\def\ashort{2ex}
			\begin{tikzpicture}
				\coordinate (hdist) at (2.5,0);
				\coordinate (vdist) at (0,2.5);
				\coordinate (P1) at ($(vdist)$);
				\coordinate (P3) at ($0.5*(vdist)-1.5*(hdist)$);
				\coordinate (P2) at ($-1*(hdist)$);
				\foreach \i in {0,1,2,3} {
					\coordinate (TrD\i_P3) at ($(P3) + \i*(hdist)$);
				}
				\foreach \i in {0,1,2} {
					\coordinate (TrD\i_P2) at ($(P2) + \i*(hdist)$);
				}
				
				\draw (P1) node {\footnotesize$M(\alpha\eta_2\beta)$};
				
				\draw (P3) node {\footnotesize$M(\varepsilon_3)$};
				\draw (TrD1_P3) node {\footnotesize$M(\eta_2\beta)$};
				\draw (TrD2_P3) node {\footnotesize$M(\alpha\eta_2)$};
				\draw (TrD3_P3) node {\footnotesize$M(\varepsilon_{1})$};
				
				\draw (P2) node {\footnotesize$M(\beta^{-1}\eta_2\beta)$};
				\draw (TrD1_P2) node {\footnotesize$M(\eta_2)$};
				\draw (TrD2_P2) node {\footnotesize$M(\alpha\eta_2\alpha^{-1})$};
				
				\draw[->, shorten <= \ashort, shorten >= \ashort] (TrD1_P3) -- (P1);
				\draw[->, shorten <= \ashort, shorten >= \ashort] (P1) -- (TrD2_P3);
				\foreach \i in {0,1,2} {
					\draw[->, shorten <= \ashort, shorten >= \ashort] (TrD\i_P3) -- (TrD\i_P2);
					\draw ($0.5*(TrD\i_P3)+0.5*(TrD\i_P2)+(0.1,0.1)$) node[rotate=-45] {\tiny$(2,1)$};
				}
				\foreach \i [evaluate=\i as \j using int(\i-1)] in {1,2,3} {
					\draw[->, shorten <= \ashort, shorten >= \ashort] (TrD\j_P2) -- (TrD\i_P3);
					\draw ($0.5*(TrD\j_P2)+0.5*(TrD\i_P3)+(-0.15,0.05)$) node[rotate=45] {\tiny$(1,2)$};
				}
			\end{tikzpicture}
		\end{equation*}
		On the other hand, the Auslander-Reiten quiver of $\mod*\wh A$ is as follows.
		\begin{equation*}
			\def\ashort{2ex}
			\begin{tikzpicture}
				\coordinate (hdist) at (2,0);
				\coordinate (vdist) at (0,2);
				\coordinate (P1) at ($(vdist)$);
				\coordinate (P3) at ($0.5*(vdist)-1.5*(hdist)$);
				\coordinate (P2) at ($-1*(hdist)$);
				\coordinate (P3p) at ($-0.5*(vdist)-1.5*(hdist)$);
				\coordinate (P1p) at ($-1*(vdist)$);
				\foreach \i in {0,1,2,3} {
					\coordinate (TrD\i_P3) at ($(P3) + \i*(hdist)$);
					\coordinate (TrD\i_P3p) at ($(P3p) + \i*(hdist)$);
				}
				\foreach \i in {0,1,2} {
					\coordinate (TrD\i_P2) at ($(P2) + \i*(hdist)$);
				}
				
				\draw[red, dashed] ($(P2)-(hdist)$) -- ($(TrD2_P2)+(hdist)$);
				\draw[white,fill=white] ($(P2)-(0.75,0.5)$) rectangle ($(P2)+(0.75,0.5)$);
				\draw[white,fill=white] ($(TrD1_P2)-(0.5,0.5)$) rectangle ($(TrD1_P2)+(0.5,0.5)$);
				\draw[white,fill=white] ($(TrD2_P2)-(0.85,0.5)$) rectangle ($(TrD2_P2)+(0.85,0.5)$);
				
				\draw (P1) node {\footnotesize$M(\widehat\alpha\widehat\beta')$};
				
				\draw (P3) node {\footnotesize$M(\varepsilon_{\widehat 3})$};
				\draw (TrD1_P3) node {\footnotesize$M(\widehat\beta')$};
				\draw (TrD2_P3) node {\footnotesize$M(\widehat\alpha)$};
				\draw (TrD3_P3) node {\footnotesize$M(\varepsilon_{\widehat 1'})$};
				
				\draw (P2) node {\footnotesize$M(\widehat\beta^{-1}\widehat\beta')$};
				\draw (TrD1_P2) node {\footnotesize$M(\varepsilon_{\widehat 2})$};
				\draw (TrD2_P2) node {\footnotesize$M(\widehat\alpha(\widehat\alpha')^{-1})$};
				
				\draw (P3p) node {\footnotesize$M(\varepsilon_{\widehat 3'})$};
				\draw (TrD1_P3p) node {\footnotesize$M(\widehat\beta^{-1})$};
				\draw (TrD2_P3p) node {\footnotesize$M((\widehat\alpha')^{-1})$};
				\draw (TrD3_P3p) node {\footnotesize$M(\varepsilon_{\widehat1})$};
				
				\draw (P1p) node {\footnotesize$M(\widehat\beta^{-1}(\widehat\alpha')^{-1})$};
				
				\draw[->, shorten <= \ashort, shorten >= \ashort] (TrD1_P3) -- (P1);
				\draw[->, shorten <= \ashort, shorten >= \ashort] (P1) -- (TrD2_P3);
				\foreach \i in {0,1,2} {
					\draw[->, shorten <= \ashort, shorten >= \ashort] (TrD\i_P3) -- (TrD\i_P2);
					\draw[->, shorten <= \ashort, shorten >= \ashort] (TrD\i_P3p) -- (TrD\i_P2);
				}
				\foreach \i [evaluate=\i as \j using int(\i-1)] in {1,2,3} {
					\draw[->, shorten <= \ashort, shorten >= \ashort] (TrD\j_P2) -- (TrD\i_P3);
					\draw[->, shorten <= \ashort, shorten >= \ashort] (TrD\j_P2) -- (TrD\i_P3p);
				}
				\draw[->, shorten <= \ashort, shorten >= \ashort] (TrD1_P3p) -- (P1p);
				\draw[->, shorten <= \ashort, shorten >= \ashort] (P1p) -- (TrD2_P3p);
			\end{tikzpicture}
		\end{equation*}
		Here, the dashed red line indicates the $\integer_2$-symmetry of the Auslander-Reiten quiver. In particular, the folding action maps vertices and arrows in the Auslander-Reiten quiver to their reflected counterparts. Moreover, in this example, the Auslander-Reiten quiver of $\mod*A$ is given by taking the quotient of the Auslander-Reiten quiver of $\mod*\wh A$ with respect to the folding action. This $\integer_2$-symmetry is explored in more detail in the following Corollary.
	\end{exam}
	
	\begin{cor} \label{cor:ARZ2}
		Let $A$ be a folded gentle algebra and $\wh A$ be the unfolded gentle algebra of $A$ obtained via the unfolding procedure. Let $\Gamma_A$ be the valued Auslander-Reiten quiver of $A$, let $\Gamma'_A$ be the underlying unvalued translation quiver of $\Gamma_A$, and let $\Gamma_{\wh A}$ be the Auslander-Reiten quiver of $\wh A$. 
		\begin{enumerate}[label = (\alph*)]
			\item There exists a group action of $\integer_2$ on $\Gamma_{\wh A}$ such that $\Gamma'_{A} \cong \Gamma_{\wh A} / \integer_2$.
			\item Let $\wh M$ be a vertex of $\Gamma_{\wh A}$ that is fixed under the action of $\integer_2$ and which corresponds to a string module. Then:
			\begin{enumerate}[label=(\roman*)]
				\item $[\wh M] \in \Gamma_A$ corresponds to the string module of a symmetric string;
				\item the arrow of source $[\wh M] \in \Gamma_A$ has a valuation of $(1,2)$; and
				\item the arrow of target $[\wh M] \in \Gamma_A$ has a valuation of $(2,1)$;
			\end{enumerate}
		\end{enumerate}
	\end{cor}
	\begin{proof}
		The action of $\integer_2$ on $\Gamma_{\wh A}$ used in this corollary is instead related to a twist of the folding action. For the purposes of clarity, we will thus label the non-identity element of $\integer_2$ as $g'$ to distinguish it from the folding action on $\wh Q$, $\wh A$, and other related sets. Suppose an indecomposable object $M \in \mod*A$ corresponds to a string module, an asymmetric band module of even parity, or a symmetric band module. If $U(M) \cong \wh M \oplus \wh M'$, then we define $g'\wh M \cong \wh M'$. On the other hand, if $U(M) \cong \wh M$, then we define $g' \wh M \cong \wh M$. The action of $g'$ on bands modules $\wh M$ of odd parity is simply given by $g' \wh M \cong \wh M$. A consequence of this construction and Propositions~\ref{prop:UStrings}, \ref{prop:UEvenBands}, \ref{prop:UOddBands} and \ref{prop:USymBand} is that the action of $\integer_2$ is closed on the following subclasses of vertices of $\Gamma_{\wh A}$:
		\begin{enumerate}[label = (\arabic*)]
			\item those corresponding to strings;
			\item those corresponding to bands that are not $\integer_2$-variant;
			\item those corresponding to bands that are of odd parity (and hence $\integer_2$-variant); and
			\item those corresponding to bands that are of even parity and $\integer_2$-variant.
		\end{enumerate}
		Moreover, this list of subclasses is mutually exclusive and collectively exhaustive.
		
		It remains to define the group action on irreducible morphisms in $\Gamma_{\wh A}$. We will begin by considering the irreducible morphisms between band modules. Since every band module resides in a homogeneous tube of the Auslander-Reiten quiver, there is an obvious action of $\integer_2$ on the irreducible morphisms between band modules $f\colon \wh M_1 \rightarrow \wh M_2$ given by $g'f\colon g'\wh M_1 \rightarrow g'\wh M_2$. A consequence of this is that $\integer_2$ acts on the collection of homogeneous tubes of $\Gamma_{\wh A}$ (containing band modules).
		
		From this, we will show that the quotient by $\integer_2$ on the collections of homogeneous tubes in $\Gamma_{\wh A}$ containing each class of band modules bijectively corresponds to the matching subclass of homogeneous tubes in $\Gamma_A$. Since $\bnd_A \cong \bnd_{\wh A} / \integer_2$ by Lemma~\ref{lem:FoldedBands}, it follows that the quotient by $\integer_2$ of the collection of homogeneous tubes in $\Gamma_{\wh A}$ associated to the vertices in (2) bijectively corresponds to the collection of homogeneous tubes in $\Gamma_A$ corresponding to asymmetric bands of even parity (c.f. Lemmata~\ref{lem:OddParity} and \ref{lem:FoldedSymmetric}(b)). Since the collection of tubes in $\Gamma_{\wh A}$ associated to subclass (3) and the collection of tubes in $\Gamma_A$ associated to bands of odd parity are both indexed by the irreducible polynomials of the set $\Pi$, and since every such vertex associated to (3) is fixed under $g'$, there is an obvious bijection between the collections in this case too. Finally, we note by Lemma~\ref{lem:SymInj} that the quotient by $\integer_2$ of the collection of tubes in $\Gamma_{\wh A}$ associated to subclass (4) bijectively corresponds to the class of homogeneous tubes of $\Gamma_A$ corresponding to symmetric band modules.
		
		Next, consider the irreducible morphisms between string modules $M(\wh w_1)$ and $M(\wh w_2)$. To this aim, consider the various cases outlined in Theorem~\ref{thm:ARSeq}(a) along with the folding action on strings: If $\wh w_2$ ends in a deep if and only if $g\wh w_2$ ends in a deep, and $\wh w_2$ starts in a deep if and only if $g\wh w_2$ starts in a deep. Consequently, for any irreducible morphism $f\colon M(\wh w_1) \rightarrow M(\wh w_2)$, there exists a corresponding irreducible morphism $g' f \colon  g'M(\wh w_1) \rightarrow g'M(\wh w_2)$, since $g'M(\wh w) \cong M(g\wh w)$ by Proposition~\ref{prop:UStrings}. Since $\rho(\wh w) \approx \rho(g\wh w)$ and hence $\str_A \cong \str_{\wh A} / \integer_2$ by Lemma~\ref{lem:FoldedStrings}, it follows that the subquiver of $\Gamma'_A$ that contains string modules is isomorphic to the quotient by $\integer_2$ of the subquiver of $\Gamma_{\wh A}$ that contains string modules.
		
		Consequently, we have shown that $\Gamma'_A \cong \Gamma_{\wh A} / \integer_2$, which completes the proof for (a). To show (b), we simply make use of existing results in this paper. Statement (b)(i) is a consequence of Lemma~\ref{lem:FoldedSymmetric}(a). Statement (b)(ii) and (iii) result from Theorem~\ref{thm:ARSeq} and Remark~\ref{rem:SymMiddleTerm}. This completes the proof.
	\end{proof}
	
	\section{Closure Under Derived Equivalence} \label{sec:DerEquiv}
	The aim of this section is to show that folded gentle algebras are closed under derived equivalence. That is, we want to show that if $A$ is a finite-dimensional folded gentle algebra and $A'$ is a finite-dimensional algebra that is derived equivalent to $A$, then $A'$ is a folded gentle algebra. The proof of this result is essentially an application of folding on the same result given for gentle algebras in \cite{SchroerZim}. We adopt the same notation throughout this section as used in the previous section, along with the following additional notation. In any given triangulated category, we will denote the $n$th shift of an object $X$ by $X[n]$. Throughout, let $D = \Hom_K({-},K)$ be the standard $K$-duality. In this section, we freely use the fact that $DA$ has a natural $A$-$A$-bimodule structure. Specifically, for any $f \in DA$ and $a,b \in A$, we have $(af)(b) = f(ba)$ and $(fa)(b) = f(ab)$.
	
	\subsection{(Un)folding repetitive algebras}
	Given a finite-dimensional algebra $A$, one can construct its \emph{repetitive algebra}, which has important applications in the theory of derived equivalences. The following definition is due to \cite{HW}, though we mainly follow the exposition of \cite{Schroer}.
	
	\begin{defn}
		Let $A$ be a finite-dimensional $K$-algebra and let $D = \Hom_K({-},K)$ be the standard $K$-duality. Consider the $K$-vector space
		\begin{equation*}
			A_R=\left\{(a\ps{i},f\ps{i})_i \in \bigoplus_{i \in \integer} A \oplus \bigoplus_{i \in \integer} DA : \text{all but finitely many } a\ps{i} \text{ and } f\ps{i} \text{ are zero} \right\}.
		\end{equation*}
		The \emph{repetitive algebra} of $A$ is the vector space $A_R$ endowed with the ring multiplicative operation
		\begin{equation*}
			(a\ps{i},f\ps{i})_i (b\ps{i},h\ps{i})_i = (a\ps{i}b\ps{i},a\ps{i}h\ps{i}+f\ps{i}b\ps{i-1})_i,
		\end{equation*}
		where $a\ps{i}b\ps{i} = (ab)\ps{i}$, $a\ps{i}h\ps{i} = (ah)\ps{i}$ and $f\ps{i}b\ps{i-1} = (fb)\ps{i}$.
	\end{defn}
	
	Now suppose that $(K,Q,Z)$ is either a folded gentle triple or a gentle triple/pair and that $A$ is its associated algebra. Let $\mathcal{P}_{Q,Z}$ be the subset of $\mathcal{P}_Q$ (from Section\ref{sec:Preliminaries}) given by removing all paths in $\mathcal{P}_Q$ that contain (as a subpath) either an ordinary relation or the square of a crease loop. Then $A_R$ has a basis given by
	\begin{equation*}
		\{p\ps{i} : p \in \mathcal{P}_{Q,Z}, i \in \integer\} \cup \{f_p\ps{i} : \mathcal{P}_{Q,Z}, i \in \integer \},
	\end{equation*}
	where $f_p\ps{i}(q) = 1$ if $p=q \in \mathcal{P}_{Q,Z}$ and $f_p\ps{i}(q)=0$ otherwise. It then follows from the multiplicative structure of $A_R$ that
	\begin{align*}
		p\ps{i}q\ps{j} &= 
		\begin{cases}
			(pq)\ps{i}	& \text{if } i=j \text{ and } pq \in \mathcal{P}_{Q,Z}, \\
			0			& \text{otherwise,}
		\end{cases} \\
		p\ps{i}f_q\ps{j} &= 
		\begin{cases}
			f_r\ps{i}	& \text{if } i=j \text{ and there exists } r \in \mathcal{P}_{Q,Z} \text{ such that } q = rp, \\
			0			& \text{otherwise,}
		\end{cases} \\
		f_q\ps{j}p\ps{i} &= 
		\begin{cases}
			f_r\ps{i}	& \text{if } i=j-1 \text{ and there exists } r \in \mathcal{P}_{Q,Z} \text{ such that } q = pr, \\
			0			& \text{otherwise,}
		\end{cases} \\
		f_p\ps{i}f_q\ps{j} &= 0
	\end{align*}
	for all $i, j \in \integer$ and $p, q \in \mathcal{P}_{Q,Z}$. For any crease $\crs_v$ satisfying $\crs_v^2 - \lambda_1\crs_v - \lambda_2\stp_v$, we note in particular that this implies that
	\begin{align*}
		\crs_v\ps{i} f_{\crs_v}\ps{i} &= f_{\crs_v}\ps{i} \crs_v\ps{i-1} = f_{\stp_v}\ps{i} \\
		\crs_v\ps{i} f_{\stp_v}\ps{i} &= f_{\stp_v}\ps{i} \crs_v\ps{i-1} = f_{\crs_v\inv}\ps{i} = \frac{1}{\lambda_2}f_{\crs_v}\ps{i} + \frac{\lambda_1}{\lambda_2}f_{\stp_v}\ps{i}.
	\end{align*}
	
	\subsubsection{The quiver and relations of a repetitive algebra}
	Suppose that $A = KQ / \langle Z \rangle$, where $(K,Q,Z)$ satisfies \ref{en:FG1}--\ref{en:FG5} (so it is either a gentle algebra or folded gentle algebra, depending on whether there are any crease loops or not). From $(K,Q,Z)$, we obtain a new triple $(K, Q_R, Z_R)$ in the following way.
	
	Firstly, $Q_R = (Q_{R,0}, Q_{R,1})$ with
	\begin{align*}
		Q_{R,0} &= \bigsqcup_{i \in \integer}  Q_0, \\
		Q_{R,1} &= \bigsqcup_{i \in \integer} ( Q_1 \cup  X),
	\end{align*}
	where $X$ is a set of additional arrows called \emph{connecting arrows}. In particular, for each vertex $ v \in  Q_0$ and arrow $ \alpha \in  Q_0$, we denote the $i$-th copies in $ Q_{R}$ by $ v\ps{i}$ and $ \alpha\ps{i}$ respectively. The elements of the set $X$ correspond to the maximal paths of $\mathcal{P}_{Q,Z}$ (that is, the paths of $\mathcal{P}_{Q,Z}$ that have no non-trivial left or right concatenation contained within $\mathcal{P}_{Q,Z}$). Specifically, for each maximal path $ p \in  \mathcal{P}_{Q,Z}$, there exists an arrow $ \chi_p\ps{i}\colon t( p \ps{i}) \rightarrow s( p \ps{i-1})$ in the $i$-th copy of $ X$.
	
	For any maximal path $ p =  p_1  p_2 \in \mathcal{P}_{Q,Z}$, we call each path $ p_2\ps{i} \chi\ps{i}_{ p}  p_1\ps{i-1}$ a \emph{full path}. The relations of $ Z_R$ are then determined as follows. 
	\begin{enumerate}[label=(RZ.\roman*)]
		\item For each ordinary relation $\alpha\beta \in  Z$, we have corresponding ordinary relations $\alpha\ps{i}\beta\ps{i} \in  Z_R$.
		\item  For each crease relation $\crs_v^2-\lambda_1\crs_v-\lambda_2\stp_v \in Z$, we have corresponding crease relations $(\crs_v\ps{i})^2-\lambda_1\crs_v\ps{i}-\lambda_2\stp_v\ps{i} \in Z_R$.
		\item For any maximal paths $ p =  p_1  p_2  p_3$ and $ q =  q_1  q_2  q_3$ in $\mathcal{P}_{Q,Z}$ with $p_2 = q_2$, we have \emph{commutativity relations} $p_3\ps{i} \chi_p\ps{i}p_1\ps{i-1} - q_3\ps{i}\chi_q\ps{i}q_1\ps{i-1} \in Z_R$.
		\item For any other path $p \in \mathcal{P}_{Q_R}$ that contains a connecting arrow, we say $p \in  Z_R$ if, up to the aforementioned commutativity relations and the crease relations, it is not equivalent to a linear combination of subpaths of full paths. 
	\end{enumerate}
	
	If $A$ is a gentle algebra, then we know from \cite[Sec. 3, Theorem]{Schroer} that $A_R \cong KQ_R / \langle Z_R \rangle$. As it happens, we have the same result for folded gentle algebras.
	
	\begin{lem} \label{lem:RepAlgQuiver}
		Let $A = KQ / \langle Z\rangle$ be a folded gentle algebra. Then $A_R \cong KQ_R / \langle Z_R \rangle$
	\end{lem}
	\begin{proof}
		The proof is identical to \cite[Sec. 3, Theorem]{Schroer}, and we use exactly the same isomorphism. Namely, define an map $\sigma\colon KQ_R / \langle Z_R \rangle  \rightarrow  A_R$ as follows. For each path $p\ps{i}$ in the $i$-th copy of the quiver $Q$ in $Q_R$, define $\sigma(p\ps{i})$ to be the corresponding path in the $i$-th copy of $A$ in $A_R$. For each maximal path $p \in \mathcal{P}_{Q,Z}$, define $\sigma(\chi_p\ps{i}) = f_p\ps{i}$ in the $i$-th copy of $DA$ in $A_R$. Thus, for any path $q \in KQ_R / \langle Z_R \rangle$ that contains a connecting arrow $\chi_p\ps{i}$, we know from the relations in $Z_R$ that $q = p_3\ps{i} \chi_p\ps{i} p_1\ps{i-1}$, where $p = p_1 p_2 p_3$. We consequently define $\sigma(q) = p_3\ps{i} f_{p}\ps{i} p_1\ps{i-1} = f_{p_2}\ps{i}$. It is easy to verify that this yields a bijective isomorphism of algebras, and so we are done.
	\end{proof}
	
	\subsubsection{The (un)folding procedure and morphisms}
	Now suppose that $\wh A = K\wh Q / \langle \wh Z \rangle$ is an unfolded gentle algebra corresponding to a folded gentle algebra $A = KQ / \langle Z \rangle$ (via the unfolding procedure). One can naturally extend the folding action of $g \in \integer_2$ on $\wh A$ to the repetitive algebra $\wh A_{R} \cong K\wh Q_R / \langle \wh Z_R \rangle$. Namely, we can define $g p\ps{i} = (gp)\ps{i}$ and $gf_{p}\ps{i} = f_{gp}\ps{i}$. Note that under the folding action, the only basis elements of $\wh A_{R}$ that are fixed under $g$ are the elements $\stp_{\wh v}\ps{i}$ and $f_{\stp_{\wh v}}\ps{i}$ such that $\pi(\wh v) \in \cv$. One can thus obtain the algebra $A_R = KQ_R / \langle Z_R \rangle$ using the following folding procedure for repetitive algebras of unfolded gentle algebras:
	\begin{enumerate}[label=(RF.\roman*)]
		\item The vertices of $Q_R$ correspond to the orbits of vertices of $\wh Q_R$ under $\integer_2$. In particular, we have $Q_{R,0} = \{ [\wh v\ps{i}] : \wh v\ps{i} \in \wh Q_{R,0}\}$.
		\item For each $[\wh v\ps{i}] \in Q_{R,0}$ such that $\wh v\ps{i} = g \wh v\ps{i}$, we have a crease loop $\crs_{\wh v}\ps{i} \in Q_{R,1}$.
		\item All other arrows of $Q_R$, which we again call \emph{ordinary arrows}, correspond to the orbits of arrows of $\wh Q_R$ under $\integer_2$. In particular, we have 
		\begin{equation*}
			Q_{R,1} = \{\crs_{[\wh v]}\ps{i} : g\wh v\ps{i} = \wh v\ps{i}\} \cup \{ [\wh \alpha\ps{i}] : \wh \alpha\ps{i} \in \wh Q_{R,1}\}.
		\end{equation*}
		\item For each crease $\crs_{[\wh v]}\ps{i} \in Q_{R,1}$, we have a \emph{crease relation} in $Z_R$ that matches the crease relation for $\crs_{[\wh v]} \in \crel$. In the alternative circumstance where we do not have a specific target folded gentle algebra in mind, we may choose these arbitrarily, provided that these still satisfy \ref{en:FG5} and $\crs_{[\wh v]}\ps{i}$ and $\crs_{[\wh v]}\ps{j}$ satisfy the same quadratic relations for each $i$ and $j$.
		\item All other relations of $Z_R$, which we will again call \emph{ordinary relations}, are given by the orbits of relations in $\wh Z_R$ under the action of $\integer_2$.
	\end{enumerate}
	
	Conversely, since $\wh A$ is obtained from $A$ via the unfolding procedure, it is not difficult to verify that $\wh A_R$ and $A_R$ are also related to each other via the following unfolding procedure:
	\begin{enumerate}[label=(RU.\roman*)]
		\item For each vertex $v\ps{i} \in Q_{R,0}$ that is not incident to a crease (which we again call an \emph{ordinary vertex}), there exist precisely two corresponding vertices $\wh v\ps{i},g\wh v\ps{i} \in \wh Q_{R,0}$.
		\item For each $v\ps{i} \in Q_{R,0}$ that is incident to a crease (which we again call a \emph{crease vertex}), there exists precisely one corresponding vertex $\wh v\ps{i} = g\wh v\ps{i} \in \wh Q_{R,0}$.
		\item For each ordinary arrow $\alpha\ps{i}\colon u\ps{i}\rightarrow v\ps{j}$, there exist precisely two corresponding arrows in $\wh Q_R$, namely $\wh \alpha\ps{i}\colon u\ps{i}\rightarrow v\ps{j}$ and $g\wh \alpha\ps{i}\colon gu\ps{i}\rightarrow gv\ps{j}$.
		\item For each ordinary zero relation $p \in Z_R$, we have precisely two corresponding relations $p,gp \in \wh Z_R$. Likewise, for each ordinary commutative relation $p - q \in Z_R$, we have precisely two corresponding relations $p - q,gp -gq \in \wh Z_R$.
	\end{enumerate}
	
	We can therefore naturally extend the maps $\pi$ and $\iota$ to $\wh A_R$ and $A_R$, though we need to amend/clarify \ref{en:QU2} slightly as follows:
	\begin{enumerate}[label=(QU2$^\ast$)]
		\item $\iota$ respects the relations of $Z_R$: if $\alpha\ps{i_1}_1\ldots\alpha\ps{i_n}_n \in Z_R$ then $\iota(\alpha\ps{i_1}_1)\ldots\iota(\alpha\ps{i_n}_n) \in \wh Z_R$, and if $\alpha\ps{i_1}_1\ldots\alpha\ps{i_n}_n - \beta\ps{j_1}_1\ldots\beta\ps{j_m}_m \in Z_R$ then $\iota(\alpha\ps{i_1}_1)\ldots\iota(\alpha\ps{i_n}_n) - \iota(\beta\ps{j_1}_1)\ldots\iota(\beta\ps{j_m}_m)\in \wh Z_R$.
	\end{enumerate}
	
	Here, caution is advised, as this can sometimes lead to an arrow $\wh\alpha \in \wh Q_1$ being such that $\wh \alpha\ps{i} \in \im \iota$ but $\wh \alpha\ps{i-1} \not\in \im \iota$, as the following example shows.
	
	\begin{exam}
		Consider the folded gentle algebra $A$ given in Example~\ref{ex:C3} and its corresponding unfolded gentle algebra $\wh A$ given in Example~\ref{ex:C3Unf}. In this example, there is precisely one maximal path $p=\crs_1\alpha\beta \in\mathcal{P}_{Q,Z}$, and precisely two maximal paths $\wh p, \wh p' \in \mathcal{P}_{\wh Q, \wh Z}$ given by $\wh p = \wh\alpha\wh\beta$ and $\wh p = \wh\alpha'\wh\beta'$. The repetitive algebra $A_R$ is then given by the following quiver with relations.
		\begin{center}
			\begin{tikzpicture}
				\draw [anchor = east] (-6,0.1) node {$Q_R\colon$};
				\draw (-5.7,0) node {$\ldots$};
				\draw (-5,0.3) node {\footnotesize$\chi^{(i+1)}_{p}$};
				\draw [->] (-5.4,0) -- (-4.6,0);
				\draw (-4.3,0) node {$1^{(i)}$};
				\draw[<-] (-4.2,0.3) .. controls (-4,0.6) and (-4.2,0.7) .. (-4.3,0.7) .. controls (-4.4,0.7) and (-4.6,0.6) .. (-4.4,0.3);
				\draw (-4.3,1) node {\footnotesize$\eta_1^{(i)}$};
				\draw [->] (-4,0) -- (-3.2,0);
				\draw (-3.6,0.2) node {\footnotesize$\alpha^{(i)}$};
				\draw (-2.9,0) node {$2^{(i)}$};
				\draw [->] (-2.6,0) -- (-1.8,0);
				\draw (-2.2,0.2) node {\footnotesize$\beta^{(i)}$};
				\draw (-1.5,0) node {$3^{(i)}$};
				\draw [->] (-1.2,0) -- (-0.4,0);
				\draw (-0.8,0.3) node {\footnotesize$\chi^{(i)}_{p}$};
				
				\draw (0.1,0) node {$1^{(i-1)}$};
				\draw[<-] (0.2,0.3) .. controls (0.4,0.6) and (0.2,0.7) .. (0.1,0.7) .. controls (0,0.7) and (-0.2,0.6) .. (0,0.3);
				\draw (0.1,1) node {\footnotesize$\eta_1^{(i-1)}$};
				\draw [->] (0.6,0) -- (1.4,0);
				\draw (1,0.2) node {\footnotesize$\alpha^{(i-1)}$};
				\draw (1.9,0) node {$2^{(i-1)}$};
				\draw [->] (2.4,0) -- (3.2,0);
				\draw (2.8,0.2) node {\footnotesize$\beta^{(i-1)}$};
				\draw (3.7,0) node {$3^{(i-1)}$};
				\draw [->] (4.2,0) -- (5,0);
				\draw (4.6,0.3) node {\footnotesize$\chi^{(i-1)}_{p}$};
				\draw (5.3,0) node {$\ldots$};
			\end{tikzpicture}
		\end{center}
		\begin{align*}
			Z_R = \{(\crs_1\ps{i})^2+\stp_1\ps{i},\chi_{p}\ps{i}\alpha\ps{i-1}, \alpha\ps{i}\beta\ps{i}\chi_{p}\ps{i}\crs_1\ps{i-1}\alpha\ps{i-1},\beta\ps{i}\chi_{p}\ps{i}\crs_1\ps{i-1}\alpha\ps{i-1}\beta\ps{i-1}, \\ \crs_1\ps{i}\alpha\ps{i}\beta\ps{i}\chi_{p}\ps{i} - \alpha\ps{i}\beta\ps{i}\chi_{p}\ps{i}\crs_1\ps{i-1} : i \in \integer\}.
		\end{align*}
		On the other hand, the repetitive algebra $\wh A_R$ is given by the following quiver with relations.
		\begin{center}
			\begin{tikzpicture}
				\draw (-6.8,0.5) node {$\ldots$};
				\draw [->] (-6.5,0.5) -- (-5.9,0.1);
				\draw (-6.8,-0.5) node {$\ldots$};
				\draw [->] (-6.5,-0.5) -- (-5.9,-0.1);
				\draw [anchor = east] (-7.1,0.1) node {$\wh Q_R\colon$};
				\draw (-5.5,0) node {$\widehat{1}^{(i)}$};
				\draw [->] (-5.2,0.1) -- (-4.6,0.5);
				\draw (-5.1,0.5) node {\footnotesize$\widehat{\alpha}^{(i)}$};
				\draw (-4.2,0.7) node {$\widehat{2}^{(i)}$};
				\draw [->] (-3.9,0.7) -- (-2.6,0.7);
				\draw (-3.2,0.9) node {\footnotesize$\widehat{\beta}^{(i)}$};
				\draw (-2.2,0.7) node {$\widehat{3}^{(i)}$};
				\draw [->] (-5.2,-0.1) -- (-4.6,-0.5);
				\draw (-5.3,-0.5) node {\footnotesize$(\widehat{\alpha}')^{(i)}$};
				\draw (-4.2,-0.7) node {$(\widehat{2}')^{(i)}$};
				\draw [->] (-3.7,-0.7) -- (-2.7,-0.7);
				\draw (-3.2,-0.5) node {\footnotesize$(\widehat{\beta}')^{(i)}$};
				\draw (-2.2,-0.7) node {$(\widehat{3}')^{(i)}$};
				\draw [->] (-1.9,0.5) -- (-1.2,0.2);
				\draw (-1.4,0.6) node {\footnotesize$\chi_{\widehat{p}}^{(i)}$};
				\draw [->] (-1.7,-0.5) -- (-1.2,-0.2);
				\draw (-1.2,-0.7) node {\footnotesize$\chi_{\widehat{p}'}^{(i)}$};
				
				\draw (-0.6,0) node {$\widehat{1}^{(i-1)}$};
				\draw [->] (0,0.2) -- (0.7,0.5);
				\draw (0,0.5) node {\footnotesize$\widehat{\alpha}^{(i-1)}$};
				\draw (1.3,0.7) node {$\widehat{2}^{(i-1)}$};
				\draw [->] (1.8,0.7) -- (3.1,0.7);
				\draw (2.5,0.9) node {\footnotesize$\widehat{\beta}^{(i-1)}$};
				\draw (3.7,0.7) node {$\widehat{3}^{(i-1)}$};
				\draw [->] (0,-0.2) -- (0.7,-0.5);
				\draw (0.9,-0.1) node {\footnotesize$(\widehat{\alpha}')^{(i-1)}$};
				\draw (1.3,-0.7) node {$(\widehat{2}')^{(i-1)}$};
				\draw [->] (2,-0.7) -- (3,-0.7);
				\draw (2.5,-0.5) node {\footnotesize$(\widehat{\beta}')^{(i-1)}$};
				\draw (3.7,-0.7) node {$(\widehat{3}')^{(i-1)}$};
				\draw [->] (4,0.5) -- (4.6,0.2);
				\draw (5,0) node {$\ldots$};
				\draw [->] (4.3,-0.4) -- (4.6,-0.2);
			\end{tikzpicture}
		\end{center}
		\begin{align*}
			\wh Z_R = \{&\chi_{\wh p}\ps{i}(\wh\alpha')\ps{i-1},
			\chi_{\wh p'}\ps{i}\wh\alpha\ps{i-1}, \\
			&\wh\alpha\ps{i}\wh\beta\ps{i}\chi_{\wh p}\ps{i}\wh\alpha\ps{i-1},
			(\wh\alpha')\ps{i}(\wh\beta')\ps{i}\chi_{\wh p'}\ps{i}(\wh\alpha')\ps{i-1}, \\
			&\wh\beta\ps{i}\chi_{\wh p}\ps{i}\wh\alpha\ps{i-1}\wh\beta\ps{i-1}, 
			(\wh\beta')\ps{i}\chi_{\wh p'}\ps{i}(\wh\alpha')\ps{i-1}(\wh\beta')\ps{i-1}, \\ 
			&\wh\alpha\ps{i}\wh\beta\ps{i}\chi_{\wh p}\ps{i} - (\wh\alpha')\ps{i}(\wh\beta')\ps{i}\chi_{\wh p'}\ps{i} : i \in \integer\}.
		\end{align*}
		Here, we have $\pi(\wh v\ps{i})=\pi((\wh v')\ps{i}) = v\ps{i}$ and $\pi(\wh\gamma\ps{i})=\pi((\wh\gamma')\ps{i}) = \gamma\ps{i}$ for each $v\ps{i} \in Q_{R,0}$ and each $\gamma\ps{i} \in Q_{R,1}$ (and each $i \in \integer$). We again have two unfolding morphisms $\iota$ and $g\iota$ that are right inverse to $\pi$. Note, however, that  the relations are such that we have $\wh\alpha\ps{i} \in \im \iota$ if and only if $(\wh\alpha')\ps{i-1} \in \im \iota$ if and only if $\wh\alpha\ps{i-1} \not\in \im \iota$.
	\end{exam}
	
	\subsection{Unfolding the module category of a repetitive algebra}
	The previous subsection motivates an extension of the unfolding functor $U$ to the module category of the repetitive algebra of a folded gentle algebra. Henceforth, we will suppress superscripts (unless necessary) for the purposes of readability. So define a functor
	\begin{equation*}
		U_R\colon \fin A_R \rightarrow \fin \wh A_R
	\end{equation*}
	as follows.
	
	Firstly, each finite-dimensional module of $M \in \fin A_R$ has an underlying $K$-vector space given by
	\begin{equation*}
		M = \bigoplus_{v \in V} M_v,
	\end{equation*}
	where $V \subset Q_{R,0}$ and $M_v = M\stp_{v}$. Define a set $\wh V = \{\wh v \in \wh Q_{R,0} : \pi(\wh v) \in V\}$. We then define the underlying vector space of $\wh M = U(M)$ as
	\begin{equation*}
		\wh M = \bigoplus_{\wh v \in \wh V} \wh M_{\wh v},
	\end{equation*}
	where each $\wh M_{\wh v}$ is a copy of $M_{\pi(\wh v)}$. For each $\wh m \in \wh M$ and $\wh v \in Q_{0,R}$, we then define
	\begin{equation*}
		\wh m \stp_{\wh v} = 
		\begin{cases}
			\wh m	& \text{if } \wh v \in \wh V \text{ and } \wh m \in \wh M_{\wh v}, \\
			0		& \text{otherwise.}
		\end{cases}
	\end{equation*}
	For the action of non-idempotent elements of $\wh A_R$ on $\wh M$, we need to extend the definition $\fw$ to the repetitive algebra. So let $\wh \sigma \in \wh Q_{R,1} \cup \wh Q_{R,1}\inv$ be any symbol of $\wh Q_R$. Then we define
	\begin{equation*}
		\fw(\wh \sigma) =
		\begin{cases}
			\crs_{u}\ps{i}\sigma
			& \text{if } \wh \sigma \not\in \im \iota, u\in \cv \text{ and } v \not\in\cv,\\
			\sigma\crs_{v}\ps{j}
			& \text{if } \wh \sigma \not\in \im \iota, u \not\in \cv \text{ and } v \in\cv,\\
			\crs_{u}\ps{i}\sigma(\crs_{v}\inv)\ps{j}
			& \text{if } \wh \sigma \not\in \im \iota \text{ and } u,v \in \cv,\\
			\sigma
			& \text{otherwise,}
		\end{cases}
	\end{equation*}
	where $u\ps{i} = \pi(s(\wh \sigma))$, $v\ps{j} = \pi(t(\wh \sigma))$ and $\sigma = \pi(\wh \sigma)$. For any non-simple $\wh w=\wh\sigma_1\ldots\wh\sigma_n \in \wrd_{\wh Q_R}$, we then define
	\begin{equation*}
		\fw(\wh w)=\fw(\wh\sigma_1)\ldots\fw(\wh\sigma_n).
	\end{equation*}
	Following the exposition of Section~\ref{sec:UnfFunctor}, for any $\wh m \in \wh M$ and any arrow $\wh \alpha \in \wh Q_{R,1}$, we then define $\wh m \wh\alpha = \wh m \fw(\wh \alpha)$ (where, as in Section~\ref{sec:UnfFunctor}, the source and target vector subspaces of $\wh M$ are given by $s(\wh\alpha)$ and $t(\wh\alpha)$ respectively).
	
	For any morphism $f \in \Hom_{A_R}(M,M')$, we again write $f = (f_v)_{v \in V}$ and $U( f) = \wh f = (\wh f_{\wh v})_{\wh v \in \wh V}$, where each $f_v \colon M\stp_v \rightarrow M'\stp_v$ and each $\wh f_{\wh v} \colon \wh M\stp_{\wh v} \rightarrow \wh M'\stp_{\wh v}$ is $K$-linear. We then define $\wh f$ by $\wh f_{\wh v} = f_{\pi(\wh v)}$.
	
	We will again call this functor the \emph{unfolding functor}. Comparing the definition of $U_R$ and $U$ (from Section~\ref{sec:UnfFunctor}), one can see that these definitions are more or less the same. It may perhaps not be too surprising to the reader then that these functors satisfy many of the same properties. For example, it is clear that $U_R$ is exact, $K$-linear and faithful. We have an analogue of Remark~\ref{rem:UImage}, where we replace $U$ with $U_R$, $A$ with $A_R$ and $\wh A$ with $\wh A_R$. We also have the following important lemma, which extends Corollary~\ref{cor:UProj} to the repetitive algebra setting.
	
	\begin{lem} \label{lem:URProj}
		Let $M \in \fin A_R$ and $\wh v \in \wh Q_{R,0}$. Then the following are equivalent.
		\begin{enumerate}[label = (\alph*)]
			\item $U_R(M)$ contains an indecomposable projective direct summand isomorphic to $P(\wh v)$.
			\item $U_R(M)$ contains a projective direct summand isomorphic to $P(\wh v) \oplus P(g\wh v)$.
			\item $M$ contains an indecomposable projective direct summand isomorphic to $P(\pi(\wh v))$.
		\end{enumerate}
		In particular, $U_R(P(v)) \cong P(\iota(v)) \oplus P(g\iota(v))$ for any $v \in Q_{R,0}$.
	\end{lem}
	\begin{proof}
		(a) $\Rightarrow$ (c): Throughout, let $L \subseteq M$ be the indecomposable direct summand such that $P(\wh v) \subseteq U_R(L)$. It is known that the algebra $\wh A_R$ is special biserial (see, for example, \cite{RingelGentle} and \cite[Sec. 4, Proposition]{Schroer}). That is, $\wh A_R$ satisfies axioms \ref{en:FG1} and \ref{en:FG2}. Consequently,
		\begin{equation*}
			P(\wh v) = \stp_{\wh v} \wh A_R = \langle \stp_{\wh v} \rangle_K + \wh N_{\wh p} + \wh N_{\wh p'},
		\end{equation*}
		as a vector space, where $\wh N_{\wh p}$ and $\wh N_{\wh p'}$ are respectively the uniserial modules given by all non-stationary subpaths of distinct maximal paths $\wh p$ and $\wh p'$ of source $\wh v$ (of which there are at most 2). If there is only one such path then $P(\wh v)$ is uniserial and we take $\wh N_{\wh p'} = 0$. It is then clear from the definition of $U_R$ and the (un)folding procedure that we must have
		\begin{equation*}
		L \cong P(\pi(\wh v)) = \stp_{\pi(\wh v)} \wh A_R = S_{\pi(\wh v)} +  N_{[\wh p]} +  N_{[\wh p']},
		\end{equation*}
		where
		\begin{equation*}
			S_{\pi(\wh v)} =
			\begin{cases}
				\langle \stp_{\pi(\wh v)} \rangle	&	\text{if } g\wh v \neq \wh v, \\
				\langle \stp_{\pi(\wh v)}, \crs_{\pi(\wh v)} \rangle  & \text{if } g\wh v = \wh v
			\end{cases}
		\end{equation*}
		and $N_{[\wh p]}$ and $N_{[\wh p']}$ are again uniserial modules given by all non-stationary subpaths of the $\integer_2$-orbits of the respective paths $\wh p$ and $\wh p'$.
		
		(c) $\Rightarrow$ (b): By the unfolding procedure, for any path $p$ of source $v$, there are two distinct corresponding paths $\wh p$ and $g\wh p$ with respective sources $\wh v$ and $g \wh v$. Note that there are no fixed arrows of $\wh Q_R$ under the action of $\integer_2$. Moreover, if $v=\pi(\wh v)$ is ordinary then $\wh v$ and $g \wh v$ are distinct vertices. It is then clear from this and the definition of $U_R$ that we must have
		\begin{equation*}
			U_R(P(v)) = U(\stp_v A_R) \cong \stp_{\wh v} \wh A_R \oplus \stp_{g\wh v} \wh A_R = P(\iota(v)) \oplus P(g \iota(v)).
		\end{equation*}
		On the other hand, if $v=\pi(\wh v)$ is a crease vertex, then $\wh v = g \wh v$. However, dimension counting reveals that we still must have
		\begin{equation*}
			U_R(P(v)) = U(\stp_v A_R) \cong \stp_{\wh v} \wh A_R \oplus \stp_{g\wh v} \wh A_R = P(\iota(v)) \oplus P(g \iota(v)) =  P(\iota(v)) \oplus P(\iota(v)),
		\end{equation*}
		since $\crs_{v} \in \stp_v A_R$ and since there are two copies of the vector subspace $P(v) \stp_{u}$ in $U_R(P(v))$ for each ordinary vertex $u$.
		
		(b) $\Rightarrow$ (a): This is obvious from the statement itself.
	\end{proof}
	
	The functor $U_R$ trivially satisfies Corollary~\ref{cor:UProps}(a). We also have an equivalent of Proposition~\ref{prop:UStrings}, which we do not list here: The maps $\theta$ and $\rho$ can naturally be extended to the classes of strings over the algebras $A_R$ and $\wh A_R$, and the proof of Proposition~\ref{prop:UStrings} does not require the gentle axioms \ref{en:FG3} and \ref{en:FG4} (\ref{en:FG5} is used implicitly in the proof, but $A_R$ also satisfies \ref{en:FG5}). 
	
	\subsection{The proof of the closure under derived equivalence} \label{sec:DerEqProof}
	We now move towards the proof of the main theorem of this section --- that folded gentle algebras are closed under derived equivalence. We begin with some preliminary results that will help with the proof.
	
	\begin{lem}\label{lem:URExt}
		Suppose $M \in \fin A_R$ is indecomposable and such that $\Ext_{A_R}^1(M,M)=0$. Then $\Ext_{\wh A_R}^1(U_R(M),U_R(M))=0$.
	\end{lem}
	\begin{proof}
		Consider a projective resolution of $M$ given by
		\begin{equation*}
			\xymatrix@1{\ldots \ar[r] & P_2 \ar[r]^-{f_1} & P_1 \ar[r]^-{f_0} & P_0 \ar[r] & M \ar[r] & 0}.
		\end{equation*}
		If $\Ext_{A_R}^1(M,M)=0$, then there exists a cochain complex
		\begin{equation*}
			\xymatrix@1{0 \ar[r] & \Hom_{A_R}(P_0,M) \ar[r]^-{\delta_0} & \Hom_{A_R}(P_1,M) \ar[r]^-{\delta_1} & \Hom_{A_R}(P_2,M) \ar[r] & \ldots}
		\end{equation*}
		such that $\Ker \delta_1 = \im \delta_0$. Now by the variant of Corollary~\ref{cor:UProps}(a) that applies to the functor $U_R$, we have a cochain complex
		\begin{equation*}
			\xymatrix@1{0 \ar[r] & \Hom_{\wh A_R}(\wh P_0,\wh M) \ar[r]^-{\wh\delta_0} & \Hom_{\wh A_R}(\wh P_1,\wh M) \ar[r]^-{\wh\delta_1} & \Hom_{\wh A_R}(\wh P_2,\wh M) \ar[r] & \ldots},
		\end{equation*}
		where (for readability purposes) $\wh X = U_R(X)$ for any $X \in \mod*A_R$ and $\wh\delta_i = U_R( {-} \circ f_i) = U_R({-}) \circ U_R(f_i)$. In particular, $\Ext_{\wh A_R}^1(\wh M, \wh M)= \Ker \wh\delta_1 / \im \wh\delta_0$. 	But since $U_R$ is exact, we must also have $\Ker \wh\delta_1 =\im \wh\delta_0$. So $\Ext_{\wh A_R}^1(U_R(M),U_R(M))=0$, as required.
	\end{proof}
	
	As a result of Lemma~\ref{lem:URProj}, we can safely consider the morphisms in the stable module category $\ufin A_R$ in terms of morphisms in the stable module category $\ufin \wh A_R$. Here, by $\ufin B$ (for some $K$-algebra $B$), we mean the category with the same objects as $\fin B$, but whose morphisms are equivalence classes of morphisms in $\fin B$ modulo those that factor through a projective module. Similarly, by $\uEnd_{B}(M)$, we mean the algebra of stable endomorphisms of $M$. That is, the algebra of all equivalence classes of endomorphisms modulo those that factor through a projective module. We begin by proving the following, which follows from \cite[Theorem 1.1]{SchroerZim} via an (un)folding argument.
	
	\begin{prop} \label{prop:FoldedEnd}
		Suppose $M \in \fin A_R$ is such that $\Ext_{A_R}^1(M,M)=0$. Then the algebra of stable endomorphisms $\uEnd_{A_R}(M)$ is a folded gentle algebra.
	\end{prop}
	\begin{proof}
		Consider the object $\wh M = U_R(M) \in \fin A_R$. By Lemma~\ref{lem:URExt}, $\Ext_{A_R}^1(\wh M,\wh M)=0$. Consequently, $\wh M$ does not contain a direct summand that is isomorphic to a band module. Thus, $\wh M$ is a direct sum of string modules and/or projective-injective modules (since $\wh A_R$ is special biserial). But then Lemma~\ref{lem:URProj} and Proposition~\ref{prop:UStrings} imply that $g \wh M \cong \wh M$. Hence, the algebra $\uEnd_{A_R}(M)$ has a well-defined $\integer_2$-action. Moreover, $\uEnd_{\wh A_R}(\wh M)$ is a gentle algebra by \cite[Theorem 1.1]{SchroerZim}. Thus, $\uEnd_{\wh A_R}(\wh M)$ is an unfolded gentle algebra.
		
		We will consider the modules $M$ and $\wh M$ in more detail. Since we are considering the algebra of stable endomorphisms of $M$, we may henceforth assume without loss of generality that $M$ (and equivalently $\wh M$ by Lemma~\ref{lem:URProj}) does not contain any direct summands that are projective. Since $\wh M$ does not have any band module direct summands, Proposition~\ref{prop:UStrings} implies that
		\begin{equation*}
			M \cong \bigoplus_{w \in W^{a}} M(w)^{k_w} \oplus \bigoplus_{w \in W^{s}} M(w)^{k_w}
		\end{equation*}
		for some subsets $W^a,W^s \subset \str_{A_R}$ and some $k_w \in \spint$, where $W^a$ consists entirely of asymmetric strings and $W^s$ consists entirely of symmetric strings. Since we are only interested in algebras up to Morita equivalence, we can assume that each $k_w =1$ (and thus, that the indecomposable direct summands are pairwise non-isomorphic). By Proposition~\ref{prop:UStrings} and Lemma~\ref{lem:FoldedSymmetric}(a), we then have
		\begin{equation*}
			\wh M \cong \bigoplus_{w \in W^{a}} (M(\theta(w)) \oplus M(g\theta(w))) \oplus \bigoplus_{w \in W^{s}} M(\theta(w))^2
		\end{equation*}
		From this, we will consider the quotient module
		\begin{equation*}
			\wh M^\ast \cong \bigoplus_{w \in W^{a}} (M(\theta(w)) \oplus M(g\theta(w))) \oplus \bigoplus_{w \in W^{s}} M(\theta(w)).
		\end{equation*}
		Clearly, $\Ext_{\wh A_R}^1(\wh M^\ast,\wh M^\ast)=0$ and is closed under $\integer_2$, and so $\uEnd_{\wh A_R}(\wh M^\ast)$ is also an unfolded gentle algebra.
		
		We will now consider the algebra $\uEnd_{\wh A_R}(\wh M^\ast)$ in more detail. We can see from the above presentation that the set
		\begin{equation*}
			\wh E = \{\wh e_{\theta(w)}, g\wh e_{\theta(w)} : w \in W^a\} \cup \{\wh e_{\theta(w)} : w \in W^s\}
		\end{equation*}
		is a complete set of primitive orthogonal idempotents of $\uEnd_{\wh A_R}(\wh M^\ast)$, where each $\wh e_{\theta(w)}\colon M(\theta(w))\rightarrow M(\theta(w))$ is the identity map. Let $\wh Q_{\wh M^\ast}$ be the quiver and $\wh Z_{\wh M^\ast}$ be the set of relations such that $\uEnd_{\wh A_R}(\wh M^\ast) \cong K\wh Q_{\wh M^\ast} / \langle \wh Z_{\wh M^\ast} \rangle$. Then the vertices of $\wh Q_{\wh M^\ast}$ are indexed by the set $\wh E$. 
		
		Since $\uEnd_{\wh A_R}(\wh M^\ast)$ is gentle, for any vertex $\wh v_{\wh w}$ of $\wh Q_{\wh M^\ast}$ indexed by $\wh e_{\wh w}$, there are at most two arrows $\wh \alpha_{\wh w, \wh w_1}$ and $\wh \alpha'_{\wh w, \wh w_2}$ of target $\wh v_{\wh w}$. Equivalently, there are at most two morphisms $\wh f_{\wh w, \wh w_1}\colon M(\wh w) \rightarrow M(\wh w_1)$ and $\wh f'_{\wh w, \wh w_2} \colon M(\wh w) \rightarrow M(\wh w_2)$ in $\uEnd_{\wh A_R}(\wh M^\ast)$ such that any other morphism of $\uEnd_{\wh A_R}(\wh M^\ast)$ whose source is the direct summand $M(\wh w)$ factors through either $\wh f_{\wh w, \wh w_1}$ or $\wh f'_{\wh w, \wh w_2}$. Similarly, there are at most two arrows $\wh \alpha_{\wh w_3, \wh w}$ and $\wh \alpha'_{\wh w_4, \wh w}$ of source $\wh v_{\wh w}$, or equivalently, at most two morphisms $\wh f_{\wh w_3, \wh w}\colon M(\wh w_3) \rightarrow M(\wh w)$ and $\wh f'_{\wh w_4, \wh w} \colon M(\wh w_4) \rightarrow M(\wh w)$ in $\uEnd_{\wh A_R}(\wh M^\ast)$ such that any other morphism of $\uEnd_{\wh A_R}(\wh M^\ast)$ whose target is the direct summand $M(\wh w)$ factors through either $\wh f_{\wh w_3, \wh w}$ or $\wh f'_{\wh w_4, \wh w}$. 
		
		Now let
		\begin{equation*}
			\wh F = \{\wh f_{\wh w, \wh w'} : M(\wh w) \text{ and } M(\wh w') \text{ are direct summands of } \wh M^{\ast}\}
		\end{equation*}
		be the complete set of all morphisms that bijectively correspond to distinct arrows $\wh v_{\wh w'} \rightarrow \wh v_{\wh w}$ in $\wh Q_{\wh M^\ast}$. Since $\uEnd_{\wh A_R}(\wh M^\ast)$ is unfolded gentle, $g \wh f_{\wh w, \wh w'} \neq \wh f_{\wh w, \wh w'}$ for any $\wh f_{\wh w, \wh w'} \in \wh F$ (this is due to Remark~\ref{rem:NoFixedArrows}). In addition, we have a relation $\wh \alpha_{\wh w_2, \wh w_3}\wh \alpha_{\wh w_1, \wh w_2} \in \wh Z_{\wh M^\ast}$ precisely when $\wh f_{\wh w_2, \wh w_3}\wh f_{\wh w_1, \wh w_2} =0$, and for any such relation, we also have $(g\wh f_{\wh w_2, \wh w_3})(g\wh f_{\wh w_1, \wh w_2}) =0$.

		We will now consider the algebra $\uEnd_{\wh A_R}( M)$ in more detail. From the presentation of $M$ above, we can see that the set
		\begin{equation*}
			E= \{e_{w} : w \in W^a \cup W^s \}
		\end{equation*}
		is a complete set of primitive orthogonal idempotents of $\uEnd_{ A_R}( M)$, where each $e_{w}\colon M(w) \rightarrow M(w)$ is the identity map. Moreover, for each $w \in W^s$, there is a linearly independent isomorphism $h_{w}\colon M(w) \rightarrow M(w)$. Specifically, for each symmetric string $w = \sigma_1\ldots\sigma_n\crs\sigma_n\inv\ldots\sigma_1\inv$, the element $h_{w}$ represents the non-trivial isomorphism
		\begin{equation*}
			M(\sigma_1\ldots\sigma_n\crs_v\sigma_n\inv\ldots\sigma_1\inv) \rightarrow M(\sigma_1\ldots\sigma_n\crs_v\inv\sigma_n\inv\ldots\sigma_1\inv)
		\end{equation*}
		given in Lemma~\ref{lem:StringIsos}(b). Moreover, $h_{w}$ satisfies the same quadratic relation as $\crs_v$. That is, if $\crs_v^2 = \lambda_1 \crs_v + \lambda_2 \stp_v$, then $h^2_{w} = \lambda_1 h_{w} + \lambda_2 e_{w}$. 
		
		Let $ Q_{ M}$ be the quiver and $ Z_{ M}$ be the set of relations such that $\uEnd_{ A_R}( M) \cong K Q_{ M} / \langle  Z_{ M} \rangle$. We can then see that the vertices $v_w$ of $Q_M$ are indexed by the set $E$. Moreover, $E$ is in bijective correspondence with the set of $\integer_2$-orbits of elements of $\wh E$. Thus, the vertices of $Q_M$ are obtained from the vertices of $\wh Q_{\wh M^\ast}$ via \ref{en:F1} of the folding procedure.
		
		Now consider any non-idempotent morphism $f\colon M(w) \rightarrow M(w')$ of string modules, where $w,w' \in W^a$. It is not difficult to verify that there are then precisely two corresponding non-idempotent morphisms $\wh f\colon M(\theta(w)) \rightarrow M(\theta(w)) \oplus M(g\theta(w))$ and $g\wh f\colon M(g\theta(w)) \rightarrow M(\theta(w)) \oplus M(g\theta(w))$ in $\uEnd_{\wh A_R}(\wh M^\ast)$. Thus, the elements of the vector subspace of all morphisms 
		\begin{equation*}
			f\colon \bigoplus_{w \in W^{a}} M(w) \rightarrow \bigoplus_{w \in W^{a}} M(w)
		\end{equation*}
		in $\uEnd_{ A_R}( M)$ are in bijective correspondence with the $\integer_2$-orbits of elements in the vector subspace of all morphisms 
		\begin{equation*}
			\wh f\colon \bigoplus_{w \in W^{a}} (M(\theta(w)) \oplus M(g\theta(w))) \rightarrow \bigoplus_{w \in W^{a}} (M(\theta(w)) \oplus M(g\theta(w))) 
		\end{equation*}
		in $\uEnd_{\wh A_R}( \wh M^\ast)$.

		Next, consider any non-idempotent morphism $f\colon M(w) \rightarrow M(w')$ of string modules, where $w \in W^a$ and $w \in W^s$. For any such morphism, there is a corresponding, linearly independent morphism $f'\colon M(w) \rightarrow M(w')$ that factors through a non-trivial linear combination of $e_{w'}$ and $h_{w'}$. For example, we have linearly independent morphisms $f$ and $h_{w'} f$. Moreover, there exist precisely two corresponding non-idempotent morphisms $\wh f\colon M(\theta(w)) \rightarrow M(\theta(w'))$ and $g\wh f\colon M(g\theta(w)) \rightarrow M(\theta(w'))$ (since $\theta(w')$ is $\integer_2$-invariant). Thus, the elements of the vector subspace of all morphisms
		\begin{equation*}
			f\colon \bigoplus_{w \in W^{a}} M(w) \rightarrow \bigoplus_{w \in W^{s}} M(w)
		\end{equation*}
		modulo those that factor through a morphism $h_w$ (with $w \in W^s$) are in bijective correspondence with the $\integer_2$-orbits of elements the vector subspace of all morphisms
		\begin{equation*}
			\wh f\colon \bigoplus_{w \in W^{a}} (M(\theta(w)) \oplus M(g\theta(w))) \rightarrow \bigoplus_{w \in W^{s}} M(\theta(w)).
		\end{equation*}
		
		Dually, for any non-idempotent morphism $f\colon M(w) \rightarrow M(w')$ with $w \in W^s$ and $w \in W^a$, there is a corresponding, linearly independent morphism $f h_w$. Similar reasoning to that used above shows that the elements of the vector subspace of all morphisms
		\begin{equation*}
			f\colon \bigoplus_{w \in W^{s}} M(w) \rightarrow \bigoplus_{w \in W^{a}} M(w)
		\end{equation*}
		modulo those that factor through a morphism $h_w$ (with $w \in W^s$) are in bijective correspondence with the $\integer_2$-orbits of elements the vector subspace of all morphisms
		\begin{equation*}
			\wh f\colon \bigoplus_{w \in W^{s}} M(\theta(w)) \rightarrow \bigoplus_{w \in W^{a}} (M(\theta(w)) \oplus M(g\theta(w))).
		\end{equation*}
		
		Finally, for any non-idempotent morphism $f\colon M(w) \rightarrow M(w')$ with $w,w' \in W^s$, one can see that the elements of the vector subspace of all morphisms
		\begin{equation*}
			f\colon \bigoplus_{w \in W^{s}} M(w) \rightarrow \bigoplus_{w \in W^{s}} M(w)
		\end{equation*}
		modulo those that factor through a morphism $h_w$ (with $w \in W^s$) are in bijective correspondence with the $\integer_2$-orbits of elements the vector subspace of all morphisms
		\begin{equation*}
			\wh f\colon \bigoplus_{w \in W^{s}} M(\theta(w)) \rightarrow \bigoplus_{w \in W^{s}} M(\theta(w)).
		\end{equation*}
		
		Consequently, the morphisms of $\uEnd_{A_R}(M)$, modulo those that factor through morphisms $\wh h_w$ (with $w \in W^s$) are in bijective correspondence with the $\integer_2$-orbits of morphisms of $\uEnd_{\wh A_R}(\wh M^\ast)$. All of this implies three things: Firstly, the morphisms $h_w$ correspond to crease loops in the quiver $Q_M$ --- so $Q_M$ is obtained from $\wh Q_{\wh M^\ast}$ via \ref{en:F3} of the folding procedure. Secondly, the arrows of $Q_M$ are in bijective correspondence with the $\integer_2$-orbits of arrows of $\wh Q_{\wh M^\ast}$ --- so $Q_M$ is obtained from $\wh Q_{\wh M^\ast}$ via \ref{en:F2} of the folding procedure. Thirdly, the relations of $Z_{M}$ are precisely those given by the $\integer_2$-orbits of the relations of $\wh Z_{\wh M^\ast}$ --- so $Z_{M}$ is obtained from $\wh Z_{\wh M^\ast}$ via \ref{en:F4} of the folding procedure. Thus, the algebra $\uEnd_{A_R}(M)$ is obtained from the unfolded gentle algebra $\uEnd_{\wh A_R}( \wh M^\ast)$ via the folding procedure, and hence, $\uEnd_{A_R}(M)$ is a folded gentle algebra.
	\end{proof}
	
	Armed with this, we can prove the main theorem of this section, which is the analogue of \cite[Corollary 1.2]{SchroerZim} for folded gentle algebras.
	
	\begin{thm} \label{thm:DerClosed}
		Let $A$ be a finite dimensional folded gentle algebra and let $B$ be an algebra that is derived equivalent to $A$. Then $B$ is a folded gentle algebra. Moreover, there are only finitely many such algebras (up to Morita equivalence) that are derived equivalent to $A$.
	\end{thm}
	\begin{proof}
		The proof is very much the same as the proof of \cite[Corollary 1.2]{SchroerZim}. Denote by $D^b(A)$ the bounded derived category of $A$. Let $B$ be an algebra such that $D^b(B) \simeq D^b(A)$. By a theorem of Rickard (\cite[Theorem 6.4, Corollary 8.3]{Rickard}, see also \cite[Ch. 3.2]{KonigZim}) this is true if and only if there exists a tilting complex $T \in D^b(A)$ such that $B \cong \End_{D^b(A)}(T)$. We refer the reader to \cite[Definition 5.2, Theorem 6.4]{Rickard} for the precise definition of a tilting complex, however we remark here that this implies that $\Hom_{D^b(A)}(T,T[i]) = 0$ for all $i \neq 0$.
		
		Now by \cite[Ch. 2.4]{Happel}, there exists a fully faithful embedding
		\begin{equation*}
			F\colon D^b(A) \rightarrow \ufin A_R.
		\end{equation*}
		In particular, since $\ufin A_R$ is a suspended category with suspension functor given by the syzygy functor $\Omega$, and since $F$ is fully faithful and respects the triangulated structure, we have
		\begin{align*}
			\Hom_{D^b(A)}(T,T[1]) &= \Hom_{D^b(A)}(T[-1],T)\\
			&=\uHom_{A_R}(F(T[-1]),F(T))\\
			&= \uHom_{A_R}(\Omega(F(T)),F(T))\\
			&= \Ext^1_{A_R}(F(T),F(T))\\
			&= 0.
		\end{align*}
		Noting that $\End_{D^b(A)}(T)$ and $\uEnd_{A_R}(F(T))$ are isomorphic, we see that $B\cong\End_{D^b(A)}(T)$ is a folded gentle algebra by Proposition~\ref{prop:FoldedEnd}, as required.
		
		For the final statement of the theorem, we note by \cite[Lemma 6.3.3]{KonigZim} that derived equivalent finite-dimensional algebras have the same number of isomorphism classes of simple modules. The isomorphism classes of simple modules of a folded gentle algebra are in bijective correspondence with the vertices of its quiver. Now it is clear from the combinatorics induced by axioms \ref{en:FG1}-\ref{en:FG6} that, for any $n$, there are only finitely many folded gentle algebras (up to Morita equivalence) with $n$ vertices. Hence, there are only finitely many algebras (up to Morita equivalence) that are derived equivalent to $A$.
	\end{proof}
	
	\appendix
	\section{An Explicit Worked Example}\label{sec:Example}
	In this appendix, we will explore a comprehensive example that showcases the entire theory of the paper. Throughout, we will work with the folded gentle algebra $A = K Q / \langle Z \rangle$, where $K = \mathbb{F}_3$, $Q$ is the quiver
	\begin{equation*}
		\begin{tikzpicture}
			\draw (-1.2,0.7) node {1};
			\draw (-1.2,-0.7) node {2};
			\draw (0,0) node {3};
			\draw (1.2,0.7) node {4};
			\draw (1.2,-0.7) node {5};
			\draw [->](-0.9,0.5) -- (-0.2,0.1);
			\draw [->](0.2,0.1) -- (0.9,0.5);
			\draw [->](-0.9,-0.5) -- (-0.2,-0.2);
			\draw [->](0.2,-0.2) -- (0.9,-0.5);
			\draw [->](-1.2,-0.4) -- (-1.2,0.4);
			\draw [->](1.2,0.4) -- (1.2,-0.4);
			\draw [->](-1.3,1) .. controls (-1.4,1.2) and (-1.4,1.4) .. (-1.2,1.4) .. controls (-1,1.4) and (-1,1.2) .. (-1.1,1);
			\draw [->](1.1,1) .. controls (1,1.2) and (1,1.4) .. (1.2,1.4) .. controls (1.4,1.4) and (1.4,1.2) .. (1.3,1);
			\draw (-2.1,0.2) node {\footnotesize $Q\colon$};
			\draw (-0.7,1.2) node {\footnotesize $\eta_1$};
			\draw (0.7,1.2) node {\footnotesize $\eta_4$};
			\draw (-1.4,-0.1) node {\footnotesize $\alpha$};
			\draw (-0.4,0.5) node {\footnotesize $\beta$};
			\draw (-0.5,-0.6) node {\footnotesize $\gamma$};
			\draw (0.4,0.5) node {\footnotesize $\delta$};
			\draw (0.5,-0.6) node {\footnotesize $\zeta$};
			\draw (1.4,0) node {\footnotesize $\xi$};
		\end{tikzpicture}
	\end{equation*}
	and
	\begin{equation*}
		Z = \{\alpha\beta, \beta\delta, \gamma\zeta, \delta\xi, \eta_1^2 + \stp_1, \eta_4^2 + \eta_4 + 2\stp_4\}
	\end{equation*}
	In particular, 
	\begin{align*}
		\ov &= \{2,3,5\}, & \cv &= \{1,4\},\\
		\oa &= \{\alpha,\beta,\gamma,\delta,\zeta,\xi\}, & \ca &= \{\eta_1,\eta_4\},\\
		\orel &= \{\alpha\beta, \beta\delta, \gamma\zeta, \delta\xi\}, & \crel &= \{\eta_1^2 + \stp_1, \eta_4^2 + \eta_4 + 2\stp_2\}.
	\end{align*}
	See Definition~\ref{def:FoldedGentleTriple} to verify the relevant axioms.
	
	\subsection{Biseriality}
	The indecomposable projective right $A$-modules are as follows.
	\begin{align*}
		\stp_1A &= \langle \varepsilon_1, \beta, \beta\zeta, \eta_1, \eta_1\beta, \eta_1\beta\zeta\rangle \cong M(\zeta^{-1}\beta^{-1}\eta_1\beta\zeta), \\
		\stp_2A &= \langle \varepsilon_2, \alpha, \alpha\eta_1, \alpha\eta_1\beta, \alpha\eta_1\beta\zeta, \gamma, \gamma\delta, \gamma\delta\eta_4, \gamma\delta\eta_4\xi\rangle \cong M(\zeta^{-1}\beta^{-1}\eta_1\inv\alpha^{-1}\gamma\delta\eta_4\xi), \\
		\stp_3A &= \langle \varepsilon_3, \zeta, \delta, \delta\eta_4, \delta\eta_4\xi\rangle \cong M(\zeta^{-1}\delta\eta_4\xi), \\
		\stp_4A &= \langle \varepsilon_4, \xi, \eta_4, \eta_4\xi\rangle \cong M(\xi^{-1}\eta_4\xi), \\
		\stp_5A &=\langle \varepsilon_5\rangle \cong M(\varepsilon_5).
	\end{align*}
	The indecomposable injective right $A$-modules are
	\begin{align*}
		D(A\varepsilon_1 ) &= D(\langle \varepsilon_1, \alpha, \eta_1, \alpha\eta_1\rangle) \cong M(\alpha\eta_1\alpha^{-1}), \\
		D(A\varepsilon_2 ) &= D(\langle \varepsilon_2\rangle) \cong M(\varepsilon_2), \\
		D(A\varepsilon_3 ) &= D(\langle \varepsilon_3, \beta, \eta_1\beta, \alpha\eta_1\beta, \gamma\rangle) \cong M(\alpha\eta_1\beta\gamma^{-1}), \\
		D(A\varepsilon_4 ) &= D(\langle \varepsilon_4, \delta, \gamma\delta, \eta_4, \delta\eta_4, \gamma\delta\eta_4\rangle) \cong M(\gamma\delta\eta_4\delta^{-1}\gamma^{-1}), \\
		D(A\varepsilon_5 ) &= D(\langle \varepsilon_5, \zeta, \beta\zeta, \eta_1\beta\zeta, \alpha\eta_1\beta\zeta, \xi, \eta_4\xi, \delta\eta_4\xi, \gamma\delta\eta_4\xi\rangle) \cong M(\alpha\eta_1\beta\zeta\xi^{-1}\eta_4\inv\delta^{-1}\gamma^{-1}).
	\end{align*}
	One can easily observe the biserial structure of $A$ from this complete list of indecomposable right projective and right injective modules. For example,
	\begin{align*}
		\rad \stp_1A &\cong \langle \beta, \beta\zeta\rangle \oplus \langle \eta_1\beta, \eta_1\beta\zeta\rangle \cong M(\zeta) \oplus M(\zeta), \\
		\rad \stp_2A &\cong \langle \alpha, \alpha\eta_1, \alpha\eta_1\beta, \alpha\eta_1\beta\zeta\rangle \oplus \langle\gamma, \gamma\delta, \gamma\delta\eta_4, \gamma\delta\eta_4\xi\rangle \cong M(\eta_1\beta\zeta) \oplus M(\delta\eta_4\xi), \\
		\soc D(A\stp_3) &\cong D(\langle \beta, \eta_1\beta, \alpha\eta_1\beta\rangle) \oplus D(\langle\gamma\rangle )\cong M(\alpha\eta_1) \oplus M(\stp_2), \\
		\soc D(A\stp_4) &\cong D(\langle \delta, \gamma\delta\rangle) \oplus D(\langle \delta\eta_4, \gamma\delta\eta_4\rangle) \cong M(\gamma) \oplus M(\gamma),
	\end{align*}
	and so on (c.f. Definition~\ref{def:Biserial} and Theorem~\ref{thm:Biserial}).
	
	\subsection{The combinatorics of unfolding and the indecomposable modules}
	The unfolded gentle algebra obtained from $A$ via the unfolding procedure (\ref{en:U1}-\ref{en:U4}) is $\wh A = K\wh Q / \langle\wh Z\rangle$, where
	\begin{equation*}
		\begin{tikzpicture}
			\draw (-1.5,1.4) node {$\widehat 2$};
			\draw (-0.3,0.7) node {$\widehat 3$};
			\draw (0.9,1.4) node {$\widehat 5$};
			\draw [->](-1.2,0.2) -- (-0.5,0.6);
			\draw [->](-0.1,0.6) -- (0.6,0.2);
			\draw [->](-1.2,1.2) -- (-0.5,0.9);
			\draw [->](-0.1,0.9) -- (0.6,1.2);
			\draw [->](-1.5,1.1) -- (-1.5,0.3);
			\draw [->](0.9,0.3) -- (0.9,1.1);
			\draw (-1.7,0.8) node {\footnotesize $\widehat\alpha$};
			\draw (-1,0.6) node {\footnotesize $\widehat\beta$};
			\draw (-0.8,1.3) node {\footnotesize $\widehat\gamma$};
			\draw (0.4,0.6) node {\footnotesize $\widehat\delta$};
			\draw (0.2,1.3) node {\footnotesize $\widehat\zeta$};
			\draw (1.1,0.7) node {\footnotesize $\widehat\xi$};
			
			\draw (-1.5,0) node {$\widehat 1$};
			\draw (-1.5,-1.4) node {$\widehat 2'$};
			\draw (-0.3,-0.7) node {$\widehat 3'$};
			\draw (0.9,0) node {$\widehat 4$};
			\draw (0.9,-1.4) node {$\widehat 5'$};
			\draw [->](-1.2,-0.2) -- (-0.5,-0.6);
			\draw [->](-0.1,-0.6) -- (0.6,-0.2);
			\draw [->](-1.2,-1.2) -- (-0.5,-0.9);
			\draw [->](-0.1,-0.9) -- (0.6,-1.2);
			\draw [->](-1.5,-1.1) -- (-1.5,-0.3);
			\draw [->](0.9,-0.3) -- (0.9,-1.1);
			\draw (-2.3,0) node {\footnotesize $\widehat Q\colon$};
			\draw (-1.7,-0.8) node {\footnotesize $\widehat\alpha'$};
			\draw (-0.7,-0.2) node {\footnotesize $\widehat\beta'$};
			\draw (-0.8,-1.3) node {\footnotesize $\widehat\gamma'$};
			\draw (0.1,-0.2) node {\footnotesize $\widehat\delta'$};
			\draw (0.2,-1.3) node {\footnotesize $\widehat\zeta'$};
			\draw (1.1,-0.7) node {\footnotesize $\widehat\xi'$};
		\end{tikzpicture}
	\end{equation*}
	\begin{equation*}
		\wh Z = \{\wh\alpha\wh\beta, \wh\alpha'\wh\beta',\wh\beta\wh\delta,\wh\beta'\wh\delta',\wh\gamma\wh\zeta,\wh\gamma'\wh\zeta',\wh\delta\wh\xi,\wh\delta'\wh\xi'\}.
	\end{equation*}
	The action of $\integer_2 = \{1,g\}$ on $\wh Q$ (as prescribed by Lemma~\ref{lem:Folding}) is given by $g\wh{i} = \wh i'$ and $g \wh{a} = \wh a'$ for each $ i \in \{2,3,5\}$ and each $ a \in \{\alpha,\beta,\gamma,\delta,\zeta,\xi\}$. In addition, $g \wh i = \wh i$ for each $i \in \{1,4\}$. The quiver folding morphism $\pi \colon \wh Q \rightarrow Q \setminus \ca$ is defined such that $\pi(\wh i) = \pi(\wh i') = i$ and $\pi(\wh a) = \pi(\wh a') = a$. In particular, $\pi(\wh i) = \pi(g\wh i)$ and $\pi(\wh a) = \pi(g\wh a)$. 
	
	There is more than one choice of quiver unfolding morphism $\iota\colon Q \setminus \ca \rightarrow \wh Q$, but we will henceforth use the unfolding defined by $\iota(i) = \wh i$ and $\iota(a) = \wh a$ for each $i \in Q_0$ and $a \in Q_1 \setminus \ca$. It is clear that $\iota$ is a right inverse of $\pi$ and respects the relations of $\orel$, and thus meets the necessary requirements outlined in Definition~\ref{def:QuiverMorphisms}.

	Henceforth, we let $U\colon \mod*A \rightarrow \mod*\wh A$ be the unfolding functor of Section~\ref{sec:FoldedModule}, which describes how the module category of the folded gentle algebra $A$ is embedded into the module category of the unfolded gentle algebra $\wh A$.
	
	\subsubsection{Unfolding string modules} 
	There are infinitely many strings in $A$ and $\wh A$, so we will give an example of an asymmetric string and a symmetric string to showcase the combinatorics of (un)folding strings and their associated modules. Consider the string $\wh w = \wh \alpha \wh\beta' \wh\zeta'(\wh\xi')\inv\wh\xi$ of $\wh A$. This string is not $\integer_2$-invariant, since $g\wh w = \wh \alpha' \wh\beta \wh\zeta\wh\xi\inv\wh\xi' \not\approx \wh w$. The corresponding folded string of $A$ is
	\begin{equation*}
		w=\rho(\wh w) = \alpha\crs_1\beta\zeta\xi\inv\crs_4\inv\xi \approx \alpha\crs_1\inv\beta\zeta\xi\inv\crs_4\xi = \rho(g\wh w)
	\end{equation*}
	which is an asymmetric string. Due to Proposition~\ref{prop:UStrings}, we know that $M(w) = M(\alpha\crs_1\beta\zeta\xi\inv\crs_4\inv\xi) \in \mod* A$ corresponds to the module
	\begin{equation*}
		U(M(w)) \cong M(\wh w) \oplus M(g\wh w)
		\cong M(\wh \alpha \wh\beta' \wh\zeta'(\wh\xi')\inv\wh\xi) \oplus M(\wh \alpha' \wh\beta \wh\zeta\wh\xi\inv\wh\xi') \in \mod*\wh A.
	\end{equation*}
	
	Now consider the string $\wh w = \wh \gamma\inv\wh\alpha(\wh\alpha')\inv\wh\gamma'$ of $\wh A$. We can see that $g\wh w = (\wh \gamma')\inv\wh\alpha'\wh\alpha\inv\wh\gamma \approx \wh w$, since $g\wh w = \wh w\inv$. So $\wh w$ is $\integer_2$-invariant. In addition, the folded string is
	\begin{equation*}
		w=\rho(\wh w) = \gamma\inv\alpha\crs_1\alpha\inv\gamma \approx \gamma\inv\alpha\crs_1\inv\alpha\inv\gamma = \rho(g\wh w),
	\end{equation*}
	which is symmetric, since $w\inv \sim_2 w$. Again due to Proposition~\ref{prop:UStrings}, we know that $M(w) = M(\gamma\inv\alpha\crs_1\alpha\inv\gamma) \in \mod* A$ corresponds to the module
	\begin{equation*}
		U(M(w)) \cong M(\wh w) \oplus M(g\wh w)
		\cong M(\wh \gamma\inv\wh\alpha(\wh\alpha')\inv\wh\gamma') \oplus M(\wh \gamma\inv\wh\alpha(\wh\alpha')\inv\wh\gamma') \in \mod*\wh A.
	\end{equation*}
	
	\subsubsection{Unfolding asymmetric band modules}
	We will now consider asymmetric bands and band modules. The each asymmetric band of $A$ yields a family of $A$-modules indexed by the set
	\begin{equation*}
		\Pi = \{p^n \in K[x] : n \in \spint \text{ and } p \text{ is monic, irreducible, and } p(0) \neq 0\}.
	\end{equation*}
	Henceforth, we let $\phi_p$ be the companion matrix of $p \in \Pi$. For example, the companion matrices associated to the the polynomials of degree at most 2 in $\Pi$ are
	\begin{align*}
		\phi_{x+1} &= \begin{pmatrix} 2 \end{pmatrix}, & 
		\phi_{x+2} &= \begin{pmatrix} 1 \end{pmatrix}, \\
		\phi_{(x+1)^2} &= \begin{pmatrix} 0 & 2 \\ 1 & 1\end{pmatrix}, & 
		\phi_{(x+2)^2} &= \begin{pmatrix} 0 & 2 \\ 1 & 2\end{pmatrix}, \\ 
		\phi_{x^2 + 1} &= \begin{pmatrix} 0 & 2 \\ 1 & 0\end{pmatrix}, & 
		\phi_{x^2 + x+2} &= \begin{pmatrix} 0 & 1 \\ 1 & 2\end{pmatrix}, & 
		\phi_{x^2 + 2x+2} &= \begin{pmatrix} 0 & 1 \\ 1 & 1\end{pmatrix}.
	\end{align*}
	
	Now consider the band $\wh w = \wh\alpha\wh\beta'\wh\zeta'(\wh\xi')\inv\wh\delta\inv\wh\gamma\inv$ of $\wh A$. This band is not $\integer_2$-invariant, since $g\wh w = \wh\alpha'\wh\beta\wh\zeta\wh\xi\inv(\wh\delta')\inv(\wh\gamma')\inv$, which is inequivalent to $\wh w$. Thus, $\wh w$ is of even parity by Lemma~\ref{lem:OddParity}. The corresponding folded band is
	\begin{equation*}
		w = \rho(\wh w) = \alpha\crs_1\beta\zeta\xi\inv\crs_4\inv\delta\inv\gamma\inv \approx \alpha\crs_1\inv\beta\zeta\xi\inv\crs_4\delta\inv\gamma\inv = \rho(g\wh w),
	\end{equation*}
	which has an even number of crease symbols and is thus also of even parity. The band $w$ is also asymmetric. Here, we note that the crease symbols of $\rho(\wh w)$ satisfy the following irreducible quadratic relations:
	\begin{equation*}
		\crs_1^2 - 2\stp_1 = 0  \qquad \text{and} \qquad \crs_4^{-2} - \crs_4^{-1} - \stp_4 =0
	\end{equation*} 
	Due to Proposition~\ref{prop:UEvenBands} and Remark~\ref{rem:LamConst}, we then know that for each $p \in \Pi$ and each $\phi \in \Aut(K^{\deg(p)})$, the module $M(w, \deg(p), \phi_p)$ corresponds to the $\wh A$-module
	\begin{equation*}
		U(M(w, \deg(p), \phi_p)) \cong M(\wh w, \deg(p), \phi_p) \oplus M(g\wh w, \deg(p), 2\phi_p).
	\end{equation*}
	Specifically, we have
	\begin{align*}
		U(M(w, 1, \phi_{x+1})) &\cong M(\wh w, 1, \phi_{x+1}) \oplus M(g\wh w, 1, \phi_{x+2}), \\
		U(M(w, 1, \phi_{x+2})) &\cong M(\wh w, 1, \phi_{x+2}) \oplus M(g\wh w, 1, \phi_{x+1}), \\
		U(M(w, 2, \phi_{(x+1)^2})) &\cong M(\wh w, 2, \phi_{(x+1)^2}) \oplus M(g\wh w, 2, \phi_{(x+2)^2}), \\
		U(M(w, 2, \phi_{(x+2)^2})) &\cong M(\wh w, 2, \phi_{(x+2)^2}) \oplus M(g\wh w, 2, \phi_{(x+1)^2}), \\
		U(M(w, 2, \phi_{x^2+1})) &\cong M(\wh w, 2, \phi_{x^2+1}) \oplus M(g\wh w, 2, \phi_{x^2+1}), \\
		U(M(w, 2, \phi_{x^2+x+2})) &\cong M(\wh w, 2, \phi_{x^2+x+2}) \oplus M(g\wh w, 2, \phi_{x^2+2x+2}), \\
		U(M(w, 2, \phi_{x^2+2x+2})) &\cong M(\wh w, 2, \phi_{x^2+2x+2}) \oplus M(g\wh w, 2, \phi_{x^2+x+2}).
	\end{align*}
	
	Now consider the band $\wh w = \wh \alpha \wh\beta' (\wh \gamma')\inv\wh\alpha'\wh\beta\wh\gamma\inv$. This band is $\integer_2$-invariant, since $g\wh w =\wh \alpha' \wh\beta \wh \gamma\wh\alpha\wh\beta'(\wh\gamma')\inv \approx \wh w$ by a rotation of the band. This also shows that $\wh w$ is of odd parity. Here we note that
	\begin{equation*}
		\fw(\wh w) = \alpha\crs_1\beta\gamma\inv\alpha\crs_1\inv\beta\gamma\inv 
		\qquad \text{and} \qquad
		w=\rho(\wh w) = \alpha\crs_1\beta\gamma\inv.
	\end{equation*}
	Thus, the folded word of $\wh w$ follows the folded band $w$ twice. We can also see that $w$ is of odd parity, which is an example of Lemma~\ref{lem:OddParity}. Due to Proposition~\ref{prop:UOddBands} and Remark~\ref{rem:LamConst}, we then know that for each $p \in \Pi$ and each $\phi \in \Aut(K^{\deg(p)})$, the module $M(w, \deg(p), \phi_p)$ corresponds to the $\wh A$-module
	\begin{equation*}
		U(M(w, \deg(p), \phi_p)) \cong M(\wh w, \deg(p), 2\phi_p^2).
	\end{equation*}
	Specifically, we have
	\begin{align*}
		U(M(w, 1, \phi_{x+1})) &\cong U(M(w, 1, \phi_{x+2})) \cong M(\wh w, 1, \phi_{x+1}), \\
		U(M(w, 2, \phi_{(x+1)^2})) &\cong U(M(w, 2, \phi_{(x+2)^2})) \cong M(\wh w, 2, \phi_{(x+1)^2}), \\
		U(M(w, 2, \phi_{x^2+1})) &\cong M(\wh w, 1, \phi_{x+2}) \oplus M(\wh w, 1, \phi_{x+2}), \\
		U(M(w, 2, \phi_{x^2+x+2})) &\cong U(M(w, 2, \phi_{x^2+2x+2})) \cong M(\wh w, 2, \phi_{x^2+1}).
	\end{align*}
	
	\subsubsection{Unfolding symmetric band modules}
	Now consider the band $\wh \beta\wh\zeta\wh\xi\inv\wh\xi'(\wh\zeta')\inv(\wh\beta')\inv$ of $\wh A$. This is a $\integer_2$-invariant band of even parity, but it is not in standard form (as the first symbol of of the folded band is not a crease symbol). We will instead consider the equivalent band $\wh w = \wh\xi'(\wh\zeta')\inv(\wh\beta')\inv\wh \beta\wh\zeta\wh\xi$, which is given by a rotation and is in standard form. The corresponding folded band is $w=\rho(\wh w)=\crs_4\xi\zeta\inv\beta\inv\crs_1\inv\beta\zeta\xi\inv$, which is symmetric. To consider the symmetric band modules of $w$ and how they unfold, we will need to consider the group action on $\Pi$ with respect to the band $w$. First, note that the crease symbols of $\rho(\wh w)$ satisfy the following irreducible quadratic relations:
	\begin{equation*}
		\crs_1^{-2} - 2\stp_1 = 0  \qquad \text{and} \qquad \crs_4^2 -2\crs_4 -\stp_4=0.
	\end{equation*} 
	Then note that the constant $\mu_{\wh w}$ from Lemma~\ref{lem:SymBandConst} is 1, as there are no crease symbols in the subword $\xi\zeta\inv\beta\inv$. Thus, for any polynomial $p \in \Pi$, the polynomial $gp \in \Pi$ is the characteristic polynomial of the matrix $2\phi\inv_p$. Thus
	\begin{align*}
		g(x+1) &= x+2, & 
		g(x+2) &= x+1, \\
		g(x+1)^2 &= (x+2)^2, & 
		g(x+2)^2 &= (x+1)^2, \\ 
		g(x^2 + x+2) &= x^2 + x+2, & 
		g(x^2 + 2x+2) &= x^2 + 2x+2, &
		g(x^2 + 1) &= x^2 + 1.
	\end{align*}
	So $[x+1],[(x+1)^2],[x^2 + 1],[x^2 + x+2],[x^2 + 2x+2] \in \Pi_w$. This indicates that there are precisely 5 indecomposable modules (up to isomorphism) of the form $M(w,m_{\psi_p},\psi_p)$ with $\deg(p) \leq 2$. One can likewise show that there are precisely 4 indecomposable modules (up to isomorphism) of the form $M(w,m_{\psi_p},\psi_p)$ with $\deg(p) = 3$ and precisely 10 indecomposable modules (up to isomorphism) with $\deg(p) = 4$, and so on. The value of $m_{\psi_p}$ is $2\deg(p)$ whenever $gp \neq p$. In general, when $gp = p$, it is less clear a priori when $m_{\psi_p} = \deg(p)$ or when $m_{\psi_p} =2\deg(p)$, other than to say that $m_{\psi_p} = \deg(p)$ if and only if $U(M(w,m_{\psi_p},\psi_p))$ is indecomposable. In the case of this specific example, we may compute representatives for these polynomials explicitly. Example representatives are as follows.
	\begin{align*}
		M_{x+1} &= M\!\left(w, 2, \left(\begin{smallmatrix} 0 & 1 \\ 2 & 0 \end{smallmatrix}\right)\right), \\
		M_{(x+1)^2} &= M\!\left(w, 4, \left(\begin{smallmatrix} 0 & 0 & 1 & 2 \\ 0 & 0 & 0 & 1 \\ 2 & 2 & 0 & 0 \\ 0 & 2 & 0 & 0 \end{smallmatrix}\right)\right), \\
		M_{x^2+x+2} &= M\!\left(w, 2, \left(\begin{smallmatrix} 1 & 2 \\ 2 & 2 \end{smallmatrix}\right)\right), \\
		M_{x^2+2x+2} &= M\!\left(w, 2, \left(\begin{smallmatrix} 2 & 1 \\ 1 & 1 \end{smallmatrix}\right)\right), \\
		M_{x^2+1} &= M\!\left(w, 4, \left(\begin{smallmatrix} 0 & 0 & 0 & 1 \\ 0 & 0 & 2 & 0 \\ 0 & 1 & 0 & 0 \\ 2 & 0 & 0 & 0 \end{smallmatrix}\right)\right).
	\end{align*}
	These representatives can be verified via an exhaustive search. As it happens, one may compute each representative via a product
	\begin{equation*}
		\psi_p = \begin{pmatrix} \id_{\frac{1}{2}m_{\psi_p}} & X \\ 0 & Y \end{pmatrix}\inv H'\begin{pmatrix} \id_{\frac{1}{2}m_{\psi_p}} & X \\ 0 & Y \end{pmatrix}
	\end{equation*}
	for some block matrices $X$ and $Y$, where $H' = \left(\begin{smallmatrix} 0 & 2 \\ 1 & 0 \end{smallmatrix}\right) \otimes_K \id_{\frac{1}{2}m_{\psi_p}}$. In particular,
	\begin{align*}
		\begin{pmatrix} 0 & 1 \\ 2 & 0 \end{pmatrix} &= \begin{pmatrix} 1 & 0 \\ 0 & 2 \end{pmatrix}\inv H' \begin{pmatrix} 1 & 0 \\ 0 & 2 \end{pmatrix} \\ 
		\begin{pmatrix} 0 & 0 & 1 & 2 \\ 0 & 0 & 0 & 1 \\ 2 & 2 & 0 & 0 \\ 0 & 2 & 0 & 0 \end{pmatrix} &= \begin{pmatrix} 1 & 0 & 0 & 0 \\ 0 & 1 & 0 & 0 \\ 0 & 0 & 2 & 1 \\ 0 & 0 & 0 & 2 \end{pmatrix}\inv H' \begin{pmatrix} 1 & 0 & 0 & 0 \\ 0 & 1 & 0 & 0 \\ 0 & 0 & 2 & 1 \\ 0 & 0 & 0 & 2 \end{pmatrix} \\ 
		\begin{pmatrix} 1 & 2 \\ 2 & 2 \end{pmatrix} &= \begin{pmatrix} 1 & 1 \\ 0 & 2 \end{pmatrix}\inv H' \begin{pmatrix} 1 & 1 \\ 0 & 2 \end{pmatrix} \\ 
		\begin{pmatrix} 2 & 1 \\ 1 & 1 \end{pmatrix} &= \begin{pmatrix} 1 & 1 \\ 0 & 1 \end{pmatrix}\inv H' \begin{pmatrix} 1 & 1 \\ 0 & 1 \end{pmatrix} \\ 
		\begin{pmatrix} 0 & 0 & 0 & 1 \\ 0 & 0 & 2 & 0 \\ 0 & 1 & 0 & 0 \\ 2 & 0 & 0 & 0 \end{pmatrix} &= \begin{pmatrix} 1 & 0 & 0 & 0 \\ 0 & 1 & 0 & 0 \\ 0 & 0 & 0 & 2 \\ 0 & 0 & 1 & 0 \end{pmatrix}\inv H' \begin{pmatrix} 1 & 0 & 0 & 0 \\ 0 & 1 & 0 & 0 \\ 0 & 0 & 0 & 2 \\ 0 & 0 & 1 & 0 \end{pmatrix}.
	\end{align*}
	An astute reader may see the rationale behind the ansatze for the block diagonal cases. The non-block diagonal cases are somewhat of a mystery to the author (particularly for irreducible polynomials of higher degree).
	
	Each of these modules unfolds in accordance to Proposition~\ref{prop:USymBand}. In particular,
	\begin{align*}
		U(M_{x+1}) &\cong M(\wh w, 1, \phi_{x+1}) \oplus M(\wh w, 1, \phi_{x+2}), \\
		U(M_{(x+1)^2}) &\cong M(\wh w, 2, \phi_{(x+1)^2}) \oplus M(\wh w, 2, \phi_{(x+2)^2}), \\
		U(M_{x^2+x+2}) &\cong M(\wh w, 2, \phi_{x^2+x+2}), \\
		U(M_{x^2+2x+2}) &\cong M(\wh w, 2, \phi_{x^2+2x+2}), \\
		U(M_{x^2+1}) &\cong M(\wh w, 2, \phi_{x^2+1}) \oplus M(\wh w, 2, \phi_{x^2+1}).
	\end{align*}
	
	\subsection{The Auslander-Reiten quiver}
	Since the algebra is representation-infinite, we will only give an example component. The component in Figure~\ref{fig:FoldedARQStrings} contains the indecomposable injective modules $D(A\stp_1)$, $D(A\stp_2)$, $D(A\stp_3)$ and $D(A\stp_4)$, and the indecomposable projective module $\stp_2 A$. A dual component contains the modules $\stp_1 A$, $\stp_3 A$, $\stp_4 A$, $\stp_5 A$ and $D(A\stp_5)$. The various irreducible morphisms between string modules are clearly given by adding/deleting hooks/cohooks, as per Theorem~\ref{thm:ARSeq}. The Auslander-Reiten components of band modules are given by tubes, as per Theorem~\ref{thm:ARSeq}.
	
	\begin{figure}
		\rotatebox{270}{
			\begin{tikzpicture}[scale = 1.2]
				\coordinate (tau) at (-3, 0);
				\coordinate (rhook) at (1.5, 1.2);
				\coordinate (lhook) at (1.5, -1.2);
				\coordinate (col) at (0, 2.4);
			
				\coordinate (o) at (0,0); 
				\draw ($(o) + -0.5*(col)$) node {\footnotesize $\zeta^{-1} \beta^{-1} \eta_1^{-1} \alpha^{-1} \gamma \delta \eta_4 \xi$}; 
				
				\draw ($(o) + 1*(lhook) -2*(col)$) node {\footnotesize $\zeta^{-1} \beta^{-1} \eta_1^{-1} \alpha^{-1} \gamma \delta \eta_4 \xi \zeta^{-1} \delta \eta_4 \xi$}; 
				\draw ($(o) + -1*(rhook) + 1*(lhook) -2*(col)$) node {\footnotesize $\delta \eta_4 \xi \zeta^{-1} \delta \eta_4 \xi$}; 
				\draw ($(o) + -1*(rhook) + -2*(col)$) node {\footnotesize $\delta \eta_4 \xi$}; 
				\draw ($(o) + -2*(rhook)+ -2*(col)$) node {\footnotesize $\gamma \delta \eta_4 \xi \zeta^{-1} \delta \eta_4 \xi$}; 
				\draw ($(o) + -2*(rhook)+ -1*(lhook)+ -2*(col)$) node {\footnotesize $\gamma \delta \eta_4 \xi$}; 
				\draw ($(o) + -3*(rhook)+ -2*(lhook)+ -2*(col)$) node {\footnotesize $\alpha \eta_1$}; 
				\draw ($(o) + -4*(rhook)+ -2*(lhook)+ -2*(col)$) node {\footnotesize $\alpha \eta_1 \beta \gamma^{-1} \alpha \eta_1$}; 
				\draw ($(o) + -3*(rhook)+ -1*(lhook)+ -2*(col)$) node {\footnotesize $\alpha \eta_1 \alpha^{-1} \gamma \delta \eta_4 \xi$}; 
				\draw ($(o) + -3*(rhook)+ -3*(lhook)+ -2*(col)$) node {\footnotesize $\alpha \eta_1 \beta \gamma^{-1}$}; 
				\draw ($(o) + -2*(rhook)+ -3*(lhook)+ -2*(col)$) node {\footnotesize $\varepsilon_2$}; 
				\draw ($(o) + -3*(rhook)+ -4*(lhook)+ -2*(col)$) node {\footnotesize $\alpha \eta_1 \beta \gamma^{-1} \alpha \eta_1^{-1} \alpha^{-1}$}; 
				\draw ($(o) + -2*(rhook)+ -4*(lhook)+ -2*(col)$) node {\footnotesize $\alpha \eta_1^{-1} \alpha^{-1}$}; 
				\draw ($(o) + -3*(rhook)+ -5*(lhook)+ -2*(col)$) node {\footnotesize $\alpha \eta_1 \beta \gamma^{-1} \alpha \eta_1^{-1} \alpha^{-1} \gamma \beta^{-1} \eta_1^{-1} \alpha^{-1}$}; 
				
				\draw ($(o) + 1*(rhook)$) node {\footnotesize $\xi^{-1} \eta_4^{-1} \xi \zeta^{-1} \beta^{-1} \eta_1^{-1} \alpha^{-1} \gamma \delta \eta_4 \xi$}; 
				\draw ($(o) + 1*(rhook) + -1*(lhook)$) node {\footnotesize $\xi^{-1} \eta_4^{-1} \xi \zeta^{-1} \beta^{-1} \eta_1^{-1}$}; 
				\draw ($(o) + -1*(lhook)$) node {\footnotesize $\zeta^{-1} \beta^{-1} \eta_1^{-1}$}; 
				\draw ($(o) + -2*(lhook)$) node {\footnotesize $\zeta^{-1} \beta^{-1} \eta_1^{-1} \beta \gamma^{-1}$}; 
				\draw ($(o) + -2*(lhook) - 1*(rhook)$) node {\footnotesize $\gamma^{-1}$};
				\draw ($(o) + -3*(lhook) - 1*(rhook)$) node {\footnotesize $\gamma^{-1} \alpha \eta_1^{-1} \alpha^{-1}$}; 
				\draw ($(o) + -2*(lhook) - 2*(rhook)$) node {\footnotesize $\gamma \delta \eta_4 \delta^{-1} \gamma^{-1}$}; 
				\draw ($(o) + -3*(lhook) - 3*(rhook)$) node {\footnotesize $\alpha \eta_1 \alpha^{-1} \gamma \delta \eta_4 \delta^{-1} \gamma^{-1} \alpha \eta_1^{-1} \alpha^{-1}$}; 
				\draw ($(o) + -3*(lhook) - 2*(rhook)$) node {\footnotesize $\gamma \delta \eta_4 \delta^{-1} \gamma^{-1} \alpha \eta_1^{-1} \alpha^{-1}$};
				
				\draw[->, shorten <= 3ex, shorten >= 3ex] ($(o) - 0.5*(col)$) -- ($(o) +(rhook)$);
				\draw[->, shorten <= 3ex, shorten >= 3ex] ($(o) +(rhook)$) -- ($(o) +2*(rhook)$);
				\foreach \i in {1,2,3} {
					\draw[->, shorten <= 3ex, shorten >= 3ex] ($(o) +(rhook) + \i*(tau)$) -- ($(o) +2*(rhook) + \i*(tau)$);
					\draw[->, shorten <= 3ex, shorten >= 3ex] ($(o) +2*(rhook) + \i*(tau)$) -- ($(o) +3*(rhook) + \i*(tau)$);
					\draw[->, shorten <= 3ex, shorten >= 3ex] ($(o) + -2*(rhook) + 1*(lhook) + \i*(tau) - 2*(col)$) -- ($(o) + -1*(rhook) + 1*(lhook) + \i*(tau) - 2*(col)$);
					\draw[->, shorten <= 3ex, shorten >= 3ex] ($(o) + -1*(rhook) + 1*(lhook) + \i*(tau) - 2*(col)$) -- ($(o) + 0*(rhook) + 1*(lhook) + \i*(tau) - 2*(col)$);
				}
				\draw[->, shorten <= 3ex, shorten >= 3ex] ($(o) + -2*(rhook) + 1*(lhook) - 2*(col)$) -- ($(o) + -1*(rhook) + 1*(lhook) - 2*(col)$);
				\draw[->, shorten <= 3ex, shorten >= 3ex] ($(o) + -1*(rhook) + 1*(lhook) - 2*(col)$) -- ($(o) + 0*(rhook) + 1*(lhook) - 2*(col)$);
				\draw[->, shorten <= 3ex, shorten >= 3ex] ($(o) +(lhook) - 2*(col)$) -- ($(o) - (tau) - 0.5*(col)$);
				\draw[->, shorten <= 3ex, shorten >= 3ex] ($(o) +(lhook) - 2*(col) +(tau)$) -- ($(o) - 0.5*(col)$);
				\foreach \i in {2,3} {
					\draw[->, shorten <= 3ex, shorten >= 3ex] ($(o) + \i*(tau)$) -- ($(o) +1*(rhook) + \i*(tau)$);
				}
				\draw[->, shorten <= 3ex, shorten >= 3ex] ($(o) + 3*(tau) + -2*(col)$) -- ($(o) +1*(rhook) + 3*(tau) + -2*(col)$);
				\draw[->, shorten <= 3ex, shorten >= 3ex] ($(o) + -1*(rhook) + 3*(tau) + -2*(col)$) -- ($(o) +0*(rhook) + 3*(tau) + -2*(col)$);
				\draw[->, shorten <= 3ex, shorten >= 3ex] ($(o) + 4*(tau) + -2*(col)$) -- ($(o) +1*(rhook) + 4*(tau) + -2*(col)$);
				\draw[->, shorten <= 3ex, shorten >= 3ex] ($(o) + 1*(rhook) + 4*(tau) + -2*(col)$) -- ($(o) +2*(rhook) + 4*(tau) + -2*(col)$);
				\draw[->, shorten <= 3ex, shorten >= 3ex] ($(o) + 1*(rhook) + 5*(tau) + -2*(col)$) -- ($(o) +2*(rhook) + 5*(tau) + -2*(col)$);
				
				\draw[->, shorten <= 3ex, shorten >= 3ex]  ($(o) +(rhook)$) -- ($(o) - (tau) - 0.5*(col)$);
				\draw[->, shorten <= 3ex, shorten >= 3ex]  ($(o) +(rhook) + (tau)$) -- ($(o) - 0.5*(col)$);
				\foreach \i in {0,1,2} {
					\draw[->, shorten <= 3ex, shorten >= 3ex]  ($(o) +(rhook) + -2*(lhook) + \i*(tau)$) -- ($(o) +(rhook) + -1*(lhook) + \i*(tau)$);
					\draw[->, shorten <= 3ex, shorten >= 3ex]  ($(o) +(rhook) + -1*(lhook) + \i*(tau)$) -- ($(o) +(rhook) + 0*(lhook) + \i*(tau)$);
					\draw[->, shorten <= 3ex, shorten >= 3ex]  ($(o) + -1*(rhook) + 0*(lhook) + \i*(tau) - 2*(col)$) -- ($(o) + -1*(rhook) + 1*(lhook)  + \i*(tau) - 2*(col)$);
					\draw[->, shorten <= 3ex, shorten >= 3ex]  ($(o) + -1*(rhook) + 1*(lhook) + \i*(tau) - 2*(col)$) -- ($(o) + -1*(rhook) + 2*(lhook)  + \i*(tau) - 2*(col)$);
				}
				\draw[->, shorten <= 3ex, shorten >= 3ex]  ($(o) + -2*(lhook) + 2*(tau)$) -- ($(o) + -1*(lhook) + 2*(tau)$);
				\draw[->, shorten <= 3ex, shorten >= 3ex]  ($(o) + -1*(lhook) + 2*(tau)$) -- ($(o) + 0*(lhook) + 2*(tau)$);
				\draw[->, shorten <= 3ex, shorten >= 3ex]  ($(o) + -1*(lhook) + 3*(tau)$) -- ($(o) + 0*(lhook) + 3*(tau)$);
				\draw[->, shorten <= 3ex, shorten >= 3ex] ($(o) - 0.5*(col)$) -- ($(o) +(lhook) - 2*(col)$);
				\draw[->, shorten <= 3ex, shorten >= 3ex] ($(o) +(lhook) - 2*(col)$) -- ($(o) +2*(lhook) - 2*(col)$);
				\draw[->, shorten <= 3ex, shorten >= 3ex] ($(o) + 3*(tau) - 1*(col)$) -- ($(o) +(lhook) + 3*(tau) - 1*(col)$);
				\draw[->, shorten <= 3ex, shorten >= 3ex] ($(o) + 4*(tau) - 1*(col)$) -- ($(o) +(lhook) + 4*(tau) - 1*(col)$);
				\draw[->, shorten <= 3ex, shorten >= 3ex] ($(o) +1*(lhook)  + 4*(tau) - 1*(col)$) -- ($(o) +2*(lhook) + 4*(tau) - 1*(col)$);
				\draw[->, shorten <= 3ex, shorten >= 3ex] ($(o) +2*(lhook)  + 4*(tau) - 1*(col)$) -- ($(o) +3*(lhook) + 4*(tau) - 1*(col)$);
				\draw[->, shorten <= 3ex, shorten >= 3ex] ($(o) +-1*(lhook)  + 3*(tau) - 3*(col)$) -- ($(o) +0*(lhook) + 3*(tau) - 3*(col)$);
				\draw[->, shorten <= 3ex, shorten >= 3ex] ($(o) +0*(lhook)  + 3*(tau) - 3*(col)$) -- ($(o) +1*(lhook) + 3*(tau) - 3*(col)$);
				
				\draw ($(o) + 2*(tau) + -0.5*(lhook) + (0.1, 0.1)$) node [rotate=-35] {\tiny$(2,1)$};
				\draw ($(o) + 2*(tau) + 0.5*(rhook) + (-0.1, 0.1)$) node [rotate=35] {\tiny$(1,2)$};
				\draw ($(o) + 3*(tau) + -0.5*(lhook) + (0.1, 0.1)$) node [rotate=-35] {\tiny$(2,1)$};
				\draw ($(o) + 3*(tau) + 0.5*(rhook) + (-0.1, 0.1)$) node [rotate=35] {\tiny$(1,2)$};
				\draw ($(o) - (col) + 3*(tau) + 0.5*(lhook) + (0.1, 0.1)$) node [rotate=-35] {\tiny$(1,2)$};
				\draw ($(o) - (col) + 3*(tau) + -0.5*(rhook) + (-0.1, 0.1)$) node [rotate=35] {\tiny$(2,1)$};
				\draw ($(o) - (col) + 4*(tau) + 0.5*(lhook) + (0.1, 0.1)$) node [rotate=-35] {\tiny$(1,2)$};
				\draw ($(o) - (col) + 4*(tau) + -0.5*(rhook) + (-0.1, 0.1)$) node [rotate=35] {\tiny$(2,1)$};
			\end{tikzpicture}
		}
		\caption{An example component of the Auslander-Reiten quiver of $\mod*A$ containing string modules. Specific modules are given by the the strings in the figure.} \label{fig:FoldedARQStrings}
	\end{figure}
	
	\subsection{The repetitive algebra and derived equivalences}
	The maximal paths of $A$ are $p=\alpha\crs_1\beta\zeta$ and $q=\gamma\delta\crs_4\xi$. Thus, the quiver of the repetitive algebra $A_R$ is given as follows.
	\begin{equation*}
		\begin{tikzpicture}
			\draw (-5.9,-0.7) node {\footnotesize $\cdots$};
			\draw[->] (-5.5,-0.6) -- (-4.4,-0.6);
			\draw[->] (-5.5,-0.8) -- (-4.4,-0.8);
			\draw (-4.9,-0.3) node {\footnotesize $\chi^{(2)}_p$};
			\draw (-4.9,-1.1) node {\footnotesize $\chi^{(2)}_q$};
			
			\draw (-4,0.7) node {$1^{(1)}$};
			\draw (-4,-0.7) node {$2^{(1)}$};
			\draw (-2.8,0) node {$3^{(1)}$};
			\draw (-1.6,0.7) node {$4^{(1)}$};
			\draw (-1.6,-0.7) node {$5^{(1)}$};
			\draw [->](-3.7,0.5) -- (-3.1,0.1);
			\draw [->](-2.5,0.1) -- (-1.9,0.5);
			\draw [->](-3.7,-0.5) -- (-3.1,-0.2);
			\draw [->](-2.5,-0.2) -- (-1.9,-0.5);
			\draw [->](-4,-0.4) -- (-4,0.4);
			\draw [->](-1.6,0.4) -- (-1.6,-0.4);
			\draw [->](-4.1,1) .. controls (-4.2,1.2) and (-4.2,1.4) .. (-4,1.4) .. controls (-3.8,1.4) and (-3.8,1.2) .. (-3.9,1);
			\draw [->](-1.7,1) .. controls (-1.8,1.2) and (-1.8,1.4) .. (-1.6,1.4) .. controls (-1.4,1.4) and (-1.4,1.2) .. (-1.5,1);
			\draw (-6.4,0.2) node {\footnotesize $Q_R\colon$};
			\draw (-3.5,1.2) node {\footnotesize $\eta_1^{(1)}$};
			\draw (-2.1,1.2) node {\footnotesize $\eta_4^{(1)}$};
			\draw (-4.3,0.1) node {\footnotesize $\alpha^{(1)}$};
			\draw (-3.2,0.5) node {\footnotesize $\beta^{(1)}$};
			\draw (-3.3,-0.6) node {\footnotesize $\gamma^{(1)}$};
			\draw (-2.4,0.5) node {\footnotesize $\delta^{(1)}$};
			\draw (-2.3,-0.6) node {\footnotesize $\zeta^{(1)}$};
			\draw (-1.3,0.1) node {\footnotesize $\xi^{(1)}$};
			\draw[->] (-1.2,-0.6) -- (-0.1,-0.6);
			\draw[->] (-1.2,-0.8) -- (-0.1,-0.8);
			\draw (-0.7,-0.3) node {\footnotesize $\chi^{(1)}_p$};
			\draw (-0.7,-1.1) node {\footnotesize $\chi^{(1)}_q$};
			
			\draw (0.3,0.7) node {$1^{(0)}$};
			\draw (0.3,-0.7) node {$2^{(0)}$};
			\draw (1.5,0) node {$3^{(0)}$};
			\draw (2.7,0.7) node {$4^{(0)}$};
			\draw (2.7,-0.7) node {$5^{(0)}$};
			\draw [->](0.6,0.5) -- (1.2,0.1);
			\draw [->](1.8,0.1) -- (2.4,0.5);
			\draw [->](0.6,-0.5) -- (1.2,-0.2);
			\draw [->](1.8,-0.2) -- (2.4,-0.5);
			\draw [->](0.3,-0.4) -- (0.3,0.4);
			\draw [->](2.7,0.4) -- (2.7,-0.4);
			\draw [->](0.2,1) .. controls (0.1,1.2) and (0.1,1.4) .. (0.3,1.4) .. controls (0.5,1.4) and (0.5,1.2) .. (0.4,1);
			\draw [->](2.6,1) .. controls (2.5,1.2) and (2.5,1.4) .. (2.7,1.4) .. controls (2.9,1.4) and (2.9,1.2) .. (2.8,1);
			\draw (0.8,1.2) node {\footnotesize $\eta_1^{(0)}$};
			\draw (2.2,1.2) node {\footnotesize $\eta_4^{(0)}$};
			\draw (0,0.1) node {\footnotesize $\alpha^{(0)}$};
			\draw (1.1,0.5) node {\footnotesize $\beta^{(0)}$};
			\draw (1,-0.6) node {\footnotesize $\gamma^{(0)}$};
			\draw (1.9,0.5) node {\footnotesize $\delta^{(0)}$};
			\draw (2,-0.6) node {\footnotesize $\zeta^{(0)}$};
			\draw (3,0.1) node {\footnotesize $\xi^{(0)}$};
			\draw[->] (3.1,-0.6) -- (4.2,-0.6);
			\draw[->] (3.1,-0.8) -- (4.2,-0.8);
			\draw (3.7,-0.3) node {\footnotesize $\chi^{(0)}_p$};
			\draw (3.7,-1.1) node {\footnotesize $\chi^{(0)}_q$};
			\draw (4.6,-0.7) node {\footnotesize $\cdots$};
		\end{tikzpicture}
	\end{equation*}
	The relations are given by
	\begin{align*}
		Z_R = \{&\alpha\ps{i}\beta\ps{i},\beta\ps{i}\delta\ps{i},\gamma\ps{i}\zeta\ps{i},\delta\ps{i}\xi\ps{i}, (\crs_1\ps{i})^2 + \stp_{1}\ps{i}, (\crs_4\ps{i})^2 +\crs_4\ps{i} + 2\stp_{4}\ps{i}, \\
		&\crs_1\ps{i}\beta\ps{i}\zeta\ps{i}\chi_p\ps{i}\alpha\ps{i-1} - \beta\ps{i}\zeta\ps{i}\chi_p\ps{i}\alpha\ps{i-1}\crs_1\ps{i-1}, 
		\alpha\ps{i}\crs_1\ps{i}\beta\ps{i}\zeta\ps{i}\chi_p\ps{i} - \gamma\ps{i}\delta\ps{i}\crs_4\ps{i}\xi\ps{i}\chi_q\ps{i}, \\
		&\zeta\ps{i}\chi_p\ps{i}\alpha\ps{i-1}\crs_1\ps{i-1}\beta\ps{i-1} - \delta\ps{i}\crs_4\ps{i}\xi\ps{i}\chi_q\ps{i}\gamma\ps{i-1},
		\crs_4\ps{i}\xi\ps{i}\chi_q\ps{i}\gamma\ps{i-1}\delta\ps{i-1} - \xi\ps{i}\chi_q\ps{i}\gamma\ps{i-1}\delta\ps{i-1}\crs_4\ps{i-1}, \\
		&\chi_p\ps{i}\alpha\ps{i-1}\crs_1\ps{i-1}\beta\ps{i-1}\zeta\ps{i-1} - \chi_q\ps{i}\gamma\ps{i-1}\delta\ps{i-1}\crs_4\ps{i-1}\xi\ps{i-1}, \\
		&\zeta\ps{i}\chi_q\ps{i}, \xi\ps{i}\chi_p\ps{i}, \chi_p\ps{i}\gamma\ps{i-1}, \chi_q\ps{i}\alpha\ps{i-1} : i \in \integer\}.
	\end{align*}
	
	The maximal paths of the unfolded gentle algebra are $\wh p = \wh\alpha\wh\beta'\wh\zeta'$, $\wh p' = g\wh p = \wh\alpha'\wh\beta\wh\zeta$, $\wh q = \wh\gamma\wh\delta\wh\xi'$, and $\wh q' = g\wh q = \wh\gamma'\wh\delta'\wh\xi$. Thus, the  quiver of the unfolded repetitive algebra $\wh A_R$ is given as follows.
	\begin{equation*}
		\begin{tikzpicture}[scale = 1.2]
			\draw (-6.6,0) node {\footnotesize $\widehat Q_R\colon$};
			\draw (-6.1,-1.4) node {\footnotesize $\cdots$};
			\draw[->] (-5.7,-1.3) -- (-4.6,-1.3);
			\draw[->] (-5.7,-1.5) -- (-4.6,-1.5);
			\draw (-5.1,-1) node {\footnotesize $\widehat\chi^{(2)'}_p$};
			\draw (-5.1,-1.8) node {\footnotesize $\widehat\chi^{(2)'}_q$};
			
			\draw (-4.2,0) node {$\widehat 1^{(1)}$};
			\draw (-4.2,-1.4) node {$\widehat 2^{(1)'}$};
			\draw (-2.9,-0.7) node {$\widehat 3^{(1)'}$};
			\draw (-1.6,0) node {$\widehat 4^{(1)}$};
			\draw (-1.6,-1.4) node {$\widehat 5^{(1)'}$};
			\draw [->](-3.9,-0.2) -- (-3.3,-0.6);
			\draw [->](-2.5,-0.6) -- (-1.9,-0.2);
			\draw [->](-3.9,-1.2) -- (-3.3,-0.9);
			\draw [->](-2.5,-0.9) -- (-1.9,-1.2);
			\draw [->](-4.2,-1.1) -- (-4.2,-0.3);
			\draw [->](-1.6,-0.3) -- (-1.6,-1.1);
			
			\draw (-4.5,-0.6) node {\footnotesize $\widehat\alpha^{(1)'}$};
			\draw (-3.4,-0.2) node {\footnotesize $\widehat\beta^{(1)'}$};
			\draw (-3.5,-1.3) node {\footnotesize $\widehat\gamma^{(1)'}$};
			\draw (-2.4,-0.2) node {\footnotesize $\widehat\delta^{(1)'}$};
			\draw (-2.3,-1.3) node {\footnotesize $\widehat\zeta^{(1)'}$};
			\draw (-1.3,-0.6) node {\footnotesize $\widehat\xi^{(1)'}$};
			\draw[->] (-1.2,-1.3) -- (-0.1,-1.3);
			\draw[->] (-1.2,-1.5) -- (-0.1,-1.5);
			\draw (-0.7,-1) node {\footnotesize $\widehat\chi^{(1)'}_p$};
			\draw (-0.7,-1.8) node {\footnotesize $\widehat\chi^{(1)'}_q$};
			
			\draw (0.3,0) node {$\widehat 1^{(0)}$};
			\draw (0.3,-1.4) node {$\widehat 2^{(0)'}$};
			\draw (1.6,-0.7) node {$\widehat 3^{(0)'}$};
			\draw (2.9,0) node {$\widehat 4^{(0)}$};
			\draw (2.9,-1.4) node {$\widehat 5^{(0)'}$};
			\draw [->](0.6,-0.2) -- (1.2,-0.6);
			\draw [->](2,-0.6) -- (2.6,-0.2);
			\draw [->](0.6,-1.2) -- (1.2,-0.9);
			\draw [->](2,-0.9) -- (2.6,-1.2);
			\draw [->](0.3,-1.1) -- (0.3,-0.3);
			\draw [->](2.9,-0.3) -- (2.9,-1.1);
			\draw (0,-0.6) node {\footnotesize $\widehat\alpha^{(0)'}$};
			\draw (1.1,-0.2) node {\footnotesize $\widehat\beta^{(0)'}$};
			\draw (1,-1.3) node {\footnotesize $\widehat\gamma^{(0)'}$};
			\draw (2.1,-0.2) node {\footnotesize $\widehat\delta^{(0)'}$};
			\draw (2.2,-1.3) node {\footnotesize $\widehat\zeta^{(0)'}$};
			\draw (3.2,-0.6) node {\footnotesize $\widehat\xi^{(0)'}$};
			\draw[->] (3.3,-1.3) -- (4.4,-1.3);
			\draw[->] (3.3,-1.5) -- (4.4,-1.5);
			\draw (3.9,-1) node {\footnotesize $\widehat\chi^{(0)'}_p$};
			\draw (3.9,-1.8) node {\footnotesize $\widehat\chi^{(0)'}_q$};
			\draw (4.8,-1.4) node {\footnotesize $\cdots$};
			
			\draw (-6.1,1.4) node {\footnotesize $\cdots$};
			\draw[->] (-5.7,1.3) -- (-4.6,1.3);
			\draw[->] (-5.7,1.5) -- (-4.6,1.5);
			\draw (-5.1,1) node {\footnotesize $\widehat\chi^{(2)}_p$};
			\draw (-5.1,1.8) node {\footnotesize $\widehat\chi^{(2)}_q$};
			
			\draw (-4.2,1.4) node {$\widehat 2^{(1)}$};
			\draw (-2.9,0.7) node {$\widehat 3^{(1)}$};
			\draw (-1.6,1.4) node {$\widehat 5^{(1)}$};
			\draw [->](-3.9,0.2) -- (-3.3,0.6);
			\draw [->](-2.5,0.6) -- (-1.9,0.2);
			\draw [->](-3.9,1.2) -- (-3.3,0.9);
			\draw [->](-2.5,0.9) -- (-1.9,1.2);
			\draw [->](-4.2,1.1) -- (-4.2,0.3);
			\draw [->](-1.6,0.3) -- (-1.6,1.1);
			\draw (-4.5,0.6) node {\footnotesize $\widehat\alpha^{(1)}$};
			\draw (-3.4,0.2) node {\footnotesize $\widehat\beta^{(1)}$};
			\draw (-3.5,1.3) node {\footnotesize $\widehat\gamma^{(1)}$};
			\draw (-2.4,0.2) node {\footnotesize $\widehat\delta^{(1)}$};
			\draw (-2.3,1.3) node {\footnotesize $\widehat\zeta^{(1)}$};
			\draw (-1.3,0.6) node {\footnotesize $\widehat\xi^{(1)}$};
			\draw[->] (-1.2,1.3) -- (-0.1,1.3);
			\draw[->] (-1.2,1.5) -- (-0.1,1.5);
			\draw (-0.7,1) node {\footnotesize $\widehat\chi^{(1)}_p$};
			\draw (-0.7,1.8) node {\footnotesize $\widehat\chi^{(1)}_q$};
			
			\draw (0.3,1.4) node {$\widehat 2^{(0)}$};
			\draw (1.6,0.7) node {$\widehat 3^{(0)}$};
			\draw (2.9,1.4) node {$\widehat 5^{(0)}$};
			\draw [->](0.6,0.2) -- (1.2,0.6);
			\draw [->](2,0.6) -- (2.6,0.2);
			\draw [->](0.6,1.2) -- (1.2,0.9);
			\draw [->](2,0.9) -- (2.6,1.2);
			\draw [->](0.3,1.1) -- (0.3,0.3);
			\draw [->](2.9,0.3) -- (2.9,1.1);
			\draw (0,0.6) node {\footnotesize $\widehat\alpha^{(0)}$};
			\draw (1.1,0.2) node {\footnotesize $\widehat\beta^{(0)}$};
			\draw (1,1.3) node {\footnotesize $\widehat\gamma^{(0)}$};
			\draw (2.1,0.2) node {\footnotesize $\widehat\delta^{(0)}$};
			\draw (2.2,1.3) node {\footnotesize $\widehat\zeta^{(0)}$};
			\draw (3.2,0.6) node {\footnotesize $\widehat\xi^{(0)}$};
			\draw[->] (3.3,1.3) -- (4.4,1.3);
			\draw[->] (3.3,1.5) -- (4.4,1.5);
			\draw (3.9,1) node {\footnotesize $\widehat\chi^{(0)}_p$};
			\draw (3.9,1.8) node {\footnotesize $\widehat\chi^{(0)}_q$};
			\draw (4.8,1.4) node {\footnotesize $\cdots$};
		\end{tikzpicture}
	\end{equation*}
	The relations are given by
	\begin{align*}
		\wh Z_R = \{
		&\alpha\ps{i}\beta\ps{i},
		\beta\ps{i}\delta\ps{i},\gamma\ps{i}\zeta\ps{i},
		\delta\ps{i}\xi\ps{i},  \\
		&\alpha\psp{i}\beta\psp{i},
		\beta\psp{i}\delta\psp{i},\gamma\psp{i}\zeta\psp{i},
		\delta\psp{i}\xi\psp{i}, \\
		&\beta\ps{i}\zeta\ps{i}\chi_p\ps{i}\alpha\ps{i-1} - \beta\psp{i}\zeta\psp{i}\chi_p\psp{i}\alpha\psp{i-1}, \\
		&\alpha\ps{i}\beta\psp{i}\zeta\psp{i}\chi_p\psp{i} - \gamma\ps{i}\delta\ps{i}\xi\psp{i}\chi_q\psp{i},
		\alpha\psp{i}\beta\ps{i}\zeta\ps{i}\chi_p\ps{i} - \gamma\psp{i}\delta\psp{i}\xi\ps{i}\chi_q\ps{i}, \\
		&\zeta\ps{i}\chi_p\ps{i}\alpha\ps{i-1}\beta\psp{i-1} - \delta\ps{i}\xi\psp{i}\chi_q\psp{i}\gamma\psp{i-1},
		\zeta\psp{i}\chi_p\psp{i}\alpha\psp{i-1}\beta\ps{i-1} - \delta\psp{i}\xi\ps{i}\chi_q\ps{i}\gamma\ps{i-1}, \\
		&\xi\ps{i}\chi_q\ps{i}\gamma\ps{i-1}\delta\ps{i-1} - \xi\psp{i}\chi_q\psp{i}\gamma\psp{i-1}\delta\psp{i-1}, \\
		&\chi_p\ps{i}\alpha\ps{i-1}\beta\psp{i-1}\zeta\psp{i-1} - \chi_q\ps{i}\gamma\ps{i-1}\delta\ps{i-1}\xi\psp{i-1},
		\chi_p\psp{i}\alpha\psp{i-1}\beta\ps{i-1}\zeta\ps{i-1} - \chi_q\psp{i}\gamma\psp{i-1}\delta\psp{i-1}\xi\ps{i-1}, \\
		&\zeta\ps{i}\chi_q\ps{i}, 
		\xi\ps{i}\chi_p\ps{i}, 
		\chi_p\ps{i}\gamma\ps{i-1}, 
		\chi_q\ps{i}\alpha\ps{i-1}, \\
		&\zeta\psp{i}\chi_q\psp{i}, 
		\xi\psp{i}\chi_p\psp{i}, 
		\chi_p\psp{i}\gamma\psp{i-1}, 
		\chi_q\psp{i}\alpha\psp{i-1}
		 : i \in \integer\}.
	\end{align*}
	
	One can easily verify that $\wh A_R$ is obtained via the unfolding procedure from $A_R$. Thus, unfolding commutes with the repetitive algebra construction.  The folding action is similar to that before, with $g \wh i\ps{j} = \wh i\psp{j}$ (for all $ i \in \{ 2, 3, 5\}$), $g \wh i\ps{j} = \wh i\ps{j}$ (for all $ i \in \{ 1, 4\}$) and $g\wh a\ps{j} = \wh a\psp{j}$ (for all $ a \in \{\alpha,\beta,\gamma,\delta,\zeta,\xi,\chi_p,\chi_q\}$). Under the unfolding functor $U_R \colon \fin A_R \rightarrow \fin \wh A_R$, we have $U_R(P(i\ps{j})) = P(\wh i\ps{j}) \oplus P(g\wh i\ps{j})$, as per Lemma~\ref{lem:URProj}.
	
	The essence of the proof on the closure under derived equivalence (Section~\ref{sec:DerEqProof}) is that tilting commutes with unfolding. We will show an example of that here. Consider the object
	\begin{equation*}
		M \cong P(1\ps{0}) \oplus P(2\ps{0}) \oplus P(3\ps{0}) \oplus P(4\ps{0}) \oplus M((\xi\ps{0})\inv\crs_4\ps{0}\xi\ps{0}(\zeta\ps{0})\inv\delta\ps{0}\crs_4\ps{0}\xi\ps{0}) \in \fin A_R
	\end{equation*}
	It is not difficult to show that this is the image of a tilting object in $D^b(A)$ under the functor given in the proof of Theorem~\ref{thm:DerClosed}. The image of this object under the unfolding functor is
	\begin{align*}
		U_R(M) \cong P(\wh 1\ps{0}) \oplus P(\wh 1\ps{0}) \oplus P(\wh 2\ps{0}) \oplus P(\wh 2\psp{0}) \oplus P(\wh 3\ps{0}) \oplus P(\wh 3\psp{0}) \oplus P(\wh 4\ps{0}) \oplus P(\wh 4\ps{0}) \oplus ~{} \\ 
		M((\wh\xi\ps{0})\inv\wh\xi\psp{0}(\wh\zeta\psp{0})\inv\wh\delta\psp{0}\wh\xi\ps{0}) \oplus M((\wh\xi\psp{0})\inv\wh\xi\ps{0}(\wh\zeta\ps{0})\inv\wh\delta\ps{0}\wh\xi\psp{0})
	\end{align*}
	From this, we consider the reduced module
	\begin{align*}
		\wh M^\ast \cong \; &P(\wh 1\ps{0}) \oplus P(\wh 2\ps{0}) \oplus P(\wh 2\psp{0}) \oplus P(\wh 3\ps{0}) \oplus P(\wh 3\psp{0}) \oplus P(\wh 4\ps{0}) \oplus ~{} \\
		&M((\wh\xi\ps{0})\inv\wh\xi\psp{0}(\wh\zeta\psp{0})\inv\wh\delta\psp{0}\wh\xi\ps{0}) \oplus M((\wh\xi\psp{0})\inv\wh\xi\ps{0}(\wh\zeta\ps{0})\inv\wh\delta\ps{0}\wh\xi\psp{0})
	\end{align*}
	which is obtained by removing duplicate direct summands.
	
	The principle of the proof is that the endomorphism algebra of $\wh M^\ast$ has a $\integer_2$-action, and that the quotient with respect to this action yields the endomorphism algebra of $M$ with the crease maps removed. In addition, crease maps exist only between the direct summands of $M$ corresponding to the direct summands of $\wh M^\ast$ that are fixed under the action of $\integer_2$. In other words, the endomorphism algebra of $M$ is given precisely by the folding procedure (Section~\ref{sec:Folding}) on the endomorphism algebra of $\wh M^\ast$. Since the endomorphism algebra of $\wh M^\ast$ is known to be (unfolded) gentle, this shows that the endomorphism algebra of $M$ is folded gentle. Such statements are obvious in this particular example. Here, $B=(\End_{A_R}(M))^{\text{op}}$ is given by the quiver
	\begin{equation*}
		\begin{tikzpicture}
			\draw (-1.2,0.7) node {1};
			\draw (-1.2,-0.7) node {2};
			\draw (0,0) node {3};
			\draw (1.2,0.7) node {4};
			\draw (1.2,-0.7) node {5};
			\draw [->](-0.9,0.5) -- (-0.2,0.1);

			\draw [->](-0.9,-0.7) -- (0.9,-0.7);
			\draw [->](0.9,-0.6) -- (0.1,-0.2);
			\draw [<-](1,-0.4) -- (0.2,0);
			\draw [->](-1.2,-0.4) -- (-1.2,0.4);
			\draw [->](1.2,-0.4) -- (1.2,0.4);
			\draw [->](-1.3,1) .. controls (-1.4,1.2) and (-1.4,1.4) .. (-1.2,1.4) .. controls (-1,1.4) and (-1,1.2) .. (-1.1,1);
			\draw [->](1.1,1) .. controls (1,1.2) and (1,1.4) .. (1.2,1.4) .. controls (1.4,1.4) and (1.4,1.2) .. (1.3,1);
			\draw (-2.1,0.2) node {\footnotesize $\widetilde Q\colon$};
			\draw (-0.7,1.2) node {\footnotesize $\eta_1$};
			\draw (0.7,1.2) node {\footnotesize $\eta_4$};
			\draw (-1.4,-0.1) node {\footnotesize $\alpha$};
			\draw (-0.4,0.5) node {\footnotesize $\beta$};
			\draw (0,-1) node {\footnotesize $\gamma$};
			\draw (0.3,-0.5) node {\footnotesize $\zeta$};
			\draw (0.7,0) node {\footnotesize $\delta$};
			\draw (1.4,0) node {\footnotesize $\xi$};
		\end{tikzpicture}
	\end{equation*}
	bound by the relations of the set
	\begin{equation*}
		\wt Z = \{\alpha\beta, \beta\delta,\delta\zeta, \gamma\xi, \eta_1^2 + \stp_1, \eta_4^2 + \eta_4 + 2\stp_4 \}.
	\end{equation*}
	On the other hand, the algebra $\wh B = (\End_{\wh A_R}(\wh M^\ast))^{\text{op}}$ is given by the quiver
	\begin{equation*}
		\begin{tikzpicture}
			\draw (-1.2,0) node {$\widehat 1$};
			\draw (-1.2,-1.6) node {$\widehat 2'$};
			\draw (0,-0.8) node {$\widehat 3'$};
			\draw (1.2,0) node {$\widehat 4$};
			\draw (1.2,-1.6) node {$\widehat 5'$};
			\draw [->](-1,-0.2) -- (-0.2,-0.7);

			\draw [->](-0.9,-1.6) -- (0.9,-1.6);
			\draw [->](0.9,-1.5) -- (0.1,-1);
			\draw [<-](1,-1.3) -- (0.2,-0.8);
			\draw [->](-1.2,-1.3) -- (-1.2,-0.3);
			\draw [->](1.2,-1.3) -- (1.2,-0.3);
			\draw (-2.1,0) node {\footnotesize $\widehat{\widetilde Q}\colon$};
			\draw (-1.4,-0.8) node {\footnotesize $\widehat\alpha'$};
			\draw (-0.7,-0.7) node {\footnotesize $\widehat\beta'$};
			\draw (0,-1.9) node {\footnotesize $\widehat\gamma'$};
			\draw (0.2,-1.37) node {\footnotesize $\widehat\zeta'$};
			\draw (0.7,-0.8) node {\footnotesize $\widehat\delta'$};
			\draw (1.4,-0.8) node {\footnotesize $\widehat\xi'$};
			
			\draw (-1.2,1.6) node {$\widehat 2$};
			\draw (0,0.8) node {$\widehat 3$};
			\draw (1.2,1.6) node {$\widehat 5$};
			\draw [->](-1,0.2) -- (-0.2,0.7);
			\draw [->](-0.9,1.6) -- (0.9,1.6);
			\draw [->](0.9,1.5) -- (0.1,1);
			\draw [<-](1,1.3) -- (0.2,0.8);
			\draw [->](-1.2,1.3) -- (-1.2,0.3);
			\draw [->](1.2,1.3) -- (1.2,0.3);
			\draw (-1.4,0.8) node {\footnotesize $\widehat\alpha$};
			\draw (-0.7,0.6) node {\footnotesize $\widehat\beta$};
			\draw (0,1.9) node {\footnotesize $\widehat\gamma$};
			\draw (0.2,1.35) node {\footnotesize $\widehat\zeta$};
			\draw (0.7,0.7) node {\footnotesize $\widehat\delta$};
			\draw (1.4,0.8) node {\footnotesize $\widehat\xi$};
		\end{tikzpicture}
	\end{equation*}
	bound by the relations in the set
	\begin{equation*}
		\wh{\wt Z} = \{\wh\alpha\wh\beta, \wh\alpha'\wh\beta', \wh\beta\wh\delta, \wh\beta'\wh\delta',\wh\delta\wh\zeta, \wh\delta'\wh\zeta', \wh\gamma\wh\xi,\wh\gamma'\wh\xi'\}.
	\end{equation*}
	One can clearly see that $B$ is folded gentle, $\wh B$ is (unfolded) gentle, and that $\wh B$ is obtained from $B$ via the unfolding procedure (and conversely, $B$ is obtained from $\wh B$ via the folding procedure). In addition, the algebra $B$ is derived equivalent to $A$, just as $\wh B$ is derived equivalent to $\wh A$. That is, unfolding commutes with derived equivalence.
	\bibliography{Bibliography}{}
	\bibliographystyle{habbrv}
\end{document}